\documentclass[12pt]{amsbook}

\usepackage[inner=3cm,outer=2.5cm,bottom=3cm,top=3cm]{geometry}

\usepackage{setspace}
\usepackage[latin1]{inputenc}
\usepackage[T1]{fontenc}
\usepackage{mathrsfs} 
\usepackage{amsfonts}
\usepackage{amsmath}
\usepackage{amssymb}
\usepackage{amsthm}
\usepackage[english]{babel}
\usepackage{mathdots}
\usepackage{hyperref}
\usepackage{url}
\usepackage{enumerate}
\usepackage{color}
\usepackage{tikz}
\usepackage{lmodern}
\usepackage{epic}
\usepackage{epstopdf}

\newtheorem{thm}{Theorem}[section]
\newtheorem{lem}[thm]{Lemma}
\newtheorem{prop}[thm]{Proposition}
\newtheorem{cor}[thm]{Corollary}
\newtheorem{defn}[thm]{Definition}

\newtheorem{ques}[thm]{Question}

\theoremstyle{remark}
\newtheorem{ex}[thm]{Example}
\newtheorem{rem}[thm]{Remark}

\numberwithin{figure}{chapter}
\numberwithin{equation}{section}

\DeclareMathOperator{\Hom}{{\rm Hom}}

\newcommand{ \lk }{ \mbox{lk} }
\newcommand{ \rk }{ \mbox{rk} }
\newcommand{ \imm }{ \mbox{Imm} }
\newcommand{ \emb }{ \mbox{Emb} }

\def\C{\mathbb C}
\def\Q{\mathbb Q}
\def\R{\mathbb R}
\def\Z{\mathbb Z}

\newcommand{\calO}{{\mathcal O}}

\newcommand{\labelpar}{\label}
\title{On certain complex surface singularities}
\author{Gerg\H{o} Pint\'{e}r}
\begin{document}

%\maketitle

\thispagestyle{empty}

\begin{center} {\Large \textsc{E\"otv\"os Lor\'and University} \\
\textsc{Institute of Mathematics}} \end{center}

\vspace{5mm}

%\begin{figure}[!htb]
%\centering
%\includegraphics[scale=.2]{elte_cimer_szines.eps}
%\end{figure}

%\begin{center}
%\includegraphics[scale=0.05]{figures/nagy.png}
%\end{center}

\vspace{3mm}

\begin{center} {\Large Ph.D. thesis}  \end{center}

\vspace{2mm}

\begin{center}{\huge \textbf{On certain complex surface singularities}}\end{center}

%\begin{center}{\Large \textbf{Komplex projekt\'iv s\'ikg\"orb\'ek \'es \\ alacsony dimenzi\'os topol\'ogia}}\end{center}

\vspace{5mm}

\begin{center}{\Large Gerg\H{o} Pint\'{e}r} \end{center}

\vspace{5mm}

\begin{center} {\Large Advisor: Andr\'as N\'emethi
\\   Professor, Doctor of Sciences}
\end{center}

\vspace{5mm}

\begin{small}
\begin{center}
{\large
Doctoral School: Mathematics
\\ Director: Istv\'{a}n Farag\'{o}
\\ Professor, Doctor of the Hungarian Academy of Sciences \\}                                                   
\end{center}

\vspace{2mm}

\begin{center}
{\large
Doctoral Program: Pure Mathematics
\\ Director: Andr\'as Sz\H{u}cs
\\ Professor, corresponding member of the Hungarian Academy of Sciences \\}
\end{center}
\end{small}

\vspace{5mm}

\begin{center}
{\large
E\"otv\"os Lor\'and University and \\
Alfr\'ed R\'enyi Institute of Mathematics\\}
\end{center}

\vspace{5mm}

\begin{center} {\large Budapest} \end{center}

%\vspace{3mm}

\begin{center} {\large 2018} \end{center}
\newpage

\frontmatter

\chapter*{Acknowledgement}

%AndrÃ¡s NÃ©methi
%AndrÃ¡s SzÅ±cs
%TamÃ¡s Terpai
%JÃ³zsef BodnÃ¡r
%David Mond
%Juan JosÃ© Nuno-Ballesteros
%Guellirmo ...
%Gabriella Keszthelyi
%Bertalan PÃ©csi
%Kinga Farkas

First of all, I would like to thank my advisor, Andr\'{a}s N\'{e}methi
for the support and work of several years. He introduced me to his
research area `complex singularity theory'. His deep knowledge,
intuition and ambitious plans guided me in our joint work, he also helped me a lot to elaborate on the details. I am very
grateful to him for the careful reading and correction of this thesis. I
hope we can continue working together.

I am very grateful to my topology professor Andr\'{a}s Sz\H{u}cs for
the lot of help, and the conversations we had related to his articles. I always rely on the knowledge I gained from him when it comes to topological topics. 

I would like to thank Tam\'{a}s Terpai for the consultations. I always
asked him when I had to understand something in topology quickly but
thoroughly.

I would like to thank David Mond and Juan J. Nu\~{n}o Ballesteros for
the opportunity to visit them and for the insights they gave me to their research areas.
These visits and the consultations served as a main
inspiration to write this thesis.

Special thanks goes to my friend and colleague J\'{o}zsef Bodn\'{a}r
for his companionship and several inspiring discussions on various
topics in mathematics. J\'{o}zsi was always disposed to discuss the
problems which many times led back to the basics. The many hours we spent
together provided me deep understanding. Moreover I would like to
thank him for the lot of help regarding the form and content of this
thesis.

I would like to thank Andr\'{a}s Stipsicz, L\'{a}szl\'{o} Feh\'{e}r and R\'{o}bert Sz\H{o}ke
for the useful courses and consultations.

I would to thank my colleagues and friends the support and interest
while I was working on this thesis. I am very grateful to Gabriella
Keszthelyi for her companionship and her pieces of advice. I would like to
thank Bertalan P\'{e}csi for the corrections and pieces of advice regarding
this thesis, and Bence Csajb\'{o}k and Szilveszter Kocsis for the many
discussions.

I have really enjoyed the company of my fellow students, friends and
colleagues at the R\'{e}nyi Institute: \'{A}d\'{a}m Gyenge, Tam\'{a}s
L\'{a}szl\'{o}, Benedek Skublics, Baldur Sigur{\dh}son, J\'{a}nos Nagy
and many others.

I would like to thank my friend Kinga Farkas, without whose support 
this thesis would not have been born.

Last but not least, I am very grateful to my family and friends for
their support and company.

The author was supported by the doctoral school of the Faculty of
Science of E\"{o}tv\"{o}s Lor\'{a}nd University between 2010 and 2013,
by the European Research Council grant \mbox{`LDTBud'} of Andr\'{a}s
Stipsicz in low-dimensional topology from 2013 to 2015 and by the
young researcher program of the Alfr\'{e}d R\'{e}nyi Institute of
Mathematics from 2015 to 2018.

\vspace{1cm}

Budapest, March 2018

\tableofcontents

\chapter*{Introduction}

\section*{Preliminaries and main results} This thesis contains the results of two joint papers with Andr\'{a}s N\'{e}methi \cite{NP, NP2} and the background related to them. 
% All the results of this thesis is joint work with Andr\'{a}s N\'{e}methi. 
N\'{e}methi's program aims to create bridges between different areas of the singularity theory, e.g. to compare the topological and analytic/algebraic invariants of complex singularities. In our case the bridge is created between immersion theory (differential topology) and two rather distinct areas of local complex singularity theory (`Milnor fibration package' and `analytic stability package'.)
%The $ \mathcal{C}^{\infty} $ category provides flexibility in contrast to the complex analytic rigidity.

The main objects of our study are holomorphic germs $ \Phi: ( \C^2, 0) \to ( \C^3, 0) $. These are triples of locally convergent complex power series  in two variables. We are mostly interested in two classes of them. The first class consists of the holomorphic germs singular only at the origin, i.e. $ \Phi|_{ \C^2 \setminus \{ 0 \}} $ is an immersion. The second class is a subset of the first one, it contains the so-called finitely $ \mathscr{A}$-determined (or $ \mathscr{A}$-finite) germs. %These are the germs which are determined by their $k$-th Taylor polynomial for some $k$ up to holomorphic change of coordinates in the source and in the target. 
By Mather--Gaffney criterion \cite{Wall} these are the germs whose restrictions $ \Phi|_{ \C^2 \setminus \{ 0 \}} $ are stable immersions.

If $\Phi: ( \C^2, 0) \to ( \C^3, 0) $ is singular only at the origin, one can associate with $\Phi$ an immersion from the $3$-sphere $S^3$ to the $5$-sphere $ S^5$ at the level of links. This immersion is the restriction of $ \Phi $ to a suitably chosen $3$-sphere around the origin of $ \C^2 $. If additionally $ \Phi $ is finitely $ \mathscr{A}$-determined, the associated immersion is stable, i.e. it has only regular simple and double values, with transverse intersection of the branches at each double value. 

In our case the associated immersion plays a similar role as the link in the case of isolated singularities. 
As it is a $ \mathcal{C}^{ \infty } $ map, it allows more flexible deformations than the holomorphic germ $ \Phi $. Up to regular homotopy (which means deformation through immersions) the immersions $ f: S^3 \looparrowright S^5 $ are completely classified by the integer valued invariant $ \Omega (f) $, the so-called Smale invariant of the immersion $f $. 
Stephen Smale published his invariant in 1959 \cite{smale}, and already two years later David Mumford in his seminal article \cite{mumford} asked for the analytic/algebraic characterization of the Smale invariant of the immersion associated with a holomorphic germ $ \Phi $.
%David Mumford asked in his seminal article \cite{mumford}, which regular homotopy classes can be represented by an immersion associated with some holomorphic germ $ \Phi $, and how can a certain regular homotopy class be identified via complex singularity theory? 
The main result of \cite{NP} provides a complete answer to the question of Mumford, see Chapter~\ref{ch:ass}. We identify the Smale invariant of the associated immersion with an analytic invariant of $ \Phi $, namely, with the number of the cross caps of a stabilization of $ \Phi $. This result implies various consequences both in singularity theory and immersion theory.

The image of a finitely $ \mathscr{A}$-determined germ $\Phi: ( \C^2, 0) \to ( \C^3, 0) $ can also be defined as the zero set of certain germ $ f: ( \C^3, 0) \to ( \C, 0) $, and provides a non-isolated hypersurface singularity $ (X, 0):= (f^{-1}(0), 0)= (\mbox{im} (\Phi), 0) \subset ( \C^3, 0)$.  Indeed, $df$ vanishes along the set of the double values of $ \Phi $. A very restricted class of non-isolated singularities occurs in this way. The normalization of $ (X, 0) $ is smooth, in fact, its normalization map is $ \Phi $, and all the transverse curves corresponding to the components of the singular locus of $ (X, 0) $ have type $ A_1$.

The Milnor fibre of $ (X, 0) = f^{-1} (0) \subset ( \C^3, 0) $ is defined as the $ \delta$-level set $ f^{-1} ( \delta ) $ of $f$ intersected with a ball $ B^{6}_{ \epsilon } $ with sufficiently small radius $ \epsilon $, $ 0 < \delta \ll \epsilon $. It plays a central role in the study of local singularities. In the case of isolated singularities, the boundary of its Milnor fibre is diffeomorphic to the link of $ (X, 0) $. Andr\'{a}s N\'{e}methi and \'{A}gnes Szil\'{a}rd in \cite{NSz} present a general algorithm, which provides the boundary of the Milnor fibre for any non-isolated hypersurface singularity $ (X, 0) \subset ( \C^3, 0) $, although it is rather technical and in concrete examples is rather computational. In \cite{NP2} we present an independent algorithm providing the Milnor fibre boundary for our restricted class of non-isolated singularities, see Chapter~\ref{ch:bound}. This algorithm uses directly the geometry of $ (X, 0) $. Namely, it produces the Milnor fibre boundary as a surgery of $ S^3 $ along the double point locus of the immersion associated with $ \Phi $. Technically the algorithm provides a plumbing graph of the Milnor fibre boundary by modifying a good embedded resolution graph of the double point locus of $ \Phi$.
% The result is presented as a plumbing graph of the boundary of the Milnor fibre.

\section*{Summary and organization of the thesis}

\subsection*{Immersions associated with holomorphic germs} The main purpose of Chapter~\ref{ch:germ} is to introduce Mond invariants $ C( \Phi ) $ and $ T ( \Phi ) $ of holomorphic germs $ \Phi: ( \C^2, 0) \to ( \C^3, 0)$. $ C ( \Phi )$ is the number of cross cap points, $ T ( \Phi ) $ is the number of triple values of a stabilization of $ \Phi $, but each can be calculated as the codimension of a suitable ideal of the local ring as well, without stabilizing $ \Phi $. Chapter~\ref{ch:germ} also serves as a collection of relations of these invariants with other concepts, which will appear in the subsequent chapters.

First we introduce the notion of germs and the finiteness property of them. We follow \cite{dejong, mond-ballesteros}. The associated smooth map $ S^{2n-1} \to S^{2p-1} $ is defined for finite germs $  ( \C^{ n}, 0) \to ( \C^{p}, 0) $, it slightly generalizes the notion of the associated immersion introduced in \cite{NP}.

The theory of stability and finite determinacy of germs were devised by e.g. H. Whitney, J. N. Mather and C. T. C. Wall \cite{Wall, matherII, matherIII, matherV, matherVI}. We introduce both concepts with respect to $ \mathscr{A}$-equivalence (left-right equivalence).
% A germ is called stable if any unfolding (local deformation) of it is equivalent with the trivial unfolding.
 The discussion of the stability also serves as a background for the singular Seifert surfaces, which are defined in Chapter~\ref{ch:imm} as stable smooth maps. 
 %We introduce the $ \mathscr{A}$ and $ \mathscr{A}_e$ codimensions of a germ, and 
 We review theorems of Mather and Gaffney characterizing the stability and finite $ \mathscr{A}$-determinacy in terms of the $ \mathscr{A}$ and $ \mathscr{A}_e$ codimensions of a germ. 
 %The stability of the Whitney-umbrella (cross cap) is calculated using Mather's theorem.
 The only stable multigerms of a map $ \C^2 \to \C^3 $ are regular simple points, regular double values with transverse intersection of the branches, regular triple values with regular intersection of the branches and simple Whitney umbrella (cross cap) points. The triple values and the Whitney umbrellas are isolated points.

The discussion of the Fitting ideals has two purposes. The defining equation $ f: ( \C^3, 0) \to ( \C, 0) $ of the image of a finite germ $ \Phi: ( \C^2, 0) \to ( \C^3, 0) $ can be determined using Fitting ideals. Moreover, by this process one can calculate the equations of the multiple point spaces of $ \Phi $ in the target as well. In particular, the number of the triple values of a stabilization can be determined with the help of Fitting ideals.
%Some examples are given to illustrate the calculation.

We introduce $ C( \Phi ) $ as the codimension of the ideal in the local ring $ \mathcal{O}_{( \C^2, 0)}$ generated by the determinants of the $2 \times 2$ minors of the Jacobian matrix of $ \Phi : ( \C^2, 0) \to ( \C^3, 0) $  \cite{Mond0, Mond2}. Similarly, $ T( \Phi ) $ is the codimension of the second Fitting ideal associated with $ \Phi $ in $ \mathcal{O}_{( \C^3, 0)}$. If $ \Phi $ is finitely $ \mathscr{A}$-determined, then both $ C( \Phi ) $ and $ T( \Phi ) $ are finite, and any stabilization of $ \Phi $ has $ C( \Phi ) $ cross caps and $ T( \Phi ) $ triple values.  The finiteness of $ C( \Phi ) $ is equivalent with the fact that $ \Phi $ is singular only at the origin, which also means that the associated map $ S^3 \to S^5 $ is an immersion. Finite $ \mathscr{A}$-determinacy is equivalent with the fact that the associated map $ S^3 \to S^5 $ is a stable immersion. Mond \cite{Mond2} introduced a third invariant $ N( \Phi ) $ for corank--$1$ germs $ \Phi : ( \C^2, 0) \to ( \C^3, 0) $, such that the finite $ \mathscr{A}$-determinacy of $ \Phi$ is equivalent with the finiteness of the three invariant $ C( \Phi ) $, $ T( \Phi ) $ and $ N( \Phi ) $. These invariants appear in the formulas expressing the image Milnor number (which is the second Betti number of the image of a stabilization of $ \Phi $), and in the formulas comparing the Milnor numbers of the four double point spaces as well \cite{Mondvan, mararmulti, nunodouble}. The structure of the double point spaces plays an important role in Chapter~\ref{ch:bound}. The end of Chapter~\ref{ch:germ} is a short review about a reinterpretation of Mond's invariants provided by W. L. Marar and J. J. Nu\~{n}o-Ballesteros \cite{slicing}. 
%This approach is closely related to our results explained in Chapter~\ref{ch:ass}, as both derive these invariants from the properties of a transverse object.

Chapter~\ref{ch:imm} provides an introduction to the Hirsch-Smale theory \cite{hirsch, smale}, which  transforms regular homotopy problems (differential topology) to homotopy theory (algebraic topology). The main purpose is to review the Hughes--Melvin definition of the integer valued Smale invariant of immersions $ S^3 \looparrowright \R^5 $ \cite{HM}, and the Ekholm--Sz\H{u}cs formulas expressing the Smale invariant in terms of the properties of singular Seifert surfaces \cite{ESz}.

Smale's theorem is the generalization of the Whitney--Graustein theorem about plane curve immersions $ S^1 \looparrowright \R^2 $, and it is a special case of the Hirsch theorem, in fact, a special case of the $h$-principle of Gromov. We review Hirsch theorem and the construction of the Smale invariant $ \Omega(f) $ of an immersion $f: S^n \looparrowright \R^q $. $ \Omega (f) $ is an element of an Abelian group depending on the dimensions $n$ and $q$, and two immersions $ S^n \looparrowright \R^q $ are regular homotopic if and only if their Smale invariants are equal. The possibility of the `sphere eversion' is a consequence of Smale's theorem as well.

J. F. Hughes and P. M. Melvin \cite{HM} proved that there are embeddings $ S^3 \hookrightarrow \R^5 $ wich are not regular homotopic to each other. They also express the Smale invariant of an embedding $ S^3 \hookrightarrow \R^5 $ with the signature of a Seifert surface, which is a $4$-manifold in $ \R^5 $ whose boundary is the image of the embedding. T. Ekholm and A. Sz\H{u}cs \cite{ESz} generalized this result to arbitrary immersions $ S^3 \looparrowright \R^5 $ using singular Seifert surfaces. One contribution of their formulas is the invariant $ L(f) $ of stable immersions $f:  S^3 \looparrowright \R^5 $ introduced by Ekholm in \cite{ekholm3}. It has several slightly different definitions \cite{ekholm3, ekholm4, ESz, saeki}. We review these definitions and the proof of their equivalence.

%In their formulas the invariant $ L(f) $ of a stable immersion $f:  S^3 \looparrowright \R^5 $ also appears. $ L(f)$ was intruduced by Ekholm in \cite{ekholm3}, and it has several slightly different definitions \cite{ekholm3, ekholm4, ESz, saeki}. We review these definitions and the proof of their equivalence.
% $ L$ is a Vassiliev type invariant of stable immersions: it is constant along regular homotopies through stable immersions, and it jumps by $ \pm 3 $ when a stable regular homotopy steps through an immersion with a triple value. 

The end of Chapter~\ref{ch:imm} is an outline of various related results, we mention here two of them. S. Kinjo defined immersions $ S^4 \looparrowright \R^4 $ associated with plumbing graphs of type $A$ and $D$. The point is that their Smale invariants agree with the Smale invariants of the immersions associated with the coverings $ \Phi: ( \C^2, 0) \to ( \C^3, 0) $ of the singularities of type $A$ and $D$, see below in the introduction. 
%However, we do not know the direct connection between the two constructions.
%There exist Smale invariant formulas using singular Seifert surfaces for immersions $  S^3 \looparrowright \R^4 $ as well \cite{hughes, EkTak}. M. Takase and S. Kinjo used these formulas to determine the Smale invariant of immersions associated with plumbing graphs. Kinjo's results for the graphs of type $A$ and $D$ agree with our results for the Smale invariants of the immersions associated with the coverings $ \Phi: ( \C^2, 0) \to ( \C^3, 0) $ of the singularities of type $A$ and $D$, see below in the introduction. However, we do not know the direct connection between the two constructions.
%We review some results about immersions $ M^3 \looparrowright \R^5 $ of arbitrary oriented $3$-manifold. The Ekholm--Sz\H{u}cs formulas are generalized for these immersions \cite{saeki}. 
%The total twist of stable immersions is a $ \Z_2$ valued cobordism invariant \cite{hughes}.
The $ \Z_2 $ valued cobordism invariant called `total twist' of immersions $ M^3 \looparrowright \R^5 $ \cite{hughes} is used later, in Chapter~\ref{ch:bound}, to conclude that $ C( \Phi ) $ and the number of the non-trivially covered double point curve components of a finitely determined germ $ \Phi : ( \C^2, 0) \to ( \C^3, 0) $ have the same parity.
 
%Using the $ \Z_2 $ valued cobordism invariant called `total twist' of immersions $ M^3 \looparrowright \R^5 $ \cite{hughes}, we conclude in Chapter~\ref{ch:bound} that $ C( \Phi ) $ and the number of the non-trivially covered double point curve components of a finitely determined germ $ \Phi : ( \C^2, 0) \to ( \C^3, 0) $ have the same parity.

%There are some results about the regular homotopy class of the embedded link of isolated complex hypersurface singularities in higher dimensions \cite{Esz2, katanaga}. We mention two examples, each of them has some edification for us: the first example shows that non-standard embedding of the sphere can be realized as a link, and in the case of the second example the regular homotopy type of the link is not well defined.

Chapter~\ref{ch:ass} includes our results published in \cite{NP}. The main theorem answers the question of Mumford mentioned above. Namely, $ \Omega ( \Phi|_{S^3} ) = - C ( \Phi ) $ holds for holomorphic germs $ \Phi : ( \C^2, 0) \to ( \C^3, 0) $ singular only at the origin, where $ \Omega (\Phi|_{S^3} ) $ is the Smale invariant of the immersion $ \Phi|_{S^3}:  S^3 \looparrowright S^5 $ associated with $ \Phi $. We refer to this result as the `main formula'.
%This formula can be considered as a complex singular Seifert surface formula for the Smale invariant: the stabilization of $ \Phi $ is a stable holomorphic map bounded by the immersion. 
There are several consequences of this result.

Using the main formula we can provide explicit (`algebraic') representatives of each regular homotopy class of immersions $ S^3 \looparrowright S^5 $. This answers a question of Smale \cite{smale}. In  contrast with the known $ \mathcal{C}^{ \infty} $ constructions, these realizations are very simple polynomial maps, and the computation of the Smale invariant via $ C ( \Phi ) $ is extremely simple.

The non-standard embeddings $ S^3 \hookrightarrow \R^5 $ of Hughes and Melvin \cite{HM} cannot be realized as the associated immersion of some $ \Phi: ( \C^2, 0) \to ( \C^3, 0) $. Although this fact follows from a deep result of Mumford \cite{mumford}, using our main formula we provide a new proof for it. Moreover, the `topological vanishing' $ \Omega ( \Phi|_{S^3} )=0 $ implies that $ \Phi $ is the regular germ via our formula.

We prove the formula $ \Omega ( \Phi|_{S^3} ) = - C ( \Phi ) $ in two steps. We introduce the newly defined `complex Smale invariant' $ \Omega_{ \C } ( \Phi ) $ of $ \Phi$. It turns out that $ \Omega_{ \C } ( \Phi )= C ( \Phi ) $ and $ \Omega_{ \C } ( \Phi ) = - \Omega ( \Phi|_{S^3} ) $. We use $ C( \Phi ) $ as the number of cross caps of a stabilization of $ \Phi $, and we do not use its algebraic definition as the codimension of the ramification ideal. However, as a by-product of the calculation we provide a new proof for the theorem of Mond declaring the equivalence of the two definitions of $ C $ in the case of corank--$1$ germs.

Note that the integer valued Smale invariant is well defined only up to sign. To determine the correct sign of the main formula, we fix generators of the corresponding infinite cyclic groups.
The Hughes--Melvin and Ekholm--Sz\H{u}cs formulas carry the sign ambiguity as well. We determine their correct sign with the help of the main formula and calculations of concrete examples. As a by-product of this procedure, we express the contributions of the Ekholm--Sz\H{u}cs formulas in terms of $ C( \Phi ) $ and $ T ( \Phi ) $ for a special singular Seifert surface, which is created from a holomorphic stabilization of $ \Phi $  by stabilizing the complex Whitney umbrella points in $ \mathcal{C}^{ \infty }$ sense. In particular, $ L ( \Phi|_{S^3} )= C ( \Phi ) - 3 T ( \Phi ) $ holds for finitely determined germs $ \Phi: ( \C^2, 0) \to ( \C^3, 0) $ and their associated stable immersions $ \Phi|_{S^3}: S^3 \looparrowright S^5 $.

It follows from our results that the analytic invariants $ C( \Phi ) $ and $ T ( \Phi ) $ are $ \mathcal{C}^{ \infty }$ invariants as well, moreover, $ C ( \Phi ) - 3 T ( \Phi ) $ is a topological invariant of $ \Phi $.

\subsection*{Boundary of the Milnor fibre}
Chapter~\ref{ch:iso} contains the definitions and some properties of the Milnor fibre and the resolution of surface singularities $ (X, 0) = (f^{-1} (0), 0) \subset ( \C^3, 0) $. The plumbing construction and the embedded resolution of plane curve singularities are also summarised in Chapter~\ref{ch:iso}.
 
The Milnor fibre of an isolated hypersurface surface singularity is well studied
and it is rather well understood. It has the homotopy type of a bouquet
of 2--spheres, it is an oriented  smooth 4--manifold whose boundary
is diffeomorphic with the link of the singular germ and also with the boundary
of any resolution of the germ \cite{MBook, Five, Egri}.
This boundary is a  plumbed
3--manifold and one can take as a plumbing graph any of the resolution graphs.
It is the basic bridge between the Milnor fibre and the resolution (both of them being complex analytic fillings of it), and this connection  produces several
nice formulas connecting the invariants of these fillings. Here primarily
we think about formulas of Laufer \cite{Laufer77b} or Durfee \cite{Du} and their generalizations, see e.g. \cite{Wa}.

For non--isolated hypersurface singularities in $(\C^3,0)$ the situation is
more complicated. First of all, the link of the germ is not smooth,
hence the boundary of the Milnor fibre cannot be isomorphic with it.
Moreover, a (any) resolution is in fact the resolution of the normalization
(which might contain considerable less information than what one needs in order to
recover the Milnor fibre $F$, or the Milnor fibre boundary $\partial F$),
see e.g. \cite{Si,NSz}. For example (see our case),
it can happen that the normalization is smooth, while $\partial F$ is rather complicated.
However, the boundary of the Milnor fibre is still a
plumbed 3--manifold, and one expects that its plumbing graph
 codifies considerable information about the germ.
 $\partial F$  can be obtained by surgery of two pieces: one of
them is the boundary of the resolution of the normalization, the other one is related to the transverse singularities associated with the singular curves of the hypersurface singularity \cite{Si,NSz,MP}. In particular, the boundary
of the Milnor fibre plays the same crucial role as in the isolated singularity case (in fact, it is the unique object in this case, which might fulfil this role):
it is the first step in the description of the Milnor fibre, and it is the bridge in the direction of the resolution and the transverse types of the components of the singular locus.

\cite{NSz} presents a
general algorithm, which provides the boundary of the Milnor fibre $\partial F$
for any non--isolated hypersurface singularity in
$(f^{-1}(0),0)\subset (\C^3,0)$.
However, this algorithm uses (some information from)
the embedded resolution of this pair, hence it is rather technical and
in concrete examples is rather computational. Therefore, for particular families
of singularities it is preferable to find more direct description of the plumbing graph
 of $\partial F$ directly from the peculiar intrinsic geometry of the germ.
For  several examples in the literature see e.g.
\cite{NSz} (homogeneous singularities, cylinders of plane curves,
$f=zf'(x,y)$, $f=f'(x^ay^b,z)$), \cite{Baldur} ($f=g(x,y) + zh(x,y)$);
or for other classes consult also
\cite{MP2} and \cite{dBM}.

Chapter~\ref{ch:bound} contains the results of \cite{NP2}. It provides an explicit construction producing
the plumbing graph for the boundary of the Milnor fiber of a non-isolated
hypersurface singularity $(X, 0)$ given by the image of a finitely determined complex
analytic map germ $ \Phi : ( \C^2, 0) \to ( \C^3, 0) $. One of
the main ingredients is the link $L = \{L_i \}_i $
in $S^3 $ of the reduced double point curve
%$ (D, 0) = \Phi^{â1} ( \Sigma, 0) \subset ( \C^2, 0) $ 
$ ( D, 0) = \Phi^{-1} ( \Sigma , 0) \subset ( \C^2, 0) $
with irreducible components $ \{ D_i \}_i$, where $( \Sigma , 0)$ is
the reduced singular locus. It is equipped with a pairing of the components
$L_i \leftrightarrow L_{ \sigma (i)} $
induced by the pairing of the double points. Then, the Milnor fiber
boundary is constructed as a surgery of $S^3$ along $L$ such that whenever $i \neq \sigma (i)$
the tubular neighbourhoods of the paired components have to be glued together,
while in the case $i = \sigma(i)$ a special $3$-manifold $Y := S^1 \times S^1 \times I/ \sim $ with torus
boundary is glued along $L_i $. 
%The manifold $Y$ (together with its plumbing presentation)
%is nicely presented with many details and extreme care in Section 3.
The description of the gluing maps uses the newly
defined invariants `vertical indices' associated with the irreducible components
of $ \Sigma $. Their relation with $ C ( \Phi ) $ is clarified whenever $ \Phi $ is a corank--$1$ germ and $ T ( \Phi ) $ vanishes. As a result, the explicit plumbing graph of the Milnor fiber boundary
is constructed from a good embedded resolution graph of $D$. The algorithm is
also illustrated on several examples.

%and The literature of singular analytic germs $\Phi:(\C^2,0)\to (\C^3,0)$ is huge with several deep and interesting results and invariants see e.g. the articles of
%Mond and Goryunov, Nu\~{n}o-Ballesteros \cite{Gor1,Gor2,Gor3,Mond1,Mond2,Mondwh, nunomulti, nunodouble, slicing}

\section*{The main examples}

\subsection*{Finitely determined germs} D. Mond's list of simple germs \cite[Table 1]{Mond2} contains finitely $ \mathscr{A}$-determined holomorphic germs $ \Phi: ( \C^2, 0) \to ( \C^3, 0) $ with $ \rk (d \Phi_0) = 1 $. These germs are organized in four families $ S_{k-1}$, $B_k$, $C_k$, $H_k$, and there are two sporadic elements, the Whitney umbrella (cross cap) and $ F_4$.

The associated immersions of $ S_{k-1} $ provide representatives of all regular homotopy classes with negative Smale invariant, see Section~\ref{s:ex}. 

The Whitney umbrella is stable as a holomorphic germ $ (\C^2, 0) \to ( \C^3, 0) $, see Example~\ref{ex:Whstab}, but not in the $ \mathcal{C}^{\infty} $ sense, as a real germ $ ( \R^4, 0) \to ( \R^6, 0) $. We present a $ \mathcal{C}^{\infty} $ stabilization of the complex Whitney umbrella in Section~\ref{s:calc} and we calculate directly the contributions of the Ekholm--Sz\H{u}cs formula.

In Section~\ref{s:milnex} we present a plumbing graph of the Milnor fibre boundary of the image of $ \Phi $ for all members of Mond's list. 

The germs $ \Phi $ of type $H_k$ are the only germs in the list with $ T ( \Phi ) \neq 0 $. In Section~\ref{s:Fitting} we calculate $ T ( \Phi ) $ and the equation of $ (\mbox{im} ( \Phi ), 0) \subset ( \C^3, 0) $ using Fitting ideals.

$ \Phi (s, t) = (s^2, t^2, s^3 + t^3+ st) $ is a corank--$2$ germ from \cite{marar}, that is, $ \rk (d \Phi_0) = 0 $. We present the Milnor fibre boundary of its image, see Section~\ref{s:milnex}.

\subsection*{Quotient singularities} The simple singularities $ (X, 0) \subset ( \C^3, 0) $ are the quotient singularities of type $A$-$D$-$E$, that is, $ (X, 0) \cong ( \C^2, 0) / G $ for a certain finite subgroup $ G \subset GL(2, \C )$. The covering map of a quotient singularity is a germ $ \Phi: ( \C^2, 0) \to ( \C^3, 0) $ whose image is $ (X, 0) $. The components of $ \Phi $ are the generators of the $G$-invariant algebra $ \mathcal{O}_{( \C^2, 0)}^G $ \cite{invariant}.

These germs $ \Phi $ are singular only at the origin, but they are not finitely $ \mathscr{A}$-determined germs, i.e. $ \Phi|_{ \C^2 \setminus \{ 0\} } $ is a nonstable immersion. In fact, if $ |G| >2 $, then every point of $ (X, 0) $ is at least triple value of $ \Phi $, and for $ |G|=2$ the transversality of the branches does not hold.

In Subsection~\ref{ex:stabilization} we present $ C( \Phi ) $ of these germs $ \Phi $, hence the Smale invariant of the associated immersions $ S^3 \to S^3/G \hookrightarrow S^5 $ follows from the main formula, see Section~\ref{s:ex}.

In Section~\ref{s:calc} we present a holomorphic stabilization of the covering germ $ \Phi $ of the $ A_1 $ singularity, and we calculate $ L ( \Phi|_{ S^3} ) $ directly. Note that its associated immersion $ S^3 \to \R \mathbb{P}^3 \hookrightarrow S^5 $ is regular homotopic with the immersions of the same structure (that is, a composition of a covering with an embedding) studied in several articles \cite{Mimm, EkTak, ekholm3, kinjo}, see the discussion in Subsection~\ref{ss:implumb}.

\section*{Notations and terminology}

\begin{itemize}

\item $ \Z $, $ \Q $, $ \R $, $ \C $: the set of integers, rational, real, complex numbers.

\item $ A \simeq B $ means that the (oriented) smooth manifold $ A $ and $ B $ are (oriented) diffeomorphic. 

\item $ A \cong B $ means that the algebraic structures or bundles $ A $ and $ B $ are isomorphic, or denotes the analytic equivalence of the analytic spaces/germs $ A $ and $ B$.

\item $ A \langle a, b \rangle $ means that the algebraic structure (module, vector space) $ A $ is generated by $ a $ and $b$. $\langle a, b \ | c, d \rangle $ denotes the group presented by generators $ a$ and $b $ and relations $c$ and $d$.

\item $ (a, b )  $ denotes the ideal in the ring/algebra $ R $ generated by the elements $ a, b \in R $ (or denotes the pair of arbitrary elements $ a $ and $b $ as well).

\item $ \dim A $ denotes the dimension of the vector space or manifold $ A$ over $ \R $. Dimension over another field $ \mathbb{F} $ is denoted by $ \dim_{ \mathbb{F} } $. The rank of a matrix/linear map $M$ is denoted by $ \rk (M) $. 

\item $ S^n $, resp. $ B^{n+1} $ denotes the (oriented) diffeomorphism type of the unit sphere, resp. the unit ball in $ \R^{n+1} $. The $2$-disc is denoted by $ D^2$. When it is important, we distinguish  the embedded and the abstract spheres and balls in the notation.

\item $[X, Y] $ denotes the set of homotopy classes of continuous maps between the topological spaces $ X$ and $Y$.

\item $ X \vee Y $ denotes the bouquet (wedge, one-point union) of the (pointed) topological spaces $ X $ and $Y$.

\item $ H_n (X, G ) $, resp. $ H^n ( X, G ) $ denotes the homology groups, resp. cohomology groups of the topological space $ X $ with coefficient group $ G $.

\item $ \pi_n (X) $ is the $n$-th homotopy group of the topological space $ X$.

\item $ d g $ ($dg_x$) denotes the differential/tangent map/Jacobian matrix of the function/map/germ $ g $ (at the point $x$).

\item We call a smooth (or analytic) germ $ g: (X, x) \to (Y, y ) $ regular (resp. singular), if the rank of $ dg_{x} $ is equal (resp. is less than) the minimum of the dimensions of $ X $ and $ Y$. 

\item $ GL (n, \mathbb{F} ) $ denotes the set of $n \times n $ matrices over $ \mathbb{F} $ with determinant $ \neq 0 $. In the real case $ GL^{+} (n, \mathbb{R} ) $ is the set of $n \times n $ matrices with determinant $ >0$. 

\item $ O(n) $ ($SO(n)$), resp. $U(n)$ ($SU(n) $) denotes the set of real ortogonal matrices (with determinant $ +1 $), resp. complex unitary matrices (with determinant $ +1$).

\item $ V_m( \mathbb{F}^q) $ (where $\mathbb{F}$ denotes $ \R$ or $ \C $) is the Stiefel manifold that consists of linearly independent $m$-frames of $ \mathbb{F}^q$.

\end{itemize}

\mainmatter

\chapter{Deformation of complex map germs}\label{ch:germ}

\section{Finite complex germs and their restrictions at the level of links}\labelpar{s:cone}

\subsection{Finite germs}  
The basic objects of the thesis are certain types of \emph{holomorphic map germs} $ \Phi: (\C^2, 0) \to ( \C^3, 0) $. The concept of the germ is an important tool in the study of the local behaviour of the maps. 

Two subsets $ X_1 $ and $X_2$ of the topological space $X$ \emph{have the same germ} at $x_0 \in X_1 \cap X_2$, if there is a neighbourhood $ U $ of $x_0$ such that $ U \cap X_1 = U \cap X_2 $. The equivalence classes of this equivalence relation are called the \emph{germs of spaces} at the point $x_0$ \cite[Definition 3.4.1.]{dejong}. The germs along a subset $ S_0$ can be defined similarly.

Let $ (X, x_0)$ and $( Y, y_0) $ be two germs of topological spaces. A \emph{germ of a continuous map} $ f: (X, x_0) \to (Y, y_0 ) $ is defined as an equivalence class of maps $ f: U \to W$, with $ f(x_0)=y_0$, where $U$ and $W$ are representatives of $ (X, x_0) $ and $ (Y, y_0)$ respectively. Two maps $f_1: U_1 \to W $ and $f_2: U_2 \to W $ are equivalent (they define the same germ) if they agree on an open neighbourhood $ V \subset U_1 \cap U_2 $ of $ p$ \cite[Definition 3.4.6.]{dejong}.
Changing $ x_0 $ to a finite subset $S_0 \subset X$ we get the notion of \emph{multi-germ}. \cite{mond-ballesteros, Wall}
%Consider the continuous maps betwen some neighbourhoods of $ x_0$ and $y_0$ wich map $x_0$ to $y_0$. Two maps $f_i: U_i \to V_i$, $f_i(x_0)=y_0$ ($i=1, 2$) \emph{define the same germ}, if there is a neighbourhood $ W \subset U_1 \cap U_2 $ of $x_0$ such that $ f_1|_W=f_2|_W $. The equivalence classes of this equivalence relation are called \emph{germs of maps} from $ (X, x_0) $ to $(Y, y_0)$. Changing $ x_0 $ to a finite subset $S_0 \subset X$ we get the notion of \emph{multi-germ}. \cite{mond-ballesteros, Wall}.

We also can define smooth (resp. analytic, holomorphic) germs of maps, as germs of smooth (resp. analytic, holomorphic) maps. We study holomorphic germs $ ( \C^n, 0) \to ( \C^p, 0) $ throughout this chapter, and germs of complex analytic spaces are studied in Chapter~\ref{ch:iso}.

The notion of germ is particularly interesting in the complex analytic category, because of uniqueness of analytic continuation: if $ U_1 $ and $ U_2 $ are open sets in $ \C^n $ with $U_1 \cap U_2 $ connected, and $f_i : U_i \to \C^p $ are complex analytic maps which coincide on some open $ V \subset U_1 \cap U_2$, they coincide on all of $ U_1 \cap U_2$ \cite{mond-ballesteros}. 

The holomorphic germs $ ( \C^n, 0) \to ( \C^p, 0) $ can be identified with the $p$-tuples of locally convergent power series in $n$ variables with constant terms $0 $, i.e. with the elements of $ \mathfrak{m}_{ ( \C^n, 0)} \cdot \mathcal{O}_{ (\C^n, 0)}^p $, where $ \mathfrak{m}_{ ( \C^n, 0)} $ is the unique maximal ideal in the local ring $ \mathcal{O}_{ (\C^n, 0)} $. 
%In fact, a holomorphic germ $ f: ( \C^n, 0) \to ( \C^p, 0) $ determines a small representative of itself: there is a neighborhood $U $ of  $ 0 \in \C^n $ such that the representatives of $ f$ defined on $U$ are equal to each other on $U$.

% However the concept of the germ is very natural in the holomorphic category, since a holomorphic germ determines a small representative of itself. Indeed, if the analytic funcions $ f_1: U \to Y $ and $f_2: U \to Y$ coincide in an open subset $ W $ of their connected domain $U$, then they coinside in $ U $. This fact corresponds to the rigidity of the complex geometry. According to this, the holopmorphic germs from $ ( \C^n, 0) $ to $ ( \C^p, 0) $ can be identified with the $p$-tuples of locally convergent power series in $n$ variables with constant terms $0 $, ie. with the elements of $ \mathfrak{m}_{ ( \C^3, 0)} ^p $, where $ \mathfrak{m}_{ ( \C^3, 0)} $ is the unique maximal ideal in the local ring $ \mathcal{O}_{ (\C^n, 0)} $.

A holomorphic germ $ \Phi: (\C^n, 0) \to ( \C^p, 0) $ is called \emph{finite}, if $ \Phi^{-1} (0)$ is a finite set, that is, $ \Phi^{-1} (0) = \{0 \} $ for a small enough representative of $ \Phi $ \cite[Definition 3.4.7., Theorem 3.4.24]{dejong}.

%  This assempion has several equvivalent versions in the holomorphic category.
% \begin{prop}\label{pr:fin} The following are eqivalent:
 
% (a) $ \Phi^{-1} (0)$ is a finite set.

% (b) $ \Phi^{-1} (0) = \{0 \} $ for a small enough representative of $ \Phi $.

% (c) $ 0$ is not an accumulation point of $ \Phi^{-1} (0) $.

 % (d) The complex dimension of $ \Phi^{-1} (0) $ is $0$.
% \end{prop}

% The equivalence of (a), (b),  and (c) is trivial by changing the domain of the representative, and the equivalence of (c) and (d) follows from the \emph{curve selection lemma}, \cite[Lemma 2.1.]{looijenga}, or \cite{MBook}. We will always choose a representative such that $ \Phi^{-1} (0) = \{0 \} $ holds. Thus we use this for definition.
% \begin{defn}\label{de:finite} A holomorphic germ $ \Phi: (\C^n, 0) \to ( \C^p, 0) $ ($n<p$) is finite if $ \Phi^{-1} (0) = \{0 \} $ for a small enough representative of $ \Phi $. 
%\end{defn}

Finiteness can be characterized by the local algebras. The germ $ \Phi: (\C^n, 0) \to ( \C^p, 0) $ induces an algebra homomorphism $ \Phi^*: \mathcal{O}_{ ( \C^p, 0) } \to \mathcal{O}_{ ( \C^n, 0) } $. In this way $ \mathcal{O}_{ ( \C^n, 0) } $ becomes a module over $ \mathcal{O}_{ ( \C^p, 0) } $, sometimes it is denoted by $ \Phi_* \mathcal{O}_{ ( \C^n, 0) } $ as well. 
%Then by \cite[Theorem 3.4.24.]{dejong} the following holds.
\begin{thm}[{\cite[Theorem 3.4.24.]{dejong}}]\label{th:dejongfin}
 The following facts are equivalent:
 
(a) $ \Phi $ is finite.

(b) $ \mathcal{O}_{ ( \C^n, 0) } $ is a finitely generated $ \mathcal{O}_{ ( \C^p, 0) } $-module.

(c) $ \mathcal{O}_{ ( \C^n, 0) }/(\Phi^*\mathfrak{m}_{ ( \C^p, 0) }) $ is a finite dimensional $ \C$-vector space.
\end{thm}
Note that $ (\Phi^*\mathfrak{m}_{ ( \C^p, 0) }) = \mathfrak{m}_{ ( \C^p, 0) } \cdot \mathcal{O}_{ ( \C^n, 0) } $ denotes the ideal generated by the coordinate functions of $ \Phi $.
%The relevance of the finiteness is shown by the following theorem, \cite{dejong, mond-ballesteros}.
%\begin{thm}[Finite mapping theorem]\label{th:fin} If $ \Phi: (\C^n, 0) \to ( \C^p, 0) $ ($n<p$) is a finite holomorphic germ, then the image of $ \Phi $ is an dimensional analytic subset of $ ( \C^p, 0)$ and $ \dim (\Phi(\C^n)) = n$.
%\end{thm}

% Finiteness is also equivalent with finite $ \mathscr{C}$-determinacy. Details can be found in \cite{mond-ballesteros, Wall}.

A finite germ $ \Phi: ( \C^n, 0) \to ( \C^p, 0) $ is \emph{generically $1$ to $1$} to its image, if there is a hypersurface $ (Y, 0)= f^{-1}(0) \subset ( \C^p, 0) $ (where $f: ( \C^p, 0) \to ( \C, 0)$) such that $ \Phi $ induces a bijection between $ \Phi^{-1} ( ( \C^p, 0) \setminus (Y, 0)) $ and $ \mbox{im} ( \Phi ) \setminus (Y, 0) $.

\subsection{Restriction of $ \Phi $ at the level of links}\labelpar{ss:link}
%The conic structure of several types of complex germs is the key tool in the study of the topology of these germs at the level of their links. For instance, a space germ with an isolated singular point is homoemorphic with the cone of its link, see \cite{looijenga, MBook}. For us the conic structure of a map germs is also significant. 
Here we present the discussion from \cite[Section 2.1.]{NP} slightly generalized. This topic is closely related to the link of complex analytic spacegerms, which is discussed in Chapter~\ref{ch:iso}.

If $(X,0)$ is a complex analytic germ with an isolated singularity $0\in X$ then its \emph{link}
 $K$ can be defined as follows. Set a real analytic map $\rho:X\to [0,\infty)$ such that
 $\rho^{-1}(0)=\{0\}$. Then, for $\epsilon>0$ sufficiently small, $K:=\rho ^{-1}(\epsilon)$
 is an oriented manifold, whose isotopy class (in $X\setminus \{0\}$)
 is independent of all the choices, cf.
  Lemma (2.2) and Proposition (2.5) of \cite{looijenga}.
 E.g., if $(X,0)$ is a subset of $(\C^N,0)$, then one can take the restriction of $\rho(z)=|z|^2$ (the norm of $z$). In this way, the link of $(\C^N,0)$ is the sphere $S^{2N-1}_\epsilon$.
 Nevertheless, the general definition  is very convenient
 even if $(X,0)=(\C^N,0)$.

Let  $ \Phi: ( \C^n, 0)  \to (\C^p, 0) $ be a finite holomorphic germ ($n<p$). Define $\rho:(\C^n,0)\to [0,\infty)$ by $\rho(z)=|\Phi(z)|^2$. Since  $\Phi^{-1}(0)=\{0\}$,  $\rho^{-1}(0)=\{0\}$ too.

\begin{lem}\label{cor:epsilon}
%Let $ S^5_{ \epsilon} \subset \C^3 $ be the standard sphere around the origin
%with radius $ \epsilon $.
There exists an $ \epsilon_0 > 0 $ sufficiently small such that
$\mathfrak{B}_\epsilon:=\Phi^{-1} ( \{z:|z|\leq \epsilon\} )$  is a
non-metric $ \mathcal{C}^{\infty}$ closed ball around the origin of $\C^n$ for
any $ \epsilon < \epsilon_0 $. Its boundary,
$\Phi^{-1}(S^{2p-1}_{\epsilon})$ is canonically diffeomorphic to $ S^{2n-1}$. In fact, for $\tilde{\epsilon}$ with
$0<\tilde{\epsilon}\ll \epsilon$, any standard metric sphere
$S^{2n-1}_{\tilde{\epsilon}}$ sits in $\mathfrak{B}_\epsilon$, and it is isotopic with
$ \Phi^{-1} (  S^{2p-1}_{ \epsilon} ) $ in  $\mathfrak{B}_\epsilon\setminus \{0\}$.
\end{lem}
%\begin{proof} This follows from with $ X= \C^2 $, $x=(0, 0) $, $ r(x, y) = | \Phi (x, y) |^2 %$, $ r' (x, y) = |x|^2 + |y|^2 $. Thus $ X_{ r = \epsilon^2 } = \Phi^{-1} (S^5_{\epsilon})
% $ and $ X_{ r' = \epsilon^2 } = S^3_{ \epsilon}  $. The assumptions of Proposition (2.5)
%are satisfied. The conclusion gives the desired diffeomorphism and the isotopy. \end{proof}
In the sequel $ \Phi^{-1} (  S^{2p-1}_{ \epsilon } ) $ and $S^{2n-1}_{\tilde{\epsilon}}\subset \C^n$
will be identified. When it is important to differentiate them we will use the notation
$ \mathfrak{S}=  \mathfrak{S}^{2n-1}:=  \Phi^{-1} (  S^{2p-1}_{ \epsilon} ) $.
We write also $S^{2n-1}=S^{2n-1}_{\tilde{\epsilon}}$ and $S^{2p-1}=S^{2p-1}_\epsilon$.

\begin{defn}\label{de:linkmap} The restriction $ \Phi |_{ \mathfrak{S} } : \mathfrak{S} \to S^{2p-1} $ is the (smooth) map associated with  $ \Phi $ at the level of links.
\end{defn} 

% By Milnor fibration theorem (see \cite{MBook, looijenga}) the link $K $ of a complex analytic isolated singularity $ (X,0)$ determines the topological type of $(X, 0)$. In fact, $(X, 0)$ is homeomorphic with the cone of $K$, also in the embedded sense, when $(X, 0) \subset ( \C^N, 0) $. This fact is generalized for arbitrary complex analytic set germs in \cite{BurgVer}. As it follows from Lemma~\ref{cor:epsilon}, a finite complex analytic map germ $ \Phi$ is topologically equivalent with the cone of $ \Phi|_{ \mathfrak{S}} $.

\section{Finite determinacy and stability} 

\subsection{Equvivalence of the germs} Several equivalence relations of the germs are studied in the literature. We refer mainly to Wall's survey paper \cite{Wall} and the book of Mond and Nu\~{n}o-Ballesteros \cite{mond-ballesteros} (which is still unpublished). The concept of stability and finite determinacy makes sense after fixing a certain equivalence relation. Geometrically $ \mathscr{A} $-equivalence is the most reasonable, that corresponds to the coordinate changes in the source and in the target of a germ. We use only $ \mathscr{A} $-equivalence in this thesis. However, the results and classification theorems  connected with $ \mathscr{A} $-equivalence require the study of some of the other equivalences as well.

Most of the equivalences come from certain group actions on the set of germs, the equivalence classes are the orbits of the action. The group $ \mathscr{L}$ (resp. $ \mathscr{R}$) consists of the germs of local automorphism of the target (resp. the source), and acts on the germs by composing from the left (resp. from the right). The corresponding equivalence relation is the left (resp. right) equivalence of the germs. The direct product of these groups is denoted by $ \mathscr{A}$, its left-right action induces $ \mathscr{A}$-equivalence. The type of the local automorphisms depends on the category of the germs: these are local homeomorphisms for continuous germs, local diffeomorphisms for smooth germs, local analytic equivalences, biholomorphisms for real analytic, holomorphic germs. 
%$ \mathscr{A}$-equivalence can be considered as a change of the coordinates in the source and in the target.
There are two other groups which are often studied, $ \mathscr{C} $ and $ \mathscr{K}$, the last one is contact equivalence. 

In this thesis we use only $ \mathscr{A}$-equivalence, but in several different categories. 
%For instance the statement from the end of Section~\ref{s:cone} can be stated more precisely as ''a finite complex analytic map germ $ \Phi$ is topologically $\mathscr{A}$-equivalent with the cone of $ \Phi|_{ \mathfrak{S}} $''. 
For the study of holomorphic germs we usually use complex analytic $\mathscr{A}$-equivalence, but some of our results can be interpreted in other categories as well: as it will turn out, some analytic invariants of the germs are invariant under topological (or smooth) $\mathscr{A}$-action as well, see Subsection~\ref{ss:cveg}. 
The flexibility to use the smooth category is provided by 
% The chance of the change of the category comes from 
the study of the link, since $ \Phi|_{ \mathfrak{S}} $ is a smooth map between smooth manifolds (spheres). 

The singular Seifert surfaces discussed in Subsection~\ref{s:ss} are defined as stable smooth maps. Thus we need $\mathscr{A}$-equivalence of global maps as well:
%, and the notion of left-right equivalence of global maps: 
that corresponds to the action of global homeomorphisms (diffeomorphism, biholomorphism) of the source and the target.  Moreover, the stability with respect to this equivalence (global $ \mathscr{A}$-equivalence) can be reduced to the local stability of the germs of the map, \cite{mond-ballesteros, matherV}.

The following general concepts can be found in \cite{Wall, mond-ballesteros}, here we present the main theorems to provide a stable background to our study. Most of them hold for $ \mathscr{L}$, $ \mathscr{R}$, $ \mathscr{C}$, $ \mathscr{K}$ equivalences too, in smooth and analytic categories as well. But they are meaningless in the continuous category, because there is no Taylor-expansion and the codimension of the orbits cannot be defined.

Note that most of the notions and invariants introduced in this chapter are $ \mathscr{A} $-invariants. That is, if $ \Phi_1 $ and $ \Phi_2 $ are $ \mathscr{A} $-equivalent germs (as complex analytic germs), then $ \Phi_1 $ is $ \mathscr{A} $-stable (resp. finitely $ \mathscr{A} $-determined) if and only if $ \Phi_2 $ is $ \mathscr{A} $-stable (resp. finitely $ \mathscr{A} $-determined), $ C ( \Phi_1) = C ( \Phi_2) $, 
$ T ( \Phi_1) = T ( \Phi_2) $, and similarly the invariants $ N$, $ \mu_I $, $ J$, $ \mu ( D^2/S_2) $ (and the Milnor number of other double point spaces) agree for $ \Phi_1 $ and $ \Phi_2 $.

\subsection{Unfolding and stability} $ \mathscr{A} $-stability means that any small perturbation of a germ is $ \mathscr{A} $-equivalent with the germ itself. To make this precise, we need the notion of unfoldings.

\begin{defn}

(a) An $r$-parameter unfolding of $ \Phi: ( \C^n, 0) \to ( \C^p, 0) $ is a germ 
\[ \tilde{ \Phi} : (\C^n \times \C^r, 0) \to ( \C^p \times \C^r, 0) \mbox{, }  \]
\[ \tilde{ \Phi} (u, v) = ( \tilde{\Phi}_v (u) , v) \mbox{,}
\] 
such that $ \tilde{\Phi}_0 = \Phi $.

(b) Two unfoldings $ \tilde{ \Phi}_1 $ and $ \tilde{ \Phi}_2 $ of $ \Phi $ are $ \mathscr{A}$-equivalent unfoldings, if there are  diffeomorphism germs $ \phi: (\C^n \times \C^r, 0) \to (\C^n \times \C^r, 0) $ and $ \psi: ( \C^p \times \C^r, 0) \to ( \C^p \times \C^r, 0) $, which are unfoldings of the identity of $(\C^n , 0) $ and $ ( \C^p , 0)$, and
\[ \tilde{ \Phi}_2 = \psi \circ \tilde{ \Phi}_1 \circ \phi \mbox{.}
\]

(c) An unfolding of $ \Phi $ is trivial if it is equivalent with the constant unfolding, $ \tilde{ \Phi} (u, v) = ( \Phi (u), v) $.

\end{defn}

If the unfolding $ \tilde{ \Phi} $ is fixed, then we will use the notation $ \Phi_v $ instead of $ \tilde{ \Phi }_v $ for the deformation of $ \Phi $ corresponding to the parameter value $v$.

An unfolding is \emph{versal} if it induces all the unfoldings up to $ \mathscr{A}$-equivalence:

\begin{defn}

(a) The pull-back of an $r$-parameter unfolding $ \tilde{ \Phi } $ of $ \Phi $ with a germ $ h: ( \C^q, 0) \to ( \C^r, 0) $ is the $q$-parameter unfolding $h^* ( \tilde{ \Phi}): ( \C^n \times \C^q, 0) \to ( \C^p \times \C^q, 0) $,
  $ h^* ( \tilde{ \Phi})(u, v)= (\tilde{\Phi}_{h(v)} (u), v) $.

(b)
An $r$-parameter unfolding $ \tilde{ \Phi } $ of $ \Phi $ is called versal unfolding, if any unfolding (with arbitrary number of parameters) is equivalent with $ h^* ( \tilde{ \Phi}) $ for a certain germ $h$.
\end{defn} 

\begin{defn}
A germ $ \Phi: ( \C^n, 0) \to ( \C^p, 0) $ is stable (more precisely, $ \mathscr{A}$-stable) if any unfolding of $ \Phi $ is trivial.
\end{defn}

The definitions of the unfolding and trivial unfolding generalize to multi-germs, replacing the base point $0$ in the source with a finite set $ S$. Then the stability of a multi-germ can be defined in the same way. Moreover the stability of the multi-germs can be reduced to the stability of its germs. For details we refer to \cite{mond-ballesteros}. Here we present only the simplest case, which is used many times in this thesis. First we need a definition.

%\begin{defn} The \emph{iso-singular locus} of a small enough representative $ \Phi: U \to V $ of a germ $ ( \C^n, 0) \to (\C^p, 0) $ ($n < p$) consists of the points $ y \in V$ such that the multi-germ of $ \Phi $ at $y$ is $ \mathscr{A} $-equivalent with the multi-germ of $ \Phi $ at $ 0$.
%\end{defn}
%The iso-singular locus is always a submanifold of $V$.
\begin{defn}
The subspaces $ E_i $ of the vector space $ E $ (of finite dimension)  \emph{meet in general position} (or \emph{have regular intersection}) if $ \Sigma_i \textrm{\emph{ codim}} E_i = \textrm{\emph{ codim}} \bigcap_i E_i $.
\end{defn}

Note that for two subspaces regular intersection means transversality.

\begin{prop}
A multigerm of regular germs (i.e., each branch is the germ of an immersion) is stable if and only if the intersection of the branches is regular.
\end{prop}

%\begin{thm}
%A multi-germ is stabe if and only if each germ of it is stable and the iso-singular loci of the germs meet in general position.
%\end{thm}

\begin{rem}\labelpar{re:globstab}
By Theorem~4 in \cite{matherV} (see also in \cite{mond-ballesteros}) the stability of \emph{proper} smooth global maps is equivalent with the stability of its multi-germs. This fact shows that we do not need to introduce the notion of `global stability'. We call a smooth map stable if all its multi-germs of are stable. 
\end{rem}

\subsection{Finite determinacy and codimension} 
%A germ is finitely $ \mathscr{A} $-determined (also-called finitely determined or $ \mathscr{A} $-finite), if its $k$-jet ($k$-th Taylor polinomial) determines it up to $ \mathscr{A} $-equivalence. More precisely:
\begin{defn}
A holomorphic germ $ \Phi: ( \C^n, 0)  \to (\C^p, 0) $ is finitely $ \mathscr{A} $-determined if there is an integer $k$ such that whenever the $k$-jet ($k$-th Taylor polynomial) of a germ equals with the $k$-jet of $ \Phi $, then the germ is $ \mathscr{A} $-equivalent with $ \Phi$.
\end{defn}

Next we define the $ \mathscr{A}$-codimension of $ \Phi$, this is the codimension of the $ \mathscr{A}$-orbit of $ \Phi $ in the space of the germs $( \C^n, 0) \to ( \C^p, 0) $. We define the $ \mathscr{A}_e$-codimension as well \cite[pg. 485.]{Wall}, \cite{mond-ballesteros}. Stability and finite determinacy can be characterized by these invariants.

Let $ \theta_{ ( \C^n, 0)} $ denote the set of the \emph{germs of vector fields} on $  ( \C^n, 0) $, these are the germs of holomorphic sections of the tangent bundle $ T \C^n \to \C^n $. $ \theta ( \Phi ) $ is the space of \emph{germ of vector fields along $ \Phi $}, these are the germs of holomorphic sections of the pull-back tangent bundle $ \Phi^* T \C^p \to \C^n $. Let $\theta_0(\Phi) \subset  \theta(\Phi)$ denote the vector fields along $ \Phi $ which are $0$ at the origin, clearly 
$ \theta_0(\Phi) = \mathfrak{m}_{ ( \C^n, 0)} \cdot \theta(\Phi) $. Note that $ \theta_{ ( \C^n, 0)} $, $ \theta(\Phi) $ and $ \theta_0(\Phi) $ are $ \mathcal{O}_{( \C^n, 0)} $-modules, in particular they are $ \C$-vector spaces. 

Furthermore, the following identifications hold by introducing local coordinates.
\[ \theta ( \Phi ) \cong \mathcal{O}_{ ( \C^n, 0)}^p \mbox{ , }
 \theta_{ ( \C^n, 0)} \cong \mathcal{O}_{ ( \C^n, 0)}^n \mbox{. }
\]
% \[ \theta_0 ( \Phi ) = \mathfrak{m}_{ ( \C^n, 0)} \cdot \mathcal{O}_{ ( \C^n, 0)}^p =
% \mathfrak{m}_{ ( \C^n, 0)}^p \mbox{. }
% \]

Define the $ \C$-linear maps
\[ t \Phi: \theta_{ ( \C^n, 0)} \to \theta ( \Phi ) \mbox{ , } 
t \Phi (w) = d \Phi \circ w \mbox{, and}
\]
\[ \omega \Phi: \theta_{ ( \C^p, 0)} \to \theta ( \Phi ) \mbox{ , } 
\omega \Phi (w) = w \circ \Phi \mbox{.}
\]
In local coordinates $ t \Phi $ is the multiplication by the Jacobian of $ \Phi $, while $ \omega \Phi $ is the substitution of the components of $ \Phi $.
%where $ \theta_{ ( \C^n, 0)} $ denotes the set of the germs of vector fields on $  ( \C^n, 0) $. 

The tangent space $T \mathscr{A} \Phi $ of the $ \mathscr{A}$-orbit of $ \Phi $ is defined as
\[  T \mathscr{A} \Phi = t \Phi ( \mathfrak{m}_{ ( \C^n, 0)} \cdot \theta_{ ( \C^n, 0)} ) +
\omega \Phi ( \mathfrak{m}_{ ( \C^n, 0)} \cdot \theta_{ ( \C^p, 0)})  \mbox{,}
\]
it is a subspace of $\theta_0 ( \Phi ) $ \cite[pg. 485.]{Wall}, \cite{mond-ballesteros}.
\begin{defn}
The $ \mathscr{A}$-codimension of $ \Phi $ is $ \dim_{ \C } ( \theta_0( \Phi ) / T \mathscr{A} \Phi ) $.
\end{defn}

The \emph{extended tangent space} $T \mathscr{A}_e \Phi $ is defined as
\[ T \mathscr{A}_e \Phi = t \Phi (  \theta_{ ( \C^n, 0)} ) +
\omega \Phi (  \theta_{ ( \C^p, 0)}) \mbox{,}
\]
it is a subspace of $\theta ( \Phi ) $ \cite{Wall}, \cite{mond-ballesteros}.
\begin{defn}
The $ \mathscr{A}_e$-codimension of $ \Phi $ is $ \dim_{ \C } ( \theta( \Phi ) / T \mathscr{A}_e \Phi ) $.
\end{defn}

To motivate these definitions note that $ \theta ( \Phi ) $ is equal to the `set of one-parameter infinitesimal deformations of $ \Phi $', that is 
\[ \theta (\Phi)= \left\{ \frac{\partial}{\partial t} \Phi_t|_{t=0} \ | \
(\Phi_t, t) \mbox{ is a one-parameter unfolding of }  \Phi_0 = \Phi \right\} \mbox{,}\]
% in other words these are the vector fields along $\Phi$. 
and $\theta_0(\Phi)$ contains the elements of $ \theta(\Phi) $ correspond to the unfoldings with $ \Phi_t (0) = 0$. $\theta_0(\Phi)$ can be considered as the tangent space of the space of the germs $( \C^n, 0) \to ( \C^p, 0) $ at $ \Phi $.
% those of them, which are $0$ at the origin.

$T \mathscr{A} \Phi $ consists of the vector fields along $ \Phi $ which correspond to trivial infinitesimal deformations and are $0$ at the origin, ie.
\[ T \mathscr{A} \Phi = 
\left\{ \frac{\partial}{\partial t} ( \psi_t \circ \Phi \circ \phi^{-1}_t )|_{t=0} \ | \
 \phi_t(0)=0 \mbox{ , } \psi_t(0)=0 
\right\} \mbox{,}
\] 
where $ \psi_t $ and $ \phi_t $ are one-parameter families of germs of biholomorphisms in the source and in the target with $ \phi_0 = \mbox{id}  $, $ \psi_0 = \mbox{id} $. Without the extra conditions $ \phi_t(0)=0 $, $ \psi_t(0)=0 $ the 
% definition of $  T \mathscr{A} \Phi $ provides the definition of
formula provides $ T \mathscr{A}_e \Phi $.
% $ T \mathscr{A} \Phi $ is a subspace of $ \theta_0 ( \Phi )$. 

The extended codimension is a more natural concept in many cases, see e.g. Theorem~\ref{th:gaf1} below. 
%Without the extra conditions $ \phi_t(0)=0 $, $ \psi_t(0)=0 $ the definition of $  T \mathscr{A} \Phi $ provides the definition of $ T \mathscr{A}_e \Phi $, the \emph{extended tangent space}. 
However $ T \mathscr{A}_e \Phi $ does not correspond to any group action or equivalence of the germs, since it allows the origin to move.

Stability and finite determinacy can be determined by the following characterization theorems by Mather, see Theorems~1.2., 3.4. and Proposition~4.5.2. in \cite{Wall}, or \cite{matherII, matherIII, mond-ballesteros}.

\begin{thm}\labelpar{th:gaf1}
Let $ \Phi: ( \C^n , 0) \to ( \C^p, 0) $ a holomorphic germ. Then the following are equivalent:

(a) $ \Phi $ is stable.

(b) The $ \mathscr{A}_e$-codimension of $ \Phi $ is $0$.
\end{thm}

Property (b) is called `infinitezimal stability'. By Theorem~\ref{th:gaf1} this is eqiuvalent with stability. In the sense of Remark~\ref{re:globstab} `global stability' is also equivalent with them.

\begin{thm}\labelpar{th:gaf2}
Let $ \Phi: ( \C^n , 0) \to ( \C^p, 0) $ be a holomorphic germ. Then the following are equivalent:

(a) $ \Phi $  is finitely $ \mathscr{A}$-determined.

(b) The $ \mathscr{A}_e$-codimension of $ \Phi $ is finite.

(c) The $ \mathscr{A}$-codimension of $ \Phi $ is finite.

(d) $ \Phi $ admits a versal unfolding.

\end{thm}

The minimal number of parameters of a versal unfolding is exactly the $ \mathscr{A}_e$-codimension of $ \Phi $. A versal unfolding with minimal numer of parameters is called \emph{miniversal} unfolding. This is unique up to $ \mathscr{A}$-equivalence of unfoldings. All the other versal unfoldings can be obtained as a trivial unfolding of the miniversal unfolding up to $ \mathscr{A}$-equivalence of unfoldings.

Finally, by Mather-Gaffney criterion \cite[Theorem 2.1.]{Wall} finite determinacy is equivalent with isolated instability, which means each multi-germ of $ \Phi|_{ \C^n \setminus \{ 0 \} } $ is stable in a sufficiantly small representative of $ \Phi $.

\begin{thm}[Mather-Gaffney criterion \cite{mararmulti, Wall, mond-ballesteros}]\labelpar{th:gaf3}
A finite holomorphic germ $ \Phi: ( \C^n , 0) \to ( \C^p, 0) $ ($ n < p $) is finitely determined if and only if $ \Phi $ has isolated instability at $ 0$.
%$ \Phi|_{ \C^n \setminus \{ 0 \} } $ is stable, that is, with a small enough representative of $ \Phi $ each multi-germ of $ \Phi|_{ \C^n \setminus \{ 0 \} } $ is stable.
\end{thm}
Note that there is a more general version of Mather-Gaffney criterion, which agrees with our version for finite germs mapping $ \Phi: ( \C^n , 0) \to ( \C^p, 0) $ with $ n < p $. Cf. \cite[Theorem 2.11.]{mararmulti}, \cite[Theorem 2.1.]{Wall} and \cite{mond-ballesteros}.
%Finiteness and the condition for the dimensions is needed because we stated here a simplified version of the theorem, see \cite[Theorem 2.11.]{mararmulti}. A more general version uses another concept of ''local stability'', which agrees with stability in our version, see \cite[Theorem 2.1.]{Wall} and \cite{mond-ballesteros}.

\begin{rem}
The concept of stability and finite determinacy can be defined similarly for $ \mathscr{L}$, $ \mathscr{R}$, $ \mathscr{C}$, $ \mathscr{K}$ equivalences too, and they determine important classes of germs. For example $ \mathscr{L}$-stable (resp. $ \mathscr{R}$-stable) germs are the germs of the immersions (resp. submersions), and a germ is finitely $ \mathscr{C}$-determined  if and only if it is finite. The codimension and the extended codimension can also be defined for the other groups, and the characterisation theorems \ref{th:gaf1}, \ref{th:gaf2}, \ref{th:gaf3} hold in the same form, if we replace $ \mathscr{A}$ with any of the other groups.

Furthermore, all the concepts and theorems \ref{th:gaf1}, \ref{th:gaf2} hold in smooth and real analytic category too. However  Theorem~\ref{th:gaf3} holds only for holomorphic germs. For real analytic germs the following criterion is useful.
%, see \cite[Proposition 1.7.]{Wall}.
\end{rem}

\begin{prop}[{\cite[Proposition 1.7.]{Wall}}]\labelpar{pr:real}
For real analytic germs the following are equivalent:

(a) The germ is finitely determined as a smooth germ.

(b) The germ is finitely determined as a real analytic germ.

(c) The complexification of the germ is finitely determined as a holomorphic germ.
\end{prop}

\subsection{Examples}\label{ss:exfindetstab}

\begin{ex}[$ \C^n \to \C$] \label{ex:findetcnc}
A germ $ \Phi: (\C^n, 0) \to ( \C, 0) $ is stable if and only if it is regular (submersion) or it is Morse type, i.e. equivalent with $ (x_i)_i \mapsto \Sigma_{i=1}^n x_i^2 $.

A germ  $ \Phi: (\C^n, 0) \to ( \C, 0) $ with isolated singularity at $0$ (which means $ \Phi $ is submersion outside the origin) is finitely determined. See \cite[Theorem 9.1.3.]{dejong} or \cite{mond-ballesteros}.

%A germ $ \Phi: (\C^n, 0) \to ( \C, 0) $ is finitely determined if and only if it has isolated singularity at $0$, that is, $ \Phi $ is submersion outside the origin.  \cite[Theorem 9.1.3.]{dejong} or \cite{mond-ballesteros}.

%Part (b) follows from (a) and the general version of Theorem~\ref{th:gaf3}, since the Morse germs have isolated singularities, thus they can be omitted  by the choice of a small enough representative of $ \Phi $. See \cite[Theorem 9.1.3.]{dejong} or \cite{mond-ballesteros}.
%In fact, \ref{th:gaf3} $ \Phi: (\C^n, 0) \to ( \C, 0) $ is finitely determined if and only if any point of $ \C \setminus \{ 0 \} $ is a regular value or the image of a Morse point, but Morse points can be omitted by the choice of a smaller representative.
\end{ex}

\begin{ex}[Whitney-umbrella]\label{ex:Whstab} The ($\mathscr{A}$-equivalence class of) the germ $ \Phi(s, t)= (s, t^2, st) $ is called \emph{Whitney umbrella} (or \emph{cross cap}, \emph{pinch point}), see Figure~\ref{fig:stabmulti}. Another often used form is 
$ \Psi (s, t) = (s, t^2, t(s+t^2))$, see Example~5.1. in \cite{slicing}, which is summarized in paragraph~\ref{ss:Slicing}, and $S_0$ from Mond's simple germ list in \cite{Mond2}. $ \Phi $ and $ \Psi $ are $ \mathscr{A}$-equivalent, since $ \Psi= \psi \circ \Phi \circ \phi $ holds with the germs of diffeomorphisms  $ \phi (s, t) = (s+t^2, t) $ and $ \psi (x, y, z) = (x-y, y, z) $.

The Whitney umbrella is stable as a holomorphic germ $ ( \C^2, 0) \to ( \C^3, 0) $, and also as real analytic or smooth germ $ ( \R^2, 0) \to (\R^3, 0) $, cf. \ref{th:gaf3}. Its stability is very important throughout the thesis, hence we present here a proof of this fact, by calculating the $ \mathscr{A}_e$-codimension of $ \Phi $ (and then referring to Theorem~\ref{th:gaf1}).

For an arbitrary germ of a vector field
\[w(x, y, z) = \left( \begin{array}{c} w_1(x, y, z) \\ w_2(x, y, z) \\ w_3(x, y, z) \\
\end{array} \right) 
\in \theta_{ ( \C^3, 0)} \mbox{,}
\]
one has
\[ \omega \Phi (w)(s, t) = \left( \begin{array}{c} w_1(s, t^2, st) \\ w_2(s, t^2, st) \\ w_3(s, t^2, st) \\
\end{array} \right) 
\in \theta ( \Phi ) \mbox{,}
\]
thus $ \omega \Phi (\theta_{ ( \C^3, 0)}) $ consists of germs of vector fields whose components are power series of $ s $, $ t^2 $ and $ st $.  It is clear that the only monomials which cannot appear in the components are the odd powers of $ t$. The Jacobian of $ \Phi $ is 
\[ d \Phi (s, t) = \left( \begin{array}{cc} 1 & 0 \\ 0 & 2t \\ t & s \\
\end{array} \right) \mbox{,}
\]
thus for a germ of a vector field
\[v(s, t) = \left( \begin{array}{c} v_1(s, t) \\ v_2 (s, t) \\
\end{array} \right) 
\in \theta_{ ( \C^2, 0)} \mbox{,}
\]
\[ t \Phi (v)(s, t) = d \Phi \cdot v = 
\left( \begin{array}{c} v_1 (s, t) \\ 2t v_2 (s, t) \\ t v_1 (s, t) + s v_2 (s, t) \\
\end{array} \right) 
\in \theta ( \Phi ) \mbox{.}
\]
Therefore $ t \Phi (\theta_{ ( \C^2, 0)}) $ contains the vector fields whose coordinate functions are odd powers of $t$. One can conclude that
\[ T \mathscr{A}_e \Phi = t \Phi (\theta_{ ( \C^2, 0)}) + \omega \Phi (\theta_{ ( \C^3, 0)}) =
\left( \begin{array}{c} \mathcal{O}_{ ( \C^2, 0)} \\ \mathcal{O}_{ ( \C^2, 0)} \\ \mathcal{O}_{ ( \C^2, 0)} \\
\end{array} \right) = \theta ( \Phi) \mbox{,}
\]
thus the $ \mathscr{A}_e$-codimension of $ \Phi $ is $0$, which means $ \Phi $ is stable.
\end{ex}

\begin{ex}[$ \C^2 \to \C^3 $]\label{ex:c23}
(a) A germ $ \Phi: (\C^2, 0) \to ( \C^3, 0) $ is stable if and only if it is 
\begin{enumerate}
\item an immersion (given by $ \Phi (s, t)= (s, t, 0) $ up to $ \mathscr{A}$-equivalence), or 
\item a Whitney umbrella.
\end{enumerate}
(b) The stable multi-germs of a map from $ \C^2 $ to $ \C^3 $ are the
\begin{enumerate}
\item regular single values (with $1$ preimage in which the germ of the map is an immersion),
\item regular double values with transverse intersection, 
%(with $2$ preimages in which the germs of the map are immersions, and the intersection of the two branches is transversal),
\item regular triple values with regular intersection,
\item single Whitney umbrella points.
\end{enumerate}
(c) A germ $ \Phi: (\C^2, 0) \to ( \C^3, 0) $ is finitely $ \mathscr{A} $-determined if and only if (a small enough representative of) $ \Phi |_{ \C^2 \setminus \{ 0 \}} $ is a stable immersion with only double values. That is, all the multi-germs of $ \Phi |_{ \C^2 \setminus \{ 0 \}} $ have type (1) or (2) from the list of the stable multi-germs in part (b).

Part (c) follows from Theorem~\ref{th:gaf3} and from the fact, that isolated stable multi-germs ((3) and (4) from the list in (b)) can be omitted with the choice of a small enough representative of $ \Phi $. 
\end{ex}

\begin{ex}[$\R^2 \to \R^3$]\label{ex:r23}
The stable germs and multigerms of the real analytic or smooth maps from $ \R^2 $ to $ \R^3 $ are formally the same as the complex ones. That is, part (a) and (b) of \ref{ex:c23} hold word for word if we change $ \C $ to $ \R$. This follows from the fact that the calculation of the $\mathscr{A}_e$-codimension is a formal algebraic procedure, cf. Example~\ref{ex:Whstab}.

However part (c) of Example~\ref{ex:c23} is not valid for real germs, only in one direction: a finitely $ \mathscr{A} $-determined real analytic or smooth germ 
$ \Phi: (\R^2, 0) \to ( \R^3, 0) $ is stable immersion outside the origin. This follows directly from Proposition~\ref{pr:real} and part (c) of Example~\ref{ex:c23}. On the other hand, $ \Phi $ is an immersion outside the origin if and only if the Jacobian ideal of $ \Phi $ is concentrated to the origin. It can happen that the real zero set of the Jacobian ideal consists only the origin, but its complex zero set is bigger: such a germ is an immersion outside the origin as a real germ, but not as a complex germ, hence it is not finitely determined. Finitely determined real germs from $ \R^2 $ to $ \R^3 $ are studied for example in \cite{nunodoodle}.
\end{ex}

\begin{center}
\begin{figure}
 \includegraphics[width=14cm]{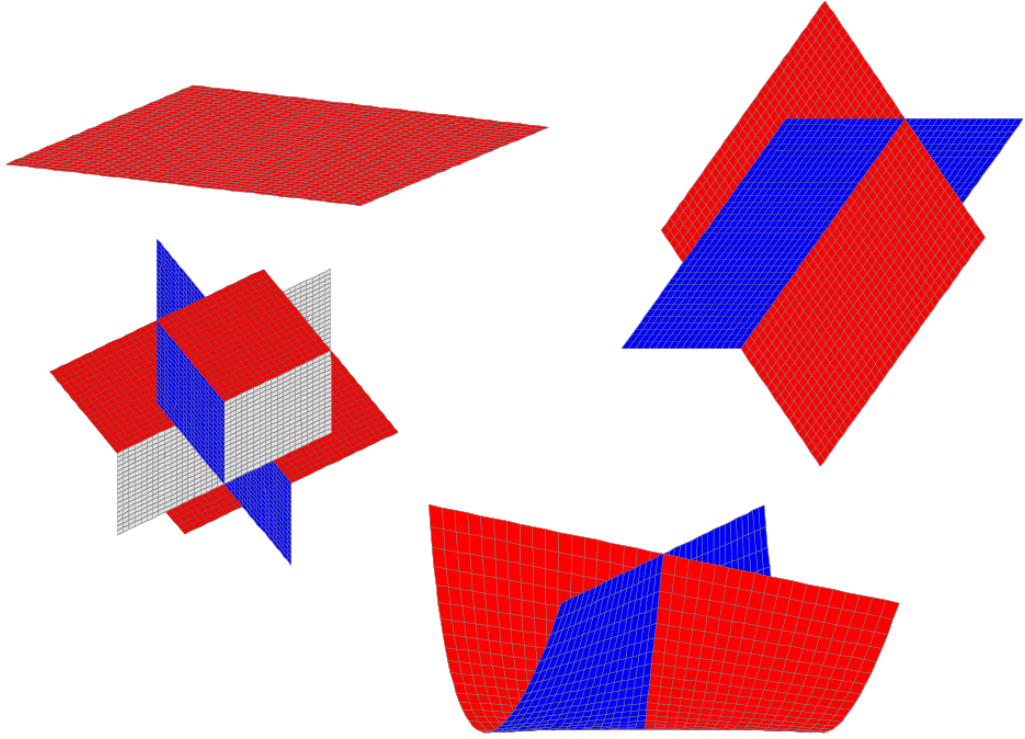}
 \caption{Stable multigerms of a maps $ \R^2 \to \R^3$: simple values, double values, isolated triple value and Whitney umbrella.}\label{fig:stabmulti}
\end{figure}
\end{center}

\begin{ex}[$ \R^4 \to \R^6 $]\label{ex:r46}
A holomorphic map from $ \C^2 $ to $ \C^3 $ can be regarded as a smooth (real analytic) map from $ \R^4 $ to $ \R^6 $. However the stability of the complex map does not imply the stability of the real map, because the complex Whitney umbrella is not stable as a germ from $ ( \R^4, 0) $ to $ ( \R^6, 0 ) $. The list of the stable multi-germs of a map from $ \R^4 $ to $ \R^6 $ consists of regular simple, double and triple values with regular intersection (cf. \ref{ex:c23}), and -- instead of the isolated  complex Whitney umbrella points -- it contains the \emph{fold points} (whose normal form is 
$ (s, t_1, t_2, u) \mapsto (s^2, st_1, st_2, t_1, t_2, u)$), which form a submanifold of dimension $1$.
% whose isosingular locus has dimension $1$. 
%See for details Section~\ref{} about the Morin singularities and singular Seifert surfaces, and
In Section~\ref{s:calc} we present a real stabilization of the complex Whitney umbrella.
\end{ex}

\section{Fitting ideals}\label{s:Fitting}

\subsection{Presentation matrix} The equation of the image and the multiple point spaces in the target can be determined by Fitting ideals. Here we present the general procedure and some examples. For more details and the properties of these ideals we refer to \cite{mond-ballesteros, mondfitting}.

Let $  \Phi: (\C^n, 0) \to ( \C^p, 0) $ be a finite holomorphic germ and $ p=n+1$. Then $ \Phi $ has a well-defined image, $ (\Phi (\C^n, 0) $ is an analytic set in $ ( \C^p, 0) $. The ideal of the functions vanishing on the image of $ \Phi $ can be determined by Fitting ideal method.

By Theorem~\ref{th:dejongfin} $ \mathcal{O}_{ ( \C^n, 0) } $ is a finitely generated $ \mathcal{O}_{ ( \C^p, 0) } $-module, see also \cite{mond-ballesteros, mondfitting}, we denote it by $\Phi_* \mathcal{O}_{ ( \C^n, 0) }$. A choice of a set of generators induces an epimorphism $ \mathcal{O}_{ ( \C^p, 0) }^M \to \Phi_* \mathcal{O}_{ ( \C^n, 0) } $. Its kernel is also finitely generated, because $ \mathcal{O}_{ ( \C^p, 0) }$ is Noetherian.
That can be summarised in the exact sequence
\[
  \mathcal{O}_{ ( \C^p, 0) }^N \to \mathcal{O}_{ ( \C^p, 0) }^M \to \Phi_* \mathcal{O}_{ ( \C^n, 0) } \to 0 \mbox{.}
\]
We call the homomorphism $ \lambda: \mathcal{O}_{ ( \C^p, 0) }^N \to \Phi_* \mathcal{O}_{ ( \C^p, 0) }^M $ the presentation of $ \mathcal{O}_{ ( \C^n, 0) } $, its matrix is the presentation matrix of $ \Phi $ (denoted also by $ \lambda $).
Each column of $ \lambda $ expresses a relation between the generators, and they together generate all the relations.

 $ \lambda $ may be chosen injective, and this forces $N=M$ ($ \lambda $ is a square matrix) \cite{mond-ballesteros}.

\begin{defn}
 The $i$-th Fitting ideal $ \mathcal{F}_i ( \Phi_* \mathcal{O}_{ ( \C^n, 0) } ) $ of $ \Phi $ is the ideal in $ \mathcal{O}_{ ( \C^p, 0) } $ generated by the determinants of the $ (N -i) \times ( N -i) $ minors of the presentation matrix $ \lambda $.
\end{defn}
It can be shown that the Fitting ideals do not depend on the choice of the presentation $ \lambda$.

$ \mathcal{F}_i ( \Phi_* \mathcal{O}_{ ( \C^n, 0) } ) $ provides an analytic structure of the $(i+1)$-fold multiple value set of $ \Phi $. 
\begin{thm}
 The zero set of $ \mathcal{F}_i (\Phi_* \mathcal{O}_{ ( \C^n, 0) }) $ consists of points in $ (\C^p, 0) $ with at least $i+1$ preimages counting with multiplicity.
\end{thm}
Especially, for $i=0$ the Fitting ideal provides an analytic structure of the image of $ \Phi $. 
\begin{thm}[\cite{tei}]
 The zero set of $ \mathcal{F}_0 (\Phi_* \mathcal{O}_{ ( \C^n, 0) }) $ is the image of $ \Phi $. In other words 
 $ \Phi (\C^n, 0)= \det ( \lambda )^{-1} (0) $.
\end{thm}
Although $ \det ( \lambda )$ is not necessarily a reduced equation, its geometrical meaning is very natural, as we will see in Example~\ref{ex:A}.

A presentation can be found by the following procedure. Let $ (x_i)_{i=1}^n $ denote the coordinates of the point $x$ in the source,  and let $ (y_i)_{i=1}^{n+1}$ denote the coordinates in the target. 
$ \Phi(x) = ( \Phi_1 (x), \Phi_2 (x), \dots , \Phi_{n+1} (x)) $, that is, $ \Phi_j = \Phi^* (y_j) \in \mathcal{O}_{ ( \C^n, 0) } $.
\begin{enumerate}
\item Define $ \bar{ \Phi }: ( \C^n, 0) \to ( \C^n, 0) $ as $ \bar{\Phi}(x) = ( \Phi_1 (x), \Phi_2 (x), \dots , \Phi_{n} (x)) $. Suppose that $ \bar{ \Phi } $ is finite. This always can be obtained by a linear change of the coordinates. %(that is, with the choice of another representative of the $ \mathscr{A} $-orbit of $ \Phi $).
\item Find generators $ g_1, \dots , g_N $ of $ \bar{ \Phi }_* \mathcal{O}_{ ( \C^n, 0) } $, i.e. $ (g_i)_{i=1}^N $ generates 
$ \mathcal{O}_{ ( \C^n, 0) } $ as a module over $ \mathcal{O}_{ ( \C^n, 0) } $ via $ \bar{\Phi} $. By Nakayama lemma that happens if and only if the classes of $ g_i $ form a basis of the $ \C$-vector space 
$ \mathcal{O}_{ ( \C^n, 0) }/( \bar{ \Phi }^* \mathfrak{m}_{(\C^n, 0)})
 = \mathcal{O}_{ ( \C^n, 0) }/ ( \Phi_1, \dots, \Phi_n ) $, which is finite dimensional since $ \bar{ \Phi } $ is finite. Note that one of the generators must be a unit, we take $ g_1=1$.
 \item Find the coefficients $ \tilde{\lambda}_{ij} \in \mathcal{O}_{ ( \C^n, 0) } $ (in the target) to express 
$ g_j \Phi_{n+1} = \Sigma_{i=1}^N \bar{\Phi}^* (\tilde{\lambda}_{ij}) g_i  $ for all $i, j=1, \dots, N $.
\item Define $ \lambda_{ij} = \tilde{\lambda}_{ij}- \delta_{ij} y_{n+1} $, where $ \delta_{ij} $ is $0$ for $ i \neq j$ and $1$ for $i=j$. Then $ \lambda = (\lambda_{ij})_{i, j=1}^N $ is a presentation matrix of $ \Phi_* \mathcal{O}_{ ( \C^n, 0) } $.
\end{enumerate}

\subsection{Examples}

\begin{ex}[$H_k$, $k\geq 1$]\label{ex:Hk} In this case $ \Phi(s, t) = (s, t^3, st+t^{3k-1} )$. This is a simple germ from Mond's list \cite{Mond2}, in fact, the first one, for which finding the equation of the image is not trivial. Here we provide it as the determinant of a presentation matrix. We use coordinates $ (s, t) $ in the source and $ (x, y, z) $ in the target.

$ \bar{\Phi}(s, t) = (s, t^3) $ is also finite, since $ \bar{\Phi}^{-1}(0, 0) = (0, 0) $. 
$ \mathcal{O}_2 / ( s, t^3 ) = \C^3 \langle 1, t, t^2 \rangle $, thus the module $ \bar{ \Phi }_* ( \mathcal{O}_2) $ is generated by $ (1, t, t^2) $. Then
\[ 1 \cdot (st+t^{3k-1}) = 0 \cdot 1 + s \cdot t + (t^3)^{k-1} \cdot t^2 =
0 \cdot 1 + x \cdot t + y^{k-1} \cdot t^2 \mbox{,}\]
\[ t \cdot (st+t^{3k-1}) = (t^3)^k \cdot 1 + 0 \cdot t + s \cdot t^2 =
y^k \cdot 1 + 0 \cdot t + x \cdot t^2
\mbox{,}\]
\[ t^2 \cdot (st+t^{3k-1}) = st^3 \cdot 1 + (t^3)^k \cdot t + 0 \cdot t^2 =
xy \cdot 1 + y^k \cdot t + 0 \cdot t^2 \mbox{.}\]
Thus 
\[ \lambda = 
\left( \begin{array}{ccc}
-z & y^k & xy \\
x & -z & y^k \\
y^{k-1} & x & -z \\
\end{array} \right)
\]
is a presentation matrix of $ \Phi_* \mathcal{O}_{ ( \C^2, 0) } $. 

The equation of the image is $ f(x, y, z) = \det ( \lambda ) = y^{3k-1} - z^3 + x^3 y + 3 zxy^k =0 $. $ \mathcal{F}_0 (\Phi_* \mathcal{O}_{ ( \C^2, 0) }) $ is the ideal generated by $ f $. The determinants of the $2 \times 2$ minors of $ \lambda $ generate the first Fitting ideal, $ \mathcal{F}_1 (\Phi_* \mathcal{O}_{ ( \C^2, 0) }) $. Its zero set consists of the double values of $ \Phi $. Finally, $ \mathcal{F}_2 (\Phi_* \mathcal{O}_{ ( \C^2, 0) }) $ is generated by the $ 1 \times 1 $ minors of $ \lambda $, thus 
$ \mathcal{F}_2 (\Phi_* \mathcal{O}_{ ( \C^2, 0) }) = ( x, z , y^{k-1} ) $. 
Its zero set is a fat point: the origin with multiplicity 
$ \dim_{ \C } ( \mathcal{O}_{ ( \C^3, 0) }) / \mathcal{F}_2 (\Phi_* \mathcal{O}_{ ( \C^2, 0) })) = k-1$. This corresponds to the fact that any stable deformation of $ \Phi $ has $ k-1$ triple values, cf. Subsection~\ref{ss:CTN}.
\end{ex}

\begin{ex}[Singularities of type $ A $]\label{ex:A} These are quotient  singularities
 of the form $ (X,0) = (\C^2, 0)/ \Z_k $, where $ \Z_k=\{\xi\in \C\,|\, \xi^k=1\}$
  denotes the cyclic group of order $k$, and the action is $ \xi* (s, t) = ( \xi s, \xi^{-1} t )$
  for $ \xi \in \Z_k $. $(X,0)$ is the image of a map $\Phi$,
 whose components  are the generators of the invariant algebra $\mathcal{O}_{ ( \C^2, 0) }^{\Z_k}$,
  see \cite[page 95]{invariant}, namely $ \Phi (s, t) = (s^k, t^k, st ) $. For $k=2$
  \[ \lambda = 
\left( \begin{array}{cccc}
-z & 0 & 0 & xy \\
0 & -z & y & 0 \\
0 & x & -z & 0 \\
1 & 0 & 0 & -z \\
\end{array} \right)
\]
is a presentation matrix of $ \Phi_* \mathcal{O}_{ ( \C^2, 0) } $. $ f(x, y, z) = \det ( \lambda ) = (z^2-xy)^2 $ is a not reduced equation of the image of $ \Phi $: it reflects to the fact that $ \Phi $ is a double cover of its image outside the origin. $ \mathcal{F}_1 (\Phi_* \mathcal{O}_{ ( \C^2, 0) }) $, the ideal generated by the determinants of the $ 3 \times 3 $ minors of $ \lambda $ equals to the ideal generated by $ z^2 - xy$, corresponding to the fact that the whole image of $ \Phi $ consists of double values.  
$ \mathcal{F}_2 (\Phi_* \mathcal{O}_{ ( \C^2, 0) }) = ( x, y, z ) = \mathfrak{m}_{ ( \C^3, 0) } $ is a codimension $ 1 $ ideal in $ \mathcal{O}_{ ( \C^3, 0) }$, that means any stable deformation of $ \Phi $ has $ 1$ triple value, see \ref{ss:CTN} and Figure~\ref{fig:A1fig}.

For arbitrary $ k$ the equation given by $ \det ( \lambda) $ is the $k$-th power a reduced germ, reflecting to the $k$-fold covering outside the origin.
\end{ex}

\begin{ex}[Cuspidal edge] Consider $ \Phi(s, t) = (s, t^2, t^3) $. The presentation  matrix is
\[ \lambda = 
\left( \begin{array}{cc}
-z & y^2  \\
y & -z \\
\end{array} \right) \mbox{.}
\]
$ \mathcal{F}_1 (\Phi_* \mathcal{O}_{ ( \C^2, 0) }) =(y, z)$. However $ \Phi $ does not have ordinary double values, the cuspidal edge $ \{ y=0, \ z=0 \} $ lies in the closure of the double values of a stable deformation of $ \Phi $.
\end{ex}

\section{Invariants of a stabilization}

\subsection{Stabilization}\label{ss:stab} A useful method to find invariants of the $\mathscr{A}$-equivalent classes of complex germs is to consider a stabilization of the germ, and count the stable multigerms that appear. In a lucky situation their number does not depend on the choice of the stabilization and it can be calculated by algebraic methods, that is, without stabilizing the germ. This is the procedure how the invariants $ C$ and $T$ can be introduced. 

A $1$-parameter unfolding is a \emph{stabilization} (stable deformation) of the germ $ \Phi: ( \C^n , 0) \to ( \C^p, 0) $, if $ \Phi_v $ is stable for all parameter values $ v \neq 0 $. However, $ \Phi_v $ is not a germ: to make the definition correct, we have to fix a representative.

Consider a $1$-parameter unfolding $ \tilde{ \Phi} (u, v) = ( \Phi_v (u), v) $, and fix small enough neighbourhoods $ W \simeq B^{2p}_{\epsilon} \subset \C^p $ and $ V \simeq D^2_{ \delta} \subset \C $. Let $ U = \tilde{ \Phi}^{-1} (W \times V) $ and 
$ U_v = \{ u \in \C^n \ | \ \Phi(u, v) \in W \} $ for each $v \in V$. Then $ \tilde{ \Phi} $ is a stabilization of $ \Phi $, if for every parameter value $ v \in V $ each multi-germ of the map $ \Phi_v: U_v \to W $ is stable. For a fixed parameter value $v \in V$, we call the map $ \Phi_v $ also a stabilization of $ \Phi $ -- hopefully without any confusion.

Stabilization always exists, if $ (n, p) $ are `nice dimensions' in the sense of Mather \cite{matherVI, mond-ballesteros, Wall}, ie. if stable maps are dense (in the $\mathcal{C}^{ \infty}$ sense). Note that in these dimensions stable maps are sometimes called \emph{generic}.

Take $p=n+1$. If $ \Phi $ is finitely determined, then it admits a versal unfolding $ \tilde{\Phi} $ with $r$ parameters.  Thus each stabilization of $ \Phi $ is equivalent with $ h^* ( \tilde{\Phi} ) $ with some germ $ h: (\C, 0) \to ( \C^r, 0) $. With the minimal number of parameters (i.e. $r$ equals to the $ \mathscr{A}_e$-codimension of $ \Phi $) the versal unfolding in unique, up to isomorphism. 

Let us fix a small enough representative of the miniversal unfolding. The \emph{bifurcation set} consists of the points $ v $ in the parameter space $ ( \C^r, 0 ) $ for which $ \Phi_v $ is not stable. The bifurcation set is an analytic subvariety, see \cite{mond-ballesteros}, indeed proper subvariety if $(n, n+1)$ are nice dimensions, which implies that the complement of it is connected in $ ( \C^r, 0 ) $. Thus if $ \Phi_{v_0} $ and $ \Phi_{v_1} $ are stable, then there is real path $ v: [0, 1] \to \C^r $ connecting $v_0$ and $v_1$ such that $ \Phi_{v_t} $ is stable for all $t$. In addition, the projection to the parameter space produces a locally trivial fibration of the image of the versal unfolding $ \tilde{ \Phi} $ over the complement of the bifurcation set. The fibres are the images of the stabilizations corresponding to different parameter values. See for details \cite{Mondvan, mararmulti, Sivan}. 

The image of a stabilization is called the \emph{disentanglement} of $ \Phi $. 
%By Lemma~1.3. of \cite{Mondvan} one has:
\begin{prop}[{\cite[Lemma 1.3.]{Mondvan}}]\label{pr:topdis}
The topology of the disentanglement of a finitely $ \mathscr{A}$-determined germ $ \Phi: ( \C^n , 0) \to ( \C^{n+1}, 0) $ is independent of the chosen stabilization.
\end{prop}

This theorem implies the well-definedness of the number of isolated stable multi-germs in each type. Especially, in the $(2, 3) $ nice dimensions:
\begin{thm}\label{th:whtrip} Let $ \Phi: ( \C^2, 0) \to ( \C^3, 0) $ be a finitely determined germ. Then the number of Whitney umbrellas and triple values are independent of the stabilization of $ \Phi$.
\end{thm}

\subsection{The invariants $C$ and $T$}\label{ss:CTN} Here we summarize the main properties of Mond's invariants $C$ and $T$. For  proofs and details we refer to \cite{Mond0, Mond1, Mond2, mondfitting, mond-ballesteros}.

Let $ \Phi: ( \C^2, 0) \to ( \C^3, 0) $ be a holomorphic germ. 
%Let $ J( \Phi ) $ denote the ideal in $ \mathcal{O}_{ ( \C^2, 0)} $ generated by the $2 \times 2 $ minors of the Jacobian of $ \Phi $. 
Let $ M_j: Hom ( \C^2 , \C^{3} ) \to \C $
 denote the determinants of the three $ 2 \times 2 $ minors ($ j=1, 2, 3 $).
 Let $ J(\Phi) $ be the ideal of the local ring $ \mathcal{O}_{( \C^2, 0)} $
 generated by the elements $ M_j\circ d\Phi$, where $d\Phi$ is the Jacobian
 matrix of $ \Phi $.
$ J( \Phi ) $ is called the \emph{Jacobian ideal} or the \emph{ramification ideal} of $ \Phi $. Define
\begin{equation}
C ( \Phi )= \dim_{ \C } \frac{ \mathcal{O}_{ ( \C^2, 0)}}{J ( \Phi) } \mbox{ and }
T ( \Phi )= \dim_{ \C } \frac{ \mathcal{O}_{ ( \C^3, 0)}}{ \mathcal{F}_2 (\Phi_* \mathcal{O}_{ ( \C^2, 0) }) } \mbox{.}
\end{equation}

\begin{thm}\label{th:CT} Let $ \Phi: ( \C^2, 0) \to ( \C^3, 0) $ be a finitely determined germ. Then $ C( \Phi ) $ and $ T( \Phi ) $ are finite and any stabilization of 
$ \Phi$ has $ C( \Phi ) $ Whitney umbrella points and $ T ( \Phi ) $ triple values.
\end{thm}

The key tool for the proof of \ref{th:CT} is the so-called \emph{Cohen-Macaulay property} of the ideals $ J ( \Phi) $ and 
$ \mathcal{F}_2 (\Phi_* \mathcal{O}_{ ( \C^2, 0) }) $ in the case of finitely determined germs
$ ( \C^2, 0) \to ( \C^3, 0) $. Cohen-Macaulay property implies a good behaviour under deformation, in our cases that means the conservation of the codimensions. E.g. 
\[ \dim_{ \C } \frac{ \mathcal{O}_{ ( \C^2, 0)}}{J ( \Phi) } =
\Sigma_{x \in U_v} \dim_{ \C } \frac{ \mathcal{O}_{ ( \C^2, x)}}{J ( \Phi_{v, x} ) } \]
%\[ \dim \frac{ \mathcal{O}_{ ( \C^3, 0)}}{ \mathcal{F}_2 (\Phi_* \mathcal{O}_{ ( \C^2, 0) }) } =
%\Sigma_{ y \in W} \dim \frac{ \mathcal{O}_{ ( \C^3, y)}}{ \mathcal{F}_2 (\Phi_{v*} \mathcal{O}_{ ( \C^2, \Phi_v^{-1}(y)) }) }
%\mbox{,}\]
holds for a (not necessarily stable) perturbation $ \Phi_v: U_v \to W $ of $ \Phi $. Here $ \Phi_{v, x} $ denotes the germ of 
$ \Phi_v $ at $ x$. Then it is enough to check that 
$ \dim_{ \C } (\mathcal{O}_{ ( \C^2, x)}/ J ( \Phi_{v, x} )) = 1 $ if $ \Phi_{v, x} $ is $ \mathscr{A}$-equivalent with a Whitney-umbrella and $ 0 $ for regular germs, hence the sum counts the Whitney umbrellas of a stabilization. Similar discussion holds for $ \mathcal{F}_2 (\Phi_* \mathcal{O}_{ ( \C^2, 0) }) $. For details we refer to \cite{Mond2, mondfitting, mond-ballesteros}.

In Remark~\ref{re:other} we give an independent proof for the equality of $ C( \Phi) $ and the number of Whitney umbrella points of a stabilization in the case of corank--1 map germs, cf. Example~\ref{ex:cor1}.

In \cite{Mond2} Mond introduced a third invariant $ N( \Phi ) $ for corank--$1$ germs, which `measures, in some sense, the non-transverse self-intersection concentrated at the origin'. Geometric interpretation of $ N( \Phi ) $ is not yet known for the author of the present thesis. Later $ N$ was replaced with more intuitive invariants, see theorems \ref{th:N} and \ref{th:N2}. Here we do not present the definition of $ N $.

We summarize some equivalent characterizations of finite determinacy of germs $  ( \C^2, 0) \to ( \C^3, 0) $. See \cite{Mond2, nunodoodle}.
\begin{thm}\label{th:fin-stab} Let $ \Phi: ( \C^2, 0) \to ( \C^3, 0) $ be a holomorphic germ. Then the following are equivalent:

(a) $ \Phi $ is finitely $ \mathscr{A}$-determined.

(b) $ \Phi|_{ \C^2 \setminus \{0 \} } $ is a stable immersion, cf. Theorem~\ref{th:gaf2}, Example~\ref{ex:c23}.

(c) The associated map of spheres $ \Phi|_{ \mathfrak{S}^3}: \mathfrak{S}^3 \to S^5 $ (cf. Definition~\ref{de:linkmap}) is a stable immersion, that is, it has only  single values and double values with transverse intersection.

\end{thm}

\begin{thm} Let $ \Phi: ( \C^2, 0) \to ( \C^3, 0) $ be a corank--$1$ germ. Then $ \Phi $ is finitely $ \mathscr{A}$-determined if and only if
 $ C( \Phi ) $, $ T( \Phi ) $ and $ N( \Phi ) $ are finite.
\end{thm}

On the other hand the finiteness of any of the three invariants has also a geometrical interpretation. E.g. the following holds for $C$. (Recall that a stabilization is defined in a sufficiently small neighbourhood of the origin.)
\begin{thm}\label{th:Csum} Let $ \Phi: ( \C^2, 0) \to ( \C^3, 0) $ be a holomorphic germ. Then the following are equivalent:

(a) $ C ( \Phi ) $ is finite.

%(b) Any stabilization of $ \Phi $ has the same number of Whitney umbrella points.

%(c) Any stabilization of $ \Phi $ has $ C( \Phi ) $ Whitney umbrella points.

(b) $ \Phi $ is singular only at the origin, that is, 
$ \rk (d \Phi_x ) < 2 $ implies $ x=0$. 
%if the rank of $ d \Phi $ is less then $ 2$ in a point $x$, then $x$ is the origin.
In other words $ \Phi|_{ \C^2 \setminus \{0 \} } $ is an immersion.

(c) The associated map of spheres $ \Phi|_{ \mathfrak{S}^3}: \mathfrak{S}^3 \to S^5 $ is an immersion.

If these claimes hold, then any stabilization of $ \Phi $ has $ C( \Phi ) $ Whitney umbrella points. Cf. \cite{Mond1, Mond2}.

\end{thm}

If $\Phi^{-1}(0) \neq \{ 0 \}$, then $\Phi^{-1}(0)$ has positive dimension, and along $ \Phi^{-1}(0) $ the rank of $d\Phi $ is $<2$. Hence corollary below follows from part (a) and (d) of Theorem~\ref{th:Csum}. %and part (d) of Proposition~\ref{pr:fin}.
\begin{cor}
If $ C ( \Phi ) $ is finite, then $ \Phi $ is a finite germ.
\end{cor}
Germs with finite $ C $, or equivalently, germs singular only at the origin are our basic objects in Chapter~\ref{ch:ass}. %We refer to  a germ is \emph{singular only at the origin}.

\begin{rem}\label{rem:whhigh} $ C( \Phi ) $ can be defined also in higher dimensions, for germs $ \Phi: ( \C^n, 0) \to ( \C^{2n-1}, 0) $, where the higher dimensional Whitney umbrellas are isolated points. These germs are studied in \cite{jorge}.
\end{rem}

\begin{prop}\label{pr:trip}
If  $T(\Phi)$ is finite, then
any stabilization of $ \Phi $ has the same number of triple points, and this number is $T(\Phi)$ (cf. \cite{mondfitting,Mond1}).
\end{prop}

By Theorem~\ref{th:Csum}, in the case of finite germs the failure of the finiteness of $C( \Phi ) $ corresponds to the existense of a \emph{cuspidal edge}, that is, a line of singular points. In general any of the following three properties implies the failure of the finite determinacy of a germ $ \Phi: ( \C^2, 0) \to ( \C^3, 0)$ \cite{nunodoodle}: (a) there is a line of singular points, (b) there is a line of triple points, (c) there is a line of non-transverse self-intersection.

%Similarly, the failure of the finiteness of $T( \Phi ) $ corresponds to the existense of a line of triple values, and $ N( \Phi ) $ is infinite, if there is a line of non-transverse self intersection.

\subsection{Examples}\label{ex:stabilization}

\begin{ex}[Finitely determined corank--$1$ germs]\label{ex:cor1}
Let $ \Phi: ( \C^2, 0) \to ( \C^3, 0) $ be a finitely determined germ with $ \rk (d \Phi_0) = 1$. Such a germ is $ \mathscr{A} $-equivalent with a germ of the form $ \Phi(s, t)= (s, p(s, t), q(s, t))$ \cite{Mond0}. Then the Jacobian ideal is generated by $ \partial_t p $ and $ \partial_t q $, thus its codimension $ C ( \Phi ) $ is the intersection multiplicity of the plane curves  $ \partial_t p = 0 $ and $ \partial_t q = 0 $ at the origin. 

Several concrete examples can be found e.g. in the list of Mond contaning the simple germs \cite{Mond2}. For example the family 
$ S_{k-1} $ given by $ \Phi (s, t) = (s, t^2, t^3 + s^k t ) $ has $ C( \Phi ) = k $ Whitney umbrellas in a stabilization, see Figure~\ref{fig:s1fig} for $k=2$. For $k=1 $ the germ itself is $ \mathscr{A}$-equivalent with the Whitney umbrella.
\end{ex}

\begin{ex}[$\Sigma^{1,0}$ type germs]\label{ex:Sig10} Assume that $\Phi(s, t) = ( s, t^2, t d(s,t) ) $, where $d(s,t) = g(s, t^2) $ for  some germ $g$, such that $d(s,t)$  is not divisible by $t$. The ($ \mathscr{A}$-equivalence classes of) these germs are often labeled by the Boardman symbol $\Sigma^{1,0}$, this notation refers to the classification of the germs up to their $2$-jet, see \cite{nunodoodle, Mond0}.
%(although it is a little bit confusing, compared with the Morin singularities \ref{}. 
They are exactly the finitely determined corank--$1$ map germs with no triple points in their stabilization, that is, with 
$T ( \Phi )=0 $, see \cite{Mond0,nunodoodle} for details.
\end{ex}

\begin{ex}[A corank--$2$ germ]\label{ex:cor2} There are also finitely determined germs $ \Phi: ( \C^2, 0) \to ( \C^3, 0) $ with $ \rk (d \Phi_0) = 0$. Several examples can be found in \cite{marar}. 
A concrete one is
$ \Phi (s, t) = (s^2, t^2, s^3 + t^3+ st) $. For this germ $ C( \Phi ) = 3$ and $ T( \Phi ) = 1$.
\end{ex}

\begin{ex}[Cuspidal edge]\label{ex:cuspedge} Let $ \Phi(s, t) = (s, t^2, t^3 ) $ (cf. Example~\ref{ex:Sig10}). Then $ C( \Phi ) = \infty$, $ \Phi $ is not finitely determined, indeed it is not an immersion outside the origin.
\end{ex} 

\begin{ex}[Singularities of type $A$]\label{ex:A2}
As in Example~\ref{ex:A}, consider the $k$-fold cover of the $A_{k-1}$ singularity, $ \Phi (s, t) = (s^k, t^k, st ) $. Then 
$ J ( \Phi ) = ( s^k, t^k, s^{k-1} t^{k-1}) $ and $ C( \Phi ) = k^2-1 $, cf. Figure~\ref{fig:A1fig} for $k=2$. Outside the origin $ \Phi $ is an immersion, but not a stable immersion. In fact, the whole image of $ \Phi $ consists of at least $2$-fold multiple values, thus the intersection of the branches is not transverse. Thus $ \Phi $ is not finitely determined. For $ k=2 $, $T( \Phi ) =1 $, cf. Example~\ref{ex:A}. However $ T( \Phi )= \infty $ for $k \geq 3 $, since the whole image consists of at least $3$-fold multiple values.
\end{ex}

The covering maps of the quotient singularities of type $D$ and $E_6$, $ E_7$ and $E_8$ are also regular outside the origin, thus $ C( \Phi ) $ is finite for them. However they are not finitely determined, from the same reason that was explained for type $A$. These germs provide us interesting examples later in Chapter~\ref{ch:ass}. We present here a possible computation of $C ( \Phi ) $ for these germs.
%and the calculation is not trivial in these cases.

\begin{ex}[Singularities of type $ D $]\label{ex:D} These are the quotient singularities of
 form $  (\C^2, 0)/ D_n $ where $ D_n $ denotes the binary dihedral group, \cite[page 89]{invariant}.
$ \Phi (s, t) = ( s^2 t^2, s^{2n} + t^{2n} , st (s^{2n} - t^{2n})) $ \cite[page 95]{invariant}.
By a computation
$ J( \Phi) = ( st (s^{2n}-t^{2n}) , s^2 t^2 (s^{2n} + t^{2n}) , (s^{2n} - t^{2n})^2 - 4n s^{2n} t^{2n}) $.
In singularity theory the quotient is the $D_{n-2}$--singularity.

A possible
 computation of $ \dim_{ \C } \left( \mathcal{O}_{(\C^2, 0)} / J( \Phi) \right) $ is based on the following facts.

\begin{lem}\label{lem:inter}
(a) Take $ f_1 , f_2, h \in \mathcal{O}_{\C^2, 0} $ such that $ f_1 f_2$ and $h $ are coprimes.
Then one has the  following exact sequence:
\[
 0 \to \mathcal{O}_{(\C^2, 0)} /( f_2 , h )  \to  \mathcal{O}_{(\C^2, 0)}
 /( f_1 f_2 , h )  \to  \mathcal{O}_{(\C^2, 0)} /( f_1 , h ) \to 0 \mbox{ .}
\]

(b) Take $ f_1 , f_2, g, h \in \mathcal{O}_{(\C^2, 0)} $ such that the ideal
$(f_1f_2, g,h) $ has finite codimension, and  $ h = f_1 h' $ for some $h'\in \mathcal{O}_{(\C^2, 0)} $.
Then one has the following exact sequence:
\[
 0 \to\mathcal{O}_{\C^2, 0} /( f_2 , g, h' )  \to  \mathcal{O}_{(\C^2, 0)} /
 ( f_1 f_2 ,g,  h )  \to  \mathcal{O}_{(\C^2, 0)} /( f_1 , g )  \to 0 \mbox{ .}\]
\end{lem}

\begin{proof}
Part (a) is well-known as the additivity property of the local intersection number of plane curves,
see e.g. \cite{fulton}. The proof of part (b) is similar.
\end{proof}

 Using these lemmas the codimension of $ J( \Phi) $ of the $D_{n-2} $ singularity can be calculated, and it is $ C( \Phi)= 4n^2 + 12n -1 $.
\end{ex}

\begin{ex}[Weighted homogeneous germs and $E_6$, $E_7$ and $E_8$]\label{ex:E} 
Assume that the three components of $\Phi$ are weighted homogeneous of weights
$w_1$ and $w_2$ and degree $d_1$, $d_2$ and $d_3$. Then, cf. \cite{Mondwh}, 
$C(\Phi)=\{d_1d_2+d_2d_3+d_3d_1-(w_1+w_2)(d_1+d_2+d_3-w_1-w_2)-w_1w_2\}/w_1w_2.$
Mond proved this identity for finitely $ \mathscr{A}$-determined germs, but the same proof works for
germs with finite $ C ( \Phi ) $.

For example, if $\Phi: (\C^2,0)\to (\C^2,0)/G\hookrightarrow (\C^3,0)$ is as in Example
\ref{ex:ade}, then all three components are homogeneous ($w_1=w_2=1$). In the case of
$A_{k-1}$ and $D_{n+2}$ the degrees are $(k,k,2)$ and $(4,2n,2n+2)$ respectively.
Hence the values $C(\Phi)$ from Examples \ref{ex:A2} and \ref{ex:D} follow in this way as well.
%\end{ex}

%\begin{ex}[$E_6$, $E_7$ and $E_8$]\label{ex:E}
For $E_6,\ E_7$ and $E_8$ singularities the degrees are
$(6,8,12)$, $(8,12,18)$ and $(12, 20,30)$ respectively, see \cite[4.5.3--4.5.5]{invariant},
hence the corresponding values
$ C( \Phi) $ are $167$, $383$, $1079$.
\end{ex}

\newpage

\begin{center}
\begin{figure}
 \includegraphics[width=13cm]{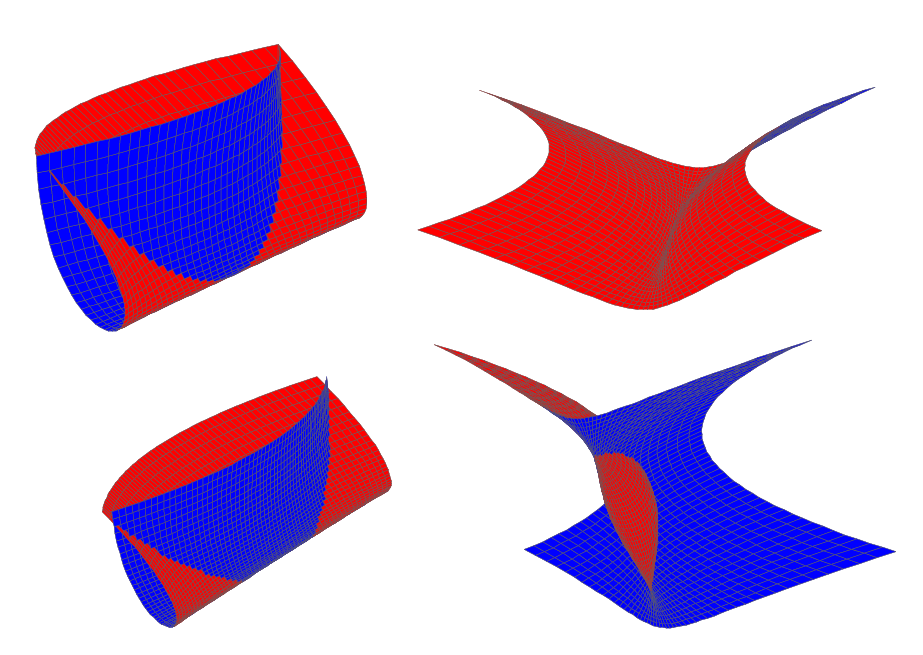}
 \caption{The real germs $ S_1^{ \mp } $, $ \Phi^{ \mp } (s, t)=(s, t^2, t^3 { \mp } s^2t)$, and their stabilizations 
 $ \Phi^{ \mp }_v(s, t)=(s, t^2, t^3 \mp s^2t { \pm } vt)$, $v \geq 0$, with two Whitney umbrellas.} \label{fig:s1fig}
\end{figure}
\end{center}

\begin{center}
\begin{figure}
 \includegraphics[width=10cm]{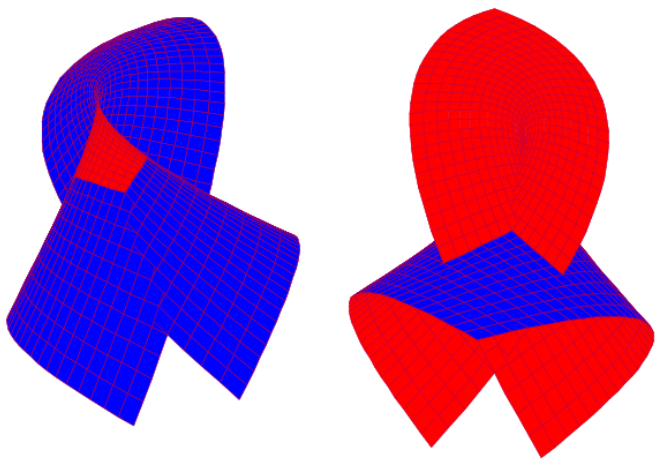}
 \caption{The real version of $A_1$ is the double covering $ \Phi (s, t)=(s^2, t^2, st) $ of the cone $ \{ xy=z^2, \ x, y \geq 0 \} $. The picture shows its stabilization $ \Phi_v (s, t)= (s(s-v), t(t-v), st)$ with the three Whitney umbrellas and one triple value.}\label{fig:A1fig}
\end{figure}
\end{center}

\subsection{Disentanglement and the image Milnor number} \label{ss:disent}
% The image of a stabilization is called the disentanglement of $ \Phi $. 
Recall for finitely determined germs from $ ( \C^n, 0) $ to $( \C^{n+1}, 0) $ the topology of the disentanglement (that is, the image of a stabilization) does not depend on the choice of the stabilization, cf. Proposition~\ref{pr:topdis}.

Note that the disentanglement is a stratified space, each stratum is a connected `isosingular locus' of stable multi-germs.

In many cases the disentanglement  plays the role of the Milnor fibre of isolated singularities, for example it has the following property, analogously with Theorem~\ref{th:milnfiso}. (In fact, the proof uses \ref{th:milnfiso}, see \cite{Mondvan}).
\begin{thm}[{\cite[Theorem 1.4.]{Mondvan}}]
The disentanglement of a finitely determined germ $ \Phi: ( \C^n , 0) \to ( \C^{n+1}, 0) $ is homotopy equivalent with a bouquet of $n$-spheres, ie. with $ \bigvee_{i=1}^{ \mu_I( \Phi)} S^n $.
\end{thm}

\begin{defn}
The number of the spheres in the bouquet (that is, the $n$-th Betti number of the disentanglement) $ \mu_I ( \Phi ) $ is called the image Milnor number of $ \Phi $.
\end{defn}

\begin{thm}[{\cite[3.2.-3.4.]{mararmulti}}, \cite{Mondvan}]
For finitely determined germs $ \Phi: ( \C^2, 0) \to ( \C^3, 0) $
\[ \mu_I ( \Phi ) = \frac{1}{2} (4 T( \Phi)- C(\Phi) - \mu (D ( \Phi )) +1 ) \]
holds, where $\mu (D ( \Phi ))$ is the Milnor number of the double point curve $ D( \Phi ) \subset \C^2 $, see paragraph~\ref{ss:double}.
\end{thm}

\begin{thm}[\cite{dejongvanstraten, Mondvan}]\label{th:Mondconj}
For a finitely determined germ $ \Phi: ( \C^2, 0) \to ( \C^3, 0) $, the image Milnor number $ \mu_I ( \Phi ) $ is bigger or equal than the $ \mathscr{A}_e $-codimension of $ \Phi $. If $ \Phi $ is weighted homogeneous, then equality holds.
\end{thm}

For finitely determined germs $ \Phi: ( \C^n, 0) \to ( \C^{n+1}, 0) $, $n \geq 3 $, the statement of Theorem~\ref{th:Mondconj} is Mond's conjecture.

We use the disentanglement as a \emph{complex singular Seifert surface} of the immersion 
$ \Phi|_{ \mathfrak{S}^3}: \mathfrak{S}^3 \looparrowright S^5 $. See Section~\ref{s:eszcomp}.

% \newpage

\section{Further invariants}

\subsection{The double point structure}\label{ss:double} 

The possible algebraic descriptions of the multiple point spaces of complex germs
% -- as well as the problems of their definitions -- 
are studied in general in \cite{nunomulti, mararmulti, Mond2, mond-ballesteros}. Here we summarize the double point structure of germs $ ( \C^2, 0) \to ( \C^3, 0)$. The Milnor numbers of the four double point curves are related to each other through the invariants $C$, $T$ and $N$. Indeed, $N$ can be  replaced by one of this Milnor numbers. In this paragraph we refer mainly to \cite{nunodouble, mararmulti}.

The double point spaces of a finite germ $ \Phi: ( \C^2, 0) \to ( \C^3, 0) $ can be summarized in the following diagram (see \cite{nunodouble, mararmulti}).

\begin{equation}\labelpar{eq:double}
\begin{array}{ccc}
D^2( \Phi) & \to & D^2( \Phi) / S_2 \\
\downarrow &     & \downarrow \\
D( \Phi ) &  \to  &  \Phi ( D ( \Phi )) \\       
\end{array}
\end{equation}

$ D^2( \Phi ) \subset ( \C^2, 0) \times ( \C^2, 0) $ is the \emph{lifting of the double point space} of $ \Phi $. A general algorithm to determine it can be found in \cite{nunodouble, Mond2}, here we sketch it in our case. We use the following notation: the coordinates of the point $ (x, x') \in \C^k \times \C^k $ are $x_i $, $x_i'$ ($i=1, 2$), 
%and $y_j$, $y_j'$ ($j=1, 2, 3$) in $ \C^3 \times \C^3 $. 
The coordinate functions of 
$ \Phi $ are $ \Phi_i $, $i=1, 2, 3$. 

Let $ I_k $ denote the ideal of the diagonal in $ ( \C^k, 0) \times ( \C^k, 0) $, that is $ I_k = (x_i - x_i')_{i=1, \dots, k}$. 
% with coordinate functions $ z_i$, $z_i'$. 
Then $ ( \Phi \times \Phi)^* ( I_3)  \subset I_2 $, thus there are $ \alpha_{ij} \in \mathcal{O}_{ (\C^2 \times \C^2, 0)} $ such that
\[ \Phi_i(x) - \Phi_i (x') = \Sigma_{j=1}^2 \alpha_{ij} (x, x') (x_j-x_j')
\]
holds for $i=1, 2, 3$. Then $ D^2( \Phi ) \in ( \C^2, 0) \times ( \C^2, 0) $ is the zero set of the ideal
\begin{equation}\labelpar{eq:doubid} ( \Phi \times \Phi)^* ( I_3) + R_2( \alpha ) \mbox{,} \end{equation}
where $ R_2( \alpha )  \subset \mathcal{O}_{ (\C^2 \times \C^2, 0)} $ is the ideal generated by the determinants of the $ 2 \times 2$ minors of the matrix $ \alpha = ( \alpha_{ij}) $.

Note that $ \alpha(x, x) = d \Phi (x) $. Hence $ D^2( \Phi ) $ contains the ordinary double points $(x, x')$, that is $ x \neq x' $ and $ \Phi (x) = \Phi ( x')$, as well as the points $(x,x)$ of the diagonal, where $x$ is a singular point of $ \Phi $.

There are several other methods to determine $ D^2( \Phi ) $ which work in special cases. For instance 
$ ( \Phi \times \Phi)^* ( I_3) : I_2 $ provides the same ideal as (\ref{eq:doubid}) for finitely determined germs, see \cite{nunomulti}. 

$D( \Phi ) \subset (\C^2, 0) $ is the \emph{double point space} in the source. It is the image of the projection 
$p_1: ( \C^2, 0) \times ( \C^2, 0) \to ( \C^2, 0) $ restricted to $ D^2 ( \Phi ) $. The ideal of $ D( \Phi ) $ can be given as the $0$-th Fitting ideal of the projection $ p_1: D^2 ( \Phi ) \to ( \C^2, 0) $. 

Finite determinacy can be characterized by the help of the curves $D^2 $ and $D$ as follows.
% see theorems 2.4., 3.4., 3.5. in \cite{nunodouble}. 

\begin{thm}[{\cite[Theorems 2.4., 3.4., 3.5.]{nunodouble}}]\label{th:findoub} Let $ \Phi: ( \C^2, 0) \to ( \C^3, 0) $ a finite and generically $1$ to $1$ complex germ. The following are equivalent:

(a) $ \Phi $ is finitely $ \mathscr{A} $-determined.

(b) $ D^2 ( \Phi ) $ is a germ of reduced curve and $ p_1: D^2 ( \Phi ) \to ( \C^2, 0) $ is generically $1$ to $1$.

(c) $D( \Phi ) $ is a germ of reduced curve.

(d) The Milnor number $ \mu ( D( \Phi )) $ of  $D( \Phi ) $ is finite. (See Subsection~\ref{ss:milnfibration} for definition).
\end{thm}

Note that if $ p_1: D^2 ( \Phi ) \to ( \C^2, 0) $ is generically $1$ to $1$, then it is an embedding for a choice of a small enough representative. $ \mu ( D( \Phi )) $ is called the \emph{Mond number} of $ \Phi $.

If $ \Phi $ is finitely determined, there is a useful method to determine the reduced equation $ d: ( \C^2, 0) \to ( \C, 0) $ of $ D( \Phi) $. Let $ f= \det \lambda: ( \C^3, 0) \to ( \C, 0 ) $ be the equation of the image of $ \Phi $. Then 
$ d= \Phi^*( \partial_i f) / M_i $ is the same for all $ i=1, 2, 3$, where $M_i$ is the determinant of the $2 \times 2$ minor of $d \Phi $ removed the $i$-th row, and $ \partial_i f $ is the $i$th partial derivative of $f$. See \cite{Piene, Bruce}.

$ D^2( \Phi) /S_2 $ is the quotient of $ D^2 ( \Phi ) $ by the $ S_2$-action $ (x, x') \mapsto (x', x) $ on $ \C^2 \times \C^2 $. 
$ D^2( \Phi) /S_2 $ embeds as a germ of analytic subspace of  $ (\C^2, 0) \times (\C^2, 0)/ S_2 $. If $ \Phi $ is a corank--1 germ, then $ D^2 ( \Phi ) $ and also $ D^2( \Phi) /S_2 $ can be embedded to $ \C^3 $.

The \emph{space of double values} $ \Phi ( D ( \Phi )) \subset ( \C^3, 0) $ is the zero set of the first Fitting ideal $ \mathcal{F}_1 (\Phi_* ( \mathcal{O}_{ ( \C^2, 0)}))$.

The four double point spaces can be defined also for maps, specially for a stabilization $ \Phi_v: U_v \to W $ of $ \Phi$. Then 
$ D^2( \Phi_v)$ is a smooth curve in $ U_v \times U_v $, and the projection $ p_1 : D^2( \Phi_v) \to U_v$ is a stable immersion with transverse double values in the triple points of $ \Phi_v $. $ D^2( \Phi_v)/ S_2$ is also a smooth curve, the factorization by $ S_2 $ is a double branched covering, ramifies at the (lifting of the) Whitney umbrella points. 
$ \Phi_v \circ p_1: D^2( \Phi_v) \to W $ is an immersion with triple values in the triple values of $ \Phi_v$.

If $ \Phi $ is a corank--$1$ finitely determined germ, then all the curves $ D^2( \Phi)$, $ D( \Phi)$, $ D^2( \Phi)/S_2$, $ \Phi(D^2( \Phi))$ are isolated complete intersection (ICIS) (cf. Chapter~\ref{ch:iso}). However the Milnor number of these curves can be defined also for corank--$2$ germs, see \cite{nunodouble}. The Milnor numbers relate to each other by the following formulas, according to the description of the double point curves of a stabilization.
\begin{thm}[{\cite[Theorem 4.3.]{nunodouble}}] If $ \Phi: (\C^2, 0) \to ( \C^3, 0) $ is a finitely determined germ, then
\[ \mu( D( \Phi ))= \mu ( D^2 ( \Phi)) + 6T( \Phi) \mbox{,}\]
\[ \mu( D^2( \Phi ))= 2 \mu ( D^2 ( \Phi)/S_2) + C( \Phi ) -1 \mbox{,}\]
\[ \mu( D( \Phi ))= 2 \mu (\Phi ( D ( \Phi))) + C( \Phi )- 2T ( \Phi) -1 \mbox{.}\]
\end{thm}

\begin{thm}[\cite{slicing, mararmulti}]\label{th:N} If $ \Phi: (\C^2, 0) \to ( \C^3, 0) $ is a finitely determined quasihomogeneous corank--$1$ germ, then 
$ N( \Phi ) = 2 \mu ( D^2 ( \Phi)/S_2) $, and
\[ \mu( D ( \Phi )) = 6 T ( \Phi ) + C( \Phi) + N( \Phi ) -1  \mbox{.} \]
\end{thm}

\subsection{Slicing}\label{ss:Slicing} In \cite{slicing} there is an interesting reinterpretation of the invariants $C$, $T$ and $N$ for corank--$1$ germs 
$ \Phi: ( \C^2, 0) \to ( \C^3, 0 ) $. This is the nearest approach to our results explained in Chapter~\ref{ch:ass}, as both derive these invariants from the properties of a transverse object. In \cite{slicing} this object is the transverse slice of $ \Phi $ defined below, while in Chapter~\ref{ch:ass} that is the immersion $ \Phi|_{ \mathfrak{S}^3}: \mathfrak{S}^3 \looparrowright S^5 $ associated with $ \Phi $.

The idea is the following. After change of coordinates the germ $ \Phi $ can be considered as a $1$-parameter unfolding of a finitely determined plane curve, called transverse slice. Then invariants of $ \Phi $ can be obtained from the parameter space of the versal unfolding of the transverse slice.

Let $ \Phi: ( \C^2, 0) \to ( \C^3, 0 ) $ be a finite, generically $1$ to $1$, corank--$1$ holomorphic map germ. Up to $\mathscr{A}$-equivalence it can be written in a form $ \Phi(s, t) = (s, p(s, t), q(s, t))$, such that 
$ \{ x=0 \} \cap \Phi ( D ( \Phi )) = \{ 0 \} $ and $ \{ x=0 \} \cap C_0( \Phi ( D ( \Phi ))) = \{ 0 \} $, where $ C_0 ( \Phi ( D ( \Phi ))) $ denotes the Zariski tangent cone of $ \Phi ( D ( \Phi )) $.

Introduce the notation $ \gamma_s (t)=(p(s, t), q(s, t))$. Then $ \Phi (s, t) = ( s, \gamma_s (t) ) $ is a $1$-parameter unfolding of the curve $ \gamma_0: ( \C, 0) \to ( \C^2, 0 ) $, called \emph{transverse slice}. The transverse slice is a finitely $ \mathscr{A} $-determined germ of plane curve. Its finite determinacy follows from the transversality conditions described in the previous paragraph. 

Note that the transverse slice is not unique, its analytic type depends on the choice of the coordinates, but its topological type does not. Its $ \delta$-invariant \\
$ \delta ( \gamma_0 )= \dim_{ \C } (\mathcal{O}_{( \C, 0)} / \gamma_0^* \mathcal{O}_{( \C^2, 0)}) =\mu ( \gamma_0) /2 $ is equal to the multiplicity of $ \Phi ( D( \Phi)) $ at $0$.

\begin{prop}[{\cite[3.1.]{slicing}}]
$\Phi (s, t)$ is finitely $ \mathscr{A} $-determined if and only if  $\Phi (s, t) = ( s, \gamma_s (t) ) $ is a stabilization of $ \gamma_0 $.
\end{prop}

Consider the miniversal deformation $( a, \gamma_{ a} (t))$ of $ \gamma $ where $a=(a_1, \dots, a_r ) \in \C^r $. Then $\Phi (s, t) = ( s, \gamma_s (t) ) $ can be pulled back from the versal deformation by a curve 
$ {\bf a}: (\C, 0)  \to ( \C^r, 0) $. If $ \Phi $ is finitely determined, then $ { \bf a} $ intersects the bifurcation set
\[ \mathcal{B} = \{ a \in \C^r \ | \ \gamma_a \mbox{ is not stable} \} \]
only at the origin.

$ \mathcal{B} $ is a hypersurface in $ \C^r$, its reduced equation decomposes as $g=g_1 g_2 g_3 $ corresponding to $\mathscr{A}_e$-codimension--$1$ multi-germs of plane curves. $\gamma_a= (t \mapsto \gamma_a (t))$ is a stable plane curve (that is, self-transverse immersion) if $ g(a) \neq 0$. For $ g_1 (a) = 0 $ the curve $ \gamma_a $ has a cusp, for $ g_2 (a) = 0 $ the curve $ \gamma_a $ has a triple point, for $ g_3 (a) = 0 $ the curve $ \gamma_a $ has a \emph{tacnode}, that is, self-tangent double point. In other words, the three components of $ \mathcal{B} $ correspond to the three \emph{Reidemeister moves} of a $1$-parameter deformation of a plane curve.

\begin{thm}[{\cite[3.7--3.10.]{slicing}}] Let the finitely determined corank--$1$ germ \\ $ \Phi : ( \C^2, 0) \to ( \C^3, 0) $ be the stabilization of its transverse slice $ \gamma_0$ pulled back from the miniversal unfolding by the curve $ { \bf a}: ( \C, 0) \to ( \C^r, 0) $.
Then
\[ C ( \Phi ) = \textrm{ord}( g_1 \circ { \bf a} ) \mbox{ and } 
T ( \Phi ) = \textrm{ord}( g_2 \circ { \bf a} ) \mbox{.} \]
\end{thm} 

The order of $ g_i \circ { \bf a} $ is the intersection multiplicity of the curves $ g_i $ and $ { \bf a} $, that is, the number of the transverse intersection points of $ \{ g_i =0 \} $ and a perturbation of $ { \bf a}$. In each intersection point $ a $ the corresponding curve $ \gamma_a $ has a cusp (resp. triple point). On the other hand, since $ { \bf a} $ describes a $1$-parameter stable unfolding of $ \gamma_0 $, which is $ \Phi $, a $1$-parameter perturbation of $ { \bf a } $ is a $ 2$-parameter unfolding of $ \gamma_0 $, that is, a $1$-parameter unfolding of $ \Phi $. The cusps and the triple points of $ \gamma_{ a} $ correspond to the Whitney umbrellas and the triple values of the unfolding of $ \Phi $.

The third invariant
\[ J ( \Phi ) := \mbox{ord}( g_3 \circ { \bf a} ) \]
is the number of values $ a \in \C^r $ of a perturbation of $ { \bf a }$, such that $ \gamma_a $ has a tacnode.  

The following relations hold to the new invariants.
\begin{thm}[{\cite[3.8. and p. 1388]{slicing}}]\label{th:N2} Let the finite, generically $1$ to $1$, corank--$1$ germ $ \Phi : ( \C^2, 0) \to ( \C^3, 0) $ be the $1$-parameter unfolding of its transverse slice $ \gamma_0$ pulled back from the miniversal unfolding by the curve $ { \bf a}: ( \C, 0) \to ( \C^r, 0) $. Then

(a) $ \Phi $ is finitely determined if and only if $ C( \Phi )$, $ T( \Phi )$ and $ J( \Phi )$ are finite (defined as 
$ \mbox{ord}( g_i \circ { \bf a} ) $).

(b) $ J( \Phi )= \delta ( \gamma_0 ) + N( \Phi )/2 -1 $.

\end{thm}

\chapter{Immersions of spheres to Euclidean spaces}\label{ch:imm}

\section{Hirsch--Smale theory}\label{s:hirsch-smale} 

\subsection{Hirsch theorem}\label{ss:hirsch} This chapter contains the basic theorems of Hirsch and Smale, and the formulae of Hughes--Melvin \cite{HM} and Ekholm--Sz\H{u}cs \cite{ESz} corresponding to the immersions of the $3$-sphere to $ \R^5 $. The last section contains other results and applications, which are related to our study presented in Chapter~\ref{ch:ass}.

Hirsch--Smale theory \cite{hirsch, smale} transforms regular homotopy problems (differential topology) to homotopy theory (algebraic topology).

Let $ M^m $ and $ Q^q $ be smooth manifolds of dimension $m$ and $q$ respectively. A smooth map $ f: M \to Q $ is an \emph{immersion}, if at every point $x \in M $ the differential $ df_x: T_x M \to T_{f(x)} Q $ is injective (monomorphism). In other words, immersions are `local embeddings'. Two immersions $ f, g: M \to Q $ are \emph{regular homotopic} if they can be deformed to each other through immersions, that is, if there exists a smooth map $ H: M \times [0, 1] \to Q $ such that $H_0= f$, $H_1 = g $ and $H_t: M \to Q $ is an immersion for all $t \in [0, 1] $. 
Regular homotopy is an equivalence relation on the set of immersions from $ M $ to $ Q $.
An injective immersion of a compact manifold $M$ is called \emph{embedding}.

We use the following notations. $ f: M \looparrowright Q $ denotes an immersion from $M$ to $Q$, and $ f: M \hookrightarrow Q $ denotes an embedding. $ \imm ( M, Q ) $ is the set of regular homotopy classes of the immersions from $M$ to $Q$, and its subset $ \emb( M, Q ) $ consists of regular homotopy classes which admit embedding representatives.

The differential of an immersion is a fibrewise injective linear bundle map $ df: TM \to TQ $, say \emph{bundle monomorphism}. We call two bundle monomorphisms homotopic if they are homotopic through bundle monomorphisms.

\begin{thm}[Hirsch \cite{hirsch}]\label{th:hirsch1} If $ m < q $, 
the differential $f \mapsto df $ induces a bijection between $ \imm(M^m, Q^q) $ and the homotopy classes of the bundle monomorphisms.
\end{thm}

If $ Q= \R^q $, a more concrete identification can be done. For an immersion $ f: M \looparrowright \R^q $ the bundle $ \mathcal{E}^q:= \{ f^* T \R^q \to M \}$ is trivial, moreover for different immersions the corresponding bundles are naturally isomorphic. Let 
$ \mbox{HOM} (TM, \mathcal{E}^q) \to M $ be the bundle of bundle homomorphisms from $ TM$ to $ \mathcal{E}^q $ (over the identity  map of $M$), and its subbundle $ \mbox{MONO} (TM, \mathcal{E}^q) \to M $ consists of bundle monomorphisms. $ \mbox{HOM} (TM, \mathcal{E}^q) \to M $ is a vector bundle with fibre $ \Hom (\R^m, \R^q ) $. $ \mbox{MONO} (TM, \mathcal{E}^q) \to M $ is a locally trivial bundle whose fibre is the \emph{Stiefel-manifold} $ V_m( \R^q) \subset \Hom (\R^m, \R^q ) $ consists of monomorphisms from $ \R^m $ to $ \R^q $, or in other words, the linearly independent $m$-frames of $ \R^q$. (The \emph{compact Stiefel manifold} consisting of orthonormal $m$-frames of $ \R^q$ is homotopy equivalent with $ V_m( \R^q) $ via its natural embedding. We do not distinguish them in the notation.)

\begin{thm}[Hirsch \cite{hirsch}]\label{th:hirsch2}
(a) If $ m < q $, 
the differential $f \mapsto df $ induces a bijection between $ \imm(M^m, \R^q) $ and the homotopy classes of the continuous sections of $  \mbox{MONO} (TM, \mathcal{E}^q) \to M $.

(b) If additionally $M$ is parallelizable (that is, $ TM $ is trivial), then a chosen trivialization of $ TM $  provides a bijection between $ \imm (M, \R^q) $ and $ [M, V_m ( \R^q)]$.
%the homotopy classes of the continuous maps from $M$ to $V_n ( \R^q)$ (denoted by $ [M, V_n ( \R^q)]$).
\end{thm}

Part (a) of Theorem~\ref{th:hirsch2} does not hold in general for $ \imm( M, Q) $, not even if all the immersions from $M$ to $Q$ are homotopic and 
$ f^*TQ \to M $ is a trivial bundle, cf. Example~\ref{ex:s1s2}.
% On the other hand, the homotpy classes of the sections in part (a) can be replaced with homotopy classes of continuous maps, as in part (b), in some other cases, when $ M$ is not parallelizable, for instance if $ M $ is a sphere, as in the next subsection.

\subsection{Smale invariant}\label{ss:Smaleorig} Although in the case of spheres the tangent bundle is not trivial in general, the homotopy classes of the sections in part (a) of Theorem~\ref{th:hirsch2} can be replaced by homotopy classes of continuous maps, as in part (b).

\begin{thm}[Smale \cite{smale}]\label{th:smale1} There is a bijection
\[ \imm( S^n, \R^q) \leftrightarrow \pi_n ( V_n ( \R^q)) \mbox{.}
\]
\end{thm}

The bijection is induced by the map $ f \mapsto \Omega (f) $, where $ \Omega (f) \in \pi_n ( V_n ( \R^q)) $ is called the \emph{Smale invariant} of the immersion $f: S^n \looparrowright \R^q $. (The relation between the bijections in Theorem~\ref{th:smale1} and in Theorem~\ref{th:hirsch2}(b) is clarified in Remark~\ref{re:parasmale}.)

Here we present the construction of the Smale invariant, cf. \cite[p. 159]{hughes}, \cite[3.2.]{ekholm3}. Let $S^n$ be the unit sphere in $ \R^{n+1} $, that is 
$ S^n = \{ x=(x_0, \dots, x_n) \in \R^{n+1} \ | \ \Sigma_{i=0}^n x_i^2 =1 \} $, and define $ S^n_+ = \{ x \in S^n \ | \ x_0 \geq 0 \} $ and $ S^n_- = \{ x \in S^n \ | \ x_0 \leq 0 \} $. Let $ U_{\pm} $ be an open neighbourhood of $ S^n_{ \pm} $ in $ S^n$. Fix a trivialization $b_+$ of the tangent bundle $TU_+$.
%Fix an atlas with charts on $U_- $ and $U_+$, so that $ TU_- $ and $TU_+$ are trivialized with the base fields $ b_- $ and $b_+$.

By a regular homotopy it can be achieved that $ f: S^n \looparrowright \R^q $ agrees with the standard embedding $ \iota: S^n \subset \R^{n+1} \subset \R^q $ on $ U_- $. We define a continuous map $ \omega(f): S^n \to V_n ( \R^q)$ as
\[ \omega(f)(x)=
\left\{ \begin{array}{ccc}
df_x (b_+ (x)) & \mbox{if} & x \in S^n_+ \mbox{,} \\
d \iota_{\tau (x)} (b_+ (\tau (x))) & \mbox{if} & x \in S^n_- \mbox{,} \\
\end{array}
\right. 
\]
where $ \tau(x_0, x_1, \dots, x_n) = (-x_0, x_1 \dots, x_n) $ is the reflection in the hyperplane $ \{ x_0=0 \} \subset \R^{n+1}$ restricted to $ S^n \subset \R^{n+1}$. Then $ \Omega(f):= [ \omega(f)] \in \pi_n (V_n ( \R^q ))$.

\begin{rem}\label{re:lie}
 If $ G $ is a connected Lie group, or a quotient of it by a
 closed connected subgroup, then $ \pi_n (G) $ can be identified with the homotopy classes
 of the continuous maps $ f: S^n \to G $ without any base point. This fact can be applied to 
 $ V_n( \R^q) = SO(q)/SO(n-q)$, cf. \cite{steenrod, husemoller}. 
 Furthermore, for Lie
 groups, the group operation of $\pi_n(G)$ agrees with that induced by
 the pointwise multiplication in $ G$; cf. \cite[p. 88 and 89]{steenrod}.
\end{rem}

Then Theorem~\ref{th:smale1} can be reformulated as follows.

\begin{thm}[Smale \cite{smale}]\label{th:smale2} For an immersion $ f: S^n \looparrowright \R^q $, the Smale invariant $ \Omega (f) \in \pi_n ( V_n ( \R^q))$ is independent of the choices, it is determined by $f$. Furthermore, two immersions $f$ and $g$ are regular homotopic if and only if $ \Omega (f) = \Omega (g)$, and each element of $ \pi_n ( V_n ( \R^q)) $ arises as $ \Omega(f) $ of some immersion $f$.
\end{thm}

\begin{rem} $ \imm ( S^n , \R^q ) $ is a group with respect to the operation `connected sum'. The connected sum $ f \# g $ of the immersions $ f$ and $g $ is the induced immersion from 
$ (S^n \setminus \mbox{int}(B^n)) \cup_{S^{n-1}} (S^n \setminus \mbox{int}(B^n) ) \simeq S^n $ to $ \R^q $ \cite[p. 164]{hughes}, \cite[3.4.]{ekholm3}.
The Smale invariant provides a group ismorphism $ \imm ( S^n, \R^q) \cong \pi_n ( V_n ( \R^q))$. In particular, in $ \imm ( S^n , \R^q ) $ the identity element is the standard embedding, thus its Smale invariant is the identity element of $ \pi_n ( V_n ( \R^q))$. See \cite{hughes, ekholm3}.
\end{rem}

\begin{rem}\label{re:parasmale} In some dimensions $ S^n $ is parallelizable, for instance if $n=1 $ or $3$. Then by part (b) of Theorem~\ref{th:hirsch2} there is a bijection $ \Theta: \imm ( S^n, \R^q) \to \pi_n ( V_n ( \R^q))$, but it depends on the choice of the trivialization of $TS^n $. Hence it does not agree with the Smale invariant in general. The relation between them is 
$ \Omega(f) = \Theta(f) - \Theta( \iota ) $, where $ \iota$ is the standard embedding. Cf. \cite{ekholm3} and Example~\ref{ex:s1r2}.
\end{rem} 

\begin{ques}\label{qu:Smale}
Smale asked in \cite{smale} to characterize the regular homotopy classes in terms of the geometry of the immersion, and give explicit representatives of each regular homotopy class. 
\end{ques}
Several results appeared in this topic, see for instance examples~\ref{ex:s2r3}, \ref{ex:snr2n} and \cite{hughes, ekholm3, kinjo}. Our results also fit to this program, cf. \cite{NP} or Chapter~\ref{ch:ass}.

\subsection{Examples}\label{ss:exsmale}

\begin{ex}[Eversion of the Sphere]\label{ex:s2r3} The Smale invariant $ \Omega (f) $ of an immersion $f: S^2 \looparrowright \R^3 $ sits in $ \pi_2 (V_2 ( \R^3 )) = \pi_2 (SO(3)) = 0 $. Thus any immersion $f: S^2 \looparrowright \R^3 $ is regular homotopic with the standard embedding $ \iota_{\mbox{st}}: S^2 \hookrightarrow \R^3 $. For instance the standard embedding composed with the hyperplane reflection $(x, y, z) \to (-x, y, z) $ is regular homotopic with the standard embedding. A regular homotopy between them is an \emph{eversion of the sphere}.
\end{ex}

\begin{ex}[$S^n \looparrowright \R^q$, $2n < q$] The Stiefel-manifold $ V_n ( \R^q ) $ is $ (q-n-1)$-connected, cf \cite{husemoller, steenrod}. Hence for $ 2n < q $ any immersion from $ S^n $ to $ \R^q $ is regular homotopic with the standard embedding.
\end{ex}

\begin{ex}[$S^n \looparrowright \R^{2n}$]\label{ex:snr2n} For an immersion $f: S^n \looparrowright \R^{2n} $ 
\[ \Omega(f) \in \pi_n (V_n ( \R^{2n}))=
\left\{ \begin{array}{ccc}
\Z & \mbox{if} & n \mbox{ is even or } n=1 \mbox{, } \\
\Z_2 & \mbox{if} &  n \mbox{ is odd and } n \neq 1 \mbox{, } \\
\end{array}
\right. 
\]
cf. \cite{husemoller, steenrod}. By a regular homotopy $f$ can be deformed to a stable immersion, which has self-transverse double values. If $n$ is even, it is possible to associate a sign to each double value, and $ \Omega (f) $ is the algebraic number of the double values. For odd $n$ ($n \neq 1$) the Smale invariant agrees with the algebraic number of double values modulo $2$. The exceptional case $n=1$ is discussed in the next example.
\end{ex}

\begin{ex}[$S^1 \looparrowright \R^2$]\label{ex:s1r2} The Smale invariant of a plane curve immersion is an element of $ \pi_1 ( V_1 (\R^2))= \pi_1 (S^1) = \Z $. By the result of Whitney \cite{Whitney}, the \emph{winding number} is a complete regular homotopy invariant. For an immersion $ f: S^1 \looparrowright \R^2 $ the winding number $w(f) $ is the degree of $ f'/|f'| : S^1 \to S^1$. Thus the winding number is the invariant according to part (b) of Theorem~\ref{th:hirsch2}, but it is not the Smale invariant, cf. Remark~\ref{re:parasmale}. However, by fixing a base point and a support line with a convention for the starting of the curve at the base point, it is possible to associate a sign to each double value of the immersion, cf. \cite{Whitney}. Then $ \Omega (f) \in \Z $ is the algebraic number of the double values, but its sign depends on the choice of the convention. By Remark~\ref{re:parasmale} $ \Omega (f) =  \pm w(f) \pm 1 $. The sign ambiguity can be avoided by fixing the conventions.
\end{ex}

The examples \ref{ex:s1s2} and \ref{ex:snsq} are results of a discussion with Tam\'{a}s Terpai.

\begin{ex}[$S^1 \looparrowright S^2$]\label{ex:s1s2} By Hirsch's Theorem~\ref{th:hirsch1} 
$ \imm(S^1, S^2)$ is in bijection with the bundle monomorphisms $ TS^1 \to TS^2 $ up to homotopy. Such a monomorphism associates to each point of $ S^1 $ a nonzero tangent vector of $ S^2 $, hence up to homotopy it can be considered as a map from $S^1 $ to $ SO(3) $. Thus $ \imm(S^1, S^2) $ is in bijection with $ \pi_1(SO(3)) \cong \Z_2 $. Even though $ f^*TS^2 \to S^1 $ is a trivial bundle for all $ f: S^1 \looparrowright S^2 $ immersions, part (a) of Theorem~\ref{th:hirsch2} cannot be applied. Indeed,  the sections of $ \mbox{MONO}(TS^1, \mathcal{E}^2)  \to S^1 $ up to homotopy are in bijection with $ \Z$.
\end{ex}

\begin{ex}[$S^n \looparrowright S^q$]\label{ex:snsq}
The inclusion $ \R^q \subset S^q= \R^q \cup \{ P \} $ (where $P$ is any point of $S^q$) induces a map from $ \imm ( S^n, \R^q ) $ to $ \imm ( S^n, S^q ) $. This map is surjective, since every regular homotopy classes can be represented by an immersion avoiding the infinity. If $ q \geq n+2 $, then the map is also injective, because the regular homotopy also can be chosen such that it avoids infinity. Hence $ \imm ( S^n, S^q ) $ can be identified with $ \imm ( S^n, \R^q ) $ for $ q \geq n+2 $.

The map from $ \imm ( S^n, \R^{n+1} ) $ to $ \imm ( S^n, S^{n+1} ) $ is surjective but not injective in general. Let $f$ and $g$ be immersions from $ S^n $ to $ \R^{n+1} \subset S^{n+1} $, and let $ H: S^n \times I \to S^{n+1} $ be a regular homotopy between $f$ and $g $ through the infinity. It can be assumed that infinity is a regular value of $ H$, and the algebraic number of its preimages is $\sharp H^{-1} ( P )=d $. Then $g$ is regular homotopic with $ f \# d \cdot \tau $ in $ \R^{n+1}$, where $ \tau: S^n \looparrowright \R^{n+1} $ is the hyperplane reflection of $ \R^{n+1} $ restricted to $ S^n \subset \R^{n+1} $, cf.
Example~\ref{ex:s2r3}, and $ \# $ denotes the connected sum of two immersions, cf. \cite{hughes, ekholm3}. 
%See picture~\ref{} for $d=1$, where $ \tau (S^n) $ is represented by a sphere around the infinity.

The normal degree $ D(f) $, i.e. the degree of the Gauss map of $f $ is a regular homotopy invariant for immersions $f: S^n \looparrowright \R^{n+1} $ with the properties $ D( \tau) = (-1)^n $ and $ D(f \# g) = D(f) + D(g) -1 $, see \cite{Mimm}. Hence for odd $n$ a regular homotopy through $ P \in S^{n+1} $ changes the normal degree with a multiple of $2$, and so changes the regular homotopy class of the immersion $f: S^{n} \looparrowright \R^{n+1}$. 
Note that $ D(f) $ is the winding number $w(f) $ for $f: S^1 \looparrowright \R^2 $, cf. examples~\ref{ex:s1r2} and \ref{ex:s1s2}. Compare also with Hughes formula~(\ref{eq:hughes2}) for immersions from $ S^3 $ to $ \R^4 $, where $ D(f)-1$ is the first component of $ \Omega (f) $, and Example~\ref{ex:s2r3}, in which case $ D( \tau ) = 1 $. 
\end{ex}

\section{Immersions of $S^3$ to $ \R^5$}\label{s:singseif} 

\subsection{Immersions and embeddings}\label{ss:s3r5}
% Determining the regular homotopy class from the multiple point structure (like in Example~\ref{ex:snr2n}) is a natural idea.
 %, there are several more examples for it, \cite{}. 
Hughes and Melvin \cite{HM} proved 
% it is not always possible: 
that there are embeddings of $ S^3 $ to $ \R^5 $ which are not regular homotopic to each other, thus it is impossible to express the Smale invariant in terms of the double point set (like e.g. in Example~\ref{ex:snr2n}). They expressed the Smale invariant of an embedding with the help of a Seifert surface.

For an immersion $ f: S^3 \looparrowright \R^5 $ the Smale invariant is in $ \pi_3(V_3 ( \R^5)) $, and 
$ \pi_3(V_3 ( \R^5)) \cong \pi_3 (SO(5)) \cong \pi_3 (SO) \cong \Z $. The isomorphisms are induced by the $SO(2)$-fibration $ SO(5) \to V_3( \R^5 ) $ and the standard embedding $ SO(5) \hookrightarrow SO $, cf. \cite{husemoller, steenrod}. The isomorphism of $\pi_3 (SO) $ and $ \Z $ is a priori non--canonical, but they can be identified with a choice of a generator in $\pi_3 (SO) $. It is described in Section~\ref{s:signconv}. 

\begin{thm}[Hughes, Melvin \cite{HM}]\label{th:HMemb}
$ \emb (S^3, \R^5) \cong 24 \cdot \Z \subset \imm( S^3, \R^5 ) \cong \Z $.
\end{thm}

Furthermore, $ \Omega(f) $ of an embedding $ f: S^3 \hookrightarrow \R^5 $ can be expressed by the \emph{signature} of a \emph{Seifert surface}.

\begin{thm}[Hughes, Melvin \cite{HM}]\label{th:HM}
 Let $ f: { S}^3 \hookrightarrow \R^5 $ be an embedding and $ \tilde{f}: M^4 \hookrightarrow \R^5 $ be a Seifert surface of $f$, i.e. $ M^4 $ is a compact oriented $4$-manifold with boundary
  $ \partial M^4 = { S}^3 $ and $ \tilde{f} $ is an embedding such that $ \tilde{f}|_{ \partial M^4} = f $.
 %Then the Smale invariant of $ f $ can be expressed via the signature
 Let $\sigma(M^4)$ be the signature of  $M^4 $. Then
\begin{equation}\label{eq:hm}
\Omega (f) = \pm \frac{3}{2} \sigma (M^4) \mbox{ .} 
\end{equation}
\end{thm}

Theeorems~\ref{th:HMemb} and \ref{th:HM} can be summarized on the following diagram:
\begin{equation}\labelpar{HMsum}
\begin{array}{ccc}
f & \mapsto & \Omega(f) \\
\imm(S^3, \R^5) & \to & \Z \\
\bigcup &     & \bigcup \\
\emb(S^3, \R^5) & \to & 24 \cdot \Z \\
f & \mapsto & \Omega(f)= \pm \frac{3}{2} \sigma (M^4) \\
\end{array} 
\end{equation}

\begin{rem}
Hughes and Melvin proved a more general result for embeddings from $ S^n $ to $ S^{n+2} $. An alternative proof for Theorems~\ref{th:HMemb} and \ref{th:HM} can be found in \cite{szucstwo}.
\end{rem}

\begin{rem}
The sign ambiguity of the formula (\ref{eq:hm}) is caused by the sign ambiguity of the Smale invariant as a $ \Z$-invariant, more precisely, by the identification of the target of $ \Omega $ with $ \Z $, and by the nature of the proof of (\ref{eq:hm}) as well. One of the goals of Chapter~\ref{ch:ass} is to indicate the correct sign in this and other Smale invariant formulae in this section, whenever the Smale invariant is replaced by the sign-refined Smale invariant, see Subsection~\ref{ss:absor}.
\end{rem}

\vspace{2mm}

For the proof of Theorems~\ref{th:HMemb} and \ref{th:HM} Hughes and Melvin introduced an alternative definition of the Smale invariant of immersions $f: S^3 \looparrowright \R^5 $. We review their construction here, since we use this definition of the Smale invariant throughout Chapter~\ref{ch:ass}. Note that Sz\H{u}cs used another construction in \cite{szucstwo}. 

Let $ U $ be a tubular neighbourhood of the standard $ S^3 \subset \R^5 $, and let $ F: U \looparrowright \R^5 $ be an orientation preserving immersion extending $f $, i.e. $ F|_{S^3} = f $. Let $ TU $
be the tangent bundle of $U$. It inherits a global trivialization from
the natural trivialization of $T\R^5 $. In particular,  there is a map (the Jacobian matrix)
\[ dF|_U : U \to GL^+ (5, \R ).
\]
Its homotopy class is the Smale invariant of $f$:
\begin{equation}\label{eq:Smale}
\Omega(f) = [ dF|_{S^3} ] \in \pi_3 (SO(5))
\end{equation}
via the  homotopy equivalence induced by the inclusion $ SO(5) \subset GL^+ (5, \R ) $. (The base point is irrelevant in $ \pi_3 (SO(5)) $, cf. Remark~\ref{re:lie}).

\begin{prop}\labelpar{prop:Smale} $ \Omega(f) $ does not depend on the choice of
$U$ and $ F $, it
depends only on the regular homotopy class of $ f $, and
$ \Omega: {\rm Imm} (S^3, \R^{5} ) \to \pi_3 (SO(5) ) $ is a bijection.
\end{prop}
Indeed, Smale proved that his original invariant gives a bijection between
$ \imm (S^3, \R^5 ) $ and $ \pi_3 (V_3 ( \R^5 )) $, cf. Subsection~\ref{ss:Smaleorig} and \cite{smale}.
Hughes and Melvin proved that their alternative definition (\ref{eq:Smale})
of the Smale invariant does not depend on the choice of $ F $ and agrees with the original Smale invariant through the natural
group isomorphism $ \pi_3 (SO(5)) \to \pi_3 (V_3 ( \R^{5} )) $.

%\subsection{Stable maps} In \cite{ESz} Hughes--Melvin Theorem~\ref{th:HM} is generalized to arbitrary immersions from $S^3 $ to $ \R^5 $ using \emph{singular Seifert surfaces}. These are stable maps, bounded by the immersion.

%\subsection{Vassiliev type invariants}\label{ss:vas}

\subsection{Ekholm's Vassiliev invariant $L$}\label{ss:L} In analogy with Arnold's invariants for plane curves $ S^1 \looparrowright \R^2 $, Ekholm \cite{ekholm3} introduced the invariants $J$ and $L$ (denoted by $ \lk$ originally) of stable immersions from $ S^3 $ to $ \R^5 $. Here we describe $L$ in detail. It appears in Subsection~\ref{s:ss} as a contribution in the Ekholm--Sz\H{u}cs formulas, and for immersions associated with holomorphic germs we express $ L$ in terms of $ C$ and $ T$, see Subsection~\ref{ss:LSt}.

A stable immersion of $S^3$ to $ \R^5$ has only regular simple and double values with transverse intersection of the branches at the double values. By \cite[Proposition 5.2.2.]{ekholm3} there are two types of codimension--$1$ immersions, i.e. immersions which can appear in a stable regular homotopy between two stable immersions: 

(a) immersion with one triple value (with regular intersection of the branches) and 

(b) immersion with self tangency at one point (that is, $ f(x)=f(y)=q \in S^5 $ is a double value of $f: S^3 \looparrowright \R^5 $, and $ \dim ( df_x (T_x S^3 ) + df_y (T_y S^3 ))=4 $ holds instead of transversality).

%Ekholm introduced the invariants $ J$ and $ L$ (denoted by $ \lk$ originally) of stable immersions. Each of them 
Ekholm's invariants $J$ and $L$ are constant along regular homotopies through stable immersions. $ J$ changes by $ \pm 1$ when the regular homotopy steps through an immersion with self tangency and does not change under triple point moves. $ L$ changes by $ \pm 3 $ under triple point moves and does not change when stepping through immersion with self tangency \cite[Theorem 2]{ekholm3}. $J(f)$ is equal to the number of the double point curve components of $f$ in the target (that is, the number of components of $ f(\gamma) $, see below).

The invariant $L(f)$ of a stable immersion $ f: S^3 \looparrowright \R^5 $ measures the linking of a shifted copy of the double values with the image of $f$. There are different versions for the definition, see below.  Here we review these definitions and their equivalence in the simplest case, for immersions $S^3 \looparrowright \R^5$, although originally they were introduced in \cite{ekholm3, ekholm4, ESz, saeki} for different levels of generalizations (for other manifolds, higher dimensions).
%wich differ from each other in the vector field used for the shift
% They can be found in \cite{ekholm3, ekholm4, ESz, saeki}, but they are introduced in different levels of generalities (for other manifolds, higher dimensions).
%, and also give a construction which includes all the definitions. This construction was resulted by a discussion with Andr\'{a}s Sz\H{u}cs.

Let $ f: S^3 \looparrowright \R^5 $ be a stable immersion, and let $ \gamma \subset S^3 $ be the double point locus of $f$, i.e. $ \gamma = \{ p \in S^3 \ | \ \exists p' \in S^3: \ p \neq p' \mbox{ and } f(p)=f(p') \} $. $ \gamma $ is a closed $1$-manifold, and $ f|_{\gamma}: \gamma \to f( \gamma ) $ is a $2$-fold covering. $ \gamma $ is endowed with an involution $ \iota: \gamma \to \gamma $ such that $ f( p )= f ( \iota (p) ) $ for all $ p \in \gamma $. 
%Ekholm's invariant $ J(f) $ is the number of components of $ f( \gamma) $.

The first definition is from \cite[6.2.]{ekholm3}. Let $v$ be a normal vector field of $ \gamma \subset S^3 $ such that $[ \tilde{\gamma} ]  $ is $0$ in $ H_1 ( S^3 \setminus \gamma , \Z) $. Such a vector field $v$ is unique up to homotopy, and for instance each component of a Seifert framing provides such a vector field.
If $ \tilde{ \gamma } \subset S^3$ is the result of pushing $ \gamma $ slightly along $v$, then the linking number $ \lk_{S^3}( \gamma, \tilde{ \gamma} )$ equals to $ 0 $. Let $ q = f(p) = f ( \iota (p)) $ be a double value of $f$. Then $ w(q) = df_p (v(p)) + df_{ \iota(p)} (v (\iota (p) ) $ defines a normal vector field $w$ along $ f( \gamma ) $. Let $ \widetilde{f(\gamma)} \subset \R^5 $ be the result of pushing $ f( \gamma) $ slightly along $w$, then $ \widetilde{f(\gamma)} $ and $ f( S^3 ) $ are disjoint. The invariant is the linking number
\begin{equation}\label{eq:L1} 
L_1 (f) := \lk_{ \R^5 } ( \widetilde{f(\gamma)}, f( S^3) )
\end{equation}
(or equivalently, $L_1 (f)= [ \widetilde{f(\gamma)} ] \in 
H_1(\R^5 \setminus f(S^3), \Z) \cong \Z $).
Note that Ekholm used an other notation: $ L_1 (f) $ is denoted by $ \lk (f) $, and $ L(f) $ is defined as $ \lfloor \lk (f) /3 \rfloor $ in \cite[2.2., 6.2.]{ekholm3}.

The second definition is \cite[Definition 11.]{ESz}, \cite[Definition 2.2.]{saeki}. The normal bundle $ \nu (f) $ of $f$ is trivial, since the oriented rank--$2$ vector bundles over $S^3 $ are classified by $ \pi_2 (SO(2)) =0$. Any two trivializations are homotopic, since their difference represents an element in $ \pi_3 (SO(2)) =0$. Let $ (v_1, v_2) $ be the homotopically unique normal framing of $f$, and in a double value $ q = f(p) = f ( \iota (p)) $ define 
$ u(q) = v_1(p) + v_1 (\iota (p))  $. $u$ is a normal vector field along $f( \gamma)$, and let $ \overline{f(\gamma)} \subset \R^5 $ be the result of pushing $ f( \gamma) $ slightly along $u$. Then $ \overline{f(\gamma)} $ and $ f( S^3 ) $ are disjoint. The invariant is the linking number (or equivalently, the homology class) 
\begin{equation}\label{eq:L2} 
L_2 (f) := \lk_{ \R^5 } ( \overline{f(\gamma)}, f( S^3) ) = [ \overline{f(\gamma)} ] \in 
H_1(\R^5 \setminus f(S^3), \Z) \cong \Z \mbox{.}
\end{equation}
Note that the framing $(v_1, v_2)$ can be replaced by an arbitrary nonzero normal vector field $v$ of $f$, since it can be extended to a framing whose first component is $v$.

The third definition is in \cite[Definition 4.]{ESz}, see also \cite[4.5., 4.6.]{ekholm4}. Let $v$ be a nonzero normal vector field of $f$ over $ \gamma $, that is, a nowhere zero section of $ \nu (f) |_{ \gamma} $. Let $ [v] $ be the homology class represented by $v$ in $ H_1 ( E_0( \nu (f)), \Z ) \cong \Z $, where $ E_0( \nu (f)) $ denotes the total space of the bundle of nonzero normal vectors of $f$. Let $u_v (q) = v(p) + v( \iota (p) ) $ be the value of the vector field $u_v$ along $ f( \gamma ) $ at the point $ q= f(p) = f ( \iota (p) )$. Let $ \overline{f(\gamma)}^{(v)} $ be the result of pushing $ f( \gamma) $ slightly along $u_v$, then $ \overline{f(\gamma)}^{(v)} $ and $ f( S^3 ) $ are disjoint. The invariant is
\begin{equation}\label{eq:Lv} 
L_v (f) := \lk_{ \R^5 } ( \overline{f(\gamma)}^{(v)}, f( S^3) ) -[v]= [ \overline{f(\gamma)}^{(v)} ] -[v] \mbox{,}
\end{equation}
where $ [ \overline{f(\gamma)}^{(v)} ] \in 
H_1(\R^5 \setminus f(S^3), \Z) \cong \Z $.

By \cite[Lemma 4.15.]{ekholm4} $ L_v (f) $ is well-defined, that is, $ L_v (f) $ does not depend on the choice of the normal field $v$. Moreover, if $ v $ is the restriction of a (global) normal vector field of $f $ to $ \gamma $, then $[v]=0$. Indeed, the restriction of the normal field of $f$ to a Seifert surface $H$ of $ \gamma $ results a surface $ \overline{H} \subset E_0( \nu (f)) $, 
%resulted by the vector field restricted to a Seifert surface $H$ of $ \gamma $ 
whose boundary is the image of $ v : \gamma \to E_0( \nu (f))$. Hence $ L_v (f) = L_2 (f) $.

%\begin{rem}
%All the homology groups $ H_1 (S^3 \setminus \gamma , \Z) $, $H_1 (\R^5 \setminus f( S^3), \Z)$ and $ H_1 (E_0( \nu (f)), \Z )$ are isomorphic with $ \Z $, and the orientations of $S^3 $, $ \gamma $ and the normal bundle $ \nu (f) $ determine generators in each of them. In each case the linking number with the removed set (that is, $ \gamma $, $f (S^3 )$ and the zero section) provides the isomorphism with $ \Z $, which takes the distinguished generator to $1$.
%\end{rem}

The invariants $L_1 $, $L_2 $ (and so $L_v$) are equal to each other. This fact follows from the basic property of $L_1$ and $L_2$ proved in \cite{ekholm3, ekholm4}, see Proposition~\ref{pr:Leq} below.
The proof  of this proposition is a result of a discussion with Andr\'{a}s Sz\H{u}cs.

\begin{prop}\label{pr:Leq}
(a) The three definitions are equivalent, i.e. $ \pm L_1 (f) = L_2 (f) = L_v (f) $.
Let us denote $L_1(f)$ by $ L(f) $ .

(b) $ L (f) $ is an invariant of stable immersions. It changes by $ \pm 3 $ under triple point moves and does not change under self tangency moves. In other words: if $f$ and $g$ are regular homotopic stable immersions, $h: S^3 \times [0, 1] \to \R^5 $ is a stable regular homotopy between them, then $\pm (L (f) -L (g))$ is equal to the algebraic number of triple values of the map $ H: S^3 \times [0, 1] \to \R^5 \times [0, 1]$, $H(x, t)=(h(x, t), t)$.
\end{prop}

\begin{proof}
Part (b) is proved independently for $L_1$ \cite[Lemma 6.2.1.]{ekholm3} and for $L_2=L_v$ \cite[Theorem 1.]{ekholm4}. Using them we prove part (a) as follows.
% stated for $L_v$. Hence part (b) holds for $L_1$ and $L_v=L_2$..
 
Since $ L_1 $ and $L_2 $ changes in the same (or opposite) way along a regular homotopy, $L_1 \pm L_2$ is a regular homotopy invariant. Moreover
$ L_1$ and $L_2 $ are additive under connected sum, see \cite[Lemma 5.2., Proposition 5.4.]{ekholm4}, \cite[6.5.]{ekholm3}. It follows that $L_1 \pm L_2$ defines a homomorphism from $ \imm (S^3, \R^5 ) $ to $ \Z$. If $f: S^3 \hookrightarrow \R^5 $ is an embedding, then $ L_1 (f) = L_2 (f) = 0 $, hence $ L_1 \pm L_2 $ is $0$ on the $24$-index subgroup $ \emb (S^3, \R^5 )$ of $ \imm ( S^3, \R^5 ) \cong \Z$. It follows that $L_1 \pm L_2 $ is $0$ for any stable immersion. 
\end{proof}

\begin{rem}\label{re:notstabL}
$L$ can be defined also for nonstable immersions which do not have triple values. Any immersion $f$ admits a small perturbation by regular homotopy to a stable immersion $ \tilde{f}$, and if $f$ does not have triple values, then any two stable perturbations can be joined with a regular homotopy without stepping through triple point. Thus $L(f)$ can be defined as $L( \tilde{f}) $ of any small stable perturbation $ \tilde{f}$ of $f$.
\end{rem}

\begin{rem}\label{re:St}
Ekholm also defined the invariant $ \mbox{St} (f) = (\Omega (f) + L(f))/3 $ for stable immersions $f: S^3 \looparrowright \R^5$, which is the analogue of Arnold's `strangeness' for plane curve immersions \cite[2.2., 6.5.]{ekholm3}. $ \mbox{St}$ changes by $ \pm 1$ under triple point moves, it is additive with respect to connected sum, and changes sign under precomposing an immersion with a reflection in the source. For immersions associated with holomorphic germs we express $ L$ and $ \mbox{St} $ in terms of $ C$ and $ T$, see Subsection~\ref{ss:LSt}.
\end{rem}

\subsection{Ekholm--Sz\H{u}cs formulas}\label{s:ss}

Ekholm and Sz\H{u}cs generalized the Hughes--Melvin formula \ref{th:HM} for immersions
via generic \emph{singular Seifert surfaces},
in two different ways:   mapped either in $\R^5$ or in $\R^6_+$ \cite{ESz},
 see also
 \cite{Esz2,saeki}.

%\begin{nota}
% In this section and in the calculations in Section~\ref{s:calc} we follow the notation introduced
%in \cite[Definition 2.]{ESz}.  For a map between two smooth manifolds $\tilde{D} $ and
%$ \tilde{\Sigma} $ denote the set of the multiple points and the set of the singular points in the
%source and for the multiple and singular values in the target we use $ D $ and $ \Sigma $.
%\end{nota}

 If $M^4$ is a compact oriented $4$-manifold and $ g: M^4 \to \R^5 $ is a stable $ \mathcal{C}^{\infty} $ map, then
  $g$ has isolated \emph{$ \Sigma^{1, 1} $--points (cusps)}. These are the singular points of $g$ restricted to its singular locus $ \Sigma^1 \subset M^4 $. Each cusp point is endowed with a well--defined sign.
  Let $\#\Sigma^{1,1}(g)$ be their algebraic number (cf.  \cite{ESz}).

\begin{thm}[Ekholm, Sz\H{u}cs \cite{ESz}] Let $ f: { S}^3 \looparrowright \R^5 $ be an immersion
and $ M^4 $ be a compact oriented $4$-manifold with boundary ${  S}^3 $. Let
$ \tilde{f}: M^4 \to \R^5 $ be a generic map such that
$ \tilde{f}|_{ \partial M^4} $ is regular homotopic to $ f $ and $ \tilde{f}$ has no singular points
near the boundary. Then %the Smale invariant of $ f $ can be expressed in the form
\begin{equation}\label{eq:cuspos}
\Omega (f) = \pm \frac{1}{2} (3 \sigma (M^4) + \# \Sigma^{1, 1} (\tilde{f}))  \mbox{ .} \end{equation}
% where $ \# \Sigma^{1, 1} (\tilde{f}) $ denotes the algebraic number of the $ \Sigma^{1, 1}
%$-points of $ \tilde{f} $.
\end{thm}

The second formula uses stable $ \mathcal{C}^{\infty} $ maps $ g: M^4 \to \R^6 $ defined on  compact oriented $4$-manifolds $M^4$.
It involves two topological invariants associated with such a map.
Next we review their definitions. They will be computed for two concrete holomorphic  maps
in order to identify the missing sign in Section~\ref{s:eszcomp}.

 If $g$ is as above, then it
 has isolated triple values (three local sheets of $ M^4$ intersecting in general position).
 Such a point is endowed with a well--defined sign \cite[2.3]{ESz}).
\begin{defn}[\cite{ESz}]
 $t(g) $ denotes the algebraic number of the triple values of $g$.
\end{defn}

Next, assume that $ \partial{M^4} = { S}^3$ and
$g: (M^4, \partial M^4) \to (\R^6_+, \partial \R^6_+ )$
is generic, it is  nonsingular near the boundary, and  $ \tilde{f}^{-1}(
\partial \R^6_+ )= \partial M^4 $.
 Here  $\R^6_+ $ is the closed half--space of $\R^6$.
 The set of  double values of $ g$ is an immersed oriented $2$-manifold,  denoted by $D(g)$.
 Its oriented boundary consists of two parts, the intersection
  of $D(g)\cap \partial \R^6_+$, and the other,
  disjoint with $\partial \R^6_+$, is the set of singular values $\Sigma(g)$ of $ g $.
 % whose set $ \Sigma(g) $ is an oriented $1$-manifold.
 Let $ \Sigma'(g) $ be a copy of $ \Sigma(g) $ shifted slightly along the outward normal vector field
 of $ \Sigma(g) $ in $D(g)$. Then $ \Sigma'(g) \cap g(M^4) = \emptyset $.

\begin{defn}[\cite{ESz}]\label{d:l}
 $ l(g) $ denotes the linking number of $g(M^4) $ and $ \Sigma'(g) $ in $ (\R^6_+, \partial \R^6_+ ) $.
\end{defn}

\begin{thm}[Ekholm, Sz\H{u}cs \cite{ESz}]\label{th:ESz}
 Let $ f: { S}^3 \looparrowright \R^5 $ be an immersion and $ M^4 $ be a compact oriented $4$-manifold with boundary $ \partial M^4 = { S}^3 $. Let $ \tilde{f}: (M^4, \partial M^4) \to (\R^6_+, \partial \R^6_+ )$ be a generic map nonsingular near the boundary, such that $ \tilde{f}^{-1}( \partial \R^6_+ )= \partial M^4 $ and $ \tilde{f}|_{\partial M^4} $ is regular homotopic to $ f $. Then
 \begin{equation}\label{eq:hurk}
 \Omega (f) = \pm \frac{1}{2} (3 \sigma (M^4) + 3 t(\tilde{f}) - 3 l(\tilde{f}) +
 L(\tilde{f}|_{\partial M^4} )) . \end{equation}
\end{thm}

\section{Outline of other related results}

\subsection{Immersions of $ S^3 $ to $ \R^4 $}\label{ss:s3r4}

Here we review a formula for the Smale invariant for immersions of $ S^3 $ to $ \R^4 $ \cite{hughes, EkTak}. For some connections with our cases see Subsection~\ref{ss:implumb} and Example~\ref{ex:quot}.
%Some of its applications relate to our examples, cf. Subsection~\ref{ss:implumb} and Example~\ref{ex:quot}.

The Smale invariant of an immersion $f: S^3 \looparrowright \R^4 $ is an element of 
$ \pi_3 ( V_3 ( \R^4 )) = \pi_3 (SO(4)) \cong \Z \oplus \Z $. The description of the generators of $ \pi_3 (SO(4)) \cong \Z \oplus \Z $ (usually denoted by $ \sigma $ and $ \rho $) can be found in \cite{steenrod} or \cite{EkTak}, and also in Section~\ref{s:signconv} of this thesis (where the first generator is denoted by $L$ instead of $ \sigma $).

The standard embedding $ i: \R^4 \hookrightarrow \R^5 $ induces an embedding $ \iota: SO(4) \hookrightarrow SO(5) $, and so a homomorphism $ \pi_3 ( \iota ): \pi_3 ( SO(4)) \to \pi_3 (SO(5)) $. This homomorphism is $  \pi_3 ( \iota ): \Z \oplus \Z \to \Z $, $ (a, b) \mapsto a+2b $ (cf. \cite{steenrod, EkTak} and Section~\ref{s:signconv}). Thus for an immersion $ f:  S^3 \looparrowright \R^4 $ with $ \Omega(f) = (a, b) $, the Smale invariant of the composition $ i \circ f: S^3 \looparrowright \R^5 $ is $ \Omega ( i \circ f) = a+2b $.

Hughes \cite{hughes} expressed the Smale invariant in terms of an immersed Seifert surface as follows. 
%surface  \cite[Theorem 3.1.]{hughes}.
\begin{thm}[{\cite[Theorem 3.1.]{hughes}}]\label{th:hughes} 
Let $f: S^3 \looparrowright \R^4 $ be an immersion which admits an immersed Seifert surface, i.e. there is a compact oriented $4$-manifold $ M^4 $ with boundary $ \partial M^4 \simeq S^3 $ and an immersion $ F: M^4 \looparrowright \R^4 $ such that $ F|_{ \partial M^4} = f $. Then
\begin{equation}\labelpar{eq:hughes}
\Omega(f) = \left( \chi (M^4 ) -1, \frac{3 \sigma ( M^4)-2( \chi(M^4)-1)}{4}  \right) \in \Z \oplus \Z \mbox{,}
\end{equation}
where $ \chi (M^4) $ is the Euler characteristic, and $ \sigma (M^4) $ is the signature of $ M^4 $.
\end{thm}

An immersion $f: S^3 \looparrowright \R^4 $ induces a stable framing of $ S^3 $ via the bundle isomorphism $ \mathcal{E}^1 \oplus T S^3 \cong f^* T \R^4 $, where $ \mathcal{E}^1 $ denotes the trivial line bundle. Indeed, the Smale invariant measures the difference of this stable framing and the standard one (induced by the standard embedding). By \cite{KirbyMelvin} the homotopy class of the stable framing is characterized by two integers, the normal degree (also called the degree) $D(f) $ of $f$ and the \emph{Hirzebruch defect} (also called signature defect) $H(f)$. See for definitions \cite{KirbyMelvin, Du}. $D(f) $ is described in detail in Example~\ref{ex:snsq}. Here we review some properties of $H(f) $ which can be used as an alternative definition in our cases. When the stable framing extends to a framing of a compact oriented $4$-manifold $M^4 $ with boundary $ \partial M^4 \simeq S^3 $, then $ D(f)= \chi (M^4 ) $ and $ H(f) = - 3 \sigma (M^4 )$, thus (\ref{eq:hughes}) can be written in the form
\begin{equation}\labelpar{eq:hughes2}
\Omega(f) = \left( D(f) -1, \frac{-H(f)-2( D(f)-1)}{4}  \right) \in \Z \oplus \Z \mbox{.}
\end{equation}
As it is proved in \cite{EkTak}, the Smale invariant formula (\ref{eq:hughes2}) holds for all immersions $ f: S^3 \looparrowright \R^4 $.

Note that the formulas (\ref{eq:hughes}) and (\ref{eq:hughes2}) are compatible with (\ref{eq:hm}). Indeed, by (\ref{eq:hughes}) and (\ref{eq:hughes2}) the Smale invariant of 
$ i \circ f: S^3 \looparrowright \R^5 $ is $ \Omega( i \circ f ) = - \frac{1}{2} H(f) =  \frac{3}{2} \sigma (M^4 ) $, where $M^4 $ is an immersed Seifert surface. See \cite{EkTak, szucstwo}.

For an arbitrary immersion $ f: S^3 \looparrowright \R^4 $ the Hirzebruch defect $H(f) $ can be expressed in terms of a singular Seifert surface, i.e., a stable map $ F: M^4 \to \R^4 $ from a compact oriented $4$-manifold with boundary $ \partial M^4 \simeq S^3 $, such that $ F|_{\partial M^4 } $ is regular homotopic with $f $ and $ F$ is regular near the boundary. Such a map $F$ has isolated $ \Sigma^{2,0} $-points, also called \emph{umbilic points}, i.e. the points where the rank of $dF$ is $2$, each of them can be equipped with a sign. Then
\begin{equation}\labelpar{eq:umb}
H(f) = -3 \sigma (M^4 ) - \sharp \Sigma^{2,0}(F) \mbox{,}
\end{equation}
where $ \sharp \Sigma^{2,0}(F)$ is the algebraic number of the $ \Sigma^{2,0} $-points. See \cite{tak, EkTak}. In \cite[Theorem 2.7]{EkTak} the formula~(\ref{eq:umb}) is generalized for non-stable maps with isolated umbilic points. In this case each umbilic point can be equipped with an index, and if $ \sharp \Sigma^{2,0}(F) $ denotes the sum of the indices, (\ref{eq:umb}) holds in the same form.

\subsection{Immersions associated with plumbing graphs}\label{ss:implumb} Ekholm and Takase \cite{EkTak} and Kinjo \cite{kinjo} defined immersions associated with plumbing graphs, and calculated their Smale invariant by using the formulae (\ref{eq:hughes2}) and (\ref{eq:umb}). Here we summarize their results.

Let $ \Gamma $ be a plumbing graph with all genera $0$ and even Euler numbers, see Subsection~\ref{ss:plumbingcalc} or \cite{neumann1}. We denote by $M^3( \Gamma) $ (resp. $M^4( \Gamma) $) the $3$-manifold (resp. the $4$-manifold with boundary) associated with $ \Gamma $. Note that $ \partial M^4( \Gamma) =M^3 ( \Gamma)$. Then $ \Gamma $ determines an immersion 
$ f_{ \Gamma }: M^4( \Gamma ) \looparrowright \R^4  $, and also 
$ f_{ \Gamma }|_{M^3 ( \Gamma)}:  M^3( \Gamma ) \looparrowright \R^4$ in the following way \cite{EkTak, kinjo}.

If the algebraic number of the double values of a stable immersion $ S^2 \looparrowright \R^4 $ is $ n$, then the Euler number of its normal bundle is $(-2n)$. Therefore, for a vertex $v \in \Gamma $ with genus $g_v=0$ and Euler number $e_v=-2n$ one associates a stable immersion $ S^2 \looparrowright \R^4 $ with $n$ double values, and its tubular neighbourhood is the image of an immersion from the disc bundle over $ S^2 $ with Euler number $(-2n)$. Plumbing these immersions for each vertex according to $ \Gamma $ provides the immersion $  f_{ \Gamma } $.

If $ M^3( \Gamma ) $ is the quotient of $ S^3 $ by an action of a finite group $ G$, then the composition of the factorization
$ S^3 \to S^3/G = M^3( \Gamma ) $ and the immersion $ f_{ \Gamma }|_{M^3 ( \Gamma)}: M^3 ( \Gamma) \looparrowright \R^4 $ defines an immersion $ g_{ \Gamma }: S^3 \looparrowright \R^4 $.

In \cite{EkTak} $ \Gamma $ is the graph with one vertex and Euler number $ 2n $, and so $ M^3( \Gamma ) $ is a lens space. The Smale invariant of the corresponding immersion $ g_n: = g_{ \Gamma }: S^3 \looparrowright \R^4 $ is
\begin{equation}\labelpar{eq:ektaksmale}
\Omega (g_n)= (4n-1, (n-1)^2) \in \Z \oplus \Z \mbox{.}
\end{equation}

In \cite{kinjo} $ \Gamma $ is the weighted Dynkin diagram $ A_{n-1} $ (resp. $D_{n+2}$)  with Euler number $ 2 $ of each vertex. So $ M^3( \Gamma )=L(n, 1) $ is a lens space (resp. $ M^3( \Gamma )=S^3/D_n $, where $ D_n $ is the binary dihedral group or also called dicyclic group). The Smale invariant of the corresponding immersion $ a_n: = g_{ A_{n-1} }: S^3 \looparrowright \R^4 $ (resp. $ d_n: = g_{ D_{n+2} }: S^3 \looparrowright \R^4 $) is
\begin{equation}\labelpar{eq:kinjosmalea}
\Omega (a_n)= (n^2-1, 0) \in \Z \oplus \Z \mbox{, resp.}
\end{equation}
\begin{equation}\labelpar{eq:kinjosmaled}
\Omega (d_n)= (4n^2+12n-1, 0) \in \Z \oplus \Z \mbox{.}
\end{equation}

%The immersions associated with the covering maps of the $A_{n-1}$ and $D_{n+2}$ quotient singularities have Smale invariant $ -(n^2-1)$, resp. $ -(4n^2+12n-1)$, see Example~\ref{ex:quot}.  Therefore they are regular homotopic with the immersions $i \circ a_n \circ \tau $, resp. $i \circ d_n \circ \tau $ of $ S^3 $ to $ \R^5 $, where $i: \R^4 \hookrightarrow \R^5 $ is the standard embedding, and $ \tau: \R^4 \to \R^4 $ is the reflection $ \tau (x, y, z, w) = (x, y, z, -w)$. 

For any plumbing graph $ \Gamma $, the composition $ i \circ f_{ \Gamma }|_{M^3 ( \Gamma)}:  M^3 ( \Gamma) \looparrowright \R^5 $ (where $i: \R^4 \hookrightarrow \R^5 $ is the standard embedding) can be deformed to an embedding by a regular homotopy in $ \R^5 $. Indeed, the composition $ S^2 \looparrowright \R^4 \hookrightarrow \R^5 $ can be deformed to a embedding by a regular homotopy, and this induces a regular homotopy of the tubular neighbourhood. Thus the immersions $ i \circ a_n$ and $ i \circ d_n $ have the same structure as the immersions associated with the covering maps $ ( \C^2, 0) \to (X, 0) \subset ( \C^3, 0) $ of the $A_{n-1}$ and $D_{n+2}$ singularities, namely $ S^3 \to S^3/G \hookrightarrow \R^5 $. 
%It would be interesting to see a direct relation between them.

The immersions associated with the covering maps of the $A_{n-1}$ and $D_{n+2}$ quotient singularities have Smale invariant $ (-(n^2-1))$, resp. $ (-(4n^2+12n-1))$, see Example~\ref{ex:quot}.  Therefore they are regular homotopic either with the immersion  $i \circ a_n $ (resp. $i \circ d_n  $) or $i \circ a_n \circ \tau $ (resp. $i \circ d_n \circ \tau $) from $ S^3 $ to $ \R^5 $, where $i: \R^4 \hookrightarrow \R^5 $ is the standard embedding, and $ \tau: \R^4 \to \R^4 $ is the reflection $ \tau (x, y, z, w) = (x, y, z, -w)$. Note that precomposing with $ \tau $ changes the sign of the Smale invariant. But there are other differences between the two types of immersions as well, which may change the sign. In our construction $ S^3/G $ is the lens space $ L(n, n-1) $, which is diffeomorphic to $ L(n, 1)$ with reversed orientation. Indeed, the plumbing graph of $ S^3/G $ in our case is $ A_{n-1}$ (resp. $D_{n+1}$) with Euler numbers $(-2)$. Moreover the author of this thesis does not know the relation of the sign conventions for the Smale invariant used in \cite{kinjo} and here, see Section~\ref{s:signconv}. It would be interesting to see a direct relation between these immersions, which applied to the $ E_6 $, $ E_7$, $E_8$ graphs may provide the Smale invariant of the immersions associated with these graphs.

The immersion  $ i \circ g_1 =i \circ a_2: S^3 \to \R \mathbb{P}^3 \hookrightarrow \R^5 $ was studied in several papers, e.g. \cite{Mimm, ekholm3}, and it has many connections with our work. Ekholm's invariant $L(i \circ g_1)$ is well-defined by Remark~\ref{re:notstabL}, as $L$ of a small stable perturbation of $i \circ g_1 $ with regular homotopy. 
A regular homotopy deforming $i \circ g_1$ to a stable immersion is given in \cite[Section 4]{EkTak}, 
%In a small enough neighborhood of the immersion the linking invariant $L$ (see subsection~\ref{s:L}) is independent of the choice of a stable representative, because $ i \circ g_1 $ has no triple points, cf. Remark~\ref{re:notstabL}
and by the calculation from that article $ L (i \circ g_1) =0 $ follows. Another stabilizing regular homotpy of $ i \circ g_1 $ (or $ i \circ g_1 \circ \tau $)  is given in Subsection~\ref{ss:A} via the stabilization of the double covering map of the $A_1$ singularity, and $L(i \circ g_1)=0$ follows too. Note that the number of the double curve components (Ekholm's $J$ invariant) is different in the two cases: the stable immersion given in \cite{EkTak} has one double curve component, while the one given in Subsection~\ref{ss:A} has $3$. The immersion associated with the corank--$2$ germ of Example~\ref{ex:cor2} can also be considered as a stable version of $ i \circ g_1 ( \circ \tau )$ with $5$ double curve components, see Subsection~\ref{ss:2cor}. Indeed, $ C( \Phi ) =3$ and $ T( \Phi )= 1$ implies $ \Omega(\Phi|_{S^3} )= -3 $ and $ L ( \Phi|_{S^3} ) = 0 $, cf. \ref{ss:LSt}, \ref{TH:MAIN}.

\subsection{Immersions of $3$-manifolds}\label{ss:m3r45}

This subsection contains a summary of the immersions of arbitrary $3$-manifolds. We mention here some results about their characterization up to regular homotopy according to \cite{Wu, saeki}, and up to cobordism according to \cite{hughes, tak}.

Since every closed oriented $3$-manifold is parallelizable, a choice of a trivialization $\tau=(v_1, v_2, v_3) $ of the tangent bundle $ TM^3$ determines a bijection between $ \imm (M^3, \R^q )$ and $ [M^3, V_3 ( \R^q) ]$, which associates with an immersion $ f: M^3 \looparrowright \R^q $ the map 
$ p \mapsto df_p (\tau) $, cf. Theorem~\ref{th:hirsch2}. If $q=4$, 
\[ [M^3, V_3 ( \R^4) ] \cong [M^3, SO(4) ] \cong [M^3, S^3 \times SO(3) ]  
\cong [M^3, S^3 ] \times [M^3, SO(3) ]\] 
holds.
Note that the degree (resp. the degree and the induced homomorphism between the fundamental groups) completely characterises the maps from $M^3 $ to $S^3 $ (resp. the maps from $M^3 $ to $SO(3) $) up to homotopy.

Consider an immersion $ f: M^3 \looparrowright \R^4 $. By \cite{Wu} the corresponding invariant in $ \Z \times \Z \times H^1 (M^3, \Z_2 ) $ can be identified as follows. 

The orientation determines a normal framing $u$ of the immersion $f$. The first integer is the normal degree $D(f)$ of $f$, i.e. the degree of the Gauss map $M^3 \to S^3$, $ p \mapsto u(p)$, cf. Example~\ref{ex:snsq} and (\ref{eq:hughes}).

%Let $(v_1, v_2, v_3, v_4) $ be the linear independent vectofields determine the chosen trivialization, and let $u_0$ be the normal vector field of the immersion $ f: M^3 \looparrowright \R^4 $. Then the map $ M^3 \to GL^+(4, \R) $ corresponding to $f$ is given by $ p \mapsto (u_0(p), df_p (v_i(p))) $.  

 In $ \R^4 \cong \mathbb{H} $ any vector $v$ determines a basis $ (v, iv, jv, kv)$, where $ (i, j, k) $ is the canonical basis of the pure imaginary quaternions. $ df_p (\tau) $ and $ ( i u(p), j u (p), k u (p) ) $ are two bases of $ u^{\bot} \subset \R^4 $, and the transition matrix between them defines a map $ M^3 \to GL^+(3, \R) $, which composed with the Gram-Schmidt process $ GL^+(3, \R) \to SO(3) $ provides the map $ \beta_f: M^3 \to SO(3)$. The degree $b_f$ of $ \beta_f $ is the second integer. The third invariant is the cohomology class $h_f$ induced by the homomorphism $ \pi_1 ( \beta_f): \pi_1 ( M^3) \to \pi_1 (SO(3)) \cong \Z_2 $.  
 
 \begin{thm}[{\cite[Theorem 1]{Wu}}] The map $ \imm (M^3, \R^4 ) \to \Z \times \Z \times H^1 (M^3, \Z_2 )$ induced by the correspondence $ f \mapsto (D(f), b_f, h_f) $ is a bijection.
 \end{thm}
 
 \vspace{2mm}
 
 In the $ q= 5$ case let us fix a generator $ g \in H^2 (V_3 (\R^5 ), \Z)$ such that $2g$ is the Euler-class of the $S^1$-bundle $ SO(5) \to V_3 ( \R^5) $. Then the \emph{Wu-invariant} of an immersion $ f: M^3 \looparrowright \R^5 $ with respect to the trivialization $\tau= (v_1, v_2, v_3)$ of $TM^3 $ is the element 
 $ c(f) := (df(\tau))^*(g) \in H^2(M^3, \Z )$, cf. \cite[Definition 3.3.]{saeki}. The normal Euler class of $f$ is equal to $ 2c(f) $ by \cite[Theorem 2]{Wu}, see also \cite[Theorem 3.1., Remark 3.2]{saeki}.
 
 Now we restrict the discussion to the regular homotopy classes of immersions with trivial normal bundle, denoted by $ \imm (M^3, \R^5)_0$. The normal Euler class of such an immersion $ f$ is $0$, hence $c(f) $ is an order--$2$ element in $ H^2(M^3, \Z)$. Let $\Gamma_2 (M^3) \subset H^2(M^3, \Z) $ denote the set of the order--$2$ elements. Let $M^3_o= M^3 \setminus B^3 $ denote the non-closed manifold obtained from $M^3$ by removing a $3$-ball.
 
\begin{thm}

(a) \cite[pg. 5]{saeki} $ f \mapsto c(f) $ induces a bijection between $ \imm ( M^3_o, \R^5)_0 $ and $\Gamma_2 (M^3)$.

(b) \cite[pg. 5, 9]{saeki} There is an integer valued regular homotopy invariant $i(f) $ such that the correspondence $ f \mapsto (c(f), i(f))$ induces
 a bijection between $ \imm ( M^3, \R^5)_0  $ and $ \Gamma_2 (M^3) \times \Z$.

(c) \cite[Theorem 3.8.]{saeki} For every element $C \in \Gamma_2( M^3) $ there exists an embedding $ g: M^3 \hookrightarrow \R^5 $ with $c(g)= C$.

\end{thm}

The invariant $i(f)$ can be defined via a singular Seifert surface, it is the analogue of the Smale invariant via formulae (\ref{eq:cuspos}) and (\ref{eq:hurk}).

Introduce the notation $ \alpha (M^3) := \dim_{\Z_2} (\tau H_1 (M^3, \Z) \otimes \Z_2) $ where $ \tau H_1 (M^3, \Z)$ is the torsion subgroup of $ H_1 (M^3, \Z) $. Any immersion $f: M^3 \looparrowright \R^5 $ admits a normal framing $\nu$, however it is not unique up to homotopy. Define $ L_{ \nu}(g) $ as (\ref{eq:L2}) for a stable immersion $g: M^3 \looparrowright \R^5 $ with a fixed normal framing $ \nu$.

 Let $ V^4 $ be a compact oriented $4$-manifold with boundary $M^3 $.
 Let
$ \tilde{f}: V^4 \to \R^5 $ be a generic map such that
$ \tilde{f}|_{ \partial V^4} $ is regular homotopic to $ f: M^3 \looparrowright \R^5 $ and $ \tilde{f}$ has no singular points
near the boundary. Define 
\begin{equation}\label{eq:Mcuspos}
i_a (f) =  \frac{3}{2} ( \sigma (V^4)- \alpha (M^3)) + \frac{1}{2} \# \Sigma^{1, 1} (\tilde{f})  \mbox{ .} 
\end{equation}

Let $ \bar{f}: (V^4, \partial V^4) \to (\R^6_+, \partial \R^6_+ )$ be a generic map nonsingular near the boundary, such that $ \bar{f}^{-1}( \partial \R^6_+ )= \partial V^4 $ and $ \bar{f}|_{\partial V^4} $ is regular homotopic to $ f $. Define
 \begin{equation}\label{eq:Mhurk}
 i_b (f) =  \frac{3}{2} ( \sigma (V^4) - \alpha (M^3)) +  \frac{1}{2} (3 t(\bar{f}) - 3 l(\bar{f}) +
 L_{\nu}(\bar{f}|_{\partial V^4} )) . \end{equation}

\begin{thm}[Lemmas 5.5., 5.7., theorems 5.6., 5.8. in \cite{saeki}] 
%\newline

(a) $ i_a (f) $ is an integer, and it does not depend on the choice of $ V^4 $ and $ \tilde{f}$. It depends only on the regular homotopy class of $ f$.

(b) $ i_b (f) $ is an integer, and it does not depend on the choice of $ V^4 $, $ \bar{f}$ and $ \nu$. It depends only on the regular homotopy class of $ f$.

(c) $ i_a (f) =  i_b (f) $. 

(d) Two immersions $ f, g: M^3 \looparrowright \R^5 $ with $ c(f)=c(g) $ are regular homotopic if and only if $ i(f) =i(g)$, where $i(f):= i_a(f)=i_b (f)$.
\end{thm}

\vspace{2mm}

Cobordism of immersions is an equivalence relation between immersions of different manifolds. Two immersions $ f: M^n \looparrowright \R^q $ and $ g: N^n \looparrowright \R^q $ of closed (resp. closed, oriented) manifolds $ M^n $ and $ N^n $ are \emph{cobordant} (resp. \emph{oriented cobordant}) if there is a compact manifold $ V^{n+1} $ with boundary $ \partial V^{n+1} \simeq M^n \sqcup N^n $ (resp. a compact oriented manifold $ V^{n+1} $ with oriented boundary $ \partial V^{n+1} \simeq M^n \sqcup - N^n $, where $-N^n $ denotes $N^n $ with opposite orientation) and an immersion $ F: V^{n+1} \looparrowright \R^q \times [0, 1] $ such that $ F|_{M^n} = f $ and $ F|_{N^n}= g$. Let $ \imm (n, k) $ (reps. 
$ \imm^{SO} (n, k) $) denotes the cobordism classes of the immersions of $n$-manifolds to $ \R^{n+k} $ (resp. the oriented cobordism classes of the immersions of oriented $n$-manifolds to $ \R^{n+k} $). $ \imm (n, k) $ and
$ \imm^{SO} (n, k) $ form groups with respect to the disjoint union, which agrees with the connected sum up to (oriented) cobordism.

$ \imm^{SO} (3, 1) \cong \pi^s(3) \cong \Z_{24}$, where $ \pi^s(3) $ denotes the third stable homotopy group of spheres  \cite[Lemma 1.4.]{szucstwo}, \cite{tak}. 
%The first isomorphism follows from the Pontrjagin construction and \cite[Lemma 1.4.]{szucstwo}. The second isomorphism is a theorem of Rokhlin, see \cite{szucstwo}, \cite{tak}. 
Each class in $ \imm^{SO} (3, 1) $ can be represented by an immersion of $ S^3 $ via the following diagram, cf. \cite[Lemma 1.8.]{szucstwo}, \cite[pg. 41, 43]{tak}.
\begin{equation}
\begin{array}{ccccc}
\imm(S^3, \R^4) & \to & \imm( S^3, \R^5) & \stackrel{J}{\rightarrow} & \imm^{SO}(3, 1) \\
\|             &       &  \|             &      &  \| \\
\pi_3(SO(4))   &  \to  &   \pi_3(SO(5))   & \stackrel{J}{\rightarrow} & \pi^s(3) \\
\|             &       &  \|             &      &  \| \\
\Z \oplus \Z    &  \to &   \Z           &    \to  &  \Z_{24} \\
(a, b)          &  \mapsto  &  a+2b     &  \mapsto & a+2b \ (\mbox{mod } 24) \\    
\end{array} 
\end{equation}
The homomorphism on the right of the diagram is called \emph{$J$-homomorphism}, its kernel is $ \emb( S^3, \R^5)$, cf. Theorem~\ref{th:HM}, \cite{HM, szucstwo}. There are also geometric formulae to determine the oriented cobordism class of an immersion $ M^3 \looparrowright \R^4 $, see \cite[Theorem 3.1., Remark 3.6.]{tak}.

By \cite[Theorem 4.1.]{hughes} $ \imm(3, 2) \cong \imm^{SO}(3, 2) \cong \Z_2 $, and the cobordism class of a stable immersion of $ M^3 $ to $ \R^5 $ is completely determined by its \emph{total twist}, which is the parity of the number of non-trivially covered double point curve components. In particular, the total twist of a stable immersion $ S^3 \looparrowright \R^5 $ is a regular homotopy invariant, it defines a homomorphism $ \imm (S^3, \R^5 ) \to \Z_2$. From this fact we conclude a correspondence between analytic invariants in Subsection~\ref{ss:compllinks}.

\subsection{Regular homotopy class of the embedded link}

We mention here two results about a very similar topic to the material discussed in Chapter~\ref{ch:ass}. The articles \cite{Esz2} and  \cite{katanaga} study links of isolated hypersurface singularities (cf. Chapter~\ref{ch:iso}) up to regular homotopy.

Although \cite{Esz2} contains a generalization of the Hughes-Melvin theorem \ref{th:HMemb} and of  singular Seifert surface formulae (\ref{eq:cuspos}) and (\ref{eq:hurk}) for immersions and embeddings of homotopy $(4k-1)$-spheres, here we only concentrate to \cite[Section 5.1.]{Esz2}. Consider the Brieskorn equations
$f_k(z)= z_1^{6k-1} + z_2^3 + z_3^2 + z_4^2 + z_5^2 $, which determine isolated singularites $ (X_k, 0) = f_k^{-1}(0) \subset ( \C^5, 0) $. Their links $ K_k = X \cap S^9_{ \epsilon} \simeq \Sigma_k^7 $ are \emph{homotopy spheres} (also called exotic spheres): $ \Sigma_k^7 $ is homeomorphic, but not diffeomorphic to $ S^7 $ for all $ k \neq 28n $. 
$ \Sigma_{28}^7 $ is diffeomorphic to $ S^7 $, and $ \Sigma_k^7 $ is diffeomorphic to $ \Sigma_l^7 $ if and only if $ k \equiv l \ (\mbox{mod } 28) $.

By \cite{Esz2}, the embeddings of $ \Sigma_{28n+k}^7 $ into $ S^9 $ are not regular homotopic for different values  of $n$, and they represent all the regular homotopy classes of the immersions which contain embeddings, i.e. $ \emb ( \Sigma_k^7 , S^9 ) \subset \imm ( \Sigma_k^7 , S^9 )$. In particular, $k=28n$ provide embeddings $ S^7 \hookrightarrow S^9 $ which are not regular homotopic to the standard embedding, each appears as the link of a complex hypersuface singularity. In contrast with this fact, nonstandard embeddings $ S^3 \hookrightarrow \R^5 $ cannot be realized as links of suface singularities in $ \C^3 $, cf. Theorem~\ref{TH:EMBINTRO}.

\vspace{2mm}

A similar question is studied in \cite{katanaga} about the link of $ (X_k, 0) = f_k^{-1}(0) \subset ( \C^4, 0) $, where 
$ f_k(x, y, z, w) = x^2 + y^2 + z^2 + w^k $. Its link depends on the parity of $k$, that is
\[
K_k = X_k \cap S^7_{ \epsilon} \simeq
\left\{ \begin{array}{ccc}
S^5 & \mbox{if} & k=2d+1 \mbox{,} \\
S^3 \times S^2 & \mbox{if} & k=2d \mbox{.} \\
\end{array}  \right. 
\]
The question is that in the case of links with different equations but with the same diffeomorphism type do they represent different regular homotopy classes? 

For $ k$ odd, $ \imm( S^5, S^7) = \pi_5 (SO(7))= 0$ (see (\ref{eq:Smale})) by Bott periodicity \cite{bottper}. Thus the question is irrelevant in this case.

In the other case $ k=2d $ it turns out that the question is not well-posed: precomposing an immersion $ f: S^3 \times S^2 \looparrowright S^7 $ with a self-diffeomorphism of $ S^3 \times S^2 $ can change the regular homotopy type of $f$. In other words the regular homotopy class of the inclusion $ K_{2d} \subset S^7 $ depends on the choice of the identification of $ K_{2d} $ with $ S^3 \times S^2 $.

A well-defined notion is the \emph{image regular homotopy type}: the immersions $f$ and $g$ $S^3 \times S^2 \looparrowright S^7$ are image regular homotopic if there is a self diffeomorphism $ \phi$ of $ S^3 \times S^2 $ such that $ g$ is regular homotopic with $ f \circ \phi $. By \cite{katanaga} there are two image regular homotopy classes of immersions $S^3 \times S^2 \looparrowright S^7$, and the image regular homotopy class of the inclusion $ K_{2d} \subset S^7 $ depends on the parity of $d$.

\chapter{Immersions associated with holomorphic germs}\label{ch:ass}

\section{Summary of the results} 

\subsection{} This chapter includes the results of the paper \cite{NP}. 
Our main goal is to analyse the \emph{complex analytic} realizations of the elements of the groups $ \imm(S^3, S^5) $ and $ \emb(S^3, S^5)$. Recall that $ \imm(S^3, S^5) \cong \pi_3 (SO(5)) \cong \Z$ via the Smale invariant. We use (\ref{eq:Smale}) for its definition throughout Chapter~\ref{ch:ass}.

Let $ \Phi: ( \C^2, 0) \to ( \C^3, 0) $ be a holomorphic germ. We assume that $ \Phi $ is singular only at the origin, that is  $ \{z\,:\, \rk (d \Phi_z)< 2\} \subset \{0\}$
in a small representative of $(\C^2,0)$. Such a germ, at the level of links of the spacegerms
$(\C^2,0)$ and $(\C^3,0)$, provides
an immersion $ f: S^3 \looparrowright S^5 $ , cf. Definition~\ref{de:linkmap} and Theorem~\ref{th:Csum}.
If an element of ${\imm}(S^3,S^5)$, or ${\emb}(S^3, S^5)$ respectively,
can be realized (up to regular homotopy)
by such an immersion,  we call it {\it holomorphic}. The corresponding subsets
will be denoted by ${\rm Imm}_{hol}(S^3,S^5)$ and  ${\rm Emb}_{hol}(S^3, S^5)$ respectively.

As we will see, ${\rm Imm}_{hol}(S^3, S^5)$ is not symmetric with respect to a sign change of $\Z$,
hence, in order to identify the subset ${\rm Imm}_{hol}(S^3,S^5)$ without any sign-ambiguity,
we will fix a `canonical' generator of  $ \pi_3 (SO(5)) $. This will be  done via the
ismorphisms  $ \pi_3 (U(3))\to \pi_3(SO(6))\to \pi_3(SO(5))$ and by fixing  a canonical generator in $ \pi_3 (U(3))$ (see \ref{ss:sign}).
Sometimes, to emphasize that we work with the Smale invariant with this fixed sign convention,
we refer to it as the {\it sign--refined Smale invariant}.
Our second goal is to determine the correct signs (compatibly with the above
choice of generators) in the Smale invariant formulas \ref{th:HM} and (\ref{eq:cuspos}), (\ref{eq:hurk}), which were stated only up to a
sign--ambiguity.

\subsection{The set ${\rm Imm}_{hol}(S^3,S^5)$}
One expects that the analytic geometry of  holomorphic
realization imposes some rigidity restrictions, and also provides some further connections with the
 properties of complex analytic spaces.
Mumford already in 1961 in his seminal article \cite{mumford} asked for the
characterization of the Smale invariant of a holomorphic (algebraic) immersion in terms of
the analytic/algebraic geometry. This chapter provides a complete answer to his question.
A more precise formulation of our guiding questions is:

\begin{ques} \label{question}

(a) Which are the regular homotopy classes ${\rm Imm}_{hol}(S^3,S^5)$ and  ${\rm Emb}_{hol}(S^3, S^5)$
 represented by holomorphic  germs?

(b) How can a certain  regular homotopy class be identified via complex singularity theory,
that is, via algebraic or analytic invariants of the involved analytic spaces?
Furthermore, if some $\Phi$ realizes some Smale invariant (e.g., if its Smale invariant is zero), then what kind of specific analytic properties $\Phi$  must have?
\end{ques}
The main results  of this chapter provide the following answer in the case of immersions.

\begin{thm}\label{TH:MAIN}

(a) ${\rm Imm}_{hol}(S^3,S^5)$ is identified via the sign--refined
Smale invariant $\Omega(f)$ by the set of non--positive integers.

(b) If the immersion $f$ is induced by the holomorphic germ $\Phi$,
then $ \Omega(f) = -C( \Phi)$, where  $C( \Phi) $ is
the number of complex Whitney umbrella points
of a stabilization of $ \Phi $. $ C( \Phi ) $ can be calculated in an algebraic way,
as the codimension of the ideal generated by the determinants of the $2\times 2$-minors
of the Jacobian matrix of $\Phi$. Cf. Subsection~\ref{ss:CTN}.
\end{thm}

The main tool of the proof of Theorem~\ref{TH:MAIN} is the concept of  \emph{complex Smale invariant}
of the germ $\Phi$. We introduce it in Section~\ref{s:co} and then we prove that it
agrees with $ C( \Phi) $. Next,  in Section~\ref{s:proof}
we identify the complex Smale invariant of a germ $ \Phi $ with the (classical) Smale invariant of the
 link of $ \Phi $. The proof of the part (b) of Theorem~\ref{TH:MAIN} is then ready up to sign. In \ref{ss:sign} we fix explicit generators of the groups $ \pi_3 (U) $ and $ \pi_3 (SO) $
 and calculate the homomorphism between them.
%which takes the complex Smale invariant of $ \Phi $
%to the Smale invariant of the link of it.
With this convention the complex Smale invariant of $ \Phi $ is equal to $ C( \Phi) $ and is opposite to
the sign--refined Smale invariant.

Part (b) of Theorem~\ref{TH:MAIN} implies that the sign--refined Smale invariant of a complex
analytic realization
is always non--positive. The proof of  part (a) is then completed by Example~\ref{ex:1}, which provides analytic  representatives for all non--positive $\Omega(f)$.

Note that in the present literature  the known ($ \mathcal{C}^{\infty}$)
 realizations of certain Smale invariants  $\Omega(f)$
 are rather involved (similarly, as the  computation of $\Omega(f)$ for any concrete $f$), see e.g. \cite{hughes, ekholm3}.
Here we provide very simple polynomial maps realizing all  non--positive Smale invariants.
Furthermore, the computation of $C(\Phi)$ for any $\Phi$ is extremely simple.

Moreover, precomposing the above complex realizations with the $ \mathcal{C}^{\infty}$ reflection
$ (s, t) \mapsto (s, \bar{t})$, we get explicit representatives for all positive Smale numbers
 as well, compare \cite[Lemma 3.4.2.]{ekholm3}. 
  In this way Theorem~\ref{TH:MAIN} together with Example~\ref{ex:1} provides an answer to Smale's question~\ref{qu:Smale}.
  
  \subsection{The set  ${\rm Emb}_{hol}(S^3,S^5)$.}
Recall that ${\rm Emb}_{hol}(S^3,S^5)$ consists of regular homotopy classes (that is,
sign--refined Smale invariants in $\Z$)
represented by holomorphic germs $\Phi$ whose induced  immersions $S^3\looparrowright S^5$
might not be embeddings, but are regular homotopic with embeddings.

A more restrictive subset consists of those regular homotopy classes (Smale invariants),
which can be represented by holomorphic gems, whose restrictions off origin are embeddings. Cf. Hughes--Melvin theorem~\ref{th:HMemb}.

\begin{thm}\label{TH:EMBINTRO}
(a)  ${\rm Emb}_{hol}(S^3,S^5)= (24\cdot \Z)\cap\Z_{\leq 0}$.

(b) Assume that the immersion $f$ is the restriction at links level of a
holomorphic germ $\Phi$ as above, $f=\Phi|_{S^3}$. Then the following facts are equivalent:

\begin{enumerate}
\item $ \rk (d\Phi_0)=2$ (hence $\Phi$ is not singular),
\item $\Omega (f)=0$,
\item $f:S^3\hookrightarrow S^5$ is an embedding,
\item $f: S^3 \hookrightarrow S^5 $ is the trivial embedding.
\end{enumerate}
\end{thm}
Again, we wish to emphasize that the previous
construction of the generator of $24\cdot \Z={\rm Emb}(S^3,S^5)$ (that is, of a smooth embedding
with $\Omega(f)=\pm 24$) is complicated, it is more existential than constructive \cite{HM}.
On the other hand, by our complex realizations, for any given $\Omega(f)\in 24\cdot \Z$
we provide several easily defined germs, which are immersions, and  are regular homotopic with embeddings.
Moreover, part (b) says that it is impossible to find holomorphic representatives $\Phi$ such that
$\Phi|_{S^3}$ is already embedding (except for $\Omega(f)=0$).

The essential  parts of Theorem \ref{TH:EMBINTRO}(b) are the implications
(2) $\Rightarrow$ (1) and  (3) $\Rightarrow$ (1), which conclude an analytic statement
from topological ones. The proof (2) $\Rightarrow$ (1) is based on Theorem \ref{TH:MAIN}, which
recovers the vanishing of the analytic invariant $C(\Phi)$ from the `topological vanishing' $\Omega(f)=0$.

A possible proof of
($\Phi|_{S^3}$ embedding) $\Rightarrow$ ($ \rk (d\Phi_0)=2$)
is based on a  deep theorem of Mumford,  which says that if the link of a complex normal surface
singularity is $S^3$ then the germ should be non--singular \cite{mumford}.
We will provide two other  possible proofs too: one of them is based on Mond's Theorem
\ref{th:CT}, the other on a theorem of Ekholm--Sz\H{u}cs \ref{th:ESz}.

\subsection{} The literature of singular analytic germs $\Phi:(\C^2,0)\to (\C^3,0)$ is huge with several deep and interesting results and invariants, see e.g. the articles of
D. Mond and V. Goryunov \cite{Gor1,Gor2,Gor3,Mond1,Mond2,Mondwh}
and the references therein, or Chapter~\ref{ch:germ} of this thesis. In singularity classifications
finitely determined or finite codimensional germs are central (with respect to some
equivalence relation). For germs $\Phi:(\C^2,0)\to (\C^3,0)$
Mond proved that the finite $\mathscr{A}$-determinacy is equivalent with the finiteness
 of three invariants $C(\Phi)$,
$T(\Phi)$ and $N(\Phi)$, see \cite{Mond2} or Subsection~\ref{ss:CTN}.
This is more restrictive than our assumption $ \{z\,:\, \rk (d \Phi_z)< 2\} \subset \{0\}$,
which requires the finiteness of $C(\Phi)$ only, cf. Theorem~\ref{th:Csum}.

However, it is advantageous   to consider this larger class, since there are
many key families of
germs  with infinite right-left codimension, but with finite $C(\Phi)$, and they
produce interesting connections with other areas as well (see e.g. the next example).

\begin{ex}\label{ex:ade} Consider a simple hypersurface singularity $(X,0)\subset (\C^3,0)$
(that is, of type A--D--E). They are quotient singularities, that is $(X,0)\simeq(\C^2,0)/G$
for certain finite subgroup $G\subset GL(2,\C)$, cf. examples \ref{ex:A}, \ref{ex:A2}, \ref{ex:D}, \ref{ex:E}. 
Let $K$ be the link of $(X,0)$ (e.g.,
it is a lens space for A-type), and consider
the regular $G$--covering $S^3\to K$.
This composed with the inclusion $K\hookrightarrow S^5$ provides an {\it immersion}
$S^3\looparrowright S^5$. Hence, the universal cover of each  A--D--E singularity
automatically provides an element
of ${\rm Imm}_{hol}(S^3,S^5)$, which usually have infinite right--left codimension.
The corresponding Smale invariants are given in Section~\ref{s:ex}. E.g.,
$-\Omega(A_{n-1})=n^2-1$, hence $A_4$ represents (up to regular homotopy) a generator of
$24\cdot \Z={\rm Emb}(S^3,S^5)$.

Recently Kinjo, using the plumbing graphs of the links of
A--D singularities and $\mathcal{C}^{\infty}$--techniques, constructed
immersions with the same Smale invariants
as our $-C(\Phi)$ up to sign. See \cite{kinjo} or the discussion in Subsection~\ref{ss:implumb}. Hence, the natural complex analytic
maps $(\C^2,0)\to (X,0)\subset (\C^3,0)$  provide analytic realizations of the $\mathcal{C}^{\infty}$
constructions of \cite{kinjo}, and emphasize their distinguished nature.
\end{ex}

\subsection{Smale invariants and the geometry of Seifert surfaces.}\label{ss:LSt}
In Section \ref{s:singseif} we reviewed three major topological  theorems,
which recover the classical Smale invariant in terms of the geometry of their (singular)
Seifert surfaces, namely the Hughes--Melvin Theorem~\ref{th:HM}, and two
 theorems of Ekholm--Sz\H{u}cs \cite{ESz}, (\ref{eq:cuspos}) and (\ref{eq:hurk}).
 All of them carry the sign ambiguity of the Smale invariant (which sometimes
is also caused by the nature of their proofs).

Section \ref{s:eszcomp} has two goals. First, we will indicate  the correct sign
in all these formulae, whenever the Smale invariant is replaced by the sign--refined
Smale invariant. Moreover, we also determine the Seifert type invariants
in terms of $C(\Phi)$ and $T(\Phi)$, whenever the immersion is induced by a holomorphic
germ $\Phi$.

When $f$ is a stable immersion, the invariant $L(f) $ of stable immersions introduced by Ekholm (cf. Subsection~\ref{ss:L}) is also expressed in terms of $C(\Phi)$ and $T(\Phi)$, namely $L(f)=
C(\Phi)-3T(\Phi)$. In other words, Ekholm's `strangeness' invariant $ \mbox{St} (f) $ is equal to $ -T( \Phi )$, see Remark~\ref{re:St}.

From this point of view our results can be considered as \emph{complex singular Seifert surface} formulae expressing the invariants of an immersion. The disentanglement (see \ref{ss:disent}) plays the role the singular Seifert surface.

\subsection{$ {C}^{\infty}$-characterisation of $C( \Phi)$ and $T( \Phi)$.}\labelpar{ss:cveg}
The formulae connecting the holomorphic invariants $ C( \Phi) $ and $ T( \Phi) $ with $ \mathcal{C}^{\infty}$-invariants $ \Omega(f) $ and $ L(f)$ have the following consequence.

\begin{thm}\label{thm:ct}
 Assume that the analytic germs $ \Phi $ and $ \Phi' :( \C^2, 0) \to ( \C^3, 0)$ are
 $\mathcal{C}^{\infty}$ $\mathscr{A}$-equivalent (that is, $ \Phi'= \Lambda \circ \Phi \circ \psi $ holds for some germs of orientation preserving diffeomorphisms $ \psi: (\R^4, 0) \to ( \R^4, 0) $ and $ \Lambda: (\R^6, 0) \to (\R^6, 0) $). Then

 (a) $ C( \Phi) = C( \Phi') $.

 (b) If additionally $ \Phi $ is finitely $ \mathscr{A}$-determined as a holomorphic germ, then $ T( \Phi) < \infty $ and $T( \Phi) = T ( \Phi')$.
\end{thm}

Part (a) of Theorem~\ref{thm:ct} is proved as Corollary~\ref{cor:C}, while part (b) in Remark~\ref{rem:T}.

%We thank the anonymous referee for suggesting us that such a consequence might follow from our characterisations of $ C $ and $ T$ in terms of $\Omega(f)$ and $L(f)$.

In fact, one might even ask for the topological analogue of Theorem \ref{thm:ct}:
 is it true that if
 $ \Phi'= \Lambda \circ \Phi \circ \psi $ holds for some germs of orientation preserving
 {\it homeomorphisms} $ \psi: (\R^4, 0) \to ( \R^4, 0) $ and
 $ \Lambda: (\R^6, 0) \to (\R^6, 0)$, then
$ C( \Phi) = C( \Phi') $ and $T( \Phi) = T ( \Phi')$?
 In Remark \ref{rem:T} we show the following.

 \begin{cor}\label{cor:NEW}
  In the presence of a topological left--right equivalence as above,
 if $ \Phi $ and $ \Phi'$ are finitely $ \mathscr{A}$-determined, 
% $f$ and $f'$ are  generic immersions 
 then $L(f)=L(f')$, hence $C(\Phi)-3T(\Phi)=C(\Phi')-3T(\Phi')$.
 \end{cor}

 \noindent
 The full extension of Theorem \ref{thm:ct} from $\mathcal{C}^{\infty}$ to topological category
 is obstructed by the following facts. Though we  identify the analytic invariant
 $C(\Phi)$ with the smooth Smale invariant $\Omega (f)$, it is not yet known if
 $\Omega(f)$ is stable with respect to  topological left--right equivalence.
 A possible way to prove this requires the extension of the formulae from Section
  \ref{s:ss} from smooth to more general Seifert surfaces, which exceeds the
  aims of the present thesis.
    (Also, we do not know how the
    topological left--right equivalence behaves with respect to analytic deformations
    used in Section \ref{s:eszcomp}.)

    \section{Distinguished generators and sign conventions}\label{s:signconv}
    
    \subsection{Smale invariant and orientation}\label{ss:absor}
    
   Recall that we use (\ref{eq:Smale}) for definition of the Smale invariant throughout Chapter~\ref{ch:ass}. We wish to emphasize the following facts regarding orientations of $S^3$ and $\R^5$
 and their effects on definition (\ref{eq:Smale}). (This might serve also as a small guide for the next sections.)

Let us think about $S^3$ as the subset of $\R^4$, the boundary of the 4--ball $B^4$ in $\R^4$,
or via embedding $\R^4\subset \R^5$, as a subset of $\R^5$. We do not wish to fix any orientation on it
as the orientation of $\partial B^4$
(that would depend on the convention how one defines the orientation of the boundary of an oriented
manifold --- called, say, `boundary convention').

Note that in the definition (\ref{eq:Smale})
of the Smale invariant, not the orientation of $S^3$ is used, but the orientation
of the tubular neighbourhood $U\subset \R^5$ and the orientation of the target $\R^5$. Moreover,
$\Omega(f)$ is unsensitive to the orientation change simultaneously in both $\R^5$.
In this way we get an element $\Omega(f)\in[S^3,SO(5)]$, which is independent of the orientation of
$\R^5$ and does not use any orientation of $S^3$. Furthermore, if we define
(this will done in \ref{ss:sign}) a generator $[L]$ in $[S^3,SO(5)]$, using again only
the embedding $S^3\subset \R^5$ (and no other orientation data), then $\Omega(f)$ identifies with
an element of $\Z$, such that its definition is independent of any orientations of $S^3$ and
$\R^5$, hence also of the `boundary convention'.

All our discussions are in this spirit (except sections \ref{s:ss} and \ref{s:eszcomp}, where
oriented Seifert surfaces are treated): we run orientation and `boundary convention' free
definitions and statements (associated with  $S^3$, regarded as a subset of  $\R^5$, and immersions
$S^3\looparrowright \R^5$).

 However, if we fix a `boundary convention', then $S^3$ (in $\R^5$) will get an orientation (as $\partial
 B^4$). Then, for any {\it oriented abstract $S^3$}, let us denote it by ${\bf S}^3$, and immersion
 ${\bf S}^3\looparrowright \R^5$, we can define the Smale invariant $\Omega^a(f)\in \Z$
 (here `a' refers to the `abstract' ${\bf S}^3$) by identifying ${\bf S}^3$ with the embedded
 $S^3\subset \R^5$ by an orientation preserving diffeomorphism and taking  $\Omega(
 S^3\to {\bf S}^3\looparrowright \R^5)$. This $\Omega^a(f)$
 depends on the `boundary convention', since the identification
 $S^3\to {\bf S}^3$ depends on it: changing the convention we change $\Omega^a(f)$ by a sign.

 This point of view should be adapted when ${\bf S}^3$ will be the (oriented) boundary of an oriented
 Seifert surface. But till Section \ref{s:eszcomp} we will focus on the first version,
 $\Omega(f)$.

 \vspace{2mm}

Next, in the definition of $\Omega(f)$,
one can replace $\R^5$ by $S^5$, where $S^5$ is the boundary of the ball in $\R^6$, and
$S^3$ is embedded naturally in $S^5$, cf. Example~\ref{ex:snsq}.
By taking a generic
point $P\in S^5$ we identify $S^5\setminus \{P\}$ with $\R^5$, and
$U$ will be replaced by a tubular neighbourhood of  $S^3$ in $S^5$. Then the definition (\ref{eq:Smale}) of $\Omega(f)$ can be repeated for any immersion $S^3\looparrowright S^5$ (where
$S^3\subset S^5$) providing an element $[S^3,SO(5)]$, which becomes an integer once
 a generator $[L]$ is constructed from the embedding $S^3\subset S^5$. Again, this Smale invariant
$\Omega(f)$ will be independent of the orientations of $S^3$ and $S^5$, hence of the
`boundary convention' as well.

For immersions defined in Subsection \ref{ss:link}, $\mathfrak{S}^3$ evidently sits naturally in
$\C^2=\R^4$ (hence also in a certain $\mathfrak{S}^5=S^5\subset \C^3=\R^6$, cf. \ref{ss:6.1}).
This together with Definition
(\ref{eq:Smale})  provide $\Omega(f)$ (which becomes an integer once $[L]$ will be constructed
in \ref{ss:sign}).

\subsection{Switch complex to real}\labelpar{ss:comprel}
There is a natural map $ \tau $ from $ \Hom (\C^3, \C^3) $ to $ \Hom (\R^6, \R^6) $, which replaces any entry
$ M_{ij} =  a+ bi $ of a matrix $ M \in \Hom (\C^3, \C^3) $ by the  real $ 2 \times 2 $--matrix
$\begin{pmatrix}a&-b\\b&a\end{pmatrix}$. $ \tau (U(3)) \subset SO(6) $ holds.
%In other words: $ U(3) $ acts faithfully on $ \C^3 = \R^6 $ with special orthogonal transformaions, this action gives a $ \tau: U(3) \hookrightarrow SO(6) $ embedding.
A map $ F: \C^3 \to \C^3 $ can be regarded as a map $ \tilde{F}: \R^6 \to \R^6 $: if we
denote by $ z_j = x_j + i y_j $  ($ j= 1, 2, 3 $) the coordinates of $ \C^3 $, then for the components of $ F $ and $ \tilde{F} $ one has
\[ F_j(z_1, z_2, z_3) = \tilde{F}_{2j-1} (x_1, y_1, x_2, y_2, x_3, y_3) + i \tilde{F}_{2j} (x_1, y_1, x_2, y_2, x_3, y_3). \]
 Then $ \tau ( d_{\C} F ) = d_{\R} \tilde{F} $ holds for the complex Jacobian of $F$ and the real Jacobian of $\tilde{F}$.

Let $ j: SO(5) \hookrightarrow SO(6) $ denote the inclusion. It is well--known (see e.g.
\cite{husemoller}) that
\begin{equation}\label{eq:pi3}
\pi_3 (j): \pi_3 (SO(5)) \to \pi_3 (SO(6)) \ \ \mbox{is an isomorphism}.
\end{equation}
\begin{prop}\label{pr:hom}
The homomorphism $ \pi_3 ( \tau ): \pi_3 (U(3)) \to \pi_3(SO(6)) $ is an isomorphism too.
\end{prop}
\begin{proof} First, we provide a more conceptual  proof, which does not identify
distinguished generators.
Both sides are in the stable range (see \cite{husemoller}), hence
 we can switch to the homomorphism $ \pi_3 (U) \to \pi_3(O) $ induced by the embedding $ \tau: U \hookrightarrow O $. By (a proof of)
 Bott periodicity,
 the quotient $ O/U $ is homotopically equivalent to the loopspace $ \Omega O $ of $ O$, cf. \cite{bottper}.
 Hence $ \pi_i (O/U) = \pi_i ( \Omega O ) = \pi_{i+1} (O) = 0 $ for $i=3$ and 4.
  Then the  isomorphism follows from
 the homotopy exact sequence of the fibration $ O \to O/U $ with fibre $ U $.
 % we get
%\[ \pi_4 (O/U) \to \pi_3 (U) \to \pi_3 (O) \to \pi_3 (O/U) \mbox{ ,} \]
%where , so the middle homomorphism is an isomorphism.
\end{proof}
In \ref{ss:sign} we will give another, more computational  proof, where we will
be able to fix distinguished  generators for $ \pi_3 (U(3)) $ and $ \pi_3(SO(6)) $,
 and via these  generators  we
identify $ \pi_3 ( \tau ) $ with multiplication by $ -1 $.

\subsection{Conventions, identifications}\labelpar{ss:sign}
First, we identify $ \mathbb{H} $ and $ \R^4 $ and $ \C^2 $ in the obvious way:
 we identify the quaternion $ q = a + bi + cj + dk = z + wj  \in \mathbb{H} $ with $ (a, b, c, d) \in \R^4 $ and with the complex pair $ (z, w) \in \C^2 $, where $ z=a+bi $ and $ w= c+ di $. Also, we identify $ S^3 $ with the quaternions of unit length:
$ S^3 = \{ q= a + bi + cj + dk  \in \mathbb{H} \ | \ a^2 + b^2 + c^2 + d^2 = 1 \} $.

We define the following maps. Set
\[ u: S^3 \to U(2) \mbox{ , } u_q =
\left(
\begin{array}{cc}
z & - \bar{w} \\
w & \bar{z} \\
\end{array}
\right),
\]
where $ q= z + wj $. $ u_q $ is the (complex) matrix of the \emph{right}
(quaternionic) multiplication
with $ q $, that is, of the map $ \mathbb{H} \to \mathbb{H} $, $ p \mapsto pq $.
Note that the left multiplication by $ q $ is not a complex unitary transformation, in general.
Next, set
\[ L: S^3 \to SO(4) \mbox{ , } L_q =
\left(
\begin{array}{cccc}
a & -b & -c & -d \\
b & a & -d & c \\
c & d & a & -b \\
d & -c & b & a \\
\end{array}
\right),
\]
where $ q= a + bi + cj + dk $.
$ L_q $ is the (real) matrix of the \emph{left} multiplication with $ q $ (i.e., of
the map $ \mathbb{H} \to \mathbb{H} $, $ p \mapsto qp $).

Let $ R:  S^3 \to SO(4) $ be the map which assigns for a $q \in S^3 $ the (real) matrix $ R_q $ of the right multiplication with $ q$ (i.e., of the map $ \mathbb{H} \to \mathbb{H} $, $ p \mapsto pq $).

Let $ \rho:  S^3 \to SO(4) $ be the map which assigns for a $q \in S^3 $ the (real) matrix $ \rho_q $ of the conjugation with $ q$ (i.e., of the map $ \mathbb{H} \to \mathbb{H} $, $ p \mapsto qpq^{-1} $).

We use the same notation for the compositions of these maps with the inclusions $ SO(4) \hookrightarrow SO $ and $ U(2) \hookrightarrow U $. Note that these inclusions commute with $ \tau $.

\begin{prop}\label{pr:hus}\cite[Section 7, Subsection 12]{husemoller}

 (a) $ \pi_3 (U(2)) = \pi_3 (U) = \Z \langle [u] \rangle $.

 (b) $ \pi_3 (SO(4)) = \Z \langle [L] \rangle \oplus \Z \langle [\rho] \rangle $.

 (c) $ \pi_3 (SO) = \Z \langle [L] \rangle $ and $ [ \rho] = 2 [L] $ in $ \pi_3 (SO) $.
 \end{prop}
 In the sequel, {\bf using these base choices $[u]$ and $[L]$ we identify
 the groups $\pi_3(U)$ and $\pi_3(SO)$ with $\Z$.}
 Now we can state the explicit version of Proposition~\ref{pr:hom}.
\begin{prop}\label{pr:homo}
$ \pi_3 (\tau) ([u]) = - [L] \in \pi_3 (SO) $ holds for $ [u] \in \pi_3 (U) $.
\end{prop}
 \begin{proof}
From definitions $ \tau \circ u = R $ and $ \rho R= L $, thus $ \pi_3 (\tau) ([ u ] ) = [ R]  = [L]-[ \rho ] = -[L] $ by part (c) of Proposition~\ref{pr:hus}.
\end{proof}

\begin{rem}
 Let $ p: U(2) \to S^3 $ be the projection (choosing the first or the second column of
 the matrix). Then $ [u] \in \pi_3 (U(2)) $ is the unique generator for which
 $ \deg (u \circ p )=1 $.
\end{rem}

\begin{rem}
The conventions we use are not universal. For example, Kirby and Melvin in \cite{KirbyMelvin}
 have chosen the same generators of $  \pi_3 (U) $ and $  \pi_3 (SO)$ (these are $ [u] $ and $ [L]$ with our notations), but they identified $ \R^4 $ and $ \C^2 $ differently than us. Namely, they identified the quaternion $ q= a+bi + cj + dk = z + jw \in \mathbb{H} $ with $ (a, b, c, d) \in \R^4 $ and the complex pair $ (z,w) \in \C^2 $, where $ z=a+bi $ and $ w= c-di $. With that identification $ u_q $ becomes the (complex) matrix of the quaternionic \emph{left} multiplication with $ q $. In that identification $ \pi_3 ( \tau )[u] $ would be equal to $ -[R] $, since that is the homotopy class of the map $ S^3 \to SO(4) $ given by the composition $ \tau \circ u \circ \kappa  $, where $ \kappa $ is the reflection $ \kappa (z, w) = (z, \bar{w}) $.
\end{rem}

\section{The complex Smale invariant}\labelpar{s:co}
\subsection{} In this section we define
 the \emph{complex Smale invariant} $ \Omega_{\C} ( \Phi ) $
for a holomorphic germ $ \Phi: ( \C^2, 0)  \to (\C^3, 0) $,
 singular only at $ 0 $.  It will be the bridge between $ C ( \Phi ) $ and $\Omega(f)$.

\begin{defn}
Consider the map (with target the complex Stiefel variety $V_2(\C^3)$):
\[ d \Phi |_{ S^3} : S^{3} \to V_2 ( \C^3) \]
defined via the natural trivialization of the complex tangent bundles $T\C^2$
and $T\C^3$.
By definition, the complex Smale invariant of $ \Phi $ is the  homotopy class:
\[ \Omega_{\C} ( \Phi ) = [ d \Phi |_{ S^{3}}  ] \in \pi_{3} ( V_2 ( \C^{3}) ). \]
\end{defn}
\noindent
By the connectivity of the group of local coordinate transformations, $\Omega_\C(\Phi)$
is independent of the choice of local coordinates in $(\C^2,0)$ and $(\C^3,0)$.

\begin{rem}\labelpar{re:stab}
The projection $ U(3) \to  V_2 ( \C^{3}) $ induces an isomorphism between
$ \pi_{3} ( V_2 ( \C^{3})) $ and $ \pi_{3} ( U(3)) = \pi_{3} ( U) $ (see e.g.
\cite{husemoller}). Hence,
 if we choose a complex normal vector field $ N_{ \Phi } $ of $ \Phi $, then the map
\[ ( d \Phi , N_{ \Phi } )|_{S^{3}} : S^{3} \to GL(3, \C)
\]
represents $ \Omega_{ \C} ( \Phi ) $ in $ \pi_{3} ( GL(3, \C)) = \pi_{3} (U(3)) = \pi_{3} (U) $.
A canonical choice of $ N_{ \Phi } $ could be the complex conjugate of the cross product of the partial derivatives of $ \Phi $:
\[ N_{ \Phi } (s, t) = \overline{  \partial_s \Phi (s, t) \times \partial_t \Phi (s, t) } \mbox{ .} \]
\end{rem}

\begin{rem}
$ \pi_3 (U) \cong \Z $ and in \ref{ss:sign}  we
 identify them through a fixed isomorphism.
 In this way $ \Omega_{ \C} ( \Phi ) $ becomes  a well-defined integer without any sign--ambiguity.
\end{rem}

\begin{prop}\label{prop:u}
 Let $ \Phi : \C^2 \to \C^3 $ be the cross cap, i.e. $ \Phi (s, t) = (s^2, st, t) $. Then $ \Omega_{\C} ( \Phi ) = [u] $.
\end{prop}

\begin{proof}
The map $ d \Phi|_{S^3}: S^3 \to V_2 ( \C^3 ) $ represents $ \Omega_{\C} ( \Phi ) $.
We should compose this with  $ V_2 (\C^3 ) \to U(3) $, then
with the inverse of the inclusion $ U(2) \to U(3) $,  and finally with the projection $ U(2) \to S^3 $;
 and then
calculate the degree of the resulting map $ S^3 \to S^3 $. In fact,
along these compositions
we will use (the homotopically equivalent groups)
$ GL(2, \C ) $ and $ GL(3, \C ) $ instead of $ U(2)  $ and $ U(3) $.
Therefore, we will arrive in
$ \C^2 \setminus \{ 0 \} $ instead of $ S^3 $.
\[ d \Phi|_{S^3} : S^3 \to V_2 (\C^3) \mbox{ , }
(s, t) \mapsto
 \left( \begin{array}{cc}
 2s & 0 \\
 t & s \\
 0 & 1 \\
\end{array} \right).
\]
The first composition gives the map
\[
 S^3 \to GL(3, \C ) \mbox{ , }
 (s, t) \mapsto
 \left( \begin{array}{ccc}
 2s & 0 & N_1 \\
 t & s & N_2 \\
 0 & 1 & N_3 \\
\end{array} \right),
\]
where $ N_1 = \bar{t} $, $ N_2 = -2 \bar{s} $ and $N_3 = 2 \bar{s}^2 $ are the coordinates of the normal vector $ N_{\Phi} $ (see Remark~\ref{re:stab}).
This is modified by the homotopy
\[
 S^3 \times [0,1] \to GL(3, \C ) \mbox{ , }
 (s, t, h) \mapsto
 \left( \begin{array}{ccc}
 2s & 0 & N_1 \\
 t & h s & N_2 \\
 0 & 1 & h N_3 \\
\end{array} \right),
\]
which maps $ (s, t, 0) $ into $ GL(2, \C ) \subset GL(3, \C) $. [Note that the determinant
 is $ |t|^2 + 4|s|^2 + 4h^2 |s|^4 \neq 0 $, thus the image is indeed in $ GL(3, \C) $.]
Hence, we obtain  the map
\[
 S^3 \to GL(2, \C ) \mbox{ , }
 (s, t) \mapsto
 \left( \begin{array}{cc}
 2s & N_1 \\
 t & N_2 \\
\end{array} \right),
\]
which composes with the projection (first column) provides
$ S^3 \to \C^2 \setminus \{0 \}$, $(s, t) \mapsto (2s, t)$.
After a normalisation, the degree of the resulting map is $ 1 $.
 %(because this map is homotopic to the identity map of $ S^3 $).
\end{proof}

The proof of Proposition~\ref{prop:u} works for all germs of the form
 \begin{equation}\labelpar{eq:rank}
  \Phi (s, t) = ( g_1 (s, t), g_2 (s, t), t )
 \end{equation}
and implies that
 \[
  \Omega_{\C} ( \Phi ) = \deg \left( S^3 \to S^3 \mbox{ , } (s, t) \mapsto \frac{(\partial_s g_1 (s, t), \partial_s g_2 (s, t))}{|(\partial_s g_1(s, t), \partial_s g_2 (s, t))|} \right).
 \]
 This degree agrees with the intersection multiplicity of $ \partial_s g_1 $ and $ \partial_s g_2 $ in $(\C^2,0)$, hence we proved the following, cf. Example~\ref{ex:cor1}.
 \begin{prop}\label{pr:otherproof}
 \begin{equation}\labelpar{eq:other}
  \Omega_{\C} ( \Phi ) = \dim_{\C}\,  \frac{ \mathcal{O}_{(\C^2, 0) }}{ (
  \partial_s g_1, \partial_s g_2 ) } = C( \Phi)
  \end{equation}
  holds for corank--$1$ germs $ \Phi: ( \C^2, 0) \to (C^3, 0) $ singular only at the origin.
 \end{prop}
% This also equals  $ C( \Phi) $ by Theorem~\ref{th:C}. This proves
% $\Omega_{\C} ( \Phi )= C(\Phi)$ for maps of corank 1.
 % (All germs which satisfy $ {\rm rank} (d \Phi_{0} )= 1 $ are right--left equivalent with germs of type  (\ref{eq:rank})).
  This identity $\Omega_{\C} ( \Phi )= C(\Phi)$ will be proved in the general case in Section~\ref{sec:5}.

\begin{rem}\label{re:other} Conversely, the identity $\Omega_{\C} ( \Phi )= C(\Phi)$ is proved in Section~\ref{sec:5}
 independently of Theorem~\ref{th:CT}, therefore (\ref{eq:other}) together with Theorem~\ref{th:Comp}
  give a new proof for Theorem~\ref{th:CT}
  in the case of germs which satisfy $ \rk (d \Phi_{0} )= 1 $.
  \end{rem}

\subsection{The identity $ \Omega_{ \C } ( \Phi ) = C ( \Phi ) $.}\label{sec:5}

Next we identify the complex Smale invariant with the number of cross caps.
\begin{thm}\label{th:Comp}
$ \Omega_{ \C } ( \Phi ) = C ( \Phi ) $.
\end{thm}

\begin{proof} Consider the following diagram:
\[
\begin{array}{cccc}
d \Phi: & \C^2 & \longrightarrow & Hom ( \C^2 , \C^3 ) \\
 & \cup & & \cup \\
d \Phi |_{ \C^2 \setminus \{ 0 \}} : & \C^2 \setminus \{ 0 \} & \longrightarrow &  Hom ( \C^2 , \C^3 ) \setminus \mathcal{D}= V_2 ( \C^3 )\\
\end{array} \mbox{ ,}
\]
where $ \mathcal{D} = \{ M \in Hom ( \C^2 , \C^3 ) \ | \ \rk (M) < 2 \} $.
 %is the common zero set of the determinants of the three minors $ M_1 $, $M_2$ and $M_3$.
$ \mathcal{D} $ is an irreducible algebraic variety of complex codimension $2$, its Zariski open set
$ \mathcal{D}^1 = \{ M \in {\mathcal D} \ | \ \rk (M) =1 \} $ is smooth.

First we prove that $ \Omega_{ \C } ( \Phi ) $ is equal to the linking number of
 $ d \Phi |_{S^3} $ and $ \mathcal{D} $ in $  Hom ( \C^2 , \C^3 ) $.
This is  defined as follows.
If  $ g: S^3 \to Hom ( \C^2 , \C^3 ) \setminus \mathcal{D} $ is a smooth map, and $\tilde{g}$
is a smooth extension defined on the ball such that $\tilde{g}|_{S^3} = g $ and
$ \tilde{g} $ intersects $ \mathcal{D} $ transversally along $ \mathcal{D}^1 $, then
the linking number of $ g $ and $ \mathcal{D} $
 is the algebraic number of the intersection points of $ \tilde{g} $ and $ \mathcal{D} $.
By standard argument it is a homotopy invariant of maps
$S^3 \to Hom ( \C^2 , \C^3 ) \setminus \mathcal{D} $.

The linking number gives a group homomorphism
$\lk: \pi_3 (V_2 ( \C^3 )) \to \Z $.
%[Let $ g_1 $ and $ g_2 : S^3 \to Hom ( \C^2 , \C^3 ) \setminus \mathcal{D} $ be two smooth maps (take the basepoint to the basepoint) and $ \tilde{g_1} $ and $ \tilde{g_2} $ two $4$-chains for them. $ \tilde{g_1} + \tilde{g_2} $ is a $ 4$-chain such that $ \partial (\tilde{g_1} + \tilde{g_2}) $ represents the homotopy class $ [g_1] + [g_2] $ and the algebraic number of the intersection points of it with $ \mathcal{D} $ is the sum of that of the $ \tilde{g_i} $-s.]
Next lemma shows  that this homomorphism is surjective.
\begin{lem}\label{lem:cr} Let $ \Phi(s, t) = (s^2, st, t) $ be the cross cap. If $ g= d \Phi|_{S^3} $ and $ \tilde{g} = d \Phi|_{B^4} $, then
$ \tilde{g} (0) \in \mathcal{D} $ is the only intersection point and the intersection is transverse at that point.
\end{lem}
\begin{proof} This is a straightforward local computation left to the reader.
The transversality follows also from the conceptual fact that the cross cap is a stable map.
\end{proof}

The sign of the intersection multiplicity at the intersection point of two
complex submanifolds is always positive. For $ g $ described in \ref{lem:cr}
 the linking number of $ g(S^3) $ and $ \mathcal{D} $ is $1$.
 This shows not only that the
 homomorphism given by the linking number is surjective (hence an isomorphism too), but also that
 this isomorphism agrees with the chosen one in \ref{ss:sign}.
  This follows from the fact that the complex Smale invariant of the cross cap is
   exactly the chosen generator, see Proposition~\ref{prop:u}. Hence, the homomorphisms  $\Omega_\C$
   and $\lk$ coincide.

 Next, we show that
 $ \lk_{ Hom ( \C^2 , \C^3 ) } ( d \Phi |_{S^3}(S^3), \mathcal{D}) = C ( \Phi ) $.
 Take a generic perturbation $ \Phi_{ \epsilon} $ of $ \Phi $. $ d \Phi_{ \epsilon} |_{S^3} $ is homotopic to $  d \Phi |_{S^3} $, hence  their linking numbers are the same. $ \Phi_{ \epsilon} $ has only cross cap singularities, their number is $ C ( \Phi ) $. This means that $ d \Phi_{ \epsilon} |_{B^4} $ intersects transversally $ \mathcal{D} $ in $ C ( \Phi ) $ points.
 Intersection of  complex manifolds provides positive signs.
\end{proof}

\begin{cor}
$ \Omega_{ \C } ( \Phi ) \geq 0 $.
\end{cor}

\section{The proof of Theorems~\ref{TH:MAIN} and \ref{TH:EMBINTRO}}\labelpar{s:proof}
\subsection{}\label{ss:6.1}

 Theorem \ref{TH:MAIN} follows from Theorem \ref{th:Comp}, Proposition \ref{pr:homo} and the next identity.
\begin{prop}\label{pr:main}
$ \pi_3 ( \tau ) ( \Omega_{ \C } ( \Phi ) ) = \pi_3 (j) (\Omega(f)) $.
\end{prop}

\begin{proof}
By the definition of the Smale invariant, one has to extend $ f $ to a neighbourhood of
the standard embedding of $ \mathfrak{S}^3 $ in an $ \R^5 $  (cf. (\ref{eq:Smale})).
On the other hand $ \Phi $ extends $ f $ in the $ \C^2 $ direction. We will compare these two
extensions using a  common extension $ F: W \to \C^3 $, where $ W $ is a suitable neighbourhood of
$ \mathfrak{S}^3 $ in $ \C^3 $.

Let us consider a fixed $\epsilon$ which satisfy the properties of
 Corollary \ref{cor:epsilon}. We also write
$B^6_\epsilon =\{z:|z|\leq \epsilon\}\subset \C^3$, $S^5_\epsilon=\partial B^6_\epsilon$,
${\mathfrak B}^4_\epsilon:=\Phi^{-1}(B^6_\epsilon)$ for the $\mathcal{C}^{\infty}$ ball in $\C^2$, and  $\mathfrak{S}^3_\epsilon:= \Phi^{-1}(S^5_\epsilon)$ for its boundary. (Late we will drop some
of the $\epsilon$'s.)
For positive  $\epsilon_1$, $\epsilon_2$ sufficiently closed to $\epsilon$, $\epsilon_1<\epsilon<\epsilon_2$, and for $0 <\rho\ll\epsilon$ one defines
$ F(s, t, r) = \Phi (s, t) + r \cdot  N_{ \Phi } (s, t) $, where $(s,t,r)\in W:=\Phi^{-1}(z:\epsilon_1<|z|<\epsilon_2)\times D^2_\rho$, $D^2_\rho$ is
the $\rho$-disc in $\C$, and $  N_{ \Phi} $ is the complex normal vector of $ \Phi $, see Remark~\ref{re:stab} .
%By construction, $F(W)$ is a tubular neighbourhood of $S^5_\epsilon$ in $\C^3$.
Since the normal bundle of $f$ in $S^5_\epsilon$
is trivial (and since the transversality is an open property), we get that $F^{-1}(S^5_\epsilon) $ is diffeomorphic to $\mathfrak{S}^3_\epsilon\times D^2_\rho$.
In fact, if $p:\C^2\times D^2_\rho\to D^2_\rho$ is the natural  projection,
then for any $r\in D^2_\rho$ we can define  $\mathfrak{S}^3_{\epsilon,r}:=
F^{-1}(S^5_\epsilon)\cap p^{-1}(r)$. Then each  $\mathfrak{S}^3_{\epsilon,r}$ is a $ \mathcal{C}^{\infty}$ 3-sphere, being the boundary of the $ \mathcal{C}^{\infty}$ 4-ball $\mathfrak{B}^4_{\epsilon,r}
\subset p^{-1}(r)$. Then $F^{-1}(S^5_\epsilon)=\cup_{r\in D^2_\rho} \mathfrak{S}^3_{\epsilon,r}$. Moreover, $\mathfrak{B}^6:=\cup_{r\in D^2_\rho} \mathfrak{B}^4_{\epsilon,r}\subset \C^2\times \C$ is a thickened tubular neighbourhood of
$\mathfrak{B}^4_\epsilon\subset \C^2\times 0$, homeomorphic to the real 6-ball.
Its corners can be smoothed, hence we think about it as a $ \mathcal{C}^{\infty}$ ball. Its boundary $\mathfrak{S}^5:=\partial\mathfrak{B}^6$ (diffeomorphic to the 5-sphere)
is the union of $F^{-1}(S^5_\epsilon)$ (diffeomorphic to
$S^3\times D^2$) and $\cup_{r\in \partial D^2_\rho} \mathfrak{B}^4_{\epsilon,r}$
(diffeomorphic to  $B^4\times S^1$).

 In a point $ (s, t, 0) \in \mathfrak{S}^3_\epsilon\times \{0\} $
 the differential of $ F $ is
\[ dF (s, t, 0)  = ( \partial_s F(s, t, 0) , \partial_t F(s, t, 0) , \partial_r F(s, t, 0) )
= ( \partial_s \Phi (s, t) , \partial_t \Phi (s, t) , N_{ \Phi } (s, t) ) \mbox{ .}
\]
Thus,  the homotopy class of $ dF|_{\mathfrak{S}_\epsilon^3} $ equals $ \Omega_{ \C } ( \Phi ) $  (cf.  \ref{re:stab}). Therefore, taking the \emph{real} function $ \tilde{F} : W \to \R^6 $
 (cf. \ref{ss:comprel}), its  real
Jacobian  satisfies $ [ d \tilde{F} |_{\mathfrak{S}_\epsilon ^3}]=
\pi_3 ( \tau ) ( \Omega_{ \C } ( \Phi ) ) $ .

On the other hand we show that $ [ d \tilde{F}|_{\mathfrak{S}^3}] = \pi_3 (j) (\Omega(f)) $.
In order to recover the Smale invariant $ \Omega (f) $ of
$ f=\Phi|_{\mathfrak{S}^3}: \mathfrak{S}^3 \looparrowright S^5 $,
first we need to fix a global coordinate system in a contractible neighbourhood of the source
$\mathfrak{S}^3$ in $\mathfrak{S}^5$
and also in $\R^5\simeq S^5_\epsilon\setminus \{\mbox{a point}\}$ containing ${\rm im}(f)$. Let us introduce
the `outward normal at the end' convention to orient  compatibly  a manifold and its boundary.
In this way we fix an orientation of $\mathfrak{S}^5 = \partial\mathfrak{B}^6 $ and $ S^5 = \partial B^6 $.
(According to \ref{ss:sign}, the output of the proof is independent of the convention choice.)

In the first case we introduce a coordinate system
in $\mathfrak{S}^5\setminus \{Q\}\simeq \R^5$ compatibly with the orientation,
 where $ Q \in \mathfrak{S^5} \setminus \mathfrak{S}^3 $ is an arbitrary point (e.g.
 $(0,0,\rho)$).
 Let $ \nu' $ denote the framing of $ T (\mathfrak{S^5} \setminus \{Q \})\simeq T\R^5$ induced by
 this coordinate system. We can extend the outward normal frame $ \nu_6 $ of $ \mathfrak{S}^3 $ in $ \C^2 $ to the rest of $ \mathfrak{S}^5 \setminus \{Q \}  $ (as the outward normal vector
 of $\mathfrak{S}^5$). This framing can be extended to
a neighbourhood $ V $ of $ \mathfrak{S}^5 \setminus \{Q \}  $ in $ \C^3 $. Let $ \nu: V \to GL^+ (6, \R) $ denote this framing (or more precisely, $ \nu $ is the transition function from the standard framing inherited from $ \R^6 $ to the one just
constructed).

The target is the standard $ S^5 \subset \R^6 $. We can choose a point $ P \in S^5 \setminus f( \mathfrak{S}^3 ) $ and a coordinate system on $ S^5 \setminus \{P\} $ compatibly with the orientation. The coordinate system induces a framing $ \eta' $ of the tangent bundle $ T ( S^5 \setminus \{P\} ) $
of $ S^5 \setminus \{P\} $. In the points of the target of $ \tilde{F} $ the vectors of $ \eta' $ and $ d \tilde{F} ( \nu_6 ) $ are linearly independent, that is,
$d \tilde{F} ( \nu_6 ) $ behaves like a normal framing
(this follows from the transversality property of
\ref{ss:link}).
We can extend it to a normal framing $\eta_6$
of $ S^5 \setminus \{P\}  $ in $ \R^6 $. In this way we get a framing of the tangent bundle of a neighbourhood $ V' $ of $ S^5 \setminus \{P\} $ in $ \R^6 $. Let $ \eta : V' \to GL^+(6, \R ) $ denote the transition from the framing on $ V' $ inherited from $ \R^6 $ to the framing just defined.

The Smale invariant $\Omega(f)$ is constructed in the following way, cf. (\ref{eq:Smale}).
Take
\[ \mathcal{J}_{( \nu', \eta')} ( \tilde{F}|_{\mathfrak{T}}),  \]
the Jacobian of $ \tilde{F} $ restricted to $ \mathfrak{T}=F^{-1}(S^5_\epsilon) $ prescribed in the framings $ \nu' $ and $ \eta' $. The homotopy class of this matrix restricted to $ \mathfrak{S}^3 $ (as a map $ \mathfrak{S}^3 \to GL^+(6, \R) $) is equal to $ \Omega(f) $. (Since $ \tilde{F} $ preserves the orientation, $ \tilde{F}|_{\mathfrak{T}} $ does as well.)
\[ \mathcal{J}_{ (\nu, \eta)} ( \tilde{F})|_{ \mathfrak{T} } = j ( \mathcal{J}_{ (\nu', \eta' )} ( \tilde{F}|_{ \mathfrak{T}})) \]
because $ d \tilde{F} (\nu_6) = \eta_6 $, thus the homotopy class of $ \mathcal{J}_{ (\nu, \eta ) } ( \tilde{F} )|_{ \mathfrak{S}^3} $ equals  $ \pi_3 (j) ( \Omega(f)) $.

On the other hand $ \mathcal{J}_{ ( \nu, \eta ) } ( \tilde{F}) = ( \eta^{-1} \circ \tilde{F}) \cdot d \tilde{F} \cdot \nu $. As maps $ \mathfrak{S}^3 \to GL^+(6, \R) $, $ ( \eta^{-1} \circ \tilde{F})|_{\mathfrak{S}^3} $ and $ \nu|_{\mathfrak{S}^3} $ are nullhomotopic because
 the vector fields are defined on the contractible spaces $\mathfrak{S}^5\setminus \{Q\}$ and
 $S^5\setminus \{P\}$. Therefore (cf. Remark~\ref{re:lie})
\[ [ \mathcal{J}_{ ( \nu, \eta ) } ( \tilde{F})|_{\mathfrak{S}^3} ]=
[( \eta^{-1} \circ \tilde{F})|_{\mathfrak{S}^3}] + [d \tilde{F}|_{\mathfrak{S}^3}] + [ \nu|_{\mathfrak{S}^3} ] = [d \tilde{F}|_{\mathfrak{S}^3}] \mbox{ .}
\]
The left hand side of this identity is $ \pi_3 (j) ( \Omega(f)) $, while the right
hand side  $ \pi_3 ( \tau ) ( \Omega_{ \C } ( \Phi ) ) $.
\end{proof}

\begin{cor}\label{cor:C}
Assume that the analytic germs $ \Phi $ and $ \Phi' :( \C^2, 0) \to ( \C^3, 0)$ are $ \mathcal{ C}^{\infty}$ left-right equivalent, that is, $ \Phi'= \Lambda \circ \Phi \circ \psi $ holds for some germs of orientation preserving diffeomorphisms $ \psi: (\R^4, 0) \to ( \R^4, 0) $ and $ \Lambda: (\R^6, 0) \to (\R^6, 0) $. Then $ C( \Phi) = C( \Phi') $.
\end{cor}

\begin{proof}
For a sufficiently small $ \epsilon $ take $ \mathfrak{S}'^3 = ( \Phi')^{-1} (S^5_{ \epsilon}) \simeq S^3 $ and let $ f': \mathfrak{S}'^3 \hookrightarrow S^5 $ be the immersion associated with the germ $  \Phi' $ (cf. \ref{ss:link}).

 As in the proof of Proposition~\ref{pr:main}, let $ \tilde{F}: W \to \R^6 $ be the extension of $ \Phi $ viewed as a real function. Then, by that proof,  $ [ d \tilde{F} |_{\mathfrak{S} ^3}]=\pi_3(j)(\Omega(f)) \in \pi_3 (GL^+(6, \R)) $, which is  $ -C( \Phi ) $ under the above identification.

Let us define $  \Psi : ( \R^6, 0) \to (\R^6, 0)$ by
\[  \Psi  (x_1, x_2, x_3, x_4, x_5, x_6) = (\psi (x_1, x_2, x_3, x_4), x_5,  x_6) \mbox{.} \]
$  \Psi $ is a germ of orientation preserving diffeomorphism extending $ \psi $.

Then $ \Lambda \circ \Phi \circ \Psi $ restricted on some neighbourhood $W' $ of $ \mathfrak{S}' $plays the role of a real extension of $ \Phi'|_{\mathfrak{S}'^3} =f'$. One can verify that the proof of Proposition~\ref{pr:main} works for this extension as well, since it sends the normal vector of $ \mathfrak{S}'^5 $ into a non-tangent vector of $ S^5$.
Therefore,
\[ [ d ( \Lambda \circ \Phi \circ \Psi ) |_{ \mathfrak{S}'^3 } ] = \pi_3 (j) ( \Omega (f')) \mbox{.} \]

Finally, note that $ [ d ( \Lambda \circ \Phi \circ \Psi ) |_{ \mathfrak{S}'^3 } ] =
 [ d \tilde{F} |_{\mathfrak{S} ^3}] $. This follows from the fact that the functions
 $ d \Lambda \circ \Phi \circ \Psi $ and $ d \Psi $ (with images in $ (GL^+(6, \R) $)
 extend to the ball $ \mathfrak{B}'^4_{ \epsilon} $.
\end{proof}

\subsection{Proof of Theorem \ref{TH:EMBINTRO}.} Part (a) follows from Theorem \ref{TH:MAIN} and
\cite{HM}.

In part (b),
the implications (1) $\Rightarrow$ (2,3,4), and (4) $\Rightarrow$ (3) are clear.

The proof of (2) $\Rightarrow$ (1): (2) implies $C(\Phi)=0$ by Theorem \ref{TH:MAIN}, while
this vanishing implies (1) via Mond's Theorem \ref{th:CT}.
For (3) $\Rightarrow$ (1) we provide three proofs, each of them emphasize a different
geometrical/topological aspect.

{\it (A) (Based on Mumford's Theorem.)} \ If $f$ is an embedding then the image
$(X,0)$ of $\Phi$ is an
isolated hypersurface singularity in $(\C^3,0)$. Moreover, its link is $S^3$, hence by Mumford's
theorem \cite{mumford} $(X,0)$ is smooth. Hence its normalization $\Phi$ is an isomorphism.

{\it (B) (Based on Mond's Theorem.)} \   Let us take the generic deformation $\Phi_\lambda$,
and consider the closure $D$ of the
preimage  of the the set of double values. It is a 1--dimensional  closed
complex analytic subspace of the disc in $\C^2$. The preimages of cross cap and triple points
are interior points of the closure of $D$, while its boundary is $D\cap S^3$ is the preimage of the
double points of the immersion of $f:S^3\looparrowright S^5$. If $f$ is an embedding then
$\partial D=\emptyset$, hence $D$ is a compact analytic curve in (the disc of) $\C^2$,
hence it should be empty.  This shows that $\Phi_\lambda$ has no cross cap and triple points either.
Hence $C(\Phi)=0$, which implies (1) by \ref{th:CT} as before.

{\it (C) (Based on Ekholm--Sz\H{u}cs Theorem.)} \ As above, we get that $\Phi_\lambda$ is an
embedding. Since $\Phi|_{S^3}$ is an embedding, this embedding is regular homotopic to
$\Phi_\lambda|_{S^3}$, hence they have the same Smale invariant. In the second case it
 can be determined by an
Ekholm--Sz\H{u}cs formula \cite{ESz} (recalled as Theorem \ref{th:ESz}): since ${\rm im}(\Phi_\lambda)$
is an embedded Seifert surface with signature zero we get $\Omega(f)=0$.
This basically proves (3) $\Rightarrow$ (2). Then we continue with the already shown
(2) $\Rightarrow$ (1).

\vspace{2mm}

In fact, the main point of this last proof is already coded in Hughes--Melvin Theorem
\cite{HM} (\ref{th:HM} here),
but in that statement the Seifert surface is in $\R^5$ and not in $\R^6_+$ (or in the
6--ball). But \ref{th:ESz} shows that that Hughes--Melvin Theorem is true even if the
4--manifold $M^4$ with boundary in $\R^5$
is embedded in $\R^6_+$ (instead of $\R^5$).

\section{Examples}\labelpar{s:ex}

\begin{ex}[$S_{k-1}$ from Mond's list \cite{Mond2}]\label{ex:1}  
 $ \Phi_{-k}(s, t) = (s, t^2, t^3 + s^k t ) $ ($k \in \Z_{\geq 0} $).
The ramification ideal $ J( \Phi_{-k}) $ is generated by  $ (2t, 3t^2 + s^k, -2kt^2 s^{k-1})= ( t, s^k ) $, cf. Subsection~\ref{ss:CTN} and Example~\ref{ex:cor1}.
Hence $  \Omega (f_{-k}) = - C( \Phi_{-k} ) = -k $.

 This family gives representatives for every regular homotopy class with non-positive
 sign--refined Smale invariant. Furthermore, we can represent any regular homotopy class with
 Smale invariant $ k $ in the form $ \Phi_{-k} \circ \tau $, where $ \tau $ is the reflection
  $ \tau (z, w) = (z, \bar{w}) $ (c.f. \cite[Lemma 3.4.2.]{ekholm3}).
  \end{ex}
  
  \begin{ex}[Quotient singularities]\label{ex:quot} The covering map germs of the quotient singularities of type $A$, $D$, $E$ are discussed in examples \ref{ex:A}, \ref{ex:A2}, \ref{ex:D} and \ref{ex:E}, where the $C$ invariant of them is calculated. Their associated immersions decompose in the form $ S^3 \to S^3/G \hookrightarrow S^5 $, cf. Example~\ref{ex:ade}. The list of their Smale invariants $ \Omega( f ) = -C( \Phi) $ is:
  \begin{equation}\labelpar{eq:adesmale} \begin{split}
  A_{n-1}: \ \ \Omega( f ) = -(n^2-1), \ \ \ D_{n+2}: \ \ \Omega( f ) = -(4 n^2+12n-1), \\
  E_6: \ \ \Omega( f ) = -167, \ \ \ E_7: \ \ \Omega( f ) = -383,
  \ \ \ E_8: \ \ \Omega( f ) = -1079 \mbox{.}
  \end{split}
  \end{equation}
  The Smale invariant of the $A$ and $D$ types agrees up to sign with the Smale invariant of the immersions constructed by Kinjo \cite{kinjo}, its discussion is in Subsection~\ref{ss:implumb}. Finding a direct relation between our immersions and Kinjo's constructions would enable us to determine the Smale invariant of the immersions associated with the graphs $E_6$, $E_7$, $E_8$. 
\end{ex}

\section{Ekholm-Sz\H{u}cs formulae for holomorphic germs $\Phi$}\labelpar{s:eszcomp}

\subsection{} In this section, from a stabilization of $\Phi$ we
construct a singular Seifert surface,
and we express the topological
 summands of (\ref{eq:hurk}) in terms of
holomorphic invariants. 
% As a corollary we specify the
% sign in the formulae (\ref{eq:hm}), (\ref{eq:cuspos}) and (\ref{eq:hurk}).

The topological formulae (\ref{eq:hm}), (\ref{eq:cuspos}) and (\ref{eq:hurk}) targeting the
Smale invariant in terms of the geometry of oriented Seifert surfaces are stated and proved only
 up to a sign ambiguity. In this next section we will  show that
the sign--refined Smale invariant
appears in all these expressions with a unique
well--defined sign, and we determine
it simultaneously for all  formulae. The discussion  has an extra output as well:
the topological ingredients in the formulae below
will get reinterpretations in terms of complex analytic invariants, provided that the immersion is
induced by a holomorphic germ $\Phi$.

First of all we have to reinterpret the formulae (\ref{eq:hm}), (\ref{eq:cuspos}) and (\ref{eq:hurk}) in the spirit of the discussion of Subsection \ref{ss:absor}. 
We replace $S^3 $ (which is the unit sphere in $ \R^4 $) with ${\bf S}^3$, the `oriented abstract $S^3$', and we also need to fix a `boundary convention', in order to have the notion
of oriented $\partial M^4$. Then $\partial M^4={\bf S}^3$ has a well-defined meaning, and $\partial B^4={ S}^3$ inherits an orientation, hence the boundary convention determines whether a diffeomorphism $  S^3 \to {\bf S}^3 $ is orientation preserving.
$\Omega(f)$ denotes
the Smale invariant (given by any of its definitions, still having its sign--ambiguity).
 Nevertheless, the sign--corrected formulae will be `boundary convention' free,
cf. Theorem \ref{thm:new}.
%Then formulae (\ref{eq:hm}), (\ref{eq:cuspos}) and (\ref{eq:hurk}) have the following forms:

\subsection{Singular Seifert surface associated with a stabilization}\labelpar{ss:assoc} 

Let $ \Phi: (\C^2, 0) \to (\C^3, 0) $ be a holomorphic germ singular only at the origin and let $ f: S^3 \looparrowright S^5 $ be the immersion associated with $ \Phi $. We take an $ \epsilon $ as in Corollary~\ref{cor:epsilon}, that is, we fix in the target a ball $B^6_\epsilon$.
We also consider a holomorphic  stabilization $ \Phi_{\lambda} $ of $ \Phi_0 = \Phi $ (cf. Section~\ref{ss:stab}),
and we fix $\lambda$ sufficiently small,  $ 0 < |\lambda| \ll \epsilon $, such that
 the cross caps and (if $ T( \Phi) < \infty$) the triple points of $ \Phi_{\lambda} $ sit in $ B^6_{\epsilon} $.
We set
$ \mathfrak{B}^4_{\epsilon, \lambda} := \Phi_{\lambda}^{-1} (B^6_{\epsilon}) $, it is a $ \mathcal{C}^{\infty}$
 non--metric ball in $\C^2$. Its  boundary is  $ \mathfrak{S}^3_{\epsilon, \lambda} :=
 \Phi_{\lambda}^{-1} (S^5_{\epsilon}) $,  it is canonically diffeomorphic to $ S^3 $.

\vspace{2mm}

%It is known (see \cite{}) that $ \Phi_{\epsilon \neq 0} $ has generic points, selftransversal duble points, cross caps and triple points. The double point sets are of complex dimension $1$ and they end in the isolated cross cap points and intersects each other in the isolated triple points in general position. Let $ T( \Phi_{\epsilon \neq 0} ) $ denote the number of triple points of $ \Phi_{\epsilon \neq 0} $. In some cases (see Section~\ref{}) $ T( \Phi_{\epsilon \neq 0} ) $ does not depend on the chosen deformation and we denote it by $ T(\Phi)$. In that cases $ T(\Phi)$ can be calculated algebraicly as the codimension of the second fitting ideal of $ \Phi $.
%\marginpar{\bf{ez a resz sztem a $ C( \Phi) $ definiciojahoz menjen majd.}}

%
The map $\Phi_\lambda$ is stable as a holomorphic map, but it is not
stable as a $\mathcal{C}^{\infty}$ map, cf. Example~\ref{ex:r46}. The $\mathcal{C}^{\infty}$ stability is obstructed by its cross cap points.
We will modify $\Phi_\lambda$ in the neighbourhood of these points according to the
following local model.

Let us fix local holomorphic coordinate systems in the source and the target  such that
$\Phi_\lambda$ in the neighbourhood of a cross cap has local equation $ \Phi^{loc} (s, t) = (s^2, st, t)$.
We consider its real smooth deformation (with $0\leq \tau \ll |\lambda|$):
\begin{equation}\label{eq:whpert}
\Phi^{loc}_\tau(s, t)=(s^2 + 2 \tau \bar{s}, st + \tau \bar{s}, t).
\end{equation}
Since the restriction of $\Phi^{loc}$ near the boundary of the local 4-ball is stable,
by a $ \mathcal{C}^{\infty}$ bump function the local deformation can be glued to the trivial deformation of
$\Phi_\lambda$ outside of local neighbourhoods of the cross caps. This gives a $ \mathcal{C}^{\infty}$ global
deformation $\Phi_{\lambda,\tau}$ of $\Phi_\lambda$ and $\Phi$.
The map
$ \tilde{f}=\Phi_{\lambda,\tau}:  (\mathfrak{B}^4_{\epsilon, \lambda}, \mathfrak{S}^3_{\epsilon, \lambda} )
  \to (B^6_\epsilon, S^5_\epsilon) $ is the  singular Seifert surface we will consider.
  %associated with the deformation $ \Phi_{\lambda} $.
Its restriction,  $ f_{\lambda} = \Phi_{\lambda,\tau} |_{ \mathfrak{S}^3_{\epsilon, \lambda}}
 =\Phi_{\lambda}|_{ \mathfrak{S}^3_{\epsilon, \lambda}}$ is
  the immersion associated with $ \Phi_{\lambda} $.

\begin{prop}\label{l:gen} \

(a)  $ \tilde{f} : {\mathfrak B}^4_{\epsilon,\lambda} \to \C^3 $
is a stable smooth map, nonsingular near the boundary.

 (b) $ f_{\lambda} $ is a stable immersion and it is regular homotopic to $ f$.

 (c) If $ f $ is a stable immersion, then $ f_{\lambda} $ is regular homotopic to $f$ through generic immersions. In this case $ L(f) = L ( f_{\lambda}) $.
\end{prop}

\begin{proof}
(a) First one checks that the local $\Phi^{loc}_\tau$
 is stable. This follows from the  computation from
 Section~\ref{ss:lwh}. % that $ \tilde{\Phi}_{0, \epsilon} $ has not any triple points,
  Its most complicated singularities are $ \Sigma^{1, 0} $ (fold) points,
  the singular values constitute  an $ S^1 $, which -- together with the double values of the image
  of the boundary of the local ball -- bounds the $2$-manifold of the double values. Cf.
  \cite[2.3.]{ESz}.

In the complement of local balls $\Phi_{\lambda,\tau} $ agrees with $ \Phi_{\lambda} $, hence it has only
simple points, self-transverse double points and isolated triple points. All of them are stable.
Hence $\tilde{f}$ has all the local properties (multigerms) of a stable map, thus it is stable by Remark~\ref{re:globstab}.
%(and, in fact, this is enough
% in the determination of all the invariants, cf. \cite{ESz}).

(b) $\Phi_\lambda|_{ \mathfrak{S}^3_{\epsilon, \lambda}} $ is stable in real sense too:
 it has only simple points and self-transverse double points.
 $\Phi_{h\lambda}|_{ \mathfrak{S}^3_{\epsilon, \lambda}} $
 is a  regular homotopy between $ f $ and $ f_{\lambda} $ ($ h \in [0, 1] $).

(c) Being a stable immersion is an open condition (cf. \cite[2.1.]{ESz}).
Furthermore, $ L $ is constant along a regular homotopy through generic immersions, cf. \cite{ekholm3}.
\end{proof}

Next,  we return back to the formula (\ref{eq:hurk}), applied for $\tilde{f}$.
%such that we put in the left hand side the sign--refined Smale invariant.
  Clearly, $\sigma(M^4)=0$.

\begin{thm}\label{thm:L} Let $ \Phi_{\lambda} $ be a holomorphic stabilization of $ \Phi $ with fixed $ \lambda \neq 0 $ and the corresponding maps $ \tilde{f} $ and $ f_{\lambda} $ as above. Then the following statements hold.

(a) $ t( \tilde{f}) = T ( \Phi_{\lambda}) $  (cf. \ref{pr:trip}).

(b)  $ l(  \tilde{f} ) = C (\Phi) $.

(c) $ L(f_{\lambda}) = C( \Phi) - 3 T ( \Phi_{\lambda}) $.

\end{thm}

For the proof see \ref{ss:pro}.

Note that $ L(f_{\lambda}) $ is an \emph{analytic} invariant of $ \Phi $, since it is defined as a (topological) invariant of an \emph{analytic} deformation.
Recall from Proposition~\ref{pr:trip} that if $T(\Phi)<\infty$, then $ T ( \Phi_{\lambda}) $ is independent of the deformation
$ \Phi_{\lambda}$ and is equal to $ T( \Phi )$.

\begin{cor}\label{cor:Lfug} If $T(\Phi)<\infty$, then $ t( \tilde{f}) = T ( \Phi) $
 and $ L(f_\lambda ) = C( \Phi) - 3 T ( \Phi)$ is also independent of the holomorphic stabilization $ \Phi_{ \lambda} $ of $ \Phi$.
 \end{cor}

 \begin{rem}\label{rem:T} 
 
 (a) Assume that $ \Phi $ is finitely determined, that is, $ \Phi |_{\mathfrak{S}^3_{ \epsilon}} = f $ is a \emph{stable} immersion, cf. Theorem~\ref{th:fin-stab}. Then $T(\Phi)$ is finite, cf. Theorem~\ref{th:CT}, and for any holomorphic stabilization $ \Phi_{\lambda} $ one has $ L( f_{ \lambda}) = L(f) $ (cf. Proposition~\ref{l:gen}), hence
 $ C( \Phi) - 3 T( \Phi) = L(f) $.

 (b) Note that if $ \Phi $ is finitely determined
% the restriction $ \Phi |_{\mathfrak{S}^3_{ \epsilon}} = f $ of a germ $ \Phi $ is a stable immersion,
  and the holomorphic germ $ \Phi' $ is $ {C}^{\infty}$ left-right equivalent with $ \Phi $ (see Corollary~\ref{cor:C} for precise definition), then the immersion $ f'$ associated with $ \Phi' $ is also stable immersion and $ L(f)= L(f') $. Therefore by Corollaries~\ref{cor:C} and \ref{cor:Lfug} we have $ T( \Phi) = T ( \Phi') $ too.

(c) More generally, $\Phi$ and $\Phi'$ are topological left--right equivalent and both of them are finitely $ \mathscr{A}$-determined (so $f$ and $f'$ are stable immersions), then $L(f)=L(f')$, hence Corollary
\ref{cor:NEW} follows too.

(d) If $ C( \Phi ) $ and $ T( \Phi ) $ are finite, then $f$ is not necessarily stable immersion, but it does not have triple points. Then by Remark~\ref{re:notstabL} $ L(f) $ is well-defined as $L$ of any small stable perturbation of $ f$ with a regular homotopy.
%A small enough regular homotopy of $f$ does not step through a triple value, hence $L$ does not change through such a regular homotopy. So $L(f) $ can be defined as $L(g)$ of any stable immersion $g$ which is regular homotopic with $f$ and $f$ and $g$ are close enough to each other, cf.~\ref{}. 
This provides an other proof for Corollary~\ref{cor:Lfug}. However it is not clear, how can  $ L(f) $ be determined from the topology of $f$ itself, without stable perturbation. That obstructs the generalization of part (b) and (c) for germs with finite $C$ and $T $.
 \end{rem}

\begin{thm}\label{thm:new} With our sign--convention, if in the left hand side of the formulae
(\ref{eq:hm}), (\ref{eq:cuspos}) and (\ref{eq:hurk}) we put the sign--refined Smale invariant
$\Omega^a(f)$, then
the formulae are valid if we put the positive sign on the right hand sides.

In particular, the validity of these sign--corrected formulae (e.g., $\Omega^a(f)=\frac{3}{2}\sigma(M^4)$)
is independent of the `boundary convention': changing the boundary convention changes the sign in both
sides of the formulae simultaneously.
\end{thm}

The proof will appear in  \ref{ss:pro}.

\begin{rem}
The formula (\ref{eq:cuspos}), involving the (algebraic) number of real cusps of maps $g:M^4\to \R^5$
is the real analogue of our theorem $\Omega(f)=-C(\Phi)$, involving the
number of  cross caps of holomorphic deformations. This suggests
that if we replace a holomorphic deformation by a smooth generic map, then we
trade each cross cup by $-2$ real cusps.
\end{rem}

\subsection{}\labelpar{ss:pro} \emph{Proof.} We prove Theorems~\ref{thm:L} and \ref{thm:new} simultanously (see also the discussion from Subsection \ref{ss:absor}).

In the definitions of the invariants $t$, $l$ and $L$ one uses very specific sign/orientation
conventions, based on the orientation of the involved subspaces in their definition.

For a triple value, the sign is determined in such a way that it is $+1$ whenever the triple
value is obtained from a holomorphic triple point (hence the orientations agree with
the complex orientations).

Since in the local deformation $\Phi^{loc}_\tau$ we do not create any new triple value,
see e.g. the computation of Section \ref{s:calc}, all the triple values of $\tilde{f}$
come from the complex triple points of the holomorphic $\Phi_\lambda$, hence (a) follows.

The proof of the remaining parts are based on computations of the invariants
$C(\Phi)$, $T(\Phi)$, $l(\tilde{f})$ and $L(f)$ for two concrete cases.
For the integers $l$ and $L$ the definitions (orientation conventions)
are not immediate even in simple cases. Therefore, in our computation
we determine them only up to a sign.
The point is that computing `sufficiently many' examples, the formula (\ref{eq:hurk}),
even with its sign ambiguity in front of the right hand side, and even with
the (new) sign ambiguities of the integers $l$ and $L$, determine uniquely all these signs.
(This also shows  that, in fact, there is a unique
universal way to fix the orientation conventions and signs in the definitions of
$l$ and $L$ such that  (\ref{eq:hurk})
works universally.)

In Section \ref{s:calc} we will determine the following data:
\begin{equation}\label{EQ:CRC}\begin{split}
 \mbox{(i)}\ \mbox{For cross cup:} \ \ \ C(\Phi)=1, \ T(\Phi)=0, \ l=\pm 1, \ L=\pm 1.\\
%\end{equation}
%\begin{equation}\label{eq:A1}
\mbox{(ii)} \ \ \ \ \ \ \ \ \ \
\mbox{For $A_1$:} \ \ \ C(\Phi)=3, \ T(\Phi)=1, \ L=0.\ \ \ \ \ \hspace{1cm}
\end{split}
\end{equation}

The singular values of $ \tilde{f} $ are concentrated near the cross caps of $ \Phi_{\lambda} $.
For $\Phi^{loc}_\tau$ the value  $ l $ is $ \pm 1 $, see (i).
Since the sign is the same for all cross caps, $ l( \tilde{f})= \pm C( \Phi ) $.

We introduce the notation
\begin{equation}\label{eq:omvessz}
 \Omega' (f_{\lambda}) :=  \frac{1}{2} ( 3 t(\tilde{f}) - 3 l(\tilde{f}) + L(f_{\lambda})).
\end{equation}

$ \Omega'(f_{\lambda}) $ agrees with $ \Omega(f) $ up to sign, thus $ \Omega'(f_{\lambda}) =
\pm C( \Phi) $. Substituting this and  the data (i) of the cross cap in (\ref{eq:omvessz})
 we conclude that $ l(\tilde{f}) = - \Omega' (f_{\lambda}) $ and
  $ L(f_{\lambda}) = \pm C( \Phi ) -  3 T ( \Phi_{\lambda})$.

Next, using the date (ii) for  $ A_1 $, all the remaining sign ambiguities can be eliminated:
$ L(f_\lambda) = C( \Phi ) -  3 T ( \Phi_\lambda)$, $ l( \tilde{f})= C( \Phi ) $ and
$ \Omega'(f_{\lambda}) = - C( \Phi) = \Omega (f) $.

The universal signs in formulae
(\ref{eq:hm}), (\ref{eq:cuspos}) and (\ref{eq:hurk}) are related by common examples, hence
one of them determines all of them.

\section{Calculations. The proof of (\ref{EQ:CRC}).}\labelpar{s:calc}

We show the main steps of the computations.
Note that if
the germ $ \Phi $ is weighted homogeneous, then $ \epsilon_0 =1 $ can be chosen.

\subsection{The case of cross cap.}\labelpar{ss:lwh}

$T(\Phi)=0$ and 
%see e.g. \cite{Mond1,Mond2}; 
$C(\Phi)=1$ is clear, cf. examples \ref{ex:cor1} and \ref{ex:Sig10}. Next we compute $l$ and $L$.
Set $ \Phi(s, t)= (s^2, st, t) $ and the  smooth perturbation
$ \tilde{f}(s, t)=(s^2 + 2 \epsilon \bar{s}, st + \epsilon \bar{s}, t)$.
The singular locus is
$\tilde{ \Sigma }=  \{ (s,t) \ | \ s=t \mbox{ , } |s|=|t|= \epsilon \} \simeq S^1$.

$ \tilde{f}|_{\tilde{\Sigma}}$ has no singular point, hence  $ \tilde{f}$  has no cusp points.
The most complicated singularities of $ \tilde{f} $ are $ \Sigma^{1, 0} $ (or fold) points.
The closure of the set of the double points $ \tilde{D} $ of $ \tilde{f} $ is
\[ \mbox{cl} ( \tilde{D} ) =
\{ (s, t) \in \mathbb{C}^2 \ | \ (s - t)t + \epsilon ( \bar{s} - \bar{t}) = 0 \}
\]
with the involution  $(s, t) \mapsto (s', t)= (2t-s, t)$.
The fix point set of the involution is $\{s=t\}$. Thus the set of the double points is
\[ \tilde{D} =
\{ (s, t) \in \mathbb{C}^2 \ | \ (s - t)t + \epsilon ( \bar{s} - \bar{t}) = 0 \}
\setminus \{s=t\} \mbox{.}
\]
Each double point has exactly one pair with the same value, hence
$ \tilde{f} $ has no triple point.

 A parametrization of $ \tilde{D}$ is
$
 (\rho, \alpha) \mapsto ( -\epsilon e^{-2 \alpha  i } + \rho e^{i \alpha} , -\epsilon e^{-2 \alpha  i }),
$
where $ \rho \in \R_+ $, $ \alpha \in [0, 2 \pi) $.

The parametrization shows that the closure of $ \tilde{D} $ is a M\"{o}bius band.
For $ \rho=0 $ we get $ \tilde{\Sigma} $, which is the midline of the M\"{o}bius band.
 The set of double values is
\begin{align*} D= \tilde{f} (\tilde{D}) &= \{ (s^2 + 2 \epsilon \bar{s}, st + \epsilon \bar{s}, t) \ | \ (s, t) \in \tilde{D} \} \\
                            &= \{ ( \rho^2 e^{2 i \alpha}+ \epsilon^2 e^{2 i \alpha } ( e^{-6 i \alpha }-2 ),
\epsilon^2 ( e^{ - 4 i \alpha } - e^{ 2 i \alpha } ) ,
-\epsilon e^{-2 \alpha  i } ) \ | \ \rho \in \R_+ \mbox{ , } \alpha \in [0, 2 \pi) \}.
\end{align*}
Writing $ \rho = 0 $ we get the singular values of $ \tilde{f} $,
\[ \Sigma= \tilde{f} (\tilde{\Sigma} ) =
\{ ( \epsilon^2 e^{2 i \alpha } ( e^{-6 i \alpha }-2 ),
\epsilon^2 ( e^{ - 4 i \alpha } - e^{ 2 i \alpha } ) ,
-\epsilon e^{-2 \alpha  i } ) \} \mbox{ .}
\]

The inward normal field of $ \Sigma  $ in $ D $ is the derivative of the curve
\[ \gamma(t)=  (t e^{2 i \alpha} + \epsilon^2 e^{2 i \alpha } ( e^{-6 i \alpha }-2 ),
\epsilon^2 ( e^{ - 4 i \alpha } - e^{ 2 i \alpha } ) ,
-\epsilon e^{-2 \alpha  i } )  \]
at $ t=0 $, that is $\gamma'(t)|_{t=0} = (e^{2 i \alpha}, 0, 0)$.
The pushing out of $ \Sigma$ (cf. Definition~\ref{d:l}) is
\[ \Sigma'= \Sigma - \delta \cdot \gamma'(t)|_{t=0}
= \{ ( - \delta e^{2 i \alpha } + \epsilon^2 e^{2 i \alpha } ( e^{-6 i \alpha }-2 ),
\epsilon^2 ( e^{ - 4 i \alpha } - e^{ 2 i \alpha } ) ,
-\epsilon e^{-2 \alpha  i } ) \} \mbox{ ,}
\]
where $ 0 < \delta \ll \epsilon $.
By Definition~\ref{d:l} we need the linking number of $ \tilde{f} ( \R^4) $ and $ \Sigma' $ in $ \R^6 $. To calculate it we fill in $ \Sigma' \simeq S^1 $ with a `membrane', which here will be   the disc
\[ H= \{ (- \delta w + \epsilon^2 (\bar{w}^2-2 w), \epsilon^2 (\bar{w}^2-w), - \epsilon \bar{w} ) \ | \ w \in \C \mbox{ , } |w| \leq 1 \}.
\]
$l(\tilde{f})$ is the algebraic number of the intersection points of $ H $ and $ \tilde{f} ( \R^4) $.
The  only solution is $ w=0$, $ (s, t)=(0,0) $, and the
intersection at this point is transverse.
Hence, for the smooth perturbation $ \tilde{f} $ of the cross cap $ l( \tilde{f}) = \pm 1 $.

Next we compute $L$.
The set of the double points of $ \Phi $ is
$\tilde{D}= \{ (s, 0) \ | \ s \neq 0 \} \subset \mathbb{C}^2$.

The set of the double values is
$
D= \Phi(\tilde{D}) = \{ (s^2, 0, 0) \ | \ s \neq 0 \} \subset \mathbb{C}^3$,
and the set of the double values of $ f $ is
$
 D_f = D \cap S^5 = \{ (s^2, 0, 0) \ | \ | s | =1 \} \subset S^5$.

The sum of the normal vectors at $ (s^2, 0, 0 ) $ is $(0,0,\bar{s}^2)$. Hence
the shifted copy of $ D $ along $ N $ is
$D'= D_f + \delta N = \{ (s^2, 0, \delta \bar {s}^2 ) \ | \ | s | =1 \}$.

Since $ D' $ does not intersect $ \Phi (\C^2 ) $ for $ \delta \in (0, 1] $,
 we can choose $ \delta = 1 $.
An injective parametrization of $ D_f + \delta N $ is
$
 D' = \{ (z, 0,  \bar {z} ) \ | \ | z | =1 \},
$
where $ z=s^2 $.
To calculate the linking number of $ \Phi( \C^2 ) $ and $ D' $ in $ \R^6 $, we need a
membrane which fills in $ D $. We take
\[
 H = \{ (z, \sqrt{1-|z|^2},  \bar {z} ) \ | \ | z | \leq 1 \} \simeq D^2 \mbox{ .}
\]
$L(f)$ is the algebraic number of the intersection points of $ \Phi( \C^2 ) $ and $ H $.
But there is only one such point, namely
$P:= \Phi ( \sqrt{\xi}, \xi) = ( \xi, \xi \sqrt{\xi}, \xi )$,
where $\xi$ is the real root of  $ g(z) :=z^3 + z^2 -1 =0 $.
Moreover, this intersection is transverse.

\subsection{The $A_1$ singularity}\labelpar{ss:A} By \ref{ex:A}, \ref{ex:A2} its covering is
given by $ \Phi_{0} (s, t) = (s^2, t^2, st) $. The immersion $ f_0$ associated with $ \Phi_0 $ is not generic, $f_0$ is the $2$-fold covering of the projective space composed with the inclusion. Thus all points of $ S^3 $ are double points of the immersion $ f_0$. Compare $f_0$ with the immersion $i \circ g_1 $ described in the end of Subsection~\ref{ss:implumb}, see also \cite[Section 4]{EkTak}. $ i \circ g_1 $ is regular homotopic to $f_0$ (up to precomposing with a reflection), and has the same structure $ S^3 \to \R \mathbb{P}^3 \hookrightarrow \R^5 $ and vanishing $L$-invariant.

On the other hand, by \ref{ex:A2}, $ C( \Phi_0 ) = 3 $, and the codimension of the second fitting ideal 
shows that $ T( \Phi_0) = 1 $, see \ref{ex:A}. The finiteness of these invariants shows that the number of cross caps and  triple points of a generic deformation of $ \Phi_0 $ are independent
of the chosen deformation. Below we give a concrete deformation $ \Phi_{\epsilon}$ of $ \Phi_0 $ and we calculate the invariant $ L $ of the generic immersion $ f_{ \epsilon} $ associated with $\Phi_{\epsilon}$.

The deformation is $\Phi_{\epsilon}(s, t) = ((s- \epsilon ) s , (t- \epsilon ) t , st ) $, see Figure~\ref{fig:A1fig}.
The vector field
\[ \tilde{N} (s, t) = \overline{  \partial_s \Phi_\epsilon (s, t) \times \partial_t \Phi
_\epsilon (s, t) } =
\left( \begin{array}{c}
\bar{t} (2 \bar{t} - \epsilon ) \\
- \bar{s} (2 \bar{s} - \epsilon ) \\
(2 \bar{s} - \epsilon ) (2 \bar{t} - \epsilon )
\end{array} \right)  \]
is $ 0 $ at the points $(0, \epsilon /2 )$, $( \epsilon/2, 0 ) $ and $( \epsilon/2, \epsilon/2 ) $. These  are the cross caps.
 %points of $ \Phi $. Out of these points $ \tilde{N} $ is a normal vector field of $ \Phi $.

The defining equation $ \Phi_\epsilon (s, t) = \Phi_\epsilon (s', t' ) $ (where $(s, t) \neq (s', t') $) of the double points leads to the system of equations
\[
 (s-s')(s+s'- \epsilon ) = 0, \ \
 (t-t')(t+t' - \epsilon ) = 0, \ \
 st=s't'.
\]
Thus the double locus $ \tilde{D} $  has three parts and these parts correspond to the three cross caps.
The first part comes from the solution $ s'=s $ and $ t'= \epsilon - t $, which implies $ s=0 $,
hence $
 \tilde{D}_1 = \{ (0, t) \ | \ t \neq \epsilon/2 \}$
with $ \Phi_\epsilon (0, t) = \Phi_\epsilon (0, \epsilon - t) $. This provide the double value set
\[
 D_1 = \Phi_\epsilon ( \tilde{D_1} ) = \{ (0, t (t- \epsilon ) , 0) \ | \ t \neq  \epsilon/2 \}.
\]
The second part comes from the solution $ s'= \epsilon - s $ and $ t'=  t $, which implies $ t=0 $,
and $
 \tilde{D}_2 = \{ (s, 0) \ | \ s \neq \epsilon/2 \}$
with $ \Phi_\epsilon (s, 0) = \Phi_\epsilon (\epsilon - s, 0) $. The set of double values is
\[
 D_2 = \Phi_\epsilon ( \tilde{D_2} ) = \{ (s (s- \epsilon ) , 0,  0) \ | \ s \neq  \epsilon/2 \}.
\]
The third part comes from the solution $ s'= \epsilon - s $ and $ t'= \epsilon - t $, which implies $ s + t = \epsilon $, and
$
 \tilde{D}_3 = \{ (s, \epsilon - s)\} \ | \ s \neq \epsilon/2 \}$
with $ \Phi_\epsilon (s, \epsilon - s) = \Phi_\epsilon (\epsilon - s, s) $. The set of double values is
\[
 D_3 = \Phi_\epsilon ( \tilde{D_3} ) = \{ (s (s- \epsilon ) , s (s- \epsilon ),  -s (s- \epsilon )) \ | \ s \neq  \epsilon/2 \}.
\]
$ D_1 $, $ D_2 $ and $ D_3 $ intersect each other in the unique triple value
$ \Phi_\epsilon (0, 0) = \Phi_\epsilon ( \epsilon, 0) = \Phi_\epsilon (0, \epsilon ) = (0, 0, 0) $.

Let $ D_i(f) = D_i \cap S^5 $ ($i=1, 2, 3$) denote the disjoint components of the set of the double
values of $ f$. Clearly $ L(f) = L_1(f) + L_2 (f) + L_3 (f) $, where $ L_i (f) $ is the linking number corresponding to the component $ D_i(f) $.
But $ L_1(f)=L_2(f)=L_3(f)$. Indeed,
%Consider the orientation preserving homeomorphisms $ \phi: \C^2 \to \C^2 $ and $ \psi: \C^3 \to \C^3 $ such that $ \psi \circ \Phi = \Phi \circ \phi $. Clearly $ \psi $ induces a self-homeomorphism of the image of $ \Phi $.
$ D_1 $ and $ D_2 $ is interchanged via the
 transformations $ \phi (s, t) = (t, s) $ (of $\C^2$)
  and $ \psi (X, Y, Z) = (Y, X, Z) $ (of $\C^3$), and
   $ D_3 $ and $ D_2 $ via $ \phi(s, t) = ( \epsilon - s - t, t) $ and $ \psi (X, Y, Z) = (X+Y+2Z, Y, -Y-Z)$. Thus, it is enough to calculate $ L_1(f) $.
The needed vector field along $ D_1 $ is
\[ N (0, t (t- \epsilon ), 0) = \tilde{N} (0, t) + \tilde{N} (0, \epsilon - t) =
((2 \bar{t} - \epsilon )^2,0,0).\]
The set of the double values of $f$  corresponding to $D_1$ is
\[
 D_1 (f) = D_1 \cap S^5 = \{ (0, t (t- \epsilon ) , 0) \ | \ |t (t- \epsilon )| = 1 \} \mbox{ .}
\]
The shifted  $ D_1 (f) $ along $ N$ is
\[
D'_1 = D_1 (f) + \delta N = \{ (\delta (2 \bar{t} - \epsilon )^2, t (t- \epsilon ) , 0) \ | \ |t (t- \epsilon )| = 1 \} \mbox{ ,}
\]
where $ \delta $ is small enough. Nevertheless, we can choose $ \delta = 1 $, because
$ D'_1  \cap \Phi ( \C^2 ) = \emptyset $ for any $ \delta \in (0, 1] $.  With the notation $ z = t(t- \epsilon ) $ we give an injective parametrization
$
D'_1 = \{ (4 \bar{z} + \epsilon^2 , z , 0) \ | \ |z  | = 1\}$.
We fill it with the membrane
\[
H = \{ (4 \bar{z} + \epsilon^2 , z , i \sqrt{1- |z|^2} ) \ | \ |z  | \leq 1 \} \mbox{ .}
\]
Computing the intersection points of $ H $ and $ \Phi( \C^2 ) $ leads to the equations
\[
4 \bar{z} + \epsilon^2 = a (a - \epsilon ), \ \
z = b (b - \epsilon ), \ \
i \sqrt{1- |z|^2} = ab,
\]
with $ | z | \leq 1 $ and $ \epsilon $ small.
The first two equations imply that $ |a| < 5 $ and $ |b| < 2 $. Multiplying the first two equations one gets
\[
z (4 \bar{z} + \epsilon^2 ) = a^2 b^2 - a^2 b \epsilon - a b^2 \epsilon + ab \epsilon^2.
\]
From the third equation follows $ a^2 b^2 = |z|^2 - 1 $, hence
\[
3 |z|^2 = -1 -  z \epsilon^2 - a^2 b \epsilon - a b^2 \epsilon + ab \epsilon^2,
\]
and the right hand side is negative if $ \epsilon $ is small enough. Hence $ H \cap \Phi ( \C^2 ) = \emptyset $, and $L(f)=0$.

\section{A note on the real case}\label{s:realcase}

\subsection{Immersions associated with real germs}
 There is a real version of  part (b) of Theorem~\ref{TH:MAIN} which follows directly from the result of Whitney and Smale.

 Let $ \Phi: (\R^{n+1},0)  \to (\R^{2n+1}, 0) $ be a real  analytic germ singular only at $0$.
  With the same method as in the complex case we can associate  an immersion $ f: S^n \looparrowright S^{2n} $ with $ \Phi$ (see \ref{ss:link}). The set of regular homotopy classes is
  \[ \imm(S^n, S^{2n}) = \left\{ \begin{array}{ccc}
\Z & \mbox{if} & n \mbox{ is even,}  \\
\Z_2 & \mbox{if} &  n \mbox{ is odd,} \\
\end{array}
\right. 
  \]
  %  A generalization of Whitney's double point formula valid for plane curve immersions \cite{Whitney} 
   % shows that the Smale invariant of $ f $
   and the invariant equals to the algebraic number of self-intersection points of  a stable immersion regular homotopic to $f$ (${\rm mod}\ 2 $ if $ n $ is odd). Cf. examples \ref{ex:snr2n}, \ref{ex:s1s2} and \ref{ex:snsq}.

   A stabilization $ \Phi' $ of $ \Phi $ has only cross cap type singularities, i.e.
   locally right--left equivalent with germs of
   the form $ (s, t) \mapsto (s^2, st, t) $, where $ s \in \R $ and $ t \in \R^n $, cf. Example~\ref{ex:r23} and Remark~\ref{rem:whhigh}.
   These cross caps are isolated, and if $n$ is even, we can associate a sign for each of them as a boundary component of the corresponding oriented double point curve. $ \Phi' $ restricted to the boundary is a stable immersion $ f': S^n \looparrowright S^{2n} $. $f'$ and $f$ are regular homotopic, and $f'$ has two kinds of double values:

   (a) double values related to a cross cap (that is, they are connected by a segment consisting
   of double values of $\Phi'$),

   (b) double values not related to a cross cap.

   When $n$ is even, the sign associated to a cross cap agrees with the sign associated with the self intersection point of $ f'$ related to the cross cap. Thus the algebraic number of such
   cross caps is equal to the algebraic number of double values of type (a) ($ {\rm mod}\ 2$ if $ n $ is odd). The double points of type (b) are pairwise joined up by segments of the double values of $ \Phi' $, thus the algebraic number of them is $ 0$.
   Moreover, it can happen that two cross caps are joined by a segment consisting of double values of $\Phi'$,
   but then they will have different algebraic sign, hence they will not contribute in the sum.
   Hence, we proved:

   \begin{prop}
   The Smale invariant of $f$ agrees with the algebraic number of the cross cap points appearing in a stabilization of $ \Phi $  (${\rm mod}\ 2 $ if $ n $ is odd).
   \end{prop}
   
   \begin{rem}
   In \cite{nunodoodle} a complete topological invariant of finitely determined real germs $ \Phi: ( \R^2, 0) \to ( \R^3, 0) $ is introduced and calculated in many cases: the Gauss word of the associated immersion $ \Phi|_{S^1}: S^1 \looparrowright S^2 $.
   \end{rem}
   
 %  \subsection{Topological characterization}\label{ss:doodles}
   
 %  \subsection{Vassiliev invariants for maps from $ \R^2$ to $ \R^3$}\label{ss:vasi-r2r3}

\chapter{Singular holomorphic spacegerms}\label{ch:iso}

\section{The Milnor fibre}

\subsection{Isolated, complete intersection and normal singularities} This chapter serves as an introduction and background to Chapter~\ref{ch:bound}, where we study the Milnor fibre of non-isolated hypersurfaces.

The aim of this subsection is to give a short introduction of the basic types of singularities, and clarify the position of our mostly interested type of surface singularities in this classification, see Remark~\ref{re:surfacenormal}. For the definitions and basic properties we refer to \cite{looijenga, dejong, Five}.

A \emph{complex analytic spacegerm} $ (X,0) \subset ( \C^{n+k}, 0) $ is the zero set of a holomorphic germ $ f=(f_1, f_2, \dots, f_k ): ( \C^{n+k}, 0) \to ( \C^k, 0) $ \cite[Definition 3.4.2]{dejong}. 
%Assume that $k$ is the minimal number of the equations defining $ (X, 0) $ in  $ ( \C^{n+k}, 0)$, that is, $k$ is the minimal number of the generators of the ideal 
We assume that the ideal $ I:=(f_1, f_2, \dots, f_k ) \subset \mathcal{O}_{ ( \C^{n+k}, 0)}$ is a radical ideal (i.e. $ \sqrt{I}=I$), in this case $I$ provides \emph{reduced analytic structure} of $ (X, 0)$. (In general, by local Nullstellensatz \cite[Theorem 3.4.4.]{dejong}, $\sqrt{I} $ is equal to the set of germs vanishing on $(X, 0)$.)

If $ k=1 $, then $ (X, 0) $ is called \emph{hypersurface} in  $ ( \C^{n+1}, 0)$, see \cite[(1.A)]{looijenga}.  If the rank of $ df_0 $ is equal to $k$, then $ (X, 0) $ is \emph{regular} (or smooth), i.e. it is the germ of a complex manifold.
% where $ \dim X $ is the complex dimension of $ X$ at a regular point 
The \emph{dimension} of $(X, 0) $ is defined as $ (X, 0) = n+k-r $, where $r$ is the rank of the Jacobian $df $ at a generic point $ p \in X$. $(X, 0) $ is called \emph{complete intersection} if $ \dim_{ \C } (X, 0) = n $
\cite[(1.5)]{looijenga}. Any hypersuface singularity is a complete intersection by \cite[(1.1)]{looijenga}. 
  If there is a small enough representative of $ (X, 0)$ such that $ \rk (df_p) =r $ for all $ p \in X \setminus \{ 0 \} $ and $ \rk (df_0) <r $, then $ (X, 0) $ has an \emph{isolated singularity} at $0$ (we also say that $ (X, 0) $ is isolated).
% \marginpar{mi az h $X$ dimenzioja?????}

$(X, 0) $ is \emph{irreducible} if for any decomposition $ (X, 0)=(X_1, 0) \cup (X_2, 0) $, where $ (X_1, 0) $ and $ (X_2,0) $  are complex analytic spacegerms, $(X_1, 0)=(X,0)$ or $ (X_2, 0)=(X, 0)$ holds \cite[Definition 3.4.17.]{dejong}.  $ (X, 0) $ is irreducible if and only if $ I \subset \mathcal{O}_{ ( \C^{n+k}, 0)} $
%a primary ideal, i.e. $ \sqrt{I}$ 
is a prime ideal \cite[Corollary 3.4.18.]{dejong}. For a hypersurface singularity $ (X, 0) \subset ( \C^{n+1}, 0) $ that means $ f$ is irreducible in $ \mathcal{O}_{ ( \C^{n+1}, 0)} $.

The local ring of analytic functions defined on $ (X, 0)$ is $ \mathcal{O}_{(X, 0)}= \mathcal{O}_{( \C^{n+k}, 0)} /\sqrt{I}$, cf. \cite[Definition 3.4.19., Lemma 3.4.20.]{dejong}. An irreducible singularity $ (X, 0) $ is called \emph{normal} if $ \mathcal{O}_{(X, 0)} $ is an integrally closed ring in its field of fractions, or equivalently: if any bounded holomorphic function $ g: X \setminus \{ 0 \} \to \C $ can be extended to an analytic function defined on $ X $, see \cite[Definition 1.3.]{Five}, and also \cite[Theorem 4.4.15.]{dejong}.

The \emph{normalization} of an irreducible germ $ (X, 0) $ is a normal germ $(\bar{X}, 0)$ together with a finite, generically $1$ to $ 1$ map germ $g: (\bar{X}, 0) \to (X, 0)$, cf. \cite[Definition 4.4.5.]{dejong}. Every irreducible germ admits a unique normalization. If $g: (\bar{X}, 0) \to (X, 0)$ is the normalization of $ (X, 0)$, then $ \mathcal{O}_{ ( \bar{X}, 0)} $ is isomorphic with the integral closure of $ \mathcal{O}_{ ( X, 0)} $ \cite[1.5.]{Five}. See \cite[Remark 4.4.6., Theorem 4.4.8.]{dejong}.

\begin{ex}[Curves]\label{ex:curvesnormal} Let $ (X, 0) $ be an irreducible curve singularity, i.e. $ \dim_{ \C } X = 1$. $(X, 0)$ is normal if and only if it is smooth. Hence the normalization of a curve is a parametrization $ p: ( \C, 0) \to (X, 0) $, see \cite[Theorem 4.4.9., 4.4.10.]{dejong}.
\end{ex}

\begin{ex}[Surfaces]\label{ex:isonormal} A  normal $2$-dimensional singularity is always isolated (or regular), and a complete intersection surface germ is normal if and only if it has (at most)   an isolated singularity \cite[1.7 (b), 1.8.]{Five}. For instance, the quotient singularities in Example~\ref{ex:ade} are normal hypersurface singularities, hence their covering germs $ \Phi: ( \C^2, 0) \to (X, 0) \subset ( \C^3, 0) $ are not normalizations.
\end{ex}

\begin{ex}[Finitely determined germs]\label{ex:findetnormal} A finitely determined germ $ \Phi: ( \C^2, 0) \to ( \C^3, 0) $ is the normalization of its image $ (X, 0) = f^{-1}(0) $, where $ f $ is the generator of the $0$-th Fitting ideal $ \mathcal{F}_0 (\Phi_* \mathcal{O}_{ ( \C^n, 0) }) $, cf. Section~\ref{s:Fitting}. It can be illustrated by the Whitney umbrella $ (X, 0) \subset ( \C^3, 0) $ parametrized by $ \Phi(s, t)= (s, t^2, st) $ or given as the zero set of $ f(x, y, z)= x^2y-z^2 $. Since $t^2-y=0$, $ t =z/x $ is an element of the integral closure of $ \mathcal{O}_{ ( X, 0)} $ in the fraction field of $ \mathcal{O}_{ ( X, 0)} $. The extension of $ \mathcal{O}_{ ( X, 0)} $ with $ t$ is  the integral closure of $ \mathcal{O}_{ ( X, 0)} $, and it is isomorphic with $ \mathcal{O}_{ ( \C^2, 0)} $.
See \cite[4.4.7. (5)]{dejong} for details. 
\end{ex}

\begin{ex}[Cuspidal edge]\label{ex:cuspedgenormal} The germ $ \Phi(s, t)=(s, t^2, t^3) $ in Example~\ref{ex:cuspedge} is also the normalization of its image.
\end{ex}

\begin{rem}\label{re:surfacenormal} 
None of the sets
$ \{ \Phi: ( \C^2, 0) \to ( \C^3, 0) \ | \ \Phi \mbox{ is singular only at $0$} \} $  and 
$ \{  \Phi: ( \C^2, 0) \to ( \C^3, 0) \ | \ \Phi \mbox{ is the normalization of its image} \} $
 includes the other one, as it is shown by the examples \ref{ex:isonormal}, \ref{ex:findetnormal}, \ref{ex:cuspedgenormal}. Indeed the second class contains exactly the finite, generically $1$ to $1$ map germs. However the intersection of the two sets includes very important classes of germs, for instance the finitely $ \mathscr{A}$-determined germs.
\end{rem}

\subsection{Milnor fibration}\label{ss:milnfibration} In this subsection we introduce the notion of the Milnor number and Milnor fibration  \cite{MBook, Egri, looijenga}. 
%We state here the basic concepts for isolated and non-isolated hypersurface singularities. In Chapter~\ref{ch:bound} we study the Milnor fibre of non-isolated hypersurfaces.

Throughout this subsection let $(X, 0) = f^{-1}(0) \subset ( \C^{n+1}, 0) $ be a hypersurface singularity of dimension $n$, where $f \in \mathcal{O}_{( \C^{n+1}, 0) } $ is square free, i.e. it is a reduced equation of $ (X, 0)$. Let 
$J=(\partial_i f)_{i=1, \dots, n+1} \subset \mathcal{O}_{( \C^{n+1}, 0) } $ be the Jacobian ideal of $f$ generated by the partial derivatives of $f$.  The \emph{Milnor number} of $ (X, 0) $ is $ \mu(X, 0):=\dim_{ \C }  \mathcal{O}_{( \C^{n+1}, 0) } /J $ \cite[(1.4)]{looijenga}.

\begin{thm}\label{th:milnornumber}

(a) \cite[Proposition (1.2)]{looijenga} $ \mu (X, 0) $ is finite if and only if $(X, 0) $ has (at most) an isolated singularity.

(b) \cite[E.1.6.]{Egri} If $ (X, 0) $ is isolated, then any stabilization of $f$ has $ \mu (X, 0) $ Morse-points.

%(c) hurkolodas

\end{thm}

Since the defining germ $f$ of an isolated singularity $ (X, 0) $ is finitely determined, and $f$ is stable if and only if it is a Morse function, cf. Example~\ref{ex:findetcnc}, Theorem~\ref{th:milnornumber} is analogue with theorems \ref{th:whtrip} and \ref{th:CT}.

%- Hurk - C

By \cite[Corollary 2.9., Remark 2.11.]{MBook}, \cite[Theorem 1.10.]{Egri} there exists an $ 0 < \epsilon_0 $ such that for all $ 0 < \epsilon \leq \epsilon_0 $ the oriented homeomorphism type of the intersection $ K_{ \epsilon} $ of the sphere $ S^{2n+1}_{ \epsilon} \subset \C^{n+1} $ with radius $ \epsilon $ and $ X $ does not depend on $ \epsilon $, that is, the pair $ (S^{2n+1}_{ \epsilon}, K_{ \epsilon}) $ has the same homeomorphism type for all $ 0 < \epsilon \leq \epsilon_0 $. Let us fix an $ 0 < \epsilon \leq \epsilon_0 $. $ K:= X \cap S^{2n+1}_{ \epsilon} $ is the \emph{link} of $ X$. The embedded link determines the embedded topological type of $ (X, 0) \subset ( \C^{n+1}, 0) $, namely the pair $ ( B^{2n+2}_{ \epsilon}, X \cap B^{2n+2}_{ \epsilon}) $ is homeomorphic to the cone over $ (S^{2n+1}_{ \epsilon}, K) $. (In the non-isolated case this fact is proved in \cite{BurgVer}. Note that the link can also be defined when $ (X, 0) \subset (\C^{n+k}, 0) $ is not hypersurface, as the intersection of $ X $ with the sphere in $ \C^{n+k} $ of small enough radius, or as the $ \epsilon $-level set of any real anlytic germ $ \rho: (X, 0) \to ( [0, \infty ), 0 ) $, $ \rho^{-1}(0)=0$ for small enough $ \epsilon $ \cite[(2.4), (2.5)]{looijenga}.)

The Milnor fibre of $ (X, 0) $ can be defined via two different fibre bundles, which happen to be isomorphic.
%, the \emph{local fibration} of $f$ and the \emph{Milnor fibration}. 
The first is the \emph{local fibration} of $f$, see \cite[Theorem 2.5.]{Egri}, \cite[Theorem 2.8.]{looijenga}.

\begin{thm}\label{th:locfibration} (a) For $ 0 < \delta \ll \epsilon $
\begin{equation}\labelpar{eq:locfibration}
f: f^{-1} (D^2_{ \delta} \setminus \{ 0 \} ) \cap B^{2n+2}_{ \epsilon} \to D^2_{ \delta} \setminus \{ 0 \}
\end{equation}
is a $ \mathcal{C}^{\infty}$ locally trivial fibration, such that its restriction $ f: f^{-1} (D^2_{ \delta} \setminus \{ 0 \} ) \cap S^{2n+1}_{ \epsilon} \to D^2_{ \delta} \setminus \{ 0 \} $ is also a $ \mathcal{C}^{\infty}$ locally trivial fibration.

(b) If, additionally, $(X, 0) $ is isolated, then the fibration 
$ f: f^{-1} (D^2_{ \delta} \setminus \{ 0 \} ) \cap S^{2n+1}_{ \epsilon} \to D^2_{ \delta} \setminus \{ 0 \} $ extends to a 
$ \mathcal{C}^{\infty}$ locally trivial fibration $ f: f^{-1} (D^2_{ \delta} ) \cap S^{2n+1}_{ \epsilon} \to D^2_{ \delta}  $, which is in fact trivial fibration, since $ D^2_{ \delta} $ is contractible.
\end{thm}

Next, we present the second fibration, the so-called 
\emph{Milnor fibration}, see \cite[Theorem 4.8.]{MBook}, \cite[Theorem 3.5.]{Egri}. Let $ U $ be an open tubular neighborhood of $ K \subset S^{2n+1}_{ \epsilon } $, and define $ \phi (z)=f(z)/|f(z)| $ for $ f(z) \neq 0$.

\begin{thm}\label{th:milnfibration} (a)
\begin{equation}\labelpar{eq:milnfibration}
\phi: S^{2n+1}_{ \epsilon } \setminus U \to S^1 
\end{equation}
is a $ \mathcal{C}^{ \infty} $ locally trivial fibration.

(b) The local fibration (\ref{eq:locfibration}) and the Milnor fibration (\ref{eq:milnfibration}) are bundle isomorphic.
\end{thm}

The \emph{Milnor fibre} $F$ is the fibre of either of the bundles (\ref{eq:locfibration}) and (\ref{eq:milnfibration}), that is, $ F \simeq f^{-1} ( \delta \theta ) \cap B^{2n+2}_{ \epsilon} \simeq \phi^{-1} ( \theta ) $ with any $ 0 < \delta \ll \epsilon $, $ |\theta|=1 $. The Milnor fibre is well defined up to orientation preserving diffeomorphism, that is, it is independent of the different choices. 

Here we collect some important properties of $ K $ and $ F $.

\begin{thm}
(a) \cite[Theorem 5.1.]{MBook} $ F$ is an oriented, compact, smooth, parallelizable  $2n$-manifold with boundary (in fact, $F$ is complex manifold), and has the homotopy type of a finite CW-complex of dimension $n$.

(b) \cite[Theorem 5.2.]{MBook} $K$ is $(n-2)$-connected.
\end{thm}

Isolated singularities have more restrictive properties for $ K $ and $ F $.

\begin{thm}\label{th:milnfiso} Let $ (X, 0 )$ be an isolated hypersurface singularity in $ ( \C^{n+1}, 0) $. Then

(a) \cite[Corollary 2.9.]{MBook} $ K $ is a smooth, oriented, closed $ (2n-1)$-submanifold of $ S^{2n+1}$.

(b) (Corollary of Theorem~\ref{th:locfibration}, part (b).) The boundary $ \partial F $ of $ F$ is diffeomorphic to $ K $, actually, they are isotopic submanifolds of $ S^{2n+1}$.

(c) \cite[Theorem 6.5.]{MBook} $ F $ is homotopically equivalent with a bouquet of $n$-spheres. 

(d) \cite[Theorem 7.2.]{MBook} The number of spheres in the bouquet is equal to the Milnor number $ \mu (X, 0) $, i.e. $ \rk (H_n (F, \Z ))= \mu (X, 0) $.

\end{thm}

\section{Resolution and plumbing}

\subsection{Resolution graph} The topological information of a resolution of the normal surface singularity $ (X, 0) $ can be encoded in the resolution graph, which also serves as a plumbing graph for the link. We refer to \cite{Five, NSz, neumann1}.

Let $ (X, 0) $ be normal singularity of dimension $2$, recall that $ (X, 0) $ is isolated. A \emph{resolution} of $ (X, 0) $ is a complex analytic map $ p: ( \tilde{X}, E) \to (X, 0) $, where $ \tilde{X} $ is a smooth complex manifold of dimension $2$ with boundary, $ p^{-1} (0) = E \subset \tilde{X} $ is a complex analytic curve and 
$ p|_{ \tilde{X} \setminus E}: \tilde{X} \setminus E \to X \setminus \{0 \} $ is a biholomorphism \cite[Definition 1.19.(a)]{Five}. Let $ E= \bigcup_{v \in \mathcal{V}} E_v $ be the irreducible decomposition of $ E $, where $ \mathcal{V} $ is a finite set. The components $E_v$ of $E$ are called \emph{exceptional divisors}. The resolution is called \emph{good} if each $ E_v $ is smooth curve, they intersect each other transversally, in particular the intersection of any distinct three of them is empty. A good resolution always exists and it is not unique.

$ \partial \tilde{X} $ is diffeomorphic with the link $ K$ of $ (X, 0) $. Hence the resolution $ \tilde{X} $ and the Milnor fibre $ F$ are two complex fillings of $K$, which fact allows to compare their invariants, see for instance the formulas of Laufer \cite{Laufer77b} or Durfee \cite{Du} and their generalizations, e.g. \cite{Wa}.

The tubular neighbourhood $N(E_v) $ of $ E_v \subset \tilde{X} $ is diffeomorphic with the total space of (the disc bundle of) the normal bundle of $E_v $, which is a $ \mathcal{C}^{ \infty} $ $D^2 $-bundle over the Riemannian surface $ E_v $. Thus the genus $g_v$ of $ E_v $ and the Euler number $ e_v $ of the normal bundle of $ E_v \subset \tilde{X} $ determines $ N(E_v) $ with its $ D^2$-bundle structure. Furthermore $ \tilde{X} $ is diffeomorphic to $ \bigcup_{v \in \mathcal{V}} N(E_v) $. The topology of $ \tilde{X} $, hence the topology of $K$ too, can be encoded in a decorated graph $ G$, called the \emph{resolution graph}, whose vertex set is $ \mathcal{V} $, and the vertexes $ v$ and $ w $ are joined  with $k$ edges if $ |E_v \cap E_w|=k $. Each vertex $v$ is decorated by two weights, the genus $g_v$ and the Euler number $e_v$ \cite[Definition 1.19.(b)]{Five}. We omit the genus if it is $0$.

Define $ E_v \cdot E_w = |E_v \cap E_w|$ for $ v \neq w $, and $ E_v \cdot E_v =e_v $. $H_2 (\tilde{X}, \Z ) \cong H_2( E, \Z) \cong \Z^{| \mathcal{V} |} $ is generated by the classes of $ E_v $, and $ E_v \cdot E_w $ is the intersection number of the corresponding cycles. These numbers form the \emph{intersection matrix}, which contains the same information as the resolution graph, whenever $g_v=0$ for all $v \in \mathcal{V}$. A decorated graph $ G $ occurs as a resolution graph of a normal surface singularity if and only if the corresponding intersection matrix is negative definite \cite[1.23.(a)]{Five}.

$ \tilde{X} $ and $ K $ can be recovered from $ G $ by plumbing, see below. On the other hand, $ K$, or equivalently, the homeomorphism type of $ (X, 0) $ determines the graph $ G $, hence the diffeomorphism type of $ \tilde{X} $ (up to plumbing operations) by \cite[Theorem 2]{neumann1}.

\subsection{Plumbing}\label{ss:plumbingcalc} In Chapter~\ref{ch:bound} we use the plumbing construction in a more general context than it occurs for the resolution of normal singularities. Let $G$ be a connected graph which can have multiple edges and loops, and each vertex $v$ is decorated by the weights ($[g_v],e_v$), $g_v \geq 0$. Furthermore each edge is decorated with a sign $ \oplus $ or $ \ominus $. Then the plumbing construction associates two manifolds to $G$, a compact oriented smooth $ 4$-manifold $ M^4(G) $, and its boundary, the closed oriented smooth $3$-manifold $M^3(G) $, as follows.  For each vertex $v$ fix a smooth, oriented $D^2$-bundle $ p_v: N_v \to E_v $ with Euler number $ e_v$, whose base space $ E_v $ is a closed oriented surface with genus $g_v$. If $v$ is the endpoint of $s_v$ edges, then fix $s_v$ basepoints in $E_v$, fix disc-neighborhoods $ \{ D_{v, i}^2 \}_{i=1, \dots , s_v} $ of these points, as well as orientation preserving local trivializations $ p_v^{-1} (D^2_{v, i}) \simeq D_{v, i}^2 \times D^2 $. Any edge between $v$ and $w$ determines a pair of multidiscs $D_{v, i}^2 \times D^2$ and $D_{w, i}^2 \times D^2$ in $N_v$ and $N_w$ respectively. Identify them by $(x, y) \sim (y, x) $, if the edge has $ \oplus $ sign, and by $ (x, y) \sim (\bar{y}, \bar{x}) $, if the edge has $ \ominus $ sign (where $\bar{x}$ is the complex conjugate of $ x$). After these gluings we obtain a $4$-manifold with corners, but these corners can be smoothed out. Cf. \cite[1.25.]{Five}, \cite{neumann1}, \cite[4.1.]{NSz}.

The resolution graph $G$ corresponding to a good resolution $ \tilde{X} $ of a normal surface singularity $ (X, 0) $ does not have loops and each edge decoration is $ \oplus $. The associated $4$-manifold $M^4(G)$ is diffeomorphic to $ \tilde{X} $, whose boundary $ M^3 (G) $ is diffeomorphic with the link $ K $. 

The $3$-manifolds in the form $ M^3(G) $ associated with a plumbing graph $G$ are called \emph{plumbed $3$-manifolds}.

The plumbing graph of a plumbed $3$-manifold is not unique. The possible operations of a plumbing graph which do not change the oriented diffeomorphism type of $M(G)$ are listed in \cite{neumann1}, \cite[4.2.]{NSz}. We mention here some of them, which will be used in the examples of Chapter~\ref{ch:bound}. We follow the notations of \cite{neumann1}.

The graph operation [R0.(a)] is reversing the signs on all edges other than loops
adjacent to any fixed  vertex. By [R0.(a)] the sign of the edges can be reversed preserving the only relevant information, the parity of the number of $ \ominus $--edges along the cycles of the graph.

[R1] corresponds to the (inverse of the) blowing up of a point $q \in E_v \subset M^4(G) $. This operation is the key tool in the construction of embedded resolution of plane curve singularities, see Subsection~\ref{ss:embres}. $q$ can be a generic point, an intersection point of $E_v$ and $E_w$, or the self intersection point of $E_v$. The three cases of [R1] are shown in the picture, where $ \epsilon = \pm 1 $ and the edge signs $ \epsilon_0 $, $ \epsilon_1 $, $ \epsilon_2 $ are related by $  \epsilon_0=- \epsilon  \epsilon_1  \epsilon_2 $.

\begin{picture}(110,50)(5,0)

% bal oldali graf:
\put(40,20){\circle*{4}} \put(100,20){\circle*{4}} % ket csucs
\put(40,20){\line(1,0){60}}                        % kozottuk el
\put(40,20){\line(-2,-1){30}}                      % elso csucstol balra elek
\put(40,20){\line(-2, 1){30}}

\put(20,25){\makebox(0,0){$\cdot$}}                % koztuk dots
\put(20,20){\makebox(0,0){$\cdot$}}
\put(20,15){\makebox(0,0){$\cdot$}}

\put(42,28){\makebox(0,0){$e_i$}}  % a csucsokon sulyok
\put(42,10){\makebox(0,0){$[g_i]$}}
\put(100,29){\makebox(0,0){$\epsilon$}}

\put(160,20){\vector(1,0){30}}                     % nyil a ket graf kozott

% jobboldali graf:
\put(240,20){\circle*{4}}                    % egy csucs
\put(244,28){\makebox(0,0){$e_i-\epsilon$}}  % a csucson sulyok
\put(244,10){\makebox(0,0){$[g_i]$}}

\put(240,20){\line(-2,-1){30}}                      % elso csucstol balra elek
\put(240,20){\line(-2, 1){30}}

\put(220,25){\makebox(0,0){$\cdot$}}                % koztuk dots
\put(220,20){\makebox(0,0){$\cdot$}}
\put(220,15){\makebox(0,0){$\cdot$}}
\end{picture}

\vspace{2mm}

\begin{picture}(100,40)(5,10)

% bal oldali graf:
\put(40,20){\circle*{4}} \put(80,20){\circle*{4}}
\put(120,20){\circle*{4}}                           % harom csucs

\put(40,20){\line(1,0){80}}                        % kozottuk el

\put(40,20){\line(-2,-1){30}}                      % elso csucstol balra elek
\put(40,20){\line(-2, 1){30}}
\put(20,25){\makebox(0,0){$\cdot$}}                % koztuk dots
\put(20,20){\makebox(0,0){$\cdot$}}
\put(20,15){\makebox(0,0){$\cdot$}}

\put(120,20){\line(2,-1){30}}                      % harmadik csucstol
                                                    % jobbra elek
\put(120,20){\line(2, 1){30}}
\put(135,25){\makebox(0,0){$\cdot$}}                % koztuk dots
\put(135,20){\makebox(0,0){$\cdot$}}
\put(135,15){\makebox(0,0){$\cdot$}}

\put(42,28){\makebox(0,0){$e_i$}}                  % a csucsokon sulyok
\put(42,10){\makebox(0,0){$[g_i]$}}
\put(80,29){\makebox(0,0){$\epsilon$}}
\put(118,28){\makebox(0,0){$e_j$}}
\put(118,10){\makebox(0,0){$[g_j]$}}

\put(60,26){\makebox(0,0){$\epsilon_1$}}      % az eleken sulyok
\put(100,26){\makebox(0,0){$\epsilon_2$}}

\put(160,20){\vector(1,0){30}}                     % nyil a ket graf kozott

% jobboldali graf:
\put(240,20){\circle*{4}}                    % baloldali csucs
\put(242,28){\makebox(0,0){$e_i-\epsilon$}}  % a csucson sulyok
\put(242,10){\makebox(0,0){$[g_i]$}}

\put(240,20){\line(-2,-1){30}}                      % elso csucstol balra elek
\put(240,20){\line(-2, 1){30}}

\put(220,25){\makebox(0,0){$\cdot$}}                % koztuk dots
\put(220,20){\makebox(0,0){$\cdot$}}
\put(220,15){\makebox(0,0){$\cdot$}}

\put(300,20){\circle*{4}} % jobboldali csucs
\put(298,28){\makebox(0,0){$e_j-\epsilon$}}  % a csucson sulyok
\put(298,10){\makebox(0,0){$[g_j]$}}

\put(300,20){\line(2,-1){30}}                      % harmadik csucstol
                                                    % jobbra elek
\put(300,20){\line(2, 1){30}}
\put(315,25){\makebox(0,0){$\cdot$}}                % koztuk dots
\put(315,20){\makebox(0,0){$\cdot$}}
\put(315,15){\makebox(0,0){$\cdot$}}

\put(240,20){\line(1,0){60}} % a ket csucs kozott el
\put(270,26){\makebox(0,0){$\epsilon_0$}}      % az elen sulyok

\end{picture}

\vspace{2mm}

\begin{picture}(150,60)(5,-5)
% eredeti graf:
\put(40,20){\circle*{4}} \put(120,20){\circle*{4}}   % ket csucs

\put(40,20){\line(-2,-1){30}}                      % elso csucstol balra elek
\put(40,20){\line(-2, 1){30}}
\put(20,25){\makebox(0,0){$\cdot$}}                % koztuk dots
\put(20,20){\makebox(0,0){$\cdot$}}
\put(20,15){\makebox(0,0){$\cdot$}}

\put(42,28){\makebox(0,0){$e_i$}}                  % a csucsokon sulyok
\put(42,10){\makebox(0,0){$[g_i]$}}
\put(120,29){\makebox(0,0){$\epsilon$}}

%a ket csucs kozott ket el
\qbezier(40,20)(80,30)(120,20)
\qbezier(40,20)(80,10)(120,20)

% a ket elen felirat
\put(80,35){\makebox(0,0){$\epsilon_1$}}
\put(80,5){\makebox(0,0){$\epsilon_2$}}

%%%%%%%%%%%%%%%%%%%%%%%%%%%%%%%%%%%%%%%%%%%%%%%%%%%%%

\put(160,20){\vector(1,0){30}}                     % nyil a ket graf kozott

%%%%%%%%%%%%%%%%%%%%%%%%%%%%%%%%%%%%%%%%%%%%%%%%%%%%%

% eredmeny graf
\put(240,20){\line(-2,1){30}}       % egy csucs es sulyai
\put(240,20){\line(-2,-1){30}}
\put(240,20){\circle*{4}}
\put(220,24){\makebox(0,0){$\vdots$}}
\put(240,32){\makebox(0,0){$e_i-2\epsilon$}}
\put(240,10){\makebox(0,0){$[g_i]$}}

% maga a hurok:
\put(240,20){\line(2,1){20}} \put(240,20){\line(2,-1){20}}
\qbezier(260,30)(290,40)(293,20) \qbezier(260,10)(290,0)(293,20)

% suly a hurkon:
\put(300,20){\makebox(0,0){$\epsilon_0$}}

\end{picture}

We also use the $0$-chain absorption [R3] and the oriented handle absorption [R5].
The edge signs
$\epsilon_i'$ ($i=1,...,s$) are related by
$\epsilon_i'=-\epsilon\overline{\epsilon}\epsilon_i$ provided that
the edge sign in question is not on a loop, and
$\epsilon_i'=\epsilon_i$, if it is on a loop.

\begin{picture}(100,50)(-5,0)

% bal oldali graf:
\put(40,20){\circle*{4}} \put(80,20){\circle*{4}}
\put(120,20){\circle*{4}}                           % harom csucs

\put(40,20){\line(1,0){80}}                        % kozottuk el

\put(40,20){\line(-2,-1){30}}                      % elso csucstol balra elek
\put(40,20){\line(-2, 1){30}}
\put(20,25){\makebox(0,0){$\cdot$}}                % koztuk dots
\put(20,20){\makebox(0,0){$\cdot$}}
\put(20,15){\makebox(0,0){$\cdot$}}

\put(120,20){\line(2,-1){30}}                      % harmadik csucstol
                                                    % jobbra elek
\put(120,20){\line(2, 1){30}}
\put(135,25){\makebox(0,0){$\cdot$}}                % koztuk dots
\put(135,20){\makebox(0,0){$\cdot$}}
\put(135,15){\makebox(0,0){$\cdot$}}

\put(42,28){\makebox(0,0){$e_i$}}                  % a csucsokon sulyok
\put(42,10){\makebox(0,0){$[g_i]$}}
\put(80,29){\makebox(0,0){$0$}}
\put(118,28){\makebox(0,0){$e_j$}}
\put(118,10){\makebox(0,0){$[g_j]$}}

\put(60,26){\makebox(0,0){$\epsilon$}}      % az eleken sulyok
\put(100,26){\makebox(0,0){$\overline{\epsilon}$}}

% sulyok a jobboldali "kivezeto" eleken
\put(138,35) {\makebox(0,0){$\epsilon_1$}}
\put(138,5){\makebox(0,0){$\epsilon_s$}}

\put(165,20){\vector(1,0){30}}                     % nyil a ket graf kozott

% jobboldali graf:
\put(240,20){\circle*{4}}                    % baloldali csucs
\put(240,30){\makebox(0,0){$e_i+e_j$}}  % a csucson sulyok
\put(240,5){\makebox(0,0){$[g_i+g_j]$}}

\put(240,20){\line(-2,-1){30}}                      % elso csucstol balra elek
\put(240,20){\line(-2, 1){30}}

\put(220,25){\makebox(0,0){$\cdot$}}                % koztuk dots
\put(220,20){\makebox(0,0){$\cdot$}}
\put(220,15){\makebox(0,0){$\cdot$}}

\put(240,20){\line(2,-1){30}}                      % csucstol
                                                    % jobbra is elek
\put(240,20){\line(2, 1){30}}
\put(260,25){\makebox(0,0){$\cdot$}}                % koztuk dots
\put(260,20){\makebox(0,0){$\cdot$}}
\put(260,15){\makebox(0,0){$\cdot$}}

% sulyok a jobboldali "kivezeto" eleken
\put(268,40) {\makebox(0,0){$\epsilon_1'$}}
\put(268,0){\makebox(0,0){$\epsilon_s'$}}

\end{picture}

\vspace{4mm}

% \noindent{\bf [O4] ()}

\vspace{3mm}

\begin{picture}(100,40)(-5,0)
% eredeti graf:
\put(40,20){\circle*{4}} \put(120,20){\circle*{4}}   % ket csucs

\put(40,20){\line(-2,-1){30}}                      % elso csucstol balra elek
\put(40,20){\line(-2, 1){30}}
\put(20,25){\makebox(0,0){$\cdot$}}                % koztuk dots
\put(20,20){\makebox(0,0){$\cdot$}}
\put(20,15){\makebox(0,0){$\cdot$}}

\put(42,28){\makebox(0,0){$e_i$}}                  % a csucsokon sulyok
\put(42,10){\makebox(0,0){$[g_i]$}}
\put(120,29){\makebox(0,0){$0$}}

%a ket csucs kozott ket el
\qbezier(40,20)(80,30)(120,20)
\qbezier(40,20)(80,10)(120,20)

% a ket elen felirat
\put(80,35){\makebox(0,0){$\ominus $}}
\put(80,5){\makebox(0,0){$ \oplus $}}

%%%%%%%%%%%%%%%%%%%%%%%%%%%%%%%%%%%%%%%%%%%%%%%%%%%%%

\put(165,20){\vector(1,0){30}}                     % nyil a ket graf kozott

%%%%%%%%%%%%%%%%%%%%%%%%%%%%%%%%%%%%%%%%%%%%%%%%%%%%%

% jobboldali graf:
\put(240,20){\circle*{4}}                    % egy csucs
\put(242,28){\makebox(0,0){$e_i$}}  % a csucson sulyok
\put(250,10){\makebox(0,0){$[g_i+1]$}}

\put(240,20){\line(-2,-1){30}}                      % elso csucstol balra elek
\put(240,20){\line(-2, 1){30}}

\put(220,25){\makebox(0,0){$\cdot$}}                % koztuk dots
\put(220,20){\makebox(0,0){$\cdot$}}
\put(220,15){\makebox(0,0){$\cdot$}}
\end{picture}

\vspace{2mm}

We also use plumbing graphs with arrowhead vertexes. An arrowhead vertex based on the vertex $v$ denotes a generic (oriented) fibre $D^2$ of $ N_v \subset M^4 (G) $. Its boundary is an oriented knot $ S^1 \subset M^3 $. Additionally, the $S^1$-bundle structure of $ \partial N_v $ determines a trivialization $ S^1 \times D^2 $ of the tubular neighbourhood of the knot.

In \cite{neumann1} there is a third decoration, the non-negative integer $h_v$, which denotes the number of boundary components of the genus--$g_v$ surface. This concept is closely related to the arrowhead vertexes. In fact, the resulted $3$-manifold has $ h_v $ boundary components for each vertex $v$, and it can be obtained from the closed $3$-manifold by removing the open tubular neighbourhoods of $h_v$ knots correspond to $h_v $ arrowhead vertexes based on $v$.

%By the (inverse of the) operation [R1] a new vertex $u_{new}$ appears with $ e_{u_{new}}= \pm 1 $. If $p$ is a generic point of $E_v$, then a new edge $ (v, u_{new} ) $ (and lso $ (v, u_{new} ) $ if $ p \in E_v \cap E_w $, and double edge, if $p$ is the self intersection point of $E_v$), and $ e_v $ changes by $ \pm 1 $ (also $ e_w $ changes by $ \pm 1 $ if $ p \in E_v \cap E_w $, and $ e_v $ changes by $ \pm 2 $ if $p$ is the self intersection point of $E_v$).

\subsection{Embedded resolution of plane curves}\label{ss:embres} Let $ (D, 0) \subset ( \C^2, 0) $ be a plane curve singularity defined as the zero set of the reduced germ $ d: ( \C^2, 0) \to ( \C, 0) $.  Let $ d= \Pi_{i=1}^l d_i $ be the irreducible decomposition of $ d\in \mathcal{O}_{( \C^2, 0)}$, and let $ (D, 0) = \bigcup_{i=1}^l (D_i, 0) $ be the irreducible decomposition of $ (D, 0)$, where $ (D_i, 0) = d_i^{-1} (0) $.

A \emph{good embedded resolution} of $ (D, 0) $ is a good resolution $\pi: (\widetilde{\C^2}, E) \to ( \C^2, 0) $ of $( \C^2, 0)$ with the additional property, that $ (d \circ \pi)^{-1} (0) $ is a normal crossing divisor in $ \widetilde{\C^2} $, that is, its irreducible components are smooth and intersects each other transversally \cite[1.28.]{Five}. Note that each component $ E_v $ of the exceptional divisor $E$ of any good resolution of $ ( \C^2, 0) $ is diffeomorphic to $ \C {\mathbb P}^1 \simeq S^2 $.

%(a) Each component $ E_v $ of $E$ is diffeomorphic to $ \C {\mathbb P}^1 \simeq S^2 $.

%(b) $ (d \circ \pi)^{-1} (0) $ is a normal crossing divisor in $ \tilde{D} $, that is, its irreducible components are smooth and intersects each other transversally. 

The irreducible components of $ (d \circ \pi)^{-1} (0) $ are the exceptional divisors $ E_v $ and the strict transforms corresponding to the components of $ (D, 0) $. That is, the components  $ \tilde{D}_i  $ of the \emph{strict transform}
$\tilde{D} =\overline{(d \circ \pi)^{-1} (0) \setminus E} \subset \widetilde{ \C^2}$
 of $ D $.
 %$\tilde{D}$, and its components by  $ \tilde{D}_i  $,
 %$1\leq i\leq l$. 
 Each $ \tilde{D}_i  $ intersects (transversally)
 only one exceptional divisor, say $ E_{v(i)} $.
 
 The \emph{total transform} of $ D_i $ is the divisor
 $ {\rm div}(d_i \circ \pi)= \Sigma_{v \in \mathcal{V}} m_i(v)
 \cdot E_v + \tilde{D}_i $,
where the multiplicity $ m_i(v) \in \Z_{>0} $ is the
vanishing order of  $d_i \circ \pi $ along $ E_v $.

Let $ \Gamma $ be the embedded resolution graph of $ (D, 0) \subset ( \C^2, 0) $ associated with the resolution ${\pi}$, i.e. the resolution graph of $ \pi $, where the arrowhead vertexes $\{a_i\}_{i=1}^l$ codify the strict transforms
$ \{\tilde{D}_i\}_{i=1}^l  $. Note that $ M^3( \Gamma ) \simeq S^3 $ since it is the link of $ ( \C^2, 0 ) $, and the link $ \bigsqcup_{i=1}^l S^1 \subset S^3$ determined by the arrowhead vertexes is the link of $ (D, 0)$.
% will be codified (as usual) by
%arrowhead vertexes $\{a_i\}_{i=1}^l$.

The multiplicities $m_i(v)$ ($v \in \mathcal{V} $, $1\leq i\leq l$) are determined by $ \Gamma $ via the identities (see e.g. \cite{Lauferbook,Five})
\begin{equation}\labelpar{eq:mult}
 \Sigma_{v \in \mathcal{V}} m_i(v) (E_v \cdot E_w)  + (\tilde{D}_i\cdot E_w) = 0 \ \ \mbox{for all $w\in \mathcal{V}$.}
 \end{equation}

A good embedded resolution of a plane curve can be obtained by a sequence of blowing-ups, whose local model is the blowing-up of the origin in $ \C^2 $. That is the complex manifold $ \mathcal{B} = \{ (p, l) \in \C^2 \times \C \mathbb{P}^1 \ | \ p \in l \} $ with its natural projection $ \pi: \mathcal{B} \to \C^2 $. The exceptional divisor $ \pi^{-1}(0) $ is $ \C \mathbb{P}^1 $, and $ \mbox{pr}_2: \mathcal{B} \to \C \mathbb{P}^1 $ defines a complex line bundle over it with Euler number $( -1)$.

$ \mathcal{B} $ admits a natural atlas with two charts whose domains are $ U_x = \pi^{-1} (\{x \neq 0 \}) $ and $ U_y = \pi^{-1} (\{y \neq 0 \}) $, where $ x $ and $y$ denote the coordinates in $ \C^2 $. The coordinates $ (s, t) $ of $ U_x $ are chosen such that $ \pi (s, t) = (s, st) $ holds, similarly, $ \pi(u, v) = (uv, v) $ holds in the coordinates $ (u, v) $ of $ U_y $. The transition map on $ U_x \cap U_y 
% = \pi^{-1} ( \{ x \neq 0 \mbox{ and } y \neq 0 \} )
$ is given by $ t=1/u $ and $ s= uv $.

\subsection{Examples} 

\begin{ex}[The quotient singularities $A$-$D$-$E$] The link of the normal hypersurface singularity described in Example~\ref{ex:ade} is the quotient of $ S^3 \subset \C^2 $ with the group action of $ G $. For these singularities the Milnor fibre $F$ is diffeomorphic to the resolution $ \tilde{X} $, however they are not biholomorphic, since $ \tilde{X} $ contains closed analytic curves (namely, the exceptional divisors $ E_i $), while $ F \subset \C^3 $ does not contain (in fact, $ F $ is always a Stein manifold). Their resolution graphs are the $A$-$D$-$E$ graphs with $ g_v =0 $ and $ e_v= -2 $ for each vertex, see \cite[1.20.(b)-(f)]{Five}.
\end{ex}

\begin{ex}[The plane curve $ A_1$]\label{ex:A1planecurve}
Let $ (X, 0) \subset ( \C^2, 0) $ be the zero set of $ f= xy $. Its link is $ S^1 \sqcup S^1 \subset S^3 $, the two components correspond to the irreducible components of $ f $. Each $ S^1 $ is unknotted in $ S^3 $ and their linking number is $ 1 $, they form a Hopf-link. Since $\mu(X, 0)=1 $, the Milnor fibre is diffeomorphic to $ S^1 \times I $. In fact, the Milnor fibre $ F = f^{-1} ( \delta ) \cap B^4_{ \epsilon } $ ($ 0 <  \delta  \ll \epsilon $) can be parametrized with $ S^1 \times [r_0, r_1] $ as $ x = \sqrt{ \delta} r e^{ \alpha i } $, $ y = \sqrt{ \delta} r^{-1} e^{- \alpha i } $, $ - \pi \leq \alpha \leq  \pi $, $ r_0 \leq r \leq r_1 $, where $ r_0 r_1 =1 $ and $ \delta (r_0^2 + r_1^2) = \epsilon^2 $. 
We use this parametrization in Chapter~\ref{ch:bound}, where $A_1 $ appears as the transverse type of the non-isolated surfaces we study.

The embedded resolution of $ (X, 0) $ can be obtained by blowing up $ ( \C^2, 0) $ once. It is easier to show this blowing up with the germ $ x^2 + y^2= (x+iy)(x-iy) $, which is equivalent with $f$, since we need only one chart for it. The total transform on the $ U_x $-chart is $ (s+ist)(s-ist)=s^2 (1+it)(1-it) $, where $ \{ s=0 \}$ is the exceptional divisor, $ \{ t= \pm i\} $ are the strict transforms of the components of $ (X, 0) $. The resolution graph is

\begin{picture}(300,60)(-100,40)
 \put(0,60){\circle*{4}}
                         \put(0,60){\vector(2,1){30}}
                         \put(0,60){\vector(2,-1){30}}
                         \put(-5,70){\makebox(0,0){$-1$}}
                          \put(-5,50){\makebox(0,0){$v$}}
                          \end{picture}
                          
\noindent with the multiplicities $ m_1 (v) = m_2 (v) =1 $.
\end{ex}

\begin{ex}\label{ex:milnex} The curve $ (D, 0) = \{ d(x, y) = xy^2 + x^k =0 \} \subset ( \C^2, 0) $ (where $k \geq 1$) appears as the double point locus of the family $ C_k $ of finitely determined germs, see Subsection~\ref{s:milnex}. We present here the calculation of its embedded resolution graph, it serves as a model also for the other examples.

The irreducible decomposition of $ d $ depends on the parity of $ k $: it is $ d(x, y) = x (y^2 + x^{k-1}) $ for $ k= 2n $ and $ d(x, y) = x (iy + x^n)(-iy + x^n) $ for $ k= 2n+1 $. After a blowing up the total transform of $ d $ is 
$ (f \circ \pi)(s, t)= s((st)^2 + s^{k-1} )= s^3 (t^2 + s^{k-3} ) $ on the $ U_x $-chart, and 
$ (f \circ \pi)(u, v)= uv(v^2 + (uv)^{k-1} )= v^3 u (1 + u^{k-1}v^{k-3} ) $ on $ U_y $. The exceptional divisor $ E_1 $ is defined by $ \{ s=0 \} $ and $ \{ v=0 \} $ on the charts, the Euler number of its normal bundle is $ (-1) $. The strict transform of $ D_1=\{ x=0 \} $ is $ \tilde{D}_1=\{ u=0 \} $ on $ U_y $, and it is not visible on $ U_x $, although one can write it as $ \tilde{D}_1 =\{ t= \infty \} $. $ \tilde{D}_1 $ is smooth and intersects the exceptional divisor $E_1 $ transversally.

The strict transform of the other component(s) $ \{ y^2 + x^{k-1} =0 \} $ is $ \{ t^2 + s^{k-3} =0 \} $ on $U_x$ (and is not visible on $U_y$). It is not smooth, if $ k > 4 $, and for $ k=4 $ it is smooth, but its intersection with $E_1$ is not transverse. We continue with the blow up of the point $ s=t=0 $. Since it is on $ E_1$, the Euler number of $E_1$ changes by $(-1)$, hence it will be $(-2)$ after the second blow up. 

The strict transform of $ \{ y^2 + x^{k-1} =0 \} $ after $ m$ blowing-ups is $ \{ t_m^2 + s_m^{k-1-2m} =0 \} $, where $ (s_m, t_m ) $ denotes the coordinates of the first chart of the $m$-th blow up. Moreover the strict transform intersects only the $m$-th exceptional divisor $ E_m $, if $ m < n$. The sequence of $m$ blow-ups results a line of $(-2)$-vertexes (a `bamboo') with a $(-1)$-vertex at the end ($ m < n$).

\vspace{2mm}

 \emph{Case 1}: $k=2n+1$.

 \vspace{2mm}
 
For $k=2n+1$ the process reaches to the strict transform $  \{ t_n^2 +1 = 0 \} = \{ t_n= \pm i \} $ after $n$ steps, which is smooth, the components are disjoint and intersect $ E_n $ transversally. Hence the resolution graph is

\begin{picture}(300,30)(-200,50)
% \put(-180,60){\makebox(0,0){$D:$}}
                         \put(100,60){\circle*{4}}
                          \put(60,60){\circle*{4}}
                           \put(0,60){\circle*{4}}
                           \put(60,60){\line(1,0){40}}
                             \put(40,60){\line(1,0){20}}
                             \put(0,60){\line(1,0){20}}
                              \put(-40,60){\circle*{4}}
                              \put(-40,60){\line(1,0){40}}
                         \put(100,60){\vector(2,1){30}}
                       \put(-40,60){\vector(-1,0){40}}
                         \put(100,60){\vector(2,-1){30}}
                         \put(95,70){\makebox(0,0){$-1$}}
                          \put(57,70){\makebox(0,0){$-2$}}
                           \put(-3,70){\makebox(0,0){$-2$}}
                            \put(-43,70){\makebox(0,0){$-2$}}
                       %   \put(93,50){\makebox(0,0){$w_n$}}
                       %   \put(-43,50){\makebox(0,0){$w_1$}}
                      %    \put(-3,50){\makebox(0,0){$w_2$}}
                          \put(29,57){\makebox(0,0){$\dots$}}
                     %     \put(60,50){\makebox(0,0){$w_{n-1}$}}
                    %       \put(-78,50){\makebox(0,0){$a_1$}}
                   % \put(140,90){\makebox(0,0){$a_2$}}
                  %   \put(140,30){\makebox(0,0){$a_3$}}
                         \end{picture}

\end{ex}

\noindent where the number of $ (-2)$-vertices in the middle is $n-1$. By (\ref{eq:mult}) the multiplicities are $m_1(E_i) =  1 $ for all $i= 1 \dots n $ and $ m_2 (E_i)= m_3 (E_i)= i $.

 \vspace{2mm}

 \emph{Case 2}: $k=2n$.

 \vspace{2mm}
 
 For $k=2n$ the process reaches to the strict transform $  \{ t_{n-1}^2 +s_{n-1} = 0 \} $ after $n-1 $ steps, which is smooth, but its intersection with $E_{n-1} $ is not transverse. The $n$-th blowing up results 
 the strict transform $ \{ u_n + v_n =0\} $ on the second chart, which is smooth and intersects $ E_n = \{ v_n=0 \} $ transversally, but it intersects also the strict transform $ \{ u_n=0 \} $ of $ E_{n-1 } = \{ s_{n-1} = 0 \}  $. One more blow-up is needed at $ u_n=v_n=0 $, which is the intersection point of $ E_n $ and the strict transform of $ E_{n-1} $. The new exceptional divisor $ E_{n+1} $ intersects (the strict transforms of) $ E_n $ and $ E_{n-1} $, and their Euler number changes by $(-1)$, hence $ e_{n-1}=-3 $, $e_n= -2 $, $e_{n+1}=-1$. The resolution graph is 

\begin{picture}(300,70)(-160,20)
% \put(-140,60){\makebox(0,0){$D:$}}
 \put(140,60){\circle*{4}}
                         \put(100,60){\circle*{4}}
                          \put(60,60){\circle*{4}}
                           \put(0,60){\circle*{4}}
                            \put(140,30){\circle*{4}}
                           \put(60,60){\line(1,0){40}}
                             \put(40,60){\line(1,0){20}}
                             \put(0,60){\line(1,0){20}}
                              \put(-40,60){\circle*{4}}
                              \put(-40,60){\line(1,0){40}}
                              \put(100,60){\line(1,0){40}}
                              \put(140,60){\line(0,-1){30}}
                         \put(140,60){\vector(1,0){40}}
                         \put(137,70){\makebox(0,0){$-1$}}
                         \put(97,70){\makebox(0,0){$-3$}}
                          \put(57,70){\makebox(0,0){$-2$}}
                           \put(-3,70){\makebox(0,0){$-2$}}
                            \put(-43,70){\makebox(0,0){$-2$}}
               %           \put(100,50){\makebox(0,0){$w_{n-1}$}}
              %            \put(-43,50){\makebox(0,0){$w_1$}}
             %             \put(-3,50){\makebox(0,0){$w_2$}}
                          \put(29,57){\makebox(0,0){$\dots$}}
            %              \put(60,50){\makebox(0,0){$w_{n-2}$}}
           %               \put(132,50){\makebox(0,0){$v_2$}}
          %                  \put(132,20){\makebox(0,0){$v_1$}}
                  \put(152,30){\makebox(0,0){$-2$}}
                  \put(-40,60){\vector(-1,0){30}}
         %         \put(-80,70){\makebox(0,0){$a_1$}}
        %          \put(180,70){\makebox(0,0){$a_2$}}
                         \end{picture}

\noindent where the number of $(-2)$-vertices in the middle is $n-2$. By (\ref{eq:mult}) the multiplicities are $ m_1(E_i)= 1 $ for $ i \neq n+1 $ and $ m_1 (E_{n+1})=2 $, $ m_2 (E_i )= 2 i $ for $ i < n $ and $ m_2 (E_n )= 2n-1$, $ m_2 (E_{n+1}) = 4n-2 $.

\section{Non-isolated case} For a non-isolated singularity $ (X, 0)$ the link $ K$, the boundary $ \partial F $ of the Milnor fibre and the boundary $ \partial \tilde{X} $ of a resolution are three different spaces. Indeed, $ \partial F $ and $ \partial \tilde{X} $ are smooth manifolds, while $ K$ is not smooth. Moreover, any resolution $ \tilde{X} $ of $ (X, 0) $ is the resolution of the normalization of $ X$ as well, thus it contains only limited information about $ (X, 0) $.

Let $ (X, 0) = f^{-1}(0) \subset ( \C^3, 0) $ be a surface singularity, whose singular set $ ( \Sigma, 0) = (df)^{-1}( 0) \cap (X, 0) $ has dimension $1 $.
The N\'{e}methi--Szil\'{a}rd book \cite{NSz} provides a general algorithm to determine $ \partial F $ as a plumbed $3$-manifold. Here we just sketch the decomposition of $ \partial F $ as the union of two pieces. For details see \cite[2.3., 3.4.]{NSz}, \cite{Si} and Section~\ref{s:boundmiln} in our special case. 

Let $ F = f^{-1} ( \delta ) \cap B_{ \epsilon }^{6} $ ($ 0 < | \delta | \ll \epsilon $) be the Milnor fibre of $ (X, 0) $, its boundary is $ \partial F = f^{-1} ( \delta ) \cap S_{ \epsilon }^{5} $. Let $ N( \Sigma ) $ be a closed tubular neighbourhood of $ \Sigma \cap S_{ \epsilon }^{5} $ in $ S_{ \epsilon }^{5} $ and let $ N^o( \Sigma ) $ be its interior. Then the two parts mentioned above are $ \partial F \setminus N^o( \Sigma ) $ and $ \partial F \cap N( \Sigma ) $. $ \partial F $ can be obtained from these pieces by gluing them along their common boundary, which is a union of tori.

$ \partial F \setminus N^o( \Sigma ) $ is diffeomorphic with $ K \setminus N^o( \Sigma ) $, where $ K= X \cap S_{ \epsilon }^{5} $ is the link of $ X $. The normalization map $n: (X_{norm}, 0) \to (X, 0) $ induces a diffeomorphism between $ n^{-1} ( K \setminus N^o( \Sigma ) ) \subset K_{norm}$ and $  K \setminus N^o( \Sigma ) $, where $ K_{norm} $ is the link of $ (X_{norm}, 0)$. Hence this piece can be inherited from the resolution graph of the normalization of $ (X, 0) $.

$ \partial F \cap N( \Sigma ) $ is the disjoint union of the pieces $ \{ \partial F \cap N( \Sigma_j ) \}_j $, where $ \Sigma = \bigcup_j \Sigma_j $ is the irreducible decomposition of $ \Sigma $.  The projection of $ \partial F \cap N( \Sigma_j ) $ to $ \Sigma_j \cap S_{ \epsilon }^{5} \simeq S^1 $ induces a bundle structure on $ \partial F \cap N( \Sigma_j ) $ over $ S^1 $. Its fibre is the Milnor fibre of the \emph{transverse curve singularity} of $ \Sigma_j $. This is the plane curve singularity defined as the zero set of $ f|_{(S, q)} : (S, q) \to ( \C, 0 ) $, where $ (S, q) \subset ( \C^3, q) $ biholomorphic to $ ( \C^2, 0) $ is a $ \Sigma_j $--transverse slice at an arbitrary point $ q \in \Sigma_j \setminus \{ 0 \} $.

\chapter{Boundary of the Milnor fibre}\label{ch:bound}

\section{Preliminaries}

\subsection{} This chapter contains the results of the article \cite{NP2}. We assume that $(X, 0)=(f^{-1}(0),0)$ is the image
of a finitely determined complex analytic germ
$ \Phi: (\C^2, 0) \to ( \C^3, 0) $.
This means that $  \Phi  $ is a
 stable immersion off the origin, cf. Theorem~\ref{th:fin-stab} or \cite{Wall,Mond2}.
 From an other approach $(X, 0)$ is a non-isolated hypersurface singularity in $ (\C^3, 0) $ such that the transverse curve of the singular set has type $ A_1 $ and the normalization of $(X, 0) $ is smooth. Then its normalization map $ \Phi: ( \C^2, 0) \to (X, 0) \subset ( \C^3, 0) $ is a finitely determined germ.
 
The main result of this chapter provides the plumbing graph of the boundary $\partial F$ of the Milnor fibre $F$
as a surgery, starting
from the embedded resolution graph  of the double points $(D,0)\subset
(\C^2,0)$. The needed additional pieces, which will be glued
to this primary object
correspond to certain fibre bundles over $S^1$
with fibres the local Milnor fibre of the
 transverse singularity type (and monodromy the corresponding
vertical monodromy). The surgery itself is
 characterized by some homologically determined integers
 combined with  the newly defined {\it `vertical index'}
associated with the irreducible components of the singular locus of $f^{-1}(0)$.

%We believe that the present method can serve as a prototype for further more
%general families as well.

This chapter is organized as follows. In this section we introduce the notations, in Section~\ref{s:Y} we describe by several characterizations the
surgery pieces associated with the components of the singular locus. In Section~\ref{s:boundmiln} and in Section~\ref{s:vertical}
we describe  the gluing and its invariants, while the last section contains several concrete
examples.

The basic notions and properties of the Milnor fibre, resolution and plumbing construction are introduced in Chapter~\ref{ch:iso}. For the properties of a finitely determined germ $ \Phi $ we refer to Chapter~\ref{ch:germ}.
% (the families are taken from the Mond's list  of simple germs \cite{Mond2}).

\subsection{Double point curves} Let $ \Phi: (\C^2, 0) \to ( \C^3, 0) $ be a
 complex analytic germ singular only at the origin.
 We assume that $ \Phi $ is finitely $\mathscr{A}$-determined.
 This happens exactly when  $  \Phi  $ is a {\it
 stable immersion} off  the origin, that is, off the origin it has only single and double values and at
 each double value the intersection of the two smooth branches is
 transverse, cf. Example~\ref{ex:c23}.
 % \cite[Theorem 2.1]{Wall} or \cite[Corollary 1.5]{Mond2}.

Write $(X,0):=({\rm im}(\Phi),0)$ and
let $ f: ( \C^3, 0) \to ( \C, 0) $ be the reduced equation of $(X,0)$. $f$ can be determined by Fitting ideal method, as a generator of $ \mathcal{F}_0 (\Phi_* \mathcal{O}_{ ( \C^n, 0) }) $, cf. Section~\ref{s:Fitting}.
 Note that $ (X, 0) $ is a  non-isolated hypersurface singularity, except when $ \Phi $ is a regular map, see Theorem~\ref{TH:EMBINTRO}.
 We denote by $ (\Sigma,0) = ( \partial_{x_1} f, \partial_{x_2} f, \partial_{x_3} f)^{-1} (0) \subset (\C^3,0) $
 the {\it reduced} singular locus of  $(X,0)$
 (which equals the closure of the set of double values of $\Phi$),
 and by $(D,0)$ the {\it reduced}
 double point curve $ \Phi^{-1} (\Sigma ) \subset (\C^2,0) $.
 (In fact, the finite determinacy of the germ $ \Phi$
 is equivalent with the fact that the double point curve is reduced;  see Theorem~\ref{th:findoub}, \cite{nunodouble}.)

Let $B^6_\epsilon$ be the $\epsilon$--ball in $\C^3$ centred at the origin,
and $ S^5_\epsilon$ its boundary ($\epsilon\in {\mathbb R}_{>0}$). Then for $\epsilon$ sufficiently small $(B_\epsilon^6,0)$ is a Milnor ball for the pair
$(\Sigma,0)\subset (X,0)$, and, furthermore,
$\mathfrak{B}_\epsilon:=\Phi^{-1}(B_\epsilon^6)$
is a (non--metric) $ \mathcal{C}^{\infty}$ ball in $(\C^2,0)$, which might serve as a Milnor
ball for $(D,0)$, cf. \ref{ss:milnfibration}, \cite{looijenga}. We set $ \mathfrak{S}^3 = \Phi^{-1} (S^5_{ \epsilon})=
\partial \mathfrak{B}_\epsilon$, diffeomorphic to $S^3$,
and we treat it as the usual Milnor--ball boundary 3--sphere.
Recall that the  immersion associated with
$ \Phi $ at the level of local neighbourhood boundaries
is $ \Phi |_{ \mathfrak{S}^3} : \mathfrak{S}^3 \to S^5 $, cf. Definition~\ref{de:linkmap} and Theorem~\ref{th:Csum}.

\subsection{Components and links of $\Sigma$ and $D$.}\label{ss:compllinks}
Let $ \Upsilon \subset S^5_\epsilon$ be the link of $\Sigma$.
It is exactly the  set of double values  of $  \Phi |_{ \mathfrak{S}^3} $.
Let $ L = \Phi^{-1}(\Upsilon) \subset \mathfrak{S}^3 $ denote the set of double points of
$  \Phi |_{ \mathfrak{S}^3} $, that is,
$ L\subset  \mathfrak {S}^3  $ is the link of $D$.
Assume that the reduced equation of $D$ is
$ d: (\C^2, 0) \to (\C, 0) $, whose irreducible decomposition is
$ d= \Pi_{i=1}^l d_i $. The irreducible components of $ D$ are denoted by
$ (D_i, 0) = d_i^{-1} (0) $
and their link components in $L$ by $L_i$,
 $1\leq  i\leq  l$.
$ D $ is equipped with an involution $ \iota: D \to D $ which pairs the double points.
 $ \iota|_{ L} $ induces a permutation (pairing) $ \sigma$ of $ \{1, 2, \dots , l \} $, such that
 $ \iota (L_i) = L_{ \sigma(i)} $. Moreover,
 $ \Phi|_{ L }: L \to \Upsilon $ is a double covering with
 $ \Phi (L_i)=\Phi (L_{ \sigma(i) }) $.
 If $ i = \sigma(i) $ for some $i$, then $ \Phi |_{L_i} $ is a nontrivial double covering
  of its image, while  above  the other components the covering  is  trivial.
 Let $ J $ be the set of pairs   $\{i, \sigma (i)\}$ ($1\leq i \leq l$), it
  is the index set of the   components of $ \Upsilon $; they will
  be denoted by $ \{\Upsilon_j\}_{ j \in J}$.

All link components are considered with  their natural orientations.

For the good embedded resolution of $ (D, 0) $ and the corresponding notations see Subsection~\ref{ss:embres}.

Note that $ 2 |J| - l $ is the number of non-trivially covered double point curve components, its parity is the total twist of $ \Phi|_{ \mathfrak{S}^3} $, cf. the end of Subsection~\ref{ss:m3r45}. Since $ 2 |J| - l   \ (\mbox{mod } 2) $ is an additive regular homotopy invariant and its value for the complex Whitney umbrella is $1$, by Theorem~\ref{TH:MAIN} it follows that
%By Theorem~\ref{TH:MAIN} and \cite[Theorem 4.1.]{hughes} it follows that
\begin{cor}
$ C( \Phi ) \equiv 2 |J| - l   \ (\mbox{mod } 2)$.
\end{cor}

\section{The manifold $ Y $}\labelpar{s:Y}
In the construction of the boundary of the Milnor fibre we will need a special
3--manifold with torus boundary. Its several realizations and properties
will be discussed in this section.

In the sequel $S^1$ (as the boundary of the unit disc of $\C$) and the real interval
$I:=[-1,1]$ are considered with their natural orientations;
 $\bar{\cdot }$ denotes  the complex conjugation of $\C$.

\subsection{The definition  of $Y$}\labelpar{ss:cY} Consider the $\Z_2$--action on $ S^1 \times S^1 \times I$ defined by the involution
%\begin{equation}\labelpar{eq:act}
 $(x, y, z) \mapsto (-x, \bar{y}, -z)$, % \mbox{,}
%\end{equation}
and define $Y$ as the quotient
\begin{equation}\labelpar{eq:Y}
 Y= \frac{S^1 \times S^1 \times I}{(x,y,z) \sim (-x, \bar{y}, -z)}.
\end{equation}

$Y$ is a $3$--manifold with a boundary diffeomorphic  with $S^1\times S^1$.
The  projections to different  components provide
 different `realizations' of $Y$.

\vspace{2mm}

(1) The projection  to the first coordinate $x$ gives a fibration
\begin{equation}\labelpar{eq:cyl}
 \begin{array}{ccc}
  S^1 \times I & \rightarrow & Y \\
             &              & \downarrow \\
             &              & S^1
 \end{array}
\end{equation}
where the base space $ S^1=S^1/\{x \sim -x\}  $
is parametrized by $ x^2 $,  and the monodromy
diffeomorphism $ S^1 \times I \to S^1 \times I $
over the base space is  $ (y, z)  \mapsto ( \bar{y} , -z) $.

\vspace{2mm}

(2) The projection  to the first two coordinates $(x,y)$ realizes $Y$ as the total space of a fibration
\begin{equation}\labelpar{eq:klein}
 \begin{array}{ccc}
   I & \rightarrow & Y \\
             &              & \downarrow \\
             &              & \mathcal{K}
 \end{array}
\end{equation}
with  fibre $I$ and  base space
%\begin{equation}
$ \mathcal{K} := (S^1 \times S^1)/\{(x,y) \sim (-x, \bar{y})\}$,
%\end{equation}
the Klein bottle.
The  factorization $S^1\times S^1\to \mathcal{K}$
 is the orientation double cover of $\mathcal{K} $. In particular,
the fibration (\ref{eq:klein})
is the segment bundle of the orientation line bundle of $ \mathcal{K}$,
hence  the orientation double cover of $ \mathcal{K} $ is realized also
 by the restriction of the bundle map to the boundary
$ \partial Y \simeq S^1 \times S^1 $.

\vspace{2mm}

(3) The projection to the $(x, z)$ coordinates realizes $Y$ as the total space of a fibration
\begin{equation}\labelpar{eq:mob}
 \begin{array}{ccc}
  S^1  & \rightarrow & Y \\
             &              & \downarrow \\
             &              & \mathcal{M}
 \end{array}
\end{equation}
over the base space
%\begin{equation}
$ \mathcal{M} := (S^1 \times I)/\{(x,z) \sim (-x, -z)\}$,
%\end{equation}
the M\"obius band.
In this way, $Y$ appears as the
tangent circle bundle of $\mathcal{M}$, i.e. as
the sub--bundle of the tangent bundle $ T \mathcal{M} $ consisting of unit tangent vectors.
This follows from the fact that both circle bundles have  the same monodromy map
 along the midline of $ \mathcal{M} $, namely $S^1 \to S^1$, $y \mapsto \bar{y}$.

\vspace{2mm}

(4) The projection  to the $(y,z)$ coordinates
realizes $Y$ as the total space of a projection
\begin{equation}\labelpar{eq:seif}
 \begin{array}{ccc}
   S^1 & \rightarrow & Y \\
             &              & \downarrow \\
             &              & D^2
 \end{array}
 \end{equation}
 to the  base space
%\begin{equation}
 $(S^1 \times I)/\{(y,z) \sim (\bar{y}, -z)\}$,
%\end{equation}
 the 2--disc $ D^2 $. Although the involution $ (y,z) \mapsto (\bar{y}, -z) $ has two fix points $(-1,0) $ and $(1,0)$,  the quotient $D^2$ can be smoothed.
 However, the projection $Y\to D^2$  is not a locally trivial fibration:
 it is a Seifert fibration with two exceptional fibres sitting
  above $(-1,0) $ and $(1,0)$.   We will refer to this $S^1$--fibration
  as the \emph{canonical Seifert fibration} of $ Y$.

 The Seifert invariants of the exceptional fibres can be calculated as in
  p. 307 of \cite{neumann1}.
The two exceptional fibres of the canonical Seifert fibration \eqref{eq:seif}
can be seen in the projection (3) as well: they correspond to the
tangent vectors of the midline. On the other hand, a generic orbit consists of those
unit tangent vectors of $\mathcal{M}$, which form non--zero
angle $\pm \alpha$ (with $\alpha$ fixed) with the  midline.
Thus,  a generic fibre in a neighbourhood of an exceptional fibre
goes around twice and  both  Seifert invariants are $(2,1)$.
Hence, cf. \cite{neumann1}, a plumbing graph of $Y$ is:

\begin{picture}(300,50)(-100,35)
\put(100,60){\circle*{4}}
\put(130,75){\circle*{4}}
\put(130,45){\circle*{4}}
\put(100,60){\line(2,1){30}}
\put(100,60){\line(2,-1){30}}
\put(82,60){\makebox(0,0){$[0,1]$}}
\put(140,75){\makebox(0,0){$-2$}}
\put(140,45){\makebox(0,0){$-2$}}
\end{picture}

\noindent
Here $[0,1]$ denotes a genus $0$ core-space with one disc removed.
The Euler number of
the $S^1$-bundle corresponding to the middle vertex is irrelevant, the resulted
$3$-manifolds with boundary are diffeomorphic with each other
 (hence with $Y$ too).
However,  the restriction of the canonical Seifert fibration to $ \partial Y
\simeq S^1 \times S^1 $ determines an $S^1$--fibration of the boundary.
Moreover, $Y$ admits a  unique  closed Seifert $3$--manifold $ \bar{Y}$,
from which $Y$ can be
 obtained by omitting a tubular neighbourhood of a generic fibre (that is,
 $ \bar{Y}$ is obtained by extending the fibration of $\partial Y$ to an
 $S^1$--fibration {\it without any  new special Seifert fibres}.
In this way there is a
canonical choice for the Euler number of the `middle' vertex, which is the
`middle' Euler number of the plumbing graph of $ \bar{Y}$.
We will calculate it below.
Using this  Euler number, the graph also determines
a parametrization (framing) of $ \partial Y\simeq S^1\times S^1$.

\subsection{Homotopical properties of $ Y $}\label{ss:homY}
In order to understand better the structure of $Y$
 we consider its fundamental domain  (in coordinates $(s,t,z)$):

\begin{picture}(300,160)(0,-5)
\put(100,10){\vector(1,0){100}}
\put(100,10){\vector(0,1){100}}
\put(100,110){\vector(1,0){100}}
\put(200,110){\vector(0,-1){100}}
\put(100,60){\vector(1,0){100}}
\put(160,140){\vector(-2,-1){60}}
\put(260,140){\vector(-2,-1){60}}
\put(160,140){\vector(1,0){100}}
\put(200,10){\vector(2,1){60}}
\put(260,140){\vector(0,-1){100}}
\dashline{2}(160,40)(160,140)
\dashline{2}(160,40)(260,40)
\dashline{2}(100,10)(160,40)
\put(100,60){\vector(2,1){60}}
\put(260,90){\vector(-2,-1){60}}
\put(160,90){\vector(1,0){100}}
\dashline{4}(130,75)(230,75)
\put(130,75){\vector(1,0){100}}

\put(151,35){\vector(2,1){10}}
\put(160,136){\vector(0,1){4}}
\put(256,40){\vector(1,0){4}}

\put(125,90){\makebox(0,0){$ z $}}
\put(140,70){\makebox(0,0){$ s $}}
\put(140,88){\makebox(0,0){$ t $}}
\put(130,75){\vector(1,0){20}}\thicklines
\put(130,75){\vector(2,1){15}}\thicklines
\put(130,75){\vector(0,1){20}}\thicklines
\put(115,72){\makebox(0,0){$ \mu $}}
\put(250,80){\makebox(0,0){$ \mu $}}
\put(172,53){\makebox(0,0){$ \lambda $}}
\put(220,97){\makebox(0,0){$ \lambda $}}
\put(180,81){\makebox(0,0){$ \bar{\lambda} $}}

\put(222,27){\makebox(0,0){$ m' $}}
\put(137,35){\makebox(0,0){$ m $}}
\put(235,120){\makebox(0,0){$ m $}}
\put(144,124){\makebox(0,0){$ m' $}}

\put(162,15){\makebox(0,0){$ c_1 $}}
\put(212,135){\makebox(0,0){$ c_2 $}}
\put(100,60){\circle*{4}}
\put(90,63){\makebox(0,0){$ P_0 $}}
\end{picture}

\noindent
where $ s \in [0, \pi] $, $ t \in [- \pi, \pi ] $ and $ z \in I= [-1, 1] $. The original coordinates  are $x= e^{si} $ and $ y= e^{ti} $. Then $Y$ is obtained  by the following  identification of the sides of the cube:
\begin{equation*}
 (0, t, z) \sim ( \pi, -t, -z) \mbox{ and } (s, -\pi, z) \sim (s, \pi, z) \mbox{.}
\end{equation*}
The boundary is
\begin{equation*}
 \partial Y = \frac{[0, \pi] \times [- \pi, \pi ] \times \{-1, 1 \}}{(0,t,\pm 1)
 \sim (\pi, -t, \mp 1), \ (s, -\pi, \pm 1 ) \sim (s, \pi, \pm 1)} \simeq S^1 \times S^1 \mbox{.}
\end{equation*}
%
%\vspace{2mm}
%
The $ S^1 $-action determining the canonical Seifert fibration (\ref{eq:seif}) is
induced by the translation along the $s$--axis. The exceptional fibres are
\begin{equation*}
 \lambda = \{(s, -\pi, 0) \ | \ s \in [0, \pi] \} = \{(s, \pi, 0) \ | \ s \in [0, \pi] \} \ \mbox{ and} \
%\end{equation*}
%\begin{equation*}
 \bar{\lambda} = \{(s, 0, 0) \ | \ s \in [0, \pi] \}  \mbox{.}
\end{equation*}
Any fixed $(t,z) \notin \{ (0, 0), (\pm \pi, 0) \} $ determines a generic fibre in the form
\begin{equation*}
 \{(s, t, z) \ | \ s \in [0, \pi] \} \cup \{ (s, -t, -z) \ | \ s \in [0, \pi ] \} \mbox{,}
\end{equation*}
where $ (0, \pm t, \pm z) $ are glued together with $(\pi, \mp t,  \mp z) $.
For example, $c= c_1 \cup c_2 \subset \partial Y$ is a generic fibre.

The base space $D^2$ of the canonical Seifert-fibration can be represented as

\begin{equation*}
 \frac{\{0 \} \times [- \pi, \pi ] \times [-1, 0]}{(0, - \pi, z)
 \sim (0, \pi, z) \mbox{ , } (0, t, 0) \sim (0, -t, 0)} \mbox{.}
\end{equation*}
Its boundary is the class of $m$, a circle.

Next, we describe the fundamental group and the homology of $ Y $. % and $ \bar{Y} $.
$ Y $ is homotopically equivalent with the Klein bottle $ \mathcal{K} $.
Let us choose the base point $ P_0=(0,-\pi, 0) $. All four vertexes of
the rectangle representing $ \mathcal{K} $ represent $ P_0$.
Thus the fundamental group of $ Y $ can be presented as
 \begin{equation}\labelpar{eq:fundY}
  \pi_1 (Y) = \langle \mu, \lambda \ | 	\  \mu \cdot \lambda \cdot \mu   = \lambda \rangle \mbox{,}
  \end{equation}
 where $ \mu $ and $ \lambda $ denote also the class of $ \mu $ and $ \lambda $ in $ \pi_1 (Y) = \pi_1 ( \mathcal{K})$;
 cf. with the description (2) from \ref{ss:cY}). A more precise description can be given via the next diagrams,
 provided by the $\{z=0\}$ subspace of $Y$ (which can be identified with ${\mathcal K}$).

\begin{picture}(300,90)(0,-5)
\put(0,60){\vector(1,0){100}}
\put(0,10){\vector(0,1){50}}
\put(0,10){\vector(1,0){100}}
\put(100,60){\vector(0,-1){50}}
\put(50,70){\makebox(0,0){$ \lambda $}}
\put(50,0){\makebox(0,0){$ \lambda $}}
\put(-10,35){\makebox(0,0){$ \mu $}}
\put(110,35){\makebox(0,0){$ \mu $}}
\put(0,35){\vector(1,0){100}}
\put(50,45){\makebox(0,0){$ \bar{ \lambda} $}}

\put(150,60){\vector(1,0){100}}
\put(150,10){\vector(0,1){50}}
\put(150,10){\vector(1,0){100}}
\put(250,60){\vector(0,-1){50}}
\put(200,70){\makebox(0,0){$ \lambda $}}
\put(200,0){\makebox(0,0){$ \lambda $}}
\put(140,35){\makebox(0,0){$ \mu $}}
\put(260,35){\makebox(0,0){$ \mu $}}
\dashline{2}(152,35)(248,35)
\dashline{2}(248,35)(248,60)
\dashline{2}(152,10)(152,35)
\put(248,56){\vector(0,1){4}}
\put(200,45){\makebox(0,0){$ \bar{ \lambda} $}}

\put(300,60){\vector(1,0){100}}
\put(300,10){\vector(0,1){50}}
\put(300,10){\vector(1,0){100}}
\put(400,10){\vector(0,1){50}}
\put(350,70){\makebox(0,0){$ \bar{\lambda} $}}
\put(350,0){\makebox(0,0){$ \lambda $}}
\put(290,35){\makebox(0,0){$ \bar{\lambda} $}}
\put(410,35){\makebox(0,0){$ \lambda $}}
\dashline{2}(300,60)(400,10)
\put(390,15){\vector(2,-1){10}}
\put(355,40){\makebox(0,0){$ \mu $}}
\end{picture}

\noindent
The first diagram shows homological cycles.
 In order to rewrite the fundamental group,
 let $ \bar{ \lambda} $ be the closed path  shown in the second diagram
  by  the dashed line. Then $ \bar{\lambda} = \mu \cdot  \lambda $ in $ \pi_1 ( Y) $. Note that
 $ \lambda^2 = \mu \cdot \lambda \cdot \mu \cdot \lambda =  \lambda \cdot \mu \cdot \lambda \cdot \mu $,
 thus $ \lambda^2 = \bar{\lambda}^2 = \mu^{-1} \cdot \lambda^2 \cdot \mu $
 and this element  commutes with $ \mu $.
The fundamental group can be also presented as
 \begin{equation}\labelpar{eq:fundY2}
  \pi_1 (Y) = \langle \lambda, \bar{\lambda} \ | 	\  \bar{\lambda}^2    = \lambda^2 \rangle \mbox{,}
 \end{equation}
according to the third  picture above.

%\begin{picture}(300,90)(150,-5)
%\put(300,60){\vector(1,0){100}}
%\put(300,10){\vector(0,1){50}}
%\put(300,10){\vector(1,0){100}}
%\put(400,10){\vector(0,1){50}}
%\put(350,70){\makebox(0,0){$ \bar{\lambda} $}}
%\put(350,0){\makebox(0,0){$ \lambda $}}
%\put(290,35){\makebox(0,0){$ \bar{\lambda} $}}
%\put(410,35){\makebox(0,0){$ \lambda $}}
%\dashline{2}(300,60)(400,10)
%\put(396,12){\vector(2,-1){4}}
%\put(355,40){\makebox(0,0){$ \mu $}}
%\end{picture}

 On the other hand, the fundamental group of the boundary is
\begin{equation}\labelpar{eq:pY}  \pi_1 ( \partial Y ) = H_1 (\partial Y , \Z ) \cong \Z \langle m \rangle
\oplus \Z \langle c \rangle \mbox{.} \end{equation}
The $ \partial Y \simeq S^1\times S^1 \hookrightarrow Y $ embedding (which is homotopically the same as the orientation covering
$S^1\times S^1\to \mathcal{K} $) induces a monomorphism $ \pi_1 (\partial Y) \to \pi_1 (Y) $. It is determined by the images of the generators, which are
\begin{equation}\label{eq:cmu}
 m \mapsto \mu \mbox{ and } c \mapsto \lambda^2= \bar{ \lambda}^2 \mbox{.}
\end{equation}
A direct computation shows that $ [ \pi_1 (Y) , \pi_1(Y) ] = \Z \langle \mu^2 \rangle $ and
\[
 H_1 (Y, \Z ) \cong \Z \langle \lambda \rangle \oplus \Z_2 \langle \mu \rangle \mbox{.}
\]
Note that $m=m'$ in $H_1(\partial Y,\Z)$, and analysing  $\{s=0\}\subset Y$ one obtains that  $m=-m'$ in $H_1(Y,\Z)$.
Hence the class of $m$ in $H_1(Y,\Z)$ has order 2, it is exactly $\mu$.

The next Lemma shows that the classes $m$ and $c$ in $H_1(\partial Y,\Z)$ have certain universal
properties with respect to the inclusion $\partial Y\subset Y$.

\begin{lem}\label{lem:UNIV}
(a) $\pm m$ are the unique primitive elements of $H_1(\partial Y,\Z)$
with the property that their doubles vanish in $H_1(Y,\Z)$.

(b) $\pm c$ are the unique primitive elements of $H_1(\partial Y,\Z)=\pi_1(\partial Y)$
whose images in $\pi_1(Y)$ are in the center of $\pi_1(Y)$.
\end{lem}
\begin{proof} (a) is clear. For (b) first note
 that any element of $\pi_1(Y)$ can be written in the form $\lambda^k\mu^l$ for some $k,l\in\Z$, and then using this one
 verifies that
 the center of $\pi_1(Y)$ is $\langle \lambda^2 \rangle$.
\end{proof}
\subsection{Homological properties of  $\bar{Y}$}\label{ss:bary} The closed Seifert
3--manifold  $\bar{Y}$ considered in \ref{ss:cY}(4) is constructed  as follows.
First we consider a new disc $D_{new}^2$ and the trivial fibration $D_{new}^2\times S^1$.
Then we paste  $m$ with the boundary of $D_{new}^2$
 and we extend the canonical Seifert--fibration of $Y$ above this  disc (as base space)
of the  trivial fibration $D_{new}^2\times S^1$. This leads to the Seifert fibred closed manifold
\begin{equation}\labelpar{eq:barY}
 \bar{Y} = \frac{Y \cup
  (D_{new}^2 \times S^1)}{\partial Y\sim \partial D_{new}^2 \times S^1,\  m\sim \partial D_{new}^2 \times *, \
  c\sim *\times  S^1 } \mbox{.}
\end{equation}
$ H_1(\bar{Y},\Z) $ can be determined by  the Mayer-Vietoris sequence of the decomposition (\ref{eq:barY}):
\[
 \begin{array}{ccccccc}
  H_1( \partial Y, \Z) & \to & H_1 (Y, \Z) \oplus H_1 (D_{new}^2 \times S^1, \Z) & \to & H_1 ( \bar{Y}, \Z) & \to & 0 \\
  \|            &     & \parallel                                   &   &  \parallel         &      &  \\
  \Z \langle m \rangle \oplus \Z \langle c \rangle & \to & \Z \langle \lambda \rangle \oplus \Z_2 \langle \mu \rangle \oplus \Z \langle c' \rangle & \to & H_1 ( \bar{Y}, \Z) & \to & 0 \\
 %( m,c) & \mapsto & ( \mu, 2 \lambda + c') & & & & \\
 % c & \mapsto &  2 \lambda + c'  & & & & \\
 \end{array}
\]
where $ m \mapsto \mu$ and $c\mapsto  2 \lambda + c'$.
Thus \begin{equation}\label{eq:Z}
 H_1 ( \bar{Y}, \Z) \cong \Z \langle \lambda \rangle .\end{equation}

\subsection{A plumbing graphs of $Y$ and $\bar{Y}$}\label{ss:plY}
By \cite{neumann2},  $ \bar{Y} $ has a plumbing graph $ G $ of the form

 \begin{picture}(300,50)(-100,35)
 \put(100,60){\circle*{4}}
 \put(130,75){\circle*{4}}
 \put(130,45){\circle*{4}}
 \put(100,60){\line(2,1){30}}
 \put(100,60){\line(2,-1){30}}
 \put(90,60){\makebox(0,0){$ e $}}
 \put(145,75){\makebox(0,0){$-2$}}
 \put(145,45){\makebox(0,0){$-2$}}
 \end{picture}

\noindent
(see also the discussion from \ref{ss:cY}(4))
and the Euler number $e$ should be chosen such that $ H_1 (M^3(G), \Z ) \cong \Z$
(cf. (\ref{eq:Z})).
%where $M^k(G)$ denotes the plumbed $k$--manifold associated with  $G$,$k=3$ or 4.
Here $ M^3(G)  $ denotes the plumbed $3$--manifold associated with the graph $G$, it is the boundary of $ M^4 (G) $, the plumbed $4$--manifold associated with $G$. %Consider the Mayer-Vietoris sequence of the pair $ (M^4 (G), M^3(G))$
%\[
% \begin{array}{ccccccc}
%   H_2 (M^4(G), \Z) & \rightarrow & H_2(M^4(G), M^3 (G), \Z) & \to & H_1 (M^3(G), \Z) & \to & 0 \\
% \end{array}
%\]
%where the first map is given by the intersection matrix. Hence $ H_1 (M^3(G), \Z) \cong \Z $ if and only if
%\[
% \det \left(
% \begin{array}{ccc}
% -2 & 0 & 1 \\
% 0 & -2 & 1 \\
% 1 & 1 & x \\
% \end{array}
% \right)
% =0 \mbox{,}
%\]
%thus $ e= -1 $.
This (via the long cohomological exact sequence of the pair
$(M^4(G),M^3(G))$) imposes the degeneracy of the intersection matrix of the plumbing (cf. e.g.
 \cite[15.1.3]{NSz}).  Hence $e=-1$.
 In particular, $\bar{Y}$ (without its Seifert fibration structure) is diffeomorphic to
 $S^1\times S^2$ (which also shows that $\bar{Y}$ admits an orientation reversing diffeomorphism).

 Furthermore,  consider the graph

 \begin{picture}(300,50)(-100,35)
 \put(100,60){\circle*{4}}
 \put(100,60){\vector(-1,0){34}}
 \put(130,75){\circle*{4}}
 \put(130,45){\circle*{4}}
 \put(100,60){\line(2,1){30}}
 \put(100,60){\line(2,-1){30}}
 \put(95,70){\makebox(0,0){$-1$}}
 \put(145,75){\makebox(0,0){$-2$}}
 \put(145,45){\makebox(0,0){$-2$}}
 \end{picture}

\noindent
The arrow denotes a knot $ K \simeq S^1 \subset  \bar{Y} $, which is a generic
 $S^1$--fibre associated with the middle vertex by the plumbing construction of $\bar{Y}$. %, and it is also a fibre of the canonical Seifert fibration.
 Let $ N(K)^\circ  $ be an open  tubular neighbourhood of $ K $ in $ \bar{Y} $.
 Then
 %We proved the following:
%
%\begin{prop} (a)
$ Y \simeq \bar{Y} \setminus N(K)^\circ  $, and
%(b) T
the induced (singular/Seifert) $S^1$--fibration associated with the
 middle vertex (by the plumbing construction) on $Y$ agrees with the canonical Seifert fibration of $Y$.
%\end{prop}
The Euler number $-1$ of the middle vertex %is irrelevant for point (a). It just
determines a parametrization (framing)
of $ \partial Y \simeq S^1 \times S^1 $.

%\subsection{}
We can present $ \pi_1 (Y) $ also from the plumbing graph using the description of \cite{mumford}. Let $ \lambda $, $ \bar{ \lambda} $ and $c$ be
oriented $S^1$--fibres associated with the three vertices provided by the plumbing construction (and extended by convenient
connecting paths to  a base point as in \cite{mumford}).
Next, let $m$ be the meridian of $K$  corresponding to the arrowhead
(and extended by a convenient path to the base point).

\begin{picture}(300,65)(-100,26)
\put(100,60){\circle*{4}}
\put(100,60){\vector(-1,0){34}}
\put(130,75){\circle*{4}}
\put(130,45){\circle*{4}}
\put(100,60){\line(2,1){30}}
\put(100,60){\line(2,-1){30}}
\put(95,70){\makebox(0,0){$-1$}}
\put(145,75){\makebox(0,0){$-2$}}
\put(145,45){\makebox(0,0){$-2$}}
\put(95,50){\makebox(0,0){$c$}}
\put(130,85){\makebox(0,0){$ \lambda $}}
\put(130,35){\makebox(0,0){$ \bar{ \lambda} $}}
\put(66,50){\makebox(0,0){$m$}}
\end{picture}

\noindent
% We also use these symbols to denote its class in $ \pi_1 (\bar{ Y}) $ and in
%$ \pi_1 (Y) $ (extended with a path from and to the base point).
Then, by \cite{mumford},  there is a choice of the connecting paths such that
$ \pi_1 (Y) $ is generated by $\lambda, \ \bar{\lambda}, \ c$ and $m$, and they satisfies the relations $\lambda^2=\bar{\lambda}^2=c$, and $c=m\lambda\bar{\lambda}$.
This is compatible with the description from Subsection \ref{ss:homY},
cf. (\ref{eq:fundY}) and (\ref{eq:cmu}).

Note that by plumbing calculus (cf. \cite{neumann1})
one has the equivalence of plumbed manifolds (where $\widetilde{\Gamma}$
is any graph):

\begin{picture}(300,60)(-30,30)
\put(20,40){\framebox(50,40)}
\put(40,60){\makebox(0,0){$\widetilde{\Gamma}$}}
 \put(100,60){\circle*{4}}
 \put(100,60){\line(-1,0){34}}
 \put(130,75){\circle*{4}}
 \put(130,45){\circle*{4}}
 \put(100,60){\line(2,1){30}}
 \put(100,60){\line(2,-1){30}}
 \put(95,70){\makebox(0,0){$-1$}}
 \put(145,75){\makebox(0,0){$-2$}}
 \put(145,45){\makebox(0,0){$-2$}}

 \put(220,40){\framebox(50,40)}
\put(240,60){\makebox(0,0){$\widetilde{\Gamma}$}}
 \put(300,60){\circle*{4}}
 \put(300,60){\line(-1,0){34}}
 \put(330,75){\circle*{4}}
 \put(330,45){\circle*{4}}
 \put(300,60){\line(2,1){30}}
 \put(300,60){\line(2,-1){30}}
 \put(295,70){\makebox(0,0){$1$}}
 \put(345,75){\makebox(0,0){$2$}}
 \put(345,45){\makebox(0,0){$2$}}

 \put(185,60){\makebox(0,0){$\simeq$}}
 \end{picture}

Hence, $Y$ or $-Y$ spliced along $K$ to any 3--manifold
give rise to  diffeomorphic manifolds.

  % \[
 % \pi_1 (\bar{Y}) = \langle \lambda, \bar{\lambda}, c \ | \ \lambda^2 = \bar{\lambda}^2 = \lambda \cdot \bar{\lambda} = c \rangle \mbox{,}
% \]
% from which $ \lambda= \bar{ \lambda}$ follows, and $ \pi_1 (M^3(G)) = \Z \langle \lambda \rangle = \Z \langle \bar{ \lambda} \rangle $.
%
% Similarly,
% \[
% \pi_1 (Y) = \langle \lambda, \bar{\lambda}, c, m \ | \ \lambda^2 = \bar{\lambda}^2 = \lambda \cdot \bar{\lambda} \cdot m = c \rangle \mbox{,}
% \]
% thus $ \lambda = \bar{ \lambda} \cdot m $ and
% \[ \pi_1 ( Y) = \langle \lambda, \bar{\lambda} \ | \lambda^2 = \bar{\lambda}^2 \rangle \mbox{ .}
% \]

 \section{The boundary of the Milnor fibre}\label{s:boundmiln}

\subsection{} Let $ F = f^{-1} ( \delta ) \cap B^6_{ \epsilon} $ be the
Milnor fibre of $ f $, where $ \delta \in \C^* $, $| \delta | \ll \epsilon $.
We wish to construct the $3$--manifold $ \partial F = f^{-1} ( \delta )
\cap S^5_{ \epsilon} $ as a surgery of $ S^3 $ along the link $ L $.

Let $ N_i $ be a sufficiently small
tubular neighbourhood of $ L_i $ in $ S^3 $.
%The choice of $ N_i $ will be specified later.
For each $j=\{i, \sigma(i) \} $ we define $X_j $ as
\begin{equation}\labelpar{eq:xj}
 X_j =   \left\{ \begin{array}{ccc} S^1\times S^1 \times I & \mbox{ if} &  i \neq \sigma(i), \\
                  Y & \mbox{ if} &  i = \sigma(i), \\
                 \end{array} \right.
                 \end{equation}
where $ Y $ is the $3$--manifold described in Section \ref{s:Y}. Recall
 that $ \partial Y \simeq S^1\times S^1$.

\begin{prop}\label{pr:gl} One has an orientation preserving diffeomorphism
 \begin{equation}\labelpar{eq:mb}  \partial F \simeq \left(\mathfrak{S}^3 \setminus
 \bigcup_{i=1}^l {\rm int}(N_i)  \right)\cup_{ \phi} \left(\bigcup_{j \in J} X_j \right) \mbox{,} \end{equation}
 where $ \phi : \partial (\mathfrak{S}^3\setminus \cup_i\,{\rm int}(N_i))\to
  -\partial (\cup_{j\in J}X_j)$ is a collection $( \phi_j)_{j \in J}$ of diffeomorphisms
% $ \phi_{ \{ i, \sigma(i) \} }: -\partial N_i \cup -\partial N_{\sigma(i)} \to -\partial X_{ \{ i, \sigma(i) \}}$.
 % That is
 \[ \phi_{\{i, \sigma(i)\} }: \left\{ \begin{array}{ccc} -\partial N_i \cup -\partial N_{\sigma(i)}
  \to -\partial
 (S^1\times S^1 \times I)  & \mbox{if} & i \neq \sigma(i), \\
- \partial N_i \to -\partial Y \hspace{1cm} & \mbox{if} & i = \sigma(i). \\
                                \end{array} \right.    \]
\end{prop}

\begin{proof} The decomposition follows from the general decomposition proved in \cite{Si}, see also
 \cite[2.3]{NSz}. For the convenience of the reader we sketch the construction.
 Recall that $\Upsilon_j\subset S^5_\epsilon$ is the link of the component $\Sigma_j$ of $\Sigma$,
 $j=\{i,\sigma(i)\}\in J$.
 Consider a sufficiently small tubular  neighbourhood $N(\Upsilon_j)$ of it in $S^5_\epsilon$.
 We can assume that $\Phi^{-1}(N(\Upsilon_j))=N_i\cup N_{\sigma(i)}$. Furthermore, for $\epsilon$ small,
 the intersection of $\partial N(\Upsilon_j)$ with $K:=X\cap S^5_\epsilon$ is transverse.
 Therefore, for $0<|\delta|\ll \epsilon$, the intersection of $\partial N(\Upsilon_j)$ with $\partial F$ is
 still transverse in $S^5_\epsilon$, and, in fact, $K\setminus \cup_j N(\Upsilon_j)$ is diffeomorphic with
 $\partial F\setminus
 \cup_j N(\Upsilon_j)$. But, the former space can be identified via $\Phi$ by $\mathfrak{S}^3\setminus \cup_i
 (N_i\cup N_{\sigma(i)})$.
 This is the space in the first parenthesis of (\ref{eq:mb}).

 The second one  is a union of spaces of type $X_j:=\partial F\cap N(\Upsilon_j)$,
 which fibres over $\Upsilon_j\simeq S^1$.  The fibre of the fibration is the
 Milnor fibre of the corresponding transverse plane curve singularity (of $\Sigma_j$).
 Since the transverse type is $A_1$, this fibre is $F_j:=S^1\times I$. The monodromy of
 the fibration is the so-called {\it geometric vertical monodromy } of the transverse type, it is orientation
 preserving  self-diffeomorphism of $ S^1 \times I $. If it does not permute
 the two components of $\partial F_j$ then it preserves the orientation of
 $I$, hence of $S^1$ too, hence up to isotopy it is the identity. If it
 permutes the components of $\partial F_j$ then
up to isotopy it is
 $ (\alpha, t) \to (\bar{ \alpha}, - t)$, where $ (\alpha, t ) \in S^1 \times I $.
  %and $ \bar{ \alpha} $ denotes the image of $ \alpha \in S^1 $ along an axial reflection (i.e. complex conjugation).
The two types of vertical monodromies provide the two choices of  $ X_j $ in formula (\ref{eq:xj}), cf. description (1) of $Y$  in
Subsection \ref{ss:cY}.
\end{proof}

\subsection{Preliminary discussion regarding the gluing}\label{ss:gluinguj}
Our next aim is to describe the gluing functions $ \phi_j $. In both cases two tori must be glued:  if $ i\neq\sigma(i)$
then basically one should identify  $ \partial N_i $ and $ -\partial N_{\sigma (i)}$,
 otherwise $\partial Y$ and $\partial N_i$.
Up to diffeotopy an orientation reversing diffeomorphism between tori is given by an invertible
 $2 \times 2 $ matrix over $ \Z $ with determinant $-1$.
 It turns out that in our cases all these gluing matrices have the form
\begin{equation}\label{eq:glue}
\left( \begin{array}{cc}
  -1 & n_j\\
  0 & 1 \\
 \end{array} \right)
\end{equation}
 \noindent hence its only relevant entry is
 the off--diagonal one.

 This integer will be determined by a newly
 introduced invariant, the {\it vertical index},  associated with each
$j\in J$.
 This is done using a special germ $H:(\C^3,0)\to (\C,0)$,
which will have a double role.
First, it provides some kind of framing along $\Sigma\setminus \{0\}$,
and also helps to
identify generators from the boundaries of  $K\setminus \cup_j N(\Upsilon_j)$ and  $\mathfrak{S}^3\setminus \cup_i (N_i\cup N_{\sigma(i)})$
respectively (constructed in two different levels: in the target and in the source of $\Phi$).

We will use three parametrizations of $ \partial N_i \simeq S^1 \times S^1 $ with the same meridian but different longitudes.
The topological longitude is the usual knot--theoretical Seifert--framing of
 $ L_i \subset S^3 $. The \emph{resolution longitude}
 is determined via  a good embedded resolution of $ (D, 0) \subset ( \C^2, 0 ) $,
  it creates the bridge with the decorations and the combinatorics of
  the resolution graph. Finally,
the \emph{sectional longitude} depends on $ H$ and it will be used for describing the gluing. In fact, the sectional longitude allows us to compare the source and the target of $ \Phi $: being defined by the geometry of $H$, the function $H$ and
its pull-back $ \Phi^* H $  plays the role of transportation of the invariants from
$\C^3 $ level to $\C^2$ level.

\subsection{The local form of $f$ along $ \Sigma $}\label{ss:local}
In the sequel we will use the notation $\Sigma^*=\Sigma\setminus \{0\}$ and $\Sigma^*_j=\Sigma_j\setminus \{0\}$.
 Recall that in a small neighbourhood (in $\C^3$) of any point $p\in \Sigma^*$ the
 space $(X,0)$ has two local
 components, both smooth and  intersecting each other transversally.

A more precise local description along $\Sigma^*$  is the following.
Let us fix a point $p_0\in \Sigma_j^*$ and let $U_0$ be a small
neighbourhood of $p_0$ in $\C^3$. In $U_0$ the function $f$ is a product $f_1\cdot f_2$,
where both $f_n$ are holomorphic, $\{f_n=0\}$ are smooth
 and intersect each other transversally.
 (The intersection is $\Sigma^*\cap U_0$; later the fact that at $p_0$ the local parametrization  of
  $\Sigma^*\cap U_0$ together with $f_1$ and $f_2$ might serve as local coordinates will be exploited further.)

 $f_1$ and $f_2$  are well--defined up to a multiplication
 by an invertible holomorphic  function $\iota$ of
 $U_0$; that is, $(f_1,f_2)$ can be replaced by $(\iota  f_1, \iota^{-1}f_2)$.
% and if  $i=\sigma(i)$ then $(f_1,f_2)$ by $(f_2,f_1)$.
 At any point $p\in \Sigma^*_j\cap U_0$ the linear term of $f_n$, $n\in\{1,2\}$,
(say, in the Taylor expansion) is $T_1(f_n)= \sum _{k=1}^3 u_{nk}(p)(x_k-p_k)$, where
$(x_1,x_2,x_3)$ are the fixed coordinates of $(\C^3,0)$.
Let us code this in the non--zero vectors $u_n(p):=(u_{n1}(p),u_{n2}(p), u_{n3}(p))$.
Hence, at any $p\in \Sigma^*_j\cap U_0$ we have two vectors
$u_1(p)$ and $u_{2}(p)$ well--defined up to
multiplication by $\iota|_{\Sigma_j^*\cap U_0}$ (in the sense described above).
%and in the case  $i=\sigma(i)$ up to a permutation.
Their classes $ [u_n(p)]\in
\C \mathbb {P} ^2$ for $p\in \Sigma^*_j\cap U_0 $
 are independent of the $\iota$--ambiguity, hence are well--defined elements.
%up to permutation whenever $i=\sigma(i)$.
In particular, they determine a global pair of elements
$[u_1(p)]$ and  $[u_2(p)]\in
\C \mathbb {P} ^2$ for $p\in \Sigma^*_j $, well--defined whenever $i\not=\sigma(i)$, and
well--defined up to permutation whenever $i=\sigma(i)$.
%(In other words, if $i=\sigma(i)$, then $\Sigma^*_j\ni p\mapsto \{[u_1(p)],
%[u_2(p)]\}$ is a two--valued function.)

In fact, we can do even more: there exists a splitting of $f$ into product $f_1\cdot f_2$ along $\Sigma^*_j$ without any invertible element ambiguity (but preserving the
permutation ambiguity whenever $i=\sigma(i)$).

Indeed, assume that we are in the trivial covering ($i\not=\sigma(i)$) case, and
let us cover $\Sigma_j^*$ by small discs $\{U_\alpha\}_\alpha$ such that
on each $U_{\alpha}$ we can fix a splitting $f(p)=f_{1,\alpha}(p)
\cdot f_{2,\alpha}(p)$, $p\in U_\alpha$. For any intersection $U_{\alpha\beta}=
U_{\alpha}\cap U_{\beta}$ the two splittings can be compared: we define
$\iota_{\alpha\beta}\in \calO(U_{\alpha\beta})^*$ by $f_{1,\alpha}|_{U_{\alpha\beta}}=\iota_{\alpha\beta}\cdot
f_{1,\beta}|_{U_{\alpha\beta}}$. From this definition follows that
$\{\iota_{\alpha\beta}\}_{\alpha, \beta}$ form a \v{C}ech 1--cocycle.
\begin{lem}
$H^1(\Sigma^*_j,\calO_{\Sigma^*_j}^*)=0$.
\end{lem}
\begin{proof}
From the exponential exact sequence $0\to \Z \to \calO\to \calO^*\to 0$ over $\Sigma_j^*$,
we get that it is enough to prove the vanishing
$H^1(\Sigma^*_j,\calO_{\Sigma^*_j})=0$, a fact which follows from Cartan's Theorem, since $\Sigma_j^*$ is Stein.
\end{proof}

Since $H^1(\Sigma^*_j,\calO_{\Sigma^*_j}^*)=0$, the cocycle
$\{\iota_{\alpha\beta}\}_{\alpha, \beta}$  is a coboundary.
This means that we can find invertible functions $\iota_\alpha $ on
each $U_\alpha$ such that on $U_{\alpha\beta}$ one has
$\iota_{\alpha}|_{U_{\alpha\beta}}=\iota_{\alpha\beta}\cdot
\iota_{\beta}|_{U_{\alpha\beta}}$. This means that  the local functions
 $\widetilde{f}_{1, \alpha}:=f_{1,\alpha }\cdot \iota_{\alpha}^{-1}$,
  $\widetilde{f}_{2, \alpha}:=f_{2,\alpha }\cdot \iota_{\alpha}$ on $U_{\alpha}$
  provide a splitting (that is, $f(p)=\widetilde{f}_{1,\alpha}(p)
\cdot \widetilde{f}_{2,\alpha}(p)$, $p\in U_\alpha$), but in this new situation
the local splittings glue globally: $\widetilde{f}_{1,\alpha}|_{U_{\alpha\beta}}=
\widetilde{f}_{1,\beta}|_{U_{\alpha\beta}}$.
If $i=\sigma(i)$ then we repeat the proof on $D_i$.

 \subsection{The special germ $H$} Next, we treat the `aid'--germ $H$.

\begin{defn}\label{def:H} Let us fix $\Phi, \ f, \ \Sigma$ as above.
 A germ $ H: ( \C^3, 0) \to ( \C ,0) $ is called \emph{transverse section} along $ \Sigma $ if
 $ \Sigma \subset H^{-1} (0) $, $H^{-1}(0)$ at any point of $\Sigma^*$ is smooth and intersects both local
 components of $(X,0)$ transversally.
\end{defn}

We claim that  transverse sections  always exist.

\begin{prop}\labelpar{pr:H}
 There exist  complex numbers  $a_1, a_2, a_3 $ such that
 $H= a_1 \partial_{x_1} f + a_2 \partial_{x_2} f + a_3 \partial_{x_3} f $
 is a transverse section.
 %Moreover, we can even assume that $H=0$ intersects $f=0$ transversally off $\Sigma$.
\end{prop}

\begin{proof} For such $H$ one has
%\begin{equation}\labelpar{eq:kifh}
$  T_1(H)  = \sum_{k=1}^3 a_k u_{1k}\cdot T_1(f_2)+
  \sum_{k=1}^3 a_k u_{2k}\cdot T_1(f_1)$.
%\end{equation}
In particular, we have to show that for certain coefficients $\{a_k\}_k$
the expressions $ \beta_n(p)=\sum_{k=1}^3 a_k u_{nk} ( p )$ (for $n\in\{1,2\}$)
 have no zeros for $p\in \Sigma^*$.
%Consider the holomorphic map
%\begin{equation*}
%  \tilde{u}_n: \Sigma_j^* \to \C \mathbb {P} ^2\,, \ \
%  \tilde{u}_n (p) = [u_{n}(p)].
%\end{equation*}

If $\gamma : X^N\to X$ is the Nash transform of $X$, then $\gamma^{-1}(0)$
is the set of limits of tangent spaces of $X\setminus \{0\}$, it is an
algebraic set of  $\C \mathbb {P} ^2$ of dimension $\leq 1$
(for details see e.g. \cite{LeTeissier} and references therein).
The existence of Nash transform guarantees that
$ [\ell_n]=\lim_{p\to 0}[u_n(p)] \in \C\mathbb{P}^2$  exist.
Indeed, the set $\{[l_1],[l_2]\}$ is the intersection of
$\gamma^{-1}(0)$ with the strict transform of $\Sigma_j$.
Then
let $[a_1:a_2:a_3]$ be generic such that
$\sum_ka_k\ell_{nk}\not=0$ for $n\in\{1,2\}$.
With this choice  $\sum_{k=1}^3 a_k u_{nk}(p) \not=0$ for $p\not=0$ and
in a small representative of $\Sigma_j$.
\end{proof}

Fix  again
$j=\{i, \sigma(i)\}\in J $, and let  $ p: (\C, 0) \to (\Sigma_j,0)\subset (\C^3, 0) $,
$\tau\mapsto p(\tau)$,  be a
parametrization (normalization) of  $ \Sigma_j $.
  For any point  $p_0=p(\tau_0)$  and neighbourhood  $\Sigma^*_j\cap U_0 \ni p(\tau)$ the
discussion from the second paragraph of \ref{ss:local}   can be repeated, in particular
we have the holomorphic
 vectors $u_n(p(\tau))$ ($n\in\{1,2\}$) (with the choice ambiguities described there).
Additionally, choose some $H$ as in Definition \ref{def:H}.
The assumption regarding $H$   guarantees that
$T_1(H)(p(\tau))=\beta_1(\tau)T_1(f_1)(p(\tau))+\beta_2(\tau)T_1(f_2)(p(\tau)) $
for some holomorphic functions $\beta_1$ and $\beta_2$ on $p^{-1}(\Sigma_j^*\cap U_0)$. If we replace
$(u_1,u_2)$ by $(\iota u_1,\iota^{-1}u_2)$ then $(\beta_1,\beta_2)$ will be replaced by
$(\iota^{-1}\beta_1,\iota\beta_2)$,
hence the product $\beta_1\beta_2$ is independent of all the $\iota$ and permutation ambiguities. It is a holomorphic function
on $p^{-1}(\Sigma_j^*\cap U_0)$ depending only on the equations $f$ and $H$. This uniqueness also guarantees
that taking different points of $\Sigma^*_j$ and repeating the construction, the output glues to a unique
holomorphic function $\mathfrak{b}_j(\tau)$ on a small punctures disc of $(\C,0)$.
Usually $\mathfrak{b}_j(\tau)$ has no analytic extension to the origin, however
one has the following.
\begin{lem}\label{lem:Laurent} $\mathfrak{b}_j(\tau)$  is
a  Laurent series on $(\C,0)$.
\end{lem}
\begin{proof}
The finiteness of the poles follows e.g. from the homological identities from
Theorem \ref{th:mainth}, or from Corollary \ref{cor:identities}  combined with (\ref{eq:alpha}).
\end{proof}

We wish to emphasize that $ \mathfrak{b}_j(\tau) $ and its pole order usually depends on the choice of
$H$.

\begin{defn}\label{de:vert} Let $a \tau^{\mathfrak{v}_j} $ (for some $a\in\C^*$) be the non--zero
monomial with smallest power of
$\tau $ in the Laurent  series of $ \mathfrak{b}_j(\tau) $.
% (if such power does not exists, then we say that
%$\mathfrak{v}_j=-\infty $, however in ??????????????? we will prove that
%$\mathfrak{v}_j\not=-\infty $).
The integer $ \mathfrak{v}_j $ is called the \emph{vertical index
of $ f$ along $ \Sigma_j $ with respect to $ H$} (or, the $H$--vertical index).
\end{defn}

\subsection{Computation of the gluing functions $ \phi_j $}

We fix a transverse section  $H$ (cf. \ref{def:H}).
%which intersects $X\setminus \Sigma$ transversally (cf. \ref{pr:H}).
Then the divisor  $ \Phi^{*} (H) $ is $ H\circ \Phi=d\cdot d_\sharp $ for some (not necessarily reduced)
germ $d_\sharp:(\C^2,0)\to (\C,0)$ (such that $d$ and $d_\sharp $ have no common components).
Let
%$d'=\prod_n (d'_n)^{\alpha_n}$ be the irreducible decomposition of  $d'$,
$(D_\sharp,0)$ be the (non--reduced) divisor associated with $d_\sharp$, and
$N(L_\sharp)$ be a small tubular neighbourhood of the reduced link
$L_\sharp:={\rm red}(D_\sharp)\cap \mathfrak{S}^3$ in $\mathfrak{S}^3$.

Let the Milnor fibre of $ \Phi^{*} (H) $ (in $\mathfrak{S}^3$)  be
\[ F_{\Phi^{*} (H)} = \{ (s, t) \in \partial \mathfrak{B}_\epsilon=\mathfrak{S}^3
 \ | \  H ( \Phi  (s, t) )>0 \}. \]
Let $ \Lambda_i $ and $ \Lambda_\sharp $
denote  the components of the oriented intersection
$F_{\Phi^{*} (H)}\cap\, \partial N_i$ and $F_{\Phi^{*} (H)}\cap \partial N(L_\sharp)$
with the tubular neighborhood boundaries
  of  $ L_i $ and $L_\sharp$ respectively.
 ($\Lambda_\sharp$ might have several components, in this notation we collect all of them.)

 Furthermore,
let $ M_i \subset \partial N_i $ be an oriented
 meridian of $ \partial N_i $ such that $ \lk (M_i, L_i) = 1 $
  %and $M_i$ bounds a disc in $ N_i$.
and fix  also (the oriented Seifert framing of $D_i $)
 $ L'_i \subset \partial N_i $ with  $ \lk (L'_i, L_i) =0 $.
 (Here the linking numbers are considered in oriented 3--sphere $\mathfrak{S}^3$.)

\begin{defn} We call $ L'_i $ the \emph{topological longitude}
 of the torus $ \partial N_i $, while
$ \Lambda_i $ the \emph{sectional longitude} of $\partial N_i$
 associated with the transverse section  $H$.
\end{defn}
Clearly we have the following facts (where $ [\cdot ] $ denotes the corresponding homology class)
\begin{equation}
\begin{split}
(a) & \ \ H_1 (N_i,\Z) \cong \Z \langle [\Lambda_i]\rangle  \cong \Z \langle [L'_i] \rangle, \\
(b) & \ \  H_1 ( \partial N_i, \Z ) \cong \Z \langle [ \Lambda_i ] \rangle \oplus \Z \langle [ M_i ] \rangle \cong
 \Z \langle [ L'_i ] \rangle \oplus \Z \langle [ M_i ] \rangle.\\
\end{split}
\end{equation}
We want to express $  [ \Lambda_i ] $ in terms of $ [ M_i ] $ and $ [L'_i ] $.

\begin{lem}\label{lem:bases}
 Define
 $\lambda_i = - \Sigma_{k \not=i } D_k \cdot D_i -D_\sharp \cdot D_i$,
 where $ C_1 \cdot C_2 $ denotes the intersection multiplicity of
 $ (C_1, 0) $ and $ (C_2, 0)  $ at $0\in \C^2$.
 Then $ [ \Lambda_i ] = [L'_i ] + \lambda_i \cdot [M_i ] $
  in $ H_1 ( \partial N_i, \Z )$.
\end{lem}

\begin{proof} First note that
$ [\Lambda_i ] = [L'_i ] + \lk ( \Lambda_i, L_i ) \cdot [M_i ] $.
Write $F'_{\Phi^{*} (H)}:=F_{\Phi^{*} (H)}\setminus (\mbox{int}(\cup_i N_i\cup N(L_\sharp)))$.
Then
$0=  \lk ( \partial F'_{\Phi^{*} (H)} , L_i )=\sum_k \lk (\Lambda_k,L_i)+
\lk(\Lambda_\sharp,L_i)=\lk(\Lambda_i,L_i)+\sum_{k\not=i}D_k\cdot D_i+D_\sharp\cdot D_i= \lk(\Lambda_i,L_i)-\lambda_i$.
\end{proof}

\begin{thm}\label{th:mainth} For any $j=\{i,\sigma(i)\}\in J$, the gluing functions $ \phi_j $ from
Proposition~\ref{pr:gl}
is characterized up to homotopy by the following identities

{\bf Case 1:} \ $ i \neq \sigma (i) $. Identify the homology groups $  H_1 (S^1\times S^1 \times \{ 1 \}, \Z ) $,
$ H_1 (S^1\times S^ 1\times \{ -1 \}, \Z ) $ and $  H_1 (S^1\times S^1 , \Z ) $ via the natural homotopies
$$S^1\times S^1\times\{-1\}\stackrel{h}\sim S^1\times S^1\times [-1,1]\stackrel{h}\sim
S^1\times S^1\times\{1\}.$$ Then in this homology group one has
\[
 \phi_{j*} ([M_i]) = - \phi_{j*} ([M_{\sigma(i)}]) \ \  \mbox{ and }  \]
 \[ \phi_{j*} ([\Lambda_i])  = \phi_{j*} ([\Lambda_{\sigma(i)}] +  \mathfrak{v}_j
 \cdot 
 % \phi_{j*} 
 [M_{\sigma(i)}]). \]
%where $ \phi_{i*} ([M_i]) $ and $ \phi_{i*} ([\Lambda_i]) $ are in ,
%$ \phi_{i*} ([M_{\sigma(i)}]) $ and $ \phi_{i*} ([\Lambda_{\sigma(i)}]) $ are in
%, and both homology groups are naturally isomorphic with .

{\bf  Case 2:} \ $ i = \sigma (i) $.
 \[ \phi_{j*} ([M_i] ) = -m \ \ \mbox{ and } \] %\phi_{j*} ([M_i]) = - \phi_{j*} ([M_{\sigma(i)}])
 \[ \phi_{j*} ([\Lambda_i] ) = c+ \mathfrak{v}_j \cdot m, \]
 where $m $ and $c $ are the two generators of $ H_1( \partial Y, \Z)$, see \ref{ss:homY} (especially
 (\ref{eq:pY}) and (\ref{eq:cmu})).
\end{thm}

%In other words, $ \partial N_i $ and $ \partial N_{ \sigma(i) } $ are glued together to form $ \partial F $ such
%that the meridian goes to meridian with opposite orientation and the algebraic longitude goes to algebraic longitude
%in orientation preserving way. In the second case, $ \partial N_i $ and $ \partial Y $ are glued together to form
%$ \partial F $ such that the meridian goes to $  \mu $ with opposite orientation and the algebraic longitude goes to
%$  \lambda^2 $ in orientation preserving way.

\begin{proof}
Recall that $\tau\mapsto p(\tau)$ is the normalization of $\Sigma_j$.
At any point $p(\tau_0)$
of $\Sigma^*_j$ one can consider $\tau$ as a local complex coordinate, which
can be completed with two other  local complex coordinates $(x,y)$ (local coordinates in a
transverse slice of $\Sigma_j$ at $p(\tau_0)$)   such that $(\tau,x,y)$
form a local coordinate system of $(\C^3,p(\tau))$, and  locally
$f(x,y)=xy $.
These two local coordinates correspond to a splitting $f=f_1\cdot f_2 $ of $f$.
According to the discussion from \ref{ss:local}, the splitting of $f$ can be done globally
along the whole $\Sigma_j^*$, hence these coordinated $(x,y)$ (corresponding to the
components $f_1$ and $f_2$) can also be chosen globally along $\Sigma_j^*$ (with the permutation ambiguity
whenever $i=\sigma(i)$, a fact which will be handled below).

 Furthermore, in these coordinates, $T_1 H=\beta_1(\tau) x+\beta_2(\tau) y$.
Since $\beta_1(\tau)$ have no zeros and poles in the small representative of $\Sigma_j^*$, $(\tau, \beta_1x, \beta_1^{-1}y)$ are also local coordinates, and in these coordinates the equations
 transform into   $f(x,y)=xy$ and $T_1 H(x,y)=x+\mathfrak{b}_j(\tau) y$.

If we concentrate on the points of  $\Upsilon_j=\Sigma_j\cap S^5_\epsilon$, and its neighbourhood in $S^5_\epsilon$, then similarly as above, we have the real
coordinate $\tau \in \Upsilon_j$, and the two complex (transverse) local
coordinates $(x,y)$, with equations  $f(x,y)=xy$ and $T_1 H(x,y)=x+\mathfrak{b}_j(\tau) y$ as before.

{\bf Case 1.}
The above  local description  globalises as follows
(compare also with the first part of the proof of Proposition \ref{pr:gl}).
 The space $\partial F\cap N(\Upsilon_j)$
has a product decomposition
$\Upsilon_j\times F_j=S^1\times S^1 \times I$, where $\Upsilon_j=S^1$
is the parameter space of $\tau$,
and $F_j$ is the local Milnor fibre $F_j=\{xy=\delta\}\cap B^4_\varepsilon$,
$0 <\delta\ll \varepsilon\ll \epsilon $,
diffeomorphic to $S^1\times I$.
In other words, $\partial F\cap N(\Upsilon_j)$ is the space
$\{(\tau,x,y)\in S^1\times B^4_\varepsilon\,:\, xy=\delta\}$,
$0 <\delta\ll \varepsilon\ll \epsilon$ (here we will use the same $\tau$ notation for the
parameter of $S^1$). Note that for $\delta\ll \varepsilon$,
the boundary of $F_j$ is `very close' to the two circles
$\{|x|=\varepsilon, \, y=0\}$ and $\{x=0, \, |y|=\varepsilon\}$ of $\partial B^4_\varepsilon$.
Using isotopy in the neighbourhoods of these two circles,
$\partial F_j$ can be identified with these two circles
(similarly as we identify via Ehresmann's fibration theorem the boundary of the
Milnor fibre of an isolated singularity with the link),
and in order to
simplify the presentation, we will make this identification.
Hence, the boundary components of $\partial F\cap N(\Upsilon_j)$ can be identified (by isotopy in $ \partial N( \Upsilon_j ) $) with
 $\partial _i:=\{(\tau,x,y)\,:\,
\tau\in S^1,\, |x|=\varepsilon, \, y=0\}$ and $\partial_{\sigma(i)}:=
\{(\tau,x,y)\,:\, \tau\in S^1,\, x=0, \, |y|=\varepsilon\}$.
(The choice of indices $i$ and $\sigma(i)$ is arbitrary and symmetric.)
These two tori are identified homologically  since they are boundary
components of $S^1\times F_j$.
In $\partial_i$ we have the meridian
$\tilde{M}_i=\{(\tau,x,y)\,:\, \tau=1, \,|x|=\varepsilon , \, y=0\}$, while in
$\partial_{\sigma(i)}$ we have the meridian
$\tilde{M}_{\sigma(i)}=\{(\tau,x,y)\,:\, \tau=1, \,x=0, \, |y|=\varepsilon\}$
(both naturally oriented as the complex unit circle).
Since $\partial (\{\tau=1\}\times F_j)
=\tilde{M}_i\cup -\tilde{M}_{\sigma(i)}$, the first wished
identity follows.

Now, we would like to study the intersection curve of $ \partial_i $ and the Milnor fibre of $ H $ associated with a positive argument, that is
$ \tilde{\Lambda}_i :=  \partial_i \cap \{ H > 0 \} $. This curve is homotopic with $ \partial_i \cap \{ T_1 H=x+\mathfrak{b}_j(\tau) y > 0 \} $ and also with $ \partial_i \cap \{ x+\tau^{\mathfrak{v}_j}y > 0 \}$ in $ \partial_i $, cf. the Definition \ref{de:vert}. Similarly, $ \tilde{\Lambda}_{\sigma (i)} :=  \partial_{\sigma (i)} \cap \{ H > 0 \} $ is homotopic with  $ \partial_{\sigma (i)} \cap \{ x+\tau^{\mathfrak{v}_j}y > 0 \}$ in $ \partial_{\sigma (i)} $.

%\marginpar{Szerintem $ \mathfrak{b} $ rossz jeloles, alig ter el
 %$ \mathfrak{v} $ tol. Talan $ \mathfrak{B} $?}

%$H(\tau,x,y)=x+\tau^{\mathfrak{v}_j}y$.
%Then $H$
%cuts (via its Milnor fibre associated with a positive argument) in
%\marginpar{jeloljuk ezeket $\tilde{M}$ es $\tilde{\Lambda}$-kel??? OK!!!}
%$\partial _i$ the circle
Thus $\tilde{\Lambda}_i $ is homotopic with $
\{(\tau ,x,y)\,:\, \tau\in S^1, \, |x|=\varepsilon,\, y=0\}\cap \{ x+\tau^{\mathfrak{v}_j}y>0\}=
\{(\tau ,x,y)\,:\, \tau\in S^1,\, x=\varepsilon, \, y=0\}$ and
 $\tilde{\Lambda}_{\sigma(i)} $ with $
\{(\tau ,x,y)\,:\, \tau\in S^1,\, x=0, \, y=\varepsilon \tau^{-\mathfrak{v}_j}\}$.
%\marginpar{kicsreltuk $b_j$-t $v_j$-re}
Hence homologically $\tilde{\Lambda}_{\sigma(i)}+\mathfrak{v}_j \tilde{M}_{\sigma(i)}$
is represented by  the circle
$\{(\tau ,x,y)\,:\, \tau\in S^1,\, x=0, \, y=\varepsilon\}$, which is
homologous in $S^1\times F_j$ with $\tilde{\Lambda}_i$. This is the second identity.
Obviously, these identities can be transferred from  the boundary
$\partial _i\cup \partial _{\sigma(i)} $ of $S^1\times F_j$
into similar identities in $\partial N_i\cup \partial N_{\sigma(i)}$, via the diagram, where all the maps are orientation preserving
 diffeomorphisms (cf. the proof of \ref{pr:gl}):
\begin{equation}\labelpar{eq:dia}
 \begin{array}{cccc}
  \Phi: & \mathfrak{S}^3 \setminus \cup_i N_i & \to  & K \setminus \cup_j N( \Upsilon_j ) \\
          & \partial N_i \sqcup \partial N_{ \sigma(i) } & \to & \partial_i \sqcup \partial_{ \sigma(i) } \\
          & \Lambda_i    , \Lambda_{ \sigma(i) }                                & \to & \tilde{ \Lambda}_i,  \tilde{\Lambda}_{ \sigma(i) } \\
          & M_i    , M_{ \sigma(i) }                                & \to & \tilde{ M}_i,  \tilde{M}_{ \sigma(i)}. \\
 \end{array}
\end{equation}

\vspace{1mm}

\noindent {\bf Case 2.}  We use similar notations and conventions as in Case 1.
%\marginpar{VERIFY}
Let us parametrize $\Upsilon_j$ as $\tau=e^{2is}$,
$s\in[0,\pi]$. Then $ \partial F\cap N(\Upsilon_j)$ is
$([0,\pi]\times F_j)/\sim $,
%\marginpar{le kellene irni ezt az azonositast pontosabban}
where by $\sim $ we identify $(0,x,y)\sim (\pi,y,x)$ for all $(x,y)\in F_j$.
Let us parametrize $ F_j $ as $ x= \sqrt{ \delta} r e^{it} $ and $ y = \sqrt{ \delta} r^{-1} e^{-it} $, where $ t \in [- \pi , \pi] $ and
$ r \in [r_0 , r_1] $ such that $ r_0 r_1=1$ and $\delta (r_0^2 + r_1^2) = \varepsilon^2 $. Denote $ z = \log_{r_1} r $.
% $ r \in [\sqrt{\delta}/ \epsilon , \epsilon / \sqrt{\delta}] $. Denote $ z = \log_{\epsilon / \sqrt{\delta}} r $.
Then we can parametrize $ \partial F\cap N(\Upsilon_j)$ by $ (s, t, z) $, thus
$  \partial F\cap N(\Upsilon_j) $ is just $([0,\pi]\times [- \pi, \pi ] \times [-1, 1])/\sim $,
where by $\sim $ we identify
$  (0, t, z) \sim ( \pi, -t, -z) \mbox{ and } (s, -\pi, z) \sim (s, \pi, z) $.
 We regard this as parametrization of
 $  \partial F\cap N(\Upsilon_j) $ by $ Y $, cf.  \ref{ss:homY}.

%\marginpar{$z$ helyett vmi mas? nem zavaro itt a $z$?
%am ha csereljuk, akkor az $Y$ leirasaban is cserelni kell}

Set
\[ \tilde{M}_{j1} = \{ ( \tau = 1,\ x= \sqrt{ \delta} r_1 e^{it},\ y = \sqrt{ \delta} r_0 e^{-it}  ) \} \mbox{ and }\]
\[ \tilde{M}_{j2}  = \{ ( \tau =1,\ x= \sqrt{ \delta} r_0 e^{-it},\ y = \sqrt{ \delta} r_1 e^{it}  ) \}. \]
The  are the two oriented meridians of $ \partial F\cap N(\Upsilon_j)$,
parametrized by $ t \in [ - \pi, \pi ] $. In terms of  $ (s, t, z) $ they are
 \[ \tilde{M}_{j1}=\{(s=0,\, t , \, z=1)\} = \{(s=\pi,-t, \, z=-1)\} \subset \partial Y \mbox{ and } \]
\[ \tilde{M}_{j2}=\{(s=0,\, -t , \, z=-1)\} = \{(s=\pi,t, \, z=1)\} \subset \partial Y \mbox{,} \]
thus with the notations of \ref{ss:homY}, $\tilde{M}_{j1} = -m'$ and $ \tilde{M}_{j2}= -m $.
%\marginpar{Irjunk at vmit? Pl. a kockan $m$ es $m'$ ha masik helyre
%lenne barajzolva mar jo lenne, nem minuszba lennenek.}

Similarly as in Case 1, by an isotopy in $ \partial N( \Upsilon_j ) $ the boundary $ \partial_j $ of $ \partial F\cap N(\Upsilon_j)$ can be identified with $ \partial_{j1} \cup \partial_{j2} $, where
\[\partial _{j1}:=\{(s,x,y)\,:\,
s \in [0, \pi],\, |x|=\varepsilon, \, y=0\} \mbox{ and }\]
\[ \partial_{j2}:=
\{(s,x,y)\,:\, s \in [0, \pi],\, x=0, \, |y|=\varepsilon\} \mbox{.} \]
The two parts of $ \partial_j $ are glued together along the image of the oriented meridians
\[ \tilde{M}_{j1}=\{(s=0,\, x= \varepsilon e^{it}, \, y=0)\} = \{(s=\pi,\, x=0, \, y=\varepsilon e^{it})\}  \mbox{ and} \]
\[ \tilde{M}_{j2} = \{(s=0,\, x=0,\, y= \varepsilon e^{it})\} = \{(s= \pi,\, x= \varepsilon e^{it}, \, y=0)\} \mbox{.} \]
% The isotopy of $ \partial_j \subset S^5 $ moves the point
% ( \tau,\ x= \sqrt{ \delta} r_1 e^{it},\ y = \sqrt{ \delta} r_0 e^{-it}  )$ to
% $ ( \tau , \ x= \varepsilon e^{it}, \ y=0) $ and
% $( \tau,\ x= \sqrt{ \delta} r_0 e^{it},\ y = \sqrt{ \delta} r_1 e^{-it}  )$ to
%$ ( \tau , \ x= 0, \ y= \varepsilon e^{-it}) $.
% Taking the pre-image of the meridians, one can express them with the $ (s, t, z) $-parametrization of $ Y $, and get the parametrizations by $ t \in [0, \pi] $ as
% $ \tilde{M}_{j1}=\{(s=0,\, t , \, z=1)\} = \{(s=\pi,-t, \, z=-1)\} \subset \partial Y$ and
% $ \tilde{M}_{j2}=\{(s=0,\, t , \, z=-1)\} = \{(s=\pi,-t, \, z=1)\} \subset \partial Y$.
(See also the description in \ref{ss:homY}.)
Since the oriented boundary $ \partial ( \{ \tau = 1 \} \times F_j ) $ is $ \tilde{M}_{j1} \sqcup \tilde{M}_{j2} $, $ \tilde{M}_{j1}$ is homologous with $ - \tilde{M}_{j2} $ in $ F_j \subset Y$. On the other hand $ \tilde{M}_{j1} $ is homologous with $ \tilde{M}_{j2} $ in $ \partial Y \subset Y $, thus $ [\tilde{M}_{j1}]= [\tilde{M}_{j2}]$ is an order--$2$ element in $H_1 ( Y, \Z ) $.

Consider the closed curve $C$ obtained as union of
$C_{1}$ and $C_{2}$, where
\[C_{1}=\{(s,x=\varepsilon, y=0)\,:\, s\in [0,\pi]\} = \{(s,t=0, z=1)\,:\, s\in [0,\pi]\}  \mbox{, and} \]
  \[ C_{2}=\{(s,x=0, y= \varepsilon)\,:\, s\in [0,\pi]\} = \{(s, t=0, z=-1) \,:\, s\in [0,\pi]\}\mbox{.} \]
  Note that
$C_{1}$ connects the points $A_1=(s=0, x=\varepsilon,y=0)=(s=0, t=0, z=1)$ with
$B_1=(s= \pi, x=\varepsilon, y=0)= (s=\pi, t=0, z=1)$,
while
$C_{2}$ connects the points $A_2=(s=0,x=0,y=\varepsilon)= (s=0, t=0, z=-1)$ with
$B_2=(s= \pi, x= 0,y=\varepsilon)= (s= \pi, t= 0, z=-1)$. Since $A_1\sim B_2$ and $A_2\sim B_1$,
they form a closed curve. Note that $ [C]=[c] $ in $ H_1 ( \partial Y, \Z ) $, see \ref{ss:homY}.

Similarly as in Case 1, the function $H$ can be replaced by
$x+\tau^{\mathfrak{v}_j}y$, hence its level set associated with a positive value determines the curve
$ \tilde{\Lambda}_j = \{ x+\tau^{\mathfrak{v}_j}y >0 \} \cap \partial_j $. This consists of two parts, $\tilde{\Lambda}_{j1}$ and $\tilde{\Lambda}_{j2}$, where
\[ \tilde{\Lambda}_{j1}=
\{(s,\, x=\varepsilon,\, y=0)\,:\, s\in [0,\pi]\}
= \{(s,t=0, z=1)\,:\, s\in [0,\pi]\} \subset \partial_{j1},  \]
while
$ \tilde{\Lambda}_{j2}=
\{(s,\, x=0,\, y=\varepsilon e^{-2is\mathfrak{v}_j})\, :\, s\in[0,\pi]\} $ equals
 \[ \{(s,t=2 \mathfrak{v}_j s \ (\mbox{mod } [- \pi, \pi ]), z=-1)\,:\, s\in [0,\pi]\} \subset \partial_{j2}. \]
$ \tilde{\Lambda}_{jn}$ has the same end--points as $C_{n}$, hence
%\marginpar{nekem az elojelek csak nem vilagosak, nekem meg a 2-es sem!
%ezt itt nem ertem}
$\tilde{\Lambda}_{j1}$ and $\tilde{\Lambda}_{j2}$ form together a closed curve, as we expect.
Furthermore, $\tilde{\Lambda}_j+\mathfrak{v}_j \tilde{M}_{j2}$ is homologous in $\partial Y$ with $C$.

The source and the target are connected by the restriction of $ \Phi $, which gives the orientation preserving  diffeomorphisms:
\begin{equation}\labelpar{eq:dia2}
 \begin{array}{cccc}
  \Phi: & \mathfrak{S}^3 \setminus \cup_i N_i & \to  & K \setminus \cup_j N( \Upsilon_j ) \\
          & \partial N_i  & \to & \partial_j  \\
          & \Lambda_i      & \to & \tilde{ \Lambda}_j \\
          & M_i    & \to & \tilde{ M}_{j1} \mbox{ homologous with } \tilde{ M}_{j2} \\
 \end{array}
\end{equation}

Since $C$ identifies with $c$ and $ \tilde{M}_{j1}$ and $ \tilde{M}_{j2}$ with $-m$,
$\tilde{\Lambda}_i=c+\mathfrak{v}_j\cdot m$ follows.
\end{proof}

\subsection{The $H$--independent description of the gluing. The `vertical index'}\label{ss:topverindex}\

Recall that the sectional longitudes $ \Lambda_i $  and the corresponding
$H$--vertical indexes  $ \mathfrak{v}_j$ depend on the choice of $H$.
 The goal of this paragraph is to replace $(M_i, \Lambda_i)$
 by the $H$--independent $(M_i, L_i')$ and  $ \mathfrak{v}_j$
 by an $H$--independent number.

\begin{defn}
 For any $ j= \{ i, \sigma (i) \} $ define $\mathfrak{vi}_j$ by

\begin{equation}\labelpar{eq:alphauj}
  \mathfrak{vi}_j := \left\{ \begin{array}{ccc}
                       \lambda_i  + \lambda_{ \sigma (i)} +  \mathfrak{v}_j  & \mbox{ if } & i \neq \sigma (i) \mbox{,} \\
                     \lambda_i  + \mathfrak{v}_j  & \mbox{ if } & i = \sigma (i) \mbox{.} \\
                     \end{array} \right.
                 \end{equation}
We call $ \mathfrak{vi}_j $ the \emph{vertical index} of $ \Sigma_j $.
\end{defn}

The next statement follows
from \ref{lem:bases} and \ref{th:mainth}  (see also  \ref{ss:gluinguj}).

\begin{cor}
For $ i \neq \sigma (i) $
\[
 \phi_{j*} ([M_i]) = - \phi_{j*} ([M_{\sigma(i)}]) \ \  \mbox{ and }  \]
\[ \phi_{j*} ([L_i']) = \phi_{j*} ([L_{ \sigma_i}']) + \mathfrak{vi}_j \cdot [ M_{ \sigma(i)}] \mbox{,}
\]
and for $ i = \sigma (i) $
\[ \phi_{j*} ([M_i] ) = -m \ \ \mbox{ and } \]
\[ \phi_{j*} ([L_i']) = c + \mathfrak{vi}_j \cdot m
\]
hold in the sense described in \ref{th:mainth}.
\end{cor}

\begin{cor} The integer
$ \mathfrak{vi}_j $ does not depend on the choice of $H$, thus it is an invariant of $f$ and $ \Sigma_j$.
\end{cor}

\subsection{The plumbing graph of $ \partial F $} We construct a plumbing graph for
$ \partial F$ by modifying a good embedded resolution graph of
$ (D,0)\subset (\C^2,0) $. (For  notations see Subsection \ref{ss:embres}.)
The gluing of the plumbing construction uses  different set of longitudes
(cf. \cite{neumann1}). First we define them and then we rewrite the
above established identities regarding $\phi_{j*}$ in this language.

We choose small tubular neighbourhoods $ N(E_v)$ of $E_v\subset\widetilde{\C^2}$ such that
\[  \pi^{-1} (\mathfrak{S}^3)\simeq
\partial \Big( \bigcup_{ v \in V} N(E_v) \Big) \ \ (\mbox{diffeomorphic
with $\mathfrak{S}^3$ via $\pi$)}\]
after smoothing the corners of $\bigcup_{ v \in V} N(E_v) $.
The tubular neighbourhood $ N(E_v)$ of $ E_v \subset \widetilde{ \C^2 } $ is a $D^2$ (real 2--disk) bundle over $ E_v$ with Euler number $ e_v $.
We can choose this bundle structure in such a way
 that $ \tilde{D}_i$ is one of the  fibres of
$ N(E_{v(i)}) $ for each $ i= 1, \dots , l$.

We choose another generic fibre $\mathcal{F}_i \simeq D^2 $ of the bundle $ N(E_{v(i)}) $ near $ \tilde{D}_i $, and we set
$ L^{\pi}_i = \pi (\partial \mathcal{F}_i) \subset \mathfrak{S}^3$.
By the choice of the bundle structure of $ N(E_{v(i)}) $ and by the choice of the fibre $ \mathcal{F}_i$ we can assume that $ L^{\pi}_i \subset \partial N_i $.
\begin{defn} $L_i^\pi \subset \partial N_i\subset \mathfrak{S}^3 $ is called the
 \emph{resolution longitude}  of
$ L_i\subset \mathfrak{S}^3 $ associated with  the resolution $\pi$.
\end{defn}

We fix a resolution longitude $ L^{\pi}_i $ for each $ L_i $.
Clearly,
\begin{equation}
  H_1 (N_i, \Z )  = \Z \langle [ L^{\pi}_i ] \rangle \ \ \mbox{and} \ \
   H_1 ( \partial N_i, \Z ) =
 \Z \langle [ L^{\pi}_i ] \rangle \oplus \Z \langle [ M_i ] \rangle.
% where $ [ \gamma ] $ denotes the homology class of a closed curve $ \gamma $.
\end{equation}
The following facts are well--known (cf. \cite{EN}).
\begin{prop}
(a) $ \lk (L^{ \pi}_i , L_i ) = m_i(v(i)) $.

 (b) $  [ L^{ \pi}_i ] = [L'_i] + m_i(v(i)) \cdot [M_i]$ holds in $ H_1 ( \partial N_i, \Z )$.
\end{prop}

%\begin{proof}
%  Part {\em (b)} follows from part {\em  (a)}. For statement (a) we consider the %Milnor fibre $ F_{d_i}= d_i^{-1} ( \delta ) \cap B^4 $ and its boundary
%$ \partial F_{d_i} = d_i^{-1} ( \delta ) \cap S^3 $. Then
 %\[
 % \lk (L^{ \pi}_i , L_i ) =  \lk (L^{ \pi}_i , \partial F_{d_i} ) \mbox{.}
% \]
%$ \lk (L^{ \pi}_i , \partial F_{d_i} ) ) $ is equal to the algebraic number
%of the intersection points of $ {\pi}(\mathcal{F}_i) $ and $ F_{d_i} $.
%That is the same to take the algebraic number of intersection points of
%$ \mathcal{F}_i $ and ${\pi}^{-1} ( F_{d_i} ) $ in $ \widetilde{ \C^2 }$,
%because $ 0 \notin F_{d_i} $. $ \mathcal{F}_i $ and ${\pi}^{-1} ( F_{d_i} ) $
%intersect each other transversally in $ m_i (v(i)) $ points, each of them
%has positive sign.
%
%\end{proof}

\begin{cor}\label{cor:REL}
 $  [ L^{ \pi}_i ] = [\Lambda_i] + ( m_i(v(i)) - \lambda_i ) \cdot [M_i]$ holds in $ H_1 ( \partial N_i, \Z )$.
\end{cor}

\begin{defn}\label{de:alpha}
For any $ j=\{i,\sigma(i)\}$ define $ \alpha_j $ by:
 \begin{equation}\labelpar{eq:alpha}
  \alpha_j = \left\{ \begin{array}{ccc}
                      - m_i(v(i)) + \lambda_i -m_{ \sigma (i)} ( v( \sigma (i))) + \lambda_{ \sigma (i)} +  \mathfrak{v}_j  & \mbox{ if } & i \neq \sigma (i), \\
                     \lambda_i -  m_i(v(i)) + \mathfrak{v}_j & \mbox{ if } & i = \sigma (i). \\
                     \end{array} \right.
 \end{equation}
\end{defn}
% Note that $ \alpha_i = \alpha_{ \sigma (i)} $.
 Then Theorem \ref{th:mainth} and Corollary \ref{cor:REL} give the following.

\begin{cor}\label{cor:identities} {\bf Case 1:} \
 For $ i \neq \sigma (i) $ in $ H_1 (S^1\times S^1, \Z )$
 the following identities  hold:
  \[ \phi_{j*} ([M_i]) = - \phi_{j*} ([M_{\sigma(i)}])\ \ \mbox{ and } \ \
  \phi_{j*} ([L^{ \pi}_i ]) =
 \phi_{j * } ( [L^{ \pi}_{ \sigma (i)}] + \alpha_j  \cdot [M_{ \sigma (i) }]).
 \]
 {\bf Case 2:} \
  For $ i= \sigma (i) $ in $ H_1 (\partial Y, \Z) $ the following
  identities hold:
 \[ \phi_{j*} ([M_i] ) = -m  \ \ \mbox{and} \ \
 \phi_{j *} ( [L^{ \pi }_i] ) = c + \alpha_j \cdot m.
 \]
\end{cor}

\subsection{The construction of the plumbing graph.} From the
embedded resolution graph
$ \Gamma $ of $ (D, 0) \subset ( \C^2, 0) $ associated with
the resolution ${\pi}$ we construct a plumbing graph $\widehat{\Gamma }$.

Recall that $\Gamma$ has $l$ arrowhead  vertices representing the strict transforms
of $\{D_i\}_{i=1}^l$, and the $i$-th arrowhead is supported (via a unique edge) by
the vertex $v(i)\in V$. We obtain the plumbing graph $\widehat{\Gamma}$ from $\Gamma$ as follows.

(1) Fix $j=\{i,\sigma(i)\}$ and assume that $i\not=\sigma(i)$. Then we identify the two
arrowheads and we replace it by a single new vertex of $\widehat{\Gamma}$.
 We define the Euler number of the new  vertex by
$\alpha_j$. Both two edges (which in $\Gamma$ supported  the arrowheads) will survive as edges of this new vertex
(connecting it with $v(i)$ and $v(\sigma(i))$ respectively),
however one of them will have a negative sign, the other one a
positive sign.  By plumbing calculus, cf. \cite[Prop. 2.1. R0(a)]{neumann1},
the choice of the edge which has negative sign -- denoted by  $\ominus$ --  is irrelevant. (The edges without any decorations, by convention, are edges
 with positive sign.)

(2) Fix $j=\{i,\sigma(i)\}$ such that $i=\sigma(i)$. Then the arrowhead associated with
$i=\sigma(i)$ will be replaced by a new vertex and it will be decorated by Euler number $\alpha_j$.
Furthermore, to this new vertex
we attach the plumbing graph of $Y$ (cf. \ref{ss:plY}) as indicated below.
The  edge connecting the
new vertex and the graph of $Y$  will have a negative sign $\ominus$.

(3) We do all these modification for all $j\in J$, otherwise we keep the shape and decorations of $\Gamma$.

\vspace{2mm}

More precisely,
if the schematic picture of $\Gamma$ is the following,

\begin{picture}(300,100)(-20,0)
\put(0,10){\framebox(130,80)}
\put(120,20){\vector(1,0){40}}
\put(120,28){\makebox(0,0){$e_{v(i')} $}}
\put(120,20){\circle*{4}}
\put(160,40){\makebox(0,0){$ \vdots $}}
\put(190,20){\makebox(0,0){$(\tilde{D}_{i'})$}}
\put(40,50){\makebox(0,0){$\Gamma$}}

\put(120,50){\vector(1,0){40}}
\put(115,58){\makebox(0,0){$e_{v(\sigma(i))} $}}
\put(120,50){\circle*{4}}
\put(190,50){\makebox(0,0){$(\tilde{D}_{\sigma(i)})$}}
\put(120,70){\vector(1,0){40}}
\put(117,78){\makebox(0,0){$e_{v(i)} $}}
\put(120,70){\circle*{4}}
\put(190,70){\makebox(0,0){$(\tilde{D}_i)$}}

\put(250,20){\makebox(0,0)[l]{$(j'=\{i',\sigma(i')\}, \ i'=\sigma(i'))$}}

\put(250,60){\makebox(0,0)[l]{$(j=\{i,\sigma(i)\}, \ i\not=\sigma(i))$}}

\end{picture}

 \noindent then the schematic picture of $\widehat{\Gamma}$ is

 \begin{picture}(300,105)(-20,-5)
\put(0,10){\framebox(130,80)}
\put(120,20){\line(1,0){80}}
\put(120,28){\makebox(0,0){$e_{v(i')} $}}
\put(120,20){\circle*{4}}

\put(160,20){\circle*{4}}
\put(200,20){\circle*{4}}
\put(230,35){\circle*{4}}
\put(230,5){\circle*{4}}
\put(200,20){\line(2,1){30}}
%\put(60,60){\line(1,0){40}}
\put(200,20){\line(2,-1){30}}
\put(200,10){\makebox(0,0){$-1$}}
\put(242,5){\makebox(0,0){$-2$}}
\put(242,35){\makebox(0,0){$-2$}}
\put(160,10){\makebox(0,0){$ \alpha_{j'}$}}
\put(181,25){\makebox(0,0){$\ominus$}}
\put(160,40){\makebox(0,0){$ \vdots $}}

\put(120,50){\line(4,1){40}}
\put(115,58){\makebox(0,0){$e_{v(\sigma(i))} $}}
\put(120,50){\circle*{4}}\put(160,60){\circle*{4}}
\put(160,70){\makebox(0,0){$\alpha_j$}}
\put(120,70){\line(4,-1){40}}
\put(117,78){\makebox(0,0){$e_{v(i)} $}}
\put(120,70){\circle*{4}}

\put(-20,50){\makebox(0,0){$\widehat{\Gamma}:$}}

\put(145,50){\makebox(0,0){$\ominus$}}

\end{picture}

In fact, by plumbing calculus (cf. Subsection~\ref{ss:plumbingcalc} or \cite[Prop. 2.1. R0(a)]{neumann1}),  whenever $i=\sigma(i)$
the edge sign from the newly created `branch' (subtree) can be omitted, however at this point we put it since this is
the graph provided by the proof (which reflects properly the corresponding base changes).

Note also that usually the graph $\widehat{\Gamma}$ (that is, the associated intersection form) is not negative definite (or,  it is not
even equivalent via plumbing calculus by a negative definite graph).

Above (when $i\not=\sigma(i)$) the coincidence $v(i)=v(\sigma(i))$ might happen, in fact, in all the cases we
analysed this coincidence (on minimal graph) does happen.

\begin{thm}
 The plumbing $3$-manifold associated with the plumbing graph
 $ \widehat{\Gamma} $ is orientation preserving diffeomorphic with $ \partial F $.
\end{thm}
\begin{proof}
This follows from Corollary \ref{cor:identities} and the relationships between this base--change
and the plumbing construction as it is described  in \cite{neumann1}, pages 318--319.
In the case $i=\sigma(i)$ the plumbing graph of $Y$ from \ref{ss:plY} should be also used.
Note that in both cases Corollary \ref{cor:identities} provides a base change matrix
from the left hand side of (\ref{eq:matrix}), this decomposes as a product as in the right hand side.
\begin{equation}\label{eq:matrix}
\begin{pmatrix}-1 & \alpha_j\\ 0& 1\end{pmatrix}=
\begin{pmatrix}0 & 1\\ 1& 0\end{pmatrix}
\begin{pmatrix}-1 & 0\\ -\alpha_j& 1\end{pmatrix}
\begin{pmatrix} 0 & -1 \\ -1& 0\end{pmatrix}.
\end{equation}
This, according to \cite[pages 318--320]{neumann1} is interpreted as a `gluing' by a string (with length one),
this is the new vertex (for each $j$) given by the construction of $\widehat{\Gamma}$.
\end{proof}

%$ \widetilde{ B^4} := p^{-1} (B^4) \subset \widetilde{ \C^2} $ is the union of tubular neighborhoods of the exceptional divisors $ E_v \subset \widetilde{ \C^2 } $.
%$ \partial \widetilde{ B^4} \simeq S^3 $ can be optained by plumbing construction along the graph $ G$.

% Each vertex $ v \in V $ marks an $ S^1 $-bundle over an exceptional divisor $\C \mathbb{P}^1 \simeq S^2 $ with Euler number $ e(v) \in \Z $. The arrows $a_1, \dots , a_l \in A $ correspond to the components $ L_i $ of the link $ \tilde{D}$. Each vertex $ v \in V $ is endowed with the multiplicities $ m_1(v), \dots, m_l(v) \in \N $ corresponding to the components $ \tilde{ \Sigma }_i $ od $ \tilde{ \Sigma } $. The edge $ e \in E $ connecting the vertexes $ v $ and $ v' $ marks the gluing of the total spaces of the bundles corresponding to $ v $ and $ v' $ along the boundaries of cutted solid tori. Each edge is endowed with a sign $ \oplus $ or $ \ominus $ which describes the type of the gluing.

% For an arrow $ a_i \in V $ let $ v(a_i) \in V $ denote the vertex where the arrow starts. $ L_i $ is an $ S^1$-fibre of the bundle corresponding to $ v(a_i) $.

% if $ v=v(i) $, and
% \[
% e_v \cdot m_i(v(i)) + \Sigma_{w \in V} (E_{v(i)} \cdot E_w) \cdot m_i(w) = 0 \mbox{.}\]

\section{The vertical index for $\Sigma^{1,0}$ type germs}\label{s:vertical}

\subsection{$\Sigma^{1,0}$ type germs} Assume that $\Phi(s, t) = ( s, t^2, t d(s,t) ) $, where $d(s,t) = g(s, t^2) $ for  some germ $g$, such that $d(s,t)$  is not divisible by $t$. In this case the equation of the image is $f= yg^2(x,y) - z^2 =0 $. These germs (more precisely, their $ \mathscr{A}$-equivalence classes) are labelled by the Boardman symbol $ \Sigma^{1,0}$, and they are exactly the finitely determined corank--$1$ map germs with no triple points in their stabilization, cf. Example~\ref{ex:Sig10},
see \cite{Mond0,nunodoodle} for details.

 The set of double points is $D = \{ (s, t) \in \C^2 \ | \  d(s,t)= 0 \}$. It is
 equipped with the involution $ \iota: D \to D $, $ \iota (s, t) = (s, -t) $.
 The set of double values is
 \[ \Sigma (f) = \{ (x, y, z) \in \C^3 \ | \  z=0 \mbox{ and } g(x,y) =0 \} \mbox{.}\]

There are several options for the choice of the  transverse section. E.g.,
Proposition \ref{pr:H} works with $a_1=a_2=a_3=1$. Furthermore,  one can also take
$H_2 (x, y, z) = z $ or $  H_3 (x, y, z) = g(x, y)$.

If we take $H(x,y,z)=z$ then the $H$--vertical indexes
$ \mathfrak{v}_j$ are zero. This fact  follows directly from the product decomposition of $f$.
(However, for different other choices it can be nonzero as well.)

\begin{prop} For $ \Sigma^{1,0} $--type germs
$ \Sigma_j \mathfrak{vi}_j = -  \Sigma_{i \neq k} D_i \cdot D_k - C( \Phi ) $, where $ C(\Phi) $ is the number of crosscaps of a stabilization of $ \Phi$, cf. Subsection~\ref{ss:CTN} or \cite{NP, Mond2}.
\end{prop}

\begin{proof}
 With the choice of $ H=z$, $ \Phi^*(H)= td(s,t) $ and $ \mathfrak{v}_j=0$. If $i \neq \sigma(i)$, then
\[ \mathfrak{vi}_j= -\Sigma_{k \not \in \{i,\sigma(i)\}}
( D_i + D_{\sigma(i)}) \cdot D_k - (D_i+D_{\sigma(i)}) \cdot \{t=0\}, \]
while for  $i= \sigma(i) $  (see \ref{ss:topverindex}) one has
\[ \mathfrak{vi}_j= -\Sigma_{k \neq i} D_i \cdot D_k - D_i \cdot \{t=0\}.
 \]
Taking the sum for all $j$ we get the identity, once we verify
 \[ \Sigma_{i=1}^l D_i \cdot \{t=0 \} = D \cdot \{t=0\} =
  \dim_{ \C} \frac{ \mathcal{O}_{ ( \C^2, 0)}}{ ( t, d(s,t) )} \mbox{.}
 \]
 Here, the last identity follows from the algebraic definition of $C( \Phi)$, as the codimension of the Jacobian ideal of $ \Phi $, see Subsection~\ref{ss:CTN} and Example~\ref{ex:cor1}
% (cf. \cite{Mond2, NP, nunodouble})
 . In our case this ideal is genereted by $ t$ and $d(s, t) $.
 \end{proof}

\section{Examples}\label{s:milnex}

In the next paragraphs we provide some concrete examples. The families are taken from
D. Mond's list  of simple germs  \cite[Table 1]{Mond2}. In the sequel
we provide the resolution graph of $D$ and the plumbing graph of $\partial F$. 
The  computations are left to the reader. Example~\ref{ex:milnex} serves as a model to determine the resolution graph of $D$. By Section~\ref{s:vertical} the equation of $D$ and all the gluing invariants can be determined for $ \Sigma^{1, 0} $ type germs.

The last two examples are more special than the previous ones: they are not of type
$\Sigma^{1,0}$. In the first family ($ H_k $ from the list \cite[Table 1]{Mond2})
the calculation  has some nontrivial steps, hence
 we provide more details. The last example \ref{ss:2cor} is a corank
 $2$ map germ from \cite{marar}, cf. Example~\ref{ex:cor2}.

\subsection{Whitney umbrella, or cross--cap} $\Phi (s, t) = (s, t^2, ts) $ and
$ f(x, y, z) = x^2y - z^2 =0$.
 The graph of $D$ is on the left, while
 our algorithm provides the  graph from the right for
$ \partial F$

\begin{picture}(300,70)(-160,25)
 \put(-100,60){\circle*{4}}
                         \put(-100,60){\vector(1,0){30}}
                         \put(-105,70){\makebox(0,0){$-1$}}

                        \put(100,60){\circle*{4}}
                           \put(20,60){\circle*{4}}
                        \put(60,60){\circle*{4}}
                         \put(130,75){\circle*{4}}
                         \put(130,45){\circle*{4}}
                         \put(100,60){\line(2,1){30}}
                         \put(20,60){\line(1,0){40}}
                         \put(60,60){\line(1,0){40}}
                         \put(100,60){\line(2,-1){30}}
                         \put(95,70){\makebox(0,0){$-1$}}
                          \put(15,70){\makebox(0,0){$-1$}}
                          \put(143,75){\makebox(0,0){$-2$}}
                          \put(58,70){\makebox(0,0){$ -2 $}}
                           \put(143,45){\makebox(0,0){$-2$}}
                           \end{picture}

\noindent which after plumbing calculus (blowing down twice followed by a $0$-chain absorption, see Subsection~\ref{ss:plumbingcalc}) transforms into
                           \begin{picture}(20,10)(50,58)
                        \put(60,60){\circle*{4}}
                          \put(58,70){\makebox(0,0){$ -4 $}}
                           \end{picture}
(Compare also with Example 10.4.2 from \cite{NSz}.)

\subsection{\bf{$S_1$}}
%\marginpar{mathbf}
Set
$ \Phi(s, t) = ( s, t^2, t^3 + s^2t ) $, hence
$f(x, y, z) = y(y+x^2)^2 - z^2$.
The graph of $D$ is the first diagram, while the other two equivalent
graphs represent $\partial F$

\begin{picture}(300,60)(-50,40)
 \put(0,60){\circle*{4}}
                         \put(0,60){\vector(2,1){30}}
                         \put(0,60){\vector(2,-1){30}}
                         \put(-5,70){\makebox(0,0){$-1$}}
                          \put(-5,50){\makebox(0,0){$v$}}

                         \put(100,60){\circle*{4}}
                       \put(140,60){\circle*{4}}
                        \put(120,60){\circle{40}}
                         \put(92,70){\makebox(0,0){$-1$}}
                             \put(148,70){\makebox(0,0){$-6$}}
                              \put(120,88){\makebox(0,0){$ \ominus $}}
                            %  \put(120,32){\makebox(0,0){$ \oplus $}}                 %\end{picture}
 \put(170,60){\makebox(0,0){$ \mbox{or} $}}
%\begin{picture}(300,100)(-50,10)
                       \put(240,60){\circle*{4}}
                          \put(220,60){\circle{40}}
                             \put(248,70){\makebox(0,0){$-4$}}
                              \put(195,60){\makebox(0,0){$ \ominus $}}
                         \end{picture}

\subsection{The family $S_{k-1} $, $k\geq 2$}
%For $k=1$ by $S_0$ one recovers the  gross--cap.
 One has $ \Phi(s, t) = ( s, t^2, t^3 + s^k t )$ and
$ f(x, y, z) = y(y+x^k)^2 - z^2 =0 $. Cf. Examples \ref{ex:cor1}, \ref{ex:1}.

 \vspace{2mm}

 \emph{Case 1}: $k=2n$.

 \vspace{2mm}

\begin{picture}(300,40)(-150,45)
                         \put(100,60){\circle*{4}}
                          \put(60,60){\circle*{4}}
                           \put(0,60){\circle*{4}}
                           \put(60,60){\line(1,0){40}}
                             \put(40,60){\line(1,0){20}}
                             \put(0,60){\line(1,0){20}}
                              \put(-40,60){\circle*{4}}
                              \put(-40,60){\line(1,0){40}}
                         \put(100,60){\vector(2,1){30}}
                         \put(100,60){\vector(2,-1){30}}
                         \put(95,70){\makebox(0,0){$-1$}}
                          \put(57,70){\makebox(0,0){$-2$}}
                           \put(-3,70){\makebox(0,0){$-2$}}
                            \put(-43,70){\makebox(0,0){$-2$}}
                         % \put(93,50){\makebox(0,0){$w_n$}}
                         % \put(-43,50){\makebox(0,0){$w_1$}}
                         % \put(-3,50){\makebox(0,0){$w_2$}}
                          \put(29,57){\makebox(0,0){$\dots$}}
                         % \put(60,50){\makebox(0,0){$w_{n-1}$}}
                           \put(-100,60){\makebox(0,0){$D:$}}
                         \end{picture}

\begin{picture}(300,60)(-150,40)
                         \put(100,60){\circle*{4}}
                          \put(60,60){\circle*{4}}
                           \put(0,60){\circle*{4}}
                           \put(140,60){\circle*{4}}
                           \put(120,60){\circle{40}}
                           \put(60,60){\line(1,0){40}}
                             \put(40,60){\line(1,0){20}}
                             \put(0,60){\line(1,0){20}}
                              \put(-40,60){\circle*{4}}
                              \put(-40,60){\line(1,0){40}}
                         \put(95,70){\makebox(0,0){$-1$}}
                          \put(57,70){\makebox(0,0){$-2$}}
                           \put(-3,70){\makebox(0,0){$-2$}}
                            \put(-43,70){\makebox(0,0){$-2$}}
                             \put(146,70){\makebox(0,0){$-3k$}}
                               \put(120,88){\makebox(0,0){$\ominus$}}
                             \put(29,57){\makebox(0,0){$\dots$}}
\put(-100,60){\makebox(0,0){$\partial F$:}}
                         \end{picture}

\noindent Here the number of $ (-2)$-vertices is $n-1$.

 \vspace{2mm}

 \emph{Case 2}: $k=2n+1$.

 \vspace{2mm}

\begin{picture}(300,50)(-100,40)
 \put(140,60){\circle*{4}}
  \put(-100,60){\makebox(0,0){$D:$}}
                         \put(100,60){\circle*{4}}
                          \put(60,60){\circle*{4}}
                           \put(0,60){\circle*{4}}
                            \put(140,30){\circle*{4}}
                           \put(60,60){\line(1,0){40}}
                             \put(40,60){\line(1,0){20}}
                             \put(0,60){\line(1,0){20}}
                              \put(-40,60){\circle*{4}}
                              \put(-40,60){\line(1,0){40}}
                              \put(100,60){\line(1,0){40}}
                              \put(140,60){\line(0,-1){30}}
                         \put(140,60){\vector(1,0){40}}
                         \put(137,70){\makebox(0,0){$-1$}}
                         \put(97,70){\makebox(0,0){$-3$}}
                          \put(57,70){\makebox(0,0){$-2$}}
                           \put(-3,70){\makebox(0,0){$-2$}}
                            \put(-43,70){\makebox(0,0){$-2$}}
                        %  \put(100,50){\makebox(0,0){$w_{n}$}}
                        %  \put(-43,50){\makebox(0,0){$w_1$}}
                         % \put(-3,50){\makebox(0,0){$w_2$}}
                          \put(29,57){\makebox(0,0){$\dots$}}
                         % \put(60,50){\makebox(0,0){$w_{n-1}$}}
                         % \put(132,50){\makebox(0,0){$v_2$}}
                         %   \put(132,30){\makebox(0,0){$v_1$}}
                  \put(152,30){\makebox(0,0){$-2$}}
                         \end{picture}

\begin{picture}(300,100)(-100,10)
\put(-100,60){\makebox(0,0){$\partial F$:}}
\put(140,60){\circle*{4}}
                         \put(100,60){\circle*{4}}
                          \put(60,60){\circle*{4}}
                           \put(0,60){\circle*{4}}
                            \put(140,30){\circle*{4}}
                           \put(60,60){\line(1,0){40}}
                             \put(40,60){\line(1,0){20}}
                             \put(0,60){\line(1,0){20}}
                              \put(-40,60){\circle*{4}}
                              \put(-40,60){\line(1,0){40}}
                              \put(100,60){\line(1,0){40}}
                              \put(140,60){\line(0,-1){30}}
                         \put(140,60){\line(1,0){40}}
                         \put(137,70){\makebox(0,0){$-1$}}
                         \put(97,70){\makebox(0,0){$-3$}}
                          \put(57,70){\makebox(0,0){$-2$}}
                           \put(-3,70){\makebox(0,0){$-2$}}
                            \put(-43,70){\makebox(0,0){$-2$}}
                          \put(29,57){\makebox(0,0){$\dots$}}
                  \put(152,30){\makebox(0,0){$-2$}}
                  \put(220,60){\circle*{4}}
                        \put(180,60){\circle*{4}}
                         \put(250,75){\circle*{4}}
                         \put(250,45){\circle*{4}}
                         \put(220,60){\line(2,1){30}}
                         \put(180,60){\line(1,0){40}}
                         \put(220,60){\line(2,-1){30}}
                         \put(210,70){\makebox(0,0){$-1$}}
                          \put(262,75){\makebox(0,0){$-2$}}
                          \put(178,70){\makebox(0,0){$ -3k $}}
                           \put(262,45){\makebox(0,0){$-2$}}
                         \end{picture}

\noindent where the number of $(-2)$-vertices on the left is $n-1$.

\subsection{The family $B_{k} $ ($k\geq 1$)}
$ \Phi(s, t) = ( s, t^2, s^2t + t^{2k+1} )$ and
$ f= y(x^2+y^k)^2 - z^2 =0 $.

The graph of $D$ is

\begin{picture}(300,60)(-150,30)
\put(-100,60){\makebox(0,0){$ D $:}}
                         \put(100,60){\circle*{4}}
                          \put(60,60){\circle*{4}}
                           \put(0,60){\circle*{4}}
                           \put(60,60){\line(1,0){40}}
                             \put(40,60){\line(1,0){20}}
                             \put(0,60){\line(1,0){20}}
                              \put(-40,60){\circle*{4}}
                              \put(-40,60){\line(1,0){40}}
                         \put(100,60){\vector(2,1){30}}
                         \put(100,60){\vector(2,-1){30}}
                         \put(95,70){\makebox(0,0){$-1$}}
                          \put(57,70){\makebox(0,0){$-2$}}
                           \put(-3,70){\makebox(0,0){$-2$}}
                            \put(-43,70){\makebox(0,0){$-2$}}
                          \put(29,57){\makebox(0,0){$\dots$}}
                         % \put(60,50){\makebox(0,0){$w_{k-1}$}}
                         \end{picture}

\noindent  where the number of $( -2)$-vertices is $k-1$.

Note that the double point curve $ D $ does not depend on the parity of $k$,
however $ \Sigma $ has one component, when $ k$ is odd, and two components,
when $k$ is even. Thus the pairing $ \sigma $ changes the components of
$ D$ in the odd case, and $ \sigma $ is the identity in the even case.

\color{black}

 \vspace{2mm}

 \emph{Case 1}: $k=2n+1$.

 \vspace{2mm}

%\begin{picture}(300,50)(-150,45)
 % \put(-100,60){\makebox(0,0){$D:$}}
 %                        \put(100,60){\circle*{4}}
 %                         \put(60,60){\circle*{4}}
 %                          \put(0,60){\circle*{4}}
 %                          \put(60,60){\line(1,0){40}}
 %                            \put(40,60){\line(1,0){20}}
 %                            \put(0,60){\line(1,0){20}}
 %                             \put(-40,60){\circle*{4}}
 %                             \put(-40,60){\line(1,0){40}}
 %                        \put(100,60){\vector(2,1){30}}
 %                        \put(100,60){\vector(2,-1){30}}
 %                        \put(95,70){\makebox(0,0){$-1$}}
 %                         \put(57,70){\makebox(0,0){$-2$}}
 %                          \put(-3,70){\makebox(0,0){$-2$}}
 %                           \put(-43,70){\makebox(0,0){$-2$}}
 %                      %   \put(93,50){\makebox(0,0){$w_k$}}
 %                       %  \put(-43,50){\makebox(0,0){$w_1$}}
 %                       %  \put(-3,50){\makebox(0,0){$w_2$}}
 %                         \put(29,57){\makebox(0,0){$\dots$}}
 %                       %  \put(60,50){\makebox(0,0){$w_{k-1}$}}
 %                        \end{picture}

\begin{picture}(300,65)(-150,30)
\put(-100,60){\makebox(0,0){$\partial F$:}}
                         \put(100,60){\circle*{4}}
                          \put(60,60){\circle*{4}}
                           \put(0,60){\circle*{4}}
                           \put(140,60){\circle*{4}}
                           \put(120,60){\circle{40}}
                           \put(60,60){\line(1,0){40}}
                             \put(40,60){\line(1,0){20}}
                             \put(0,60){\line(1,0){20}}
                              \put(-40,60){\circle*{4}}
                              \put(-40,60){\line(1,0){40}}
                         \put(95,70){\makebox(0,0){$-1$}}
                          \put(57,70){\makebox(0,0){$-2$}}
                           \put(-3,70){\makebox(0,0){$-2$}}
                            \put(-43,70){\makebox(0,0){$-2$}}
                             \put(158,70){\makebox(0,0){$-4k-2$}}
                               \put(120,88){\makebox(0,0){$\ominus$}}
                             \put(29,57){\makebox(0,0){$\dots$}}
                         \end{picture}

\noindent
where the number of $( -2)$-vertices is $k-1$.

 \vspace{2mm}

 \emph{Case 2}: $k=2n$.

 \vspace{2mm}

%\begin{picture}(300,100)(-150,10)
%                         \put(100,60){\circle*{4}}
%                          \put(60,60){\circle*{4}}
%                           \put(0,60){\circle*{4}}
%                           \put(60,60){\line(1,0){40}}
%                             \put(40,60){\line(1,0){20}}
%                             \put(0,60){\line(1,0){20}}
%                              \put(-40,60){\circle*{4}}
%                              \put(-40,60){\line(1,0){40}}
%                         \put(100,60){\vector(1,1){30}}
%                         \put(100,60){\vector(1,-1){30}}
%                         \put(95,70){\makebox(0,0){$-1$}}
%                          \put(57,70){\makebox(0,0){$-2$}}
%                           \put(-3,70){\makebox(0,0){$-2$}}
%                            \put(-43,70){\makebox(0,0){$-2$}}
%                          \put(93,50){\makebox(0,0){$w_k$}}
%                          \put(-43,50){\makebox(0,0){$w_1$}}
%                          \put(-3,50){\makebox(0,0){$w_2$}}
%                          \put(29,57){\makebox(0,0){$\dots$}}
%                          \put(60,50){\makebox(0,0){$w_{k-1}$}}
%                         \end{picture}

\begin{picture}(300,140)(-150,-20)
\put(-100,60){\makebox(0,0){$\partial F$:}}
                         \put(100,60){\circle*{4}}
                          \put(60,60){\circle*{4}}
                           \put(0,60){\circle*{4}}
                           \put(60,60){\line(1,0){40}}
                             \put(40,60){\line(1,0){20}}
                             \put(0,60){\line(1,0){20}}
                              \put(-40,60){\circle*{4}}
                              \put(-40,60){\line(1,0){40}}
                         \put(100,60){\line(0,1){40}}
                         \put(100,60){\line(0,-1){40}}
                         \put(100,100){\circle*{4}}
                         \put(100,20){\circle*{4}}
                         \put(90,70){\makebox(0,0){$-1$}}
                          \put(57,70){\makebox(0,0){$-2$}}
                           \put(-3,70){\makebox(0,0){$-2$}}
                            \put(-43,70){\makebox(0,0){$-2$}}
                          \put(29,57){\makebox(0,0){$\dots$}}
                             \put(140,100){\circle*{4}}
                         \put(140,20){\circle*{4}}
                          \put(100,100){\line(1,0){40}}
                           \put(100,20){\line(1,0){40}}
                            \put(140,100){\line(2,1){30}}
                               \put(140,100){\line(2,-1){30}}
                                  \put(140,20){\line(2,1){30}}
                                     \put(140,20){\line(2,-1){30}}
                                      \put(170,115){\circle*{4}}
                                       \put(170,85){\circle*{4}}
                                        \put(170,35){\circle*{4}}
                                         \put(170,5){\circle*{4}}
                             \put(90,110){\makebox(0,0){$-2k-1$}}
                              \put(77,25){\makebox(0,0){$-2k-1$}}
                               \put(137,110){\makebox(0,0){$-1$}}
                              \put(137,30){\makebox(0,0){$-1$}}
                               \put(90,110){\makebox(0,0){$-2k-1$}}
                              \put(185, 115){\makebox(0,0){$-2$}}
                           \put(185, 85){\makebox(0,0){$-2$}}
                            \put(185, 35){\makebox(0,0){$-2$}}
                             \put(185,5){\makebox(0,0){$-2$}}
                         \end{picture}

\noindent
where the number of $(-2)$-vertexes on the left is again $k-1$.

\subsection{The family $C_k $ ($k\geq 1$)}
 $\Phi(s, t) = ( s, t^2, st^3 + s^k t )$ and
$ f = y(xy+x^k)^2 - z^2 =0$. The resolution of $ D$ is calculated in Example~\ref{ex:milnex}.

 \vspace{2mm}

 \emph{Case 1}: $k=2n+1$.

 \vspace{2mm}

\begin{picture}(300,50)(-200,40)
 \put(-180,60){\makebox(0,0){$D:$}}
                         \put(100,60){\circle*{4}}
                          \put(60,60){\circle*{4}}
                           \put(0,60){\circle*{4}}
                           \put(60,60){\line(1,0){40}}
                             \put(40,60){\line(1,0){20}}
                             \put(0,60){\line(1,0){20}}
                              \put(-40,60){\circle*{4}}
                              \put(-40,60){\line(1,0){40}}
                         \put(100,60){\vector(2,1){30}}
                       \put(-40,60){\vector(-1,0){40}}
                         \put(100,60){\vector(2,-1){30}}
                         \put(95,70){\makebox(0,0){$-1$}}
                          \put(57,70){\makebox(0,0){$-2$}}
                           \put(-3,70){\makebox(0,0){$-2$}}
                            \put(-43,70){\makebox(0,0){$-2$}}
                       %   \put(93,50){\makebox(0,0){$w_n$}}
                       %   \put(-43,50){\makebox(0,0){$w_1$}}
                      %    \put(-3,50){\makebox(0,0){$w_2$}}
                          \put(29,57){\makebox(0,0){$\dots$}}
                     %     \put(60,50){\makebox(0,0){$w_{n-1}$}}
                    %       \put(-78,50){\makebox(0,0){$a_1$}}
                   % \put(140,90){\makebox(0,0){$a_2$}}
                  %   \put(140,30){\makebox(0,0){$a_3$}}
                         \end{picture}

\begin{picture}(300,70)(-200,40)
\put(-180,60){\makebox(0,0){$\partial F$:}}
                         \put(100,60){\circle*{4}}
                          \put(60,60){\circle*{4}}
                           \put(0,60){\circle*{4}}
                           \put(140,60){\circle*{4}}
                           \put(120,60){\circle{40}}
                           \put(60,60){\line(1,0){40}}
                             \put(40,60){\line(1,0){20}}
                             \put(0,60){\line(1,0){20}}
                              \put(-40,60){\circle*{4}}
                              \put(-40,60){\line(1,0){40}}
                         \put(95,70){\makebox(0,0){$-1$}}
                          \put(57,70){\makebox(0,0){$-2$}}
                           \put(-3,70){\makebox(0,0){$-2$}}
                            \put(-43,70){\makebox(0,0){$-2$}}
                             \put(160,70){\makebox(0,0){$-3k+1$}}
                               \put(120,88){\makebox(0,0){$\ominus$}}
                             \put(29,57){\makebox(0,0){$\dots$}}
                              \put(-40,60){\line(-1,0){40}}
                                \put(-120,60){\circle*{4}}
                                \put(-120,60){\line(-2,1){30}}
                                \put(-120,60){\line(-2,-1){30}}
                                 \put(-150,75){\circle*{4}}
                                 \put(-150,45){\circle*{4}}
                                  \put(-80,60){\circle*{4}}
                                   \put(-80,60){\line(-1,0){40}}
                                    \put(-138,75){\makebox(0,0){$-2$}}
                                          \put(-138,45){\makebox(0,0){$-2$}}
                                           \put(-80,70){\makebox(0,0){$-4$}}
                                            \put(-115,70){\makebox(0,0){$-1$}}
                         \end{picture}

\noindent where the number of $ (-2)$-vertices in the middle is $n-1$.

 \vspace{2mm}

 \emph{Case 2}: $k=2n$.

 \vspace{2mm}

\begin{picture}(300,50)(-160,30)
 \put(-140,60){\makebox(0,0){$D:$}}
 \put(140,60){\circle*{4}}
                         \put(100,60){\circle*{4}}
                          \put(60,60){\circle*{4}}
                           \put(0,60){\circle*{4}}
                            \put(140,30){\circle*{4}}
                           \put(60,60){\line(1,0){40}}
                             \put(40,60){\line(1,0){20}}
                             \put(0,60){\line(1,0){20}}
                              \put(-40,60){\circle*{4}}
                              \put(-40,60){\line(1,0){40}}
                              \put(100,60){\line(1,0){40}}
                              \put(140,60){\line(0,-1){30}}
                         \put(140,60){\vector(1,0){40}}
                         \put(137,70){\makebox(0,0){$-1$}}
                         \put(97,70){\makebox(0,0){$-3$}}
                          \put(57,70){\makebox(0,0){$-2$}}
                           \put(-3,70){\makebox(0,0){$-2$}}
                            \put(-43,70){\makebox(0,0){$-2$}}
               %           \put(100,50){\makebox(0,0){$w_{n-1}$}}
              %            \put(-43,50){\makebox(0,0){$w_1$}}
             %             \put(-3,50){\makebox(0,0){$w_2$}}
                          \put(29,57){\makebox(0,0){$\dots$}}
            %              \put(60,50){\makebox(0,0){$w_{n-2}$}}
           %               \put(132,50){\makebox(0,0){$v_2$}}
          %                  \put(132,20){\makebox(0,0){$v_1$}}
                  \put(152,30){\makebox(0,0){$-2$}}
                  \put(-40,60){\vector(-1,0){30}}
         %         \put(-80,70){\makebox(0,0){$a_1$}}
        %          \put(180,70){\makebox(0,0){$a_2$}}
                         \end{picture}

\noindent and the graph of $\partial F$ is

\begin{picture}(300,70)(-160,25)
\put(140,60){\circle*{4}}
                         \put(100,60){\circle*{4}}
                          \put(60,60){\circle*{4}}
                           \put(0,60){\circle*{4}}
                            \put(140,30){\circle*{4}}
                           \put(60,60){\line(1,0){40}}
                             \put(40,60){\line(1,0){20}}
                             \put(0,60){\line(1,0){20}}
                              \put(-40,60){\circle*{4}}
                              \put(-40,60){\line(1,0){40}}
                              \put(100,60){\line(1,0){40}}
                              \put(140,60){\line(0,-1){30}}
                         \put(140,60){\line(1,0){40}}
                         \put(137,70){\makebox(0,0){$-1$}}
                         \put(97,70){\makebox(0,0){$-3$}}
                          \put(57,70){\makebox(0,0){$-2$}}
                           \put(-3,70){\makebox(0,0){$-2$}}
                            \put(-43,70){\makebox(0,0){$-2$}}
                          \put(29,57){\makebox(0,0){$\dots$}}
                  \put(150,30){\makebox(0,0){$-2$}}
                  \put(220,60){\circle*{4}}
                        \put(180,60){\circle*{4}}
                         \put(250,75){\circle*{4}}
                         \put(250,45){\circle*{4}}
                         \put(220,60){\line(2,1){30}}
                         \put(180,60){\line(1,0){40}}
                         \put(220,60){\line(2,-1){30}}
                         \put(213,70){\makebox(0,0){$-1$}}
                          \put(237,75){\makebox(0,0){$-2$}}
                          \put(176,70){\makebox(0,0){$ -3k +1 $}}
                           \put(237,45){\makebox(0,0){$-2$}}
                             \put(-120,60){\circle*{4}}
                                \put(-120,60){\line(-2,1){30}}
                                \put(-120,60){\line(-2,-1){30}}
                                 \put(-80,60){\line(1,0){40}}
                                 \put(-150,75){\circle*{4}}
                                 \put(-150,45){\circle*{4}}
                                  \put(-80,60){\circle*{4}}
                                   \put(-80,60){\line(-1,0){40}}
                                    \put(-138,75){\makebox(0,0){$-2$}}
                                          \put(-138,45){\makebox(0,0){$-2$}}
                                           \put(-80,70){\makebox(0,0){$-4$}}
                                            \put(-115,70){\makebox(0,0){$-1$}}
                         \end{picture}

\noindent
where the number of $(-2)$-vertices in the middle is $n-2$.

 \subsection{$F_{4} $}
 $ \Phi(s, t) = ( s, t^2, s^3t + t^5 )$ and
$ f= y(x^3+y^2)^2 - z^2 =0$.

 \begin{picture}(300,70)(-200,30)

                           \put(-150,60){\circle*{4}}
                         \put(-110,60){\circle*{4}}
                             \put(-150,60){\line(1,0){40}}
                              \put(-110,60){\line(0,-1){30}}
                              \put(-190,60){\circle*{4}}
                                \put(-110,30){\circle*{4}}
                              \put(-190,60){\line(1,0){40}}
                         \put(-110,60){\vector(1,0){40}}
                           \put(-153,70){\makebox(0,0){$-2$}}
                             \put(-113,70){\makebox(0,0){$-1$}}
          %                   \put(30,50){\makebox(0,0){$w_4$}}
                            \put(-193,70){\makebox(0,0){$-2$}}
            %              \put(-43,50){\makebox(0,0){$w_1$}}
           %               \put(-3,50){\makebox(0,0){$w_2$}}
                           \put(-123,30){\makebox(0,0){$-4$}}
         %                   \put(52,20){\makebox(0,0){$w_3$}}

%\put(-100,60){\makebox(0,0){$\partial F$:}}
                           \put(0,60){\circle*{4}}
                         \put(40,60){\circle*{4}}
                         \put(80,60){\circle*{4}}
                         \put(120,60){\circle*{4}}
                         \put(150,75){\circle*{4}}
                         \put(150,45){\circle*{4}}
                             \put(0,60){\line(1,0){40}}
                              \put(40,60){\line(0,-1){30}}
                              \put(-40,60){\circle*{4}}
                                \put(40,30){\circle*{4}}
                              \put(-40,60){\line(1,0){40}}
                         \put(40,60){\line(1,0){40}}
                          \put(80,60){\line(1,0){40}}
                          \put(120,60){\line(2,1){30}}
                          \put(120,60){\line(2,-1){30}}
                           \put(-3,70){\makebox(0,0){$-2$}}
                             \put(37,70){\makebox(0,0){$-1$}}
                            \put(-43,70){\makebox(0,0){$-2$}}
                             \put(77,70){\makebox(0,0){$-15$}}
                             \put(113,70){\makebox(0,0){$-1$}}
                             \put(136,75){\makebox(0,0){$-2$}}
                              \put(136,45){\makebox(0,0){$-2$}}
                           \put(27,30){\makebox(0,0){$-4$}}
                         \end{picture}

  \subsection{The family $H_{k} $, $k\geq 1$}\labelpar{ss:Hk}
  In this case
$ \Phi(s, t) = ( s, st + t^{3k-1}, t^3 ) $ and the equation of the image is calculated as the $0^{th}$ fitting ideal of $ \Phi_* ( \mathcal{O}_{ ( \C^2, 0 ) } ) $ in Example~\ref{ex:Hk}, $ f(x, y, z) = z^{3k-1} - y^3 + x^3 z + 3 xyz^k =0 $.

When $ k>1$, the local form of $ T^2 f $ along $ p ( \tau ) = ( \tau^{3k-2}, - \tau^{3k-1}, \tau^3 ) \in \Sigma^* $ is
\[
 \frac{\tau^{3k-1}}{12 k^2 - 12k + 4}
 \big[-\big((3k-3)- (3k-1) \sqrt{3} i \big) \tau x' +
 \big(3k +(3k-2) \sqrt{3} i\big) y' +
 (6 k^2 - 6k +2 ) \tau^{3k-4} z'
 \big] \cdot \]
\[ \cdot \big[-\big((3k-3)+ (3k-1) \sqrt{3} i\big) \tau x' +
 \big(3k-(3k-2) \sqrt{3} i\big) y' +
 (6 k^2 - 6k +2 ) \tau^{3k-4} z'
 \big]
 \mbox{,}
\]
where $ x'= x-\tau^{3k-2} $, $y'= y+ \tau^{3k-1} $, $ z'= z- \tau^3 $. With the choice
 $ H(x, y, z) = \partial_x f (x, y, z)+\partial_y f(x, y, z)+\partial_z f(x, y, z)$,
 one has $\mathfrak{v} = 3k-1$.

For $ k=1 $ one gets
 \[
 T^2(f) = \big[z' + \big( \frac{3}{2} + \frac{ \sqrt{3}}{2} i \big)
 \tau y' + \sqrt{3} i \tau^2 x' \big] \cdot
 \big[z' + \big( \frac{3}{2} - \frac{ \sqrt{3}}{2} i \big) \tau y' -
 \sqrt{3} i \tau^2 x' \big] \mbox{,}
\]
where $ x'= x-\tau $, $y'= y+ \tau^2 $, $ z'= z- \tau^3 $.
 In this case  $ \mathfrak{v} = 0 $.

In all cases  $ \lambda_1 + \lambda_2 + \mathfrak{v} = -3k-1 $.

\begin{picture}(300,60)(-150,40)
 \put(-140,60){\makebox(0,0){$D$:}}
                         \put(100,60){\circle*{4}}
                          \put(60,60){\circle*{4}}
                           \put(0,60){\circle*{4}}
                           \put(60,60){\line(1,0){40}}
                             \put(40,60){\line(1,0){20}}
                             \put(0,60){\line(1,0){20}}
                              \put(-40,60){\circle*{4}}
                              \put(-40,60){\line(1,0){40}}
                         \put(100,60){\vector(2,1){30}}
                         \put(100,60){\vector(2,-1){30}}
                         \put(95,70){\makebox(0,0){$-1$}}
                          \put(57,70){\makebox(0,0){$-2$}}
                           \put(-3,70){\makebox(0,0){$-2$}}
                            \put(-43,70){\makebox(0,0){$-2$}}
                          \put(29,57){\makebox(0,0){$\dots$}}
                         \end{picture}

\begin{picture}(300,70)(-150,40)
 \put(-140,60){\makebox(0,0){$\partial F$:}}
                         \put(100,60){\circle*{4}}
                          \put(60,60){\circle*{4}}
                           \put(0,60){\circle*{4}}
                           \put(140,60){\circle*{4}}
                           \put(120,60){\circle{40}}
                           \put(60,60){\line(1,0){40}}
                             \put(40,60){\line(1,0){20}}
                             \put(0,60){\line(1,0){20}}
                              \put(-40,60){\circle*{4}}
                              \put(-40,60){\line(1,0){40}}
                         \put(95,70){\makebox(0,0){$-1$}}
                          \put(57,70){\makebox(0,0){$-2$}}
                           \put(-3,70){\makebox(0,0){$-2$}}
                            \put(-43,70){\makebox(0,0){$-2$}}
                             \put(160,70){\makebox(0,0){$-9k +3$}}
                               \put(120,88){\makebox(0,0){$\ominus$}}
                             \put(29,57){\makebox(0,0){$\dots$}}
                         \end{picture}

 \noindent where the number of ($-2$)-vertexes is $ 3k-3$.

In the special case $k=1$ the germs
 $ H_1 $ and $S_1 $ are analytic $ \mathscr{A}$-equivalent.
 %$ \partial F $:
%
%                         \begin{picture}(300,100)(-100,10)
 %                        \put(100,60){\circle*{4}}
 %                          \put(140,60){\circle*{4}}
 %                          \put(120,60){\circle{40}}
 %                        \put(92,70){\makebox(0,0){$-1$}}
%                             \put(150,70){\makebox(0,0){$-6$}}
%                               \put(120,88){\makebox(0,0){$\ominus$}}
%                         \end{picture}

 \subsection{A corank $2$ map germ}\labelpar{ss:2cor} In this case
$ \Phi (s, t) = (s^2, t^2, s^3 + t^3+ st) $ and $ f(x, y, z) = x^6-2x^4y + x^2y^2 - 2x^3y^3-2xy^4 + y^6 - 8x^2y^2z - 2x^3z^2-2xyz^2-2y^3z^2+z^4 =0 $.

\begin{picture}(300,70)(-50,10)
\put(-40,60){\makebox(0,0){$ D $:}}
                         \put(100,60){\circle*{4}}
                         \put(100,60){\vector(1,0){50}}
                          \put(100,60){\vector(-1,0){50}}
                           \put(100,60){\vector(-1,-1){40}}
                           \put(100,60){\vector(1,-1){40}}
                           \put(100,60){\vector(0,-1){50}}
                         \put(95,70){\makebox(0,0){$-1$}}
                        %  \put(95,50){\makebox(0,0){$v$}}
                         \end{picture}

                         \begin{picture}(300, 250)(-70,-150)
                         \put(-50,100){\makebox(0,0){$\partial F$:}}
                         \put(100,60){\circle*{4}}
                         \put(150,60){\circle*{4}}
                         \put(100,10){\circle*{4}}
                         \put(50,60){\circle*{4}}
                         \put(60,20){\circle*{4}}
                         \put(140,20){\circle*{4}}
                         \put(200,60){\circle*{4}}
                       \put(100,-40){\circle*{4}}
                       \put(0,60){\circle*{4}}
                       \put(20,-20){\circle*{4}}
                       \put(180,-20){\circle*{4}}
                       \put(250,60){\circle*{4}}
                       \put(200,10){\circle*{4}}
                       \put(60,-80){\circle*{4}}
                       \put(140,-80){\circle*{4}}
                       \put(-50,60){\circle*{4}}
                       \put(0,10){\circle*{4}}
                        \put(-30,-20){\circle*{4}}
                         \put(20,-70){\circle*{4}}
                          \put(230,-20){\circle*{4}}
                           \put(180,-70){\circle*{4}}
                         \put(100,60){\line(1,0){50}}
                          \put(100,60){\line(-1,0){50}}
                           \put(100,60){\line(-1,-1){40}}
                           \put(100,60){\line(1,-1){40}}
                           \put(100,60){\line(0,-1){50}}
                          \put(150,60){\line(1,0){50}}
                          \put(100,10){\line(0,-1){50}}
                          \put(50,60){\line(-1,0){50}}
                          \put(60,20){\line(-1,-1){40}}
                          \put(140,20){\line(1,-1){40}}
                              \put(200,60){\line(1,0){50}}
                              \put(200,60){\line(0,-1){50}}
                                  \put(100,-40){\line(-1,-1){40}}
                                   \put(100,-40){\line(1,-1){40}}
                                    \put(0,60){\line(-1,0){50}}
                                    \put(0,60){\line(0,-1){50}}
                                       \put(20, -20){\line(0,-1){50}}
                                          \put(20, -20){\line(-1,0){50}}
                                          \put(180,-20){\line(1,0){50}}
                                          \put(180,-20){\line(0,-1){50}}
                                              \put(95,70){\makebox(0,0){$-1$}}
                                                \put(145,70){\makebox(0,0){$-5$}}
                                                \put(45,70){\makebox(0,0){$-5$}}
                                                 \put(55,30){\makebox(0,0){$-5$}}
                                                  \put(138,30){\makebox(0,0){$-5$}}
\put(107,20){\makebox(0,0){$-5$}}
\put(-5,70){\makebox(0,0){$-1$}}
 \put(195,70){\makebox(0,0){$-1$}}
\put(15,-10){\makebox(0,0){$-1$}}
 \put(179,-10){\makebox(0,0){$-1$}}
 \put(108,-30){\makebox(0,0){$-1$}}
 \put(-55,70){\makebox(0,0){$-2$}}
 \put(-10,20){\makebox(0,0){$-2$}}
\put(-35,-10){\makebox(0,0){$-2$}}
 \put(10,-60){\makebox(0,0){$-2$}}
 \put(55,-70){\makebox(0,0){$-2$}}
\put(138,-70){\makebox(0,0){$-2$}}
\put(245,70){\makebox(0,0){$-2$}}
\put(207,20){\makebox(0,0){$-2$}}
 \put(230,-10){\makebox(0,0){$-2$}}
\put(192,-60){\makebox(0,0){$-2$}}
                         \end{picture}

\backmatter

\def\cprime{$'$}

\newpage
\thispagestyle{empty}
\mbox{}

\newpage
\phantomsection
\thispagestyle{plain}
%\chapter*{}
\addcontentsline{toc}{chapter}{Summary}

{\bf {\Large Summary}}
\medskip

{\setstretch{1.0} 
In the thesis we study holomorphic germs from $ \C^2 $ to $ \C^3 $, which are singular only at the origin, with a special emphasis on the distinguished class of finitely determined germs. The results are published in two articles, joint with Andr\'{a}s N\'{e}methi, they are contained in Chapter 3 and 5.

Stability and finite determinacy of germs are discussed in Chapter 1. Here we review several theorems of Mather and Gaffney following the articles of Wall, Mond, Marar and Nu\~{n}o-Ballesteros. Furthermore, we review the analytic invariants of germs introduced by Mond, namely the number of Whitney umbrellas and triple values of a stabilization.

In Chapter 2 we introduce the Smale invariant, which classifies the immersions of spheres up to regular homotopy. We review formulas of Hughes--Melvin and Ekholm--Sz\H{u}cs, which express the Smale invariant of an immersion from the $3$-sphere to the $5$-sphere using singular Seifert surfaces. We introduce Ekholm's invariants for stable immersions. Here we mainly refer to the articles of Hughes, Sz\H{u}cs, Ekholm and Takase.

In Chapter 3 we identify the Smale invariant of immersions associated with holomorphic germs with an analytic invariant of the germs, namely with the number of Whitney umbrellas of a stabilization, answering a question of Mumford from 1961. We give representatives for all regular homotopy classes of immersions from the $3$-sphere to the $5$-sphere using the simple germs from Mond's list, answering Smale's question in this dimension. We give a new proof for Mond's theorem for corank--$1$ germs about the number of Whitney umbrellas via the newly defined complex Smale invariant. We identify the correct sign in the Hughes--Melvin and Ekholm--Sz\H{u}cs formulas. We express the Ekholm invariant of stable immersions associated with finitely determined germs in terms of the analytic invariants of Mond. We conclude that these analytic invariants are $ \mathcal{C}^{ \infty} $ invariants, and their combination expressing the Ekholm invariant is a topological invariant of the finitely determined germs.

The image of a finitely determined germ provides a non-isolated hypersurface singularity in $ \C^3 $. In Chapter 5 we determine the boundary of the Milnor fibre of these singularities. We summarize in Chapter 4 the main properties of the Milnor fibre and the resolution of isolated singularities, in which case the boundary of the Milnor fibre is diffeomorphic to the link and to the boundary of any resolution of the singularity as well. Here we use the de Jong--Pfister book and the articles of Milnor, Looijenga and N\'{e}methi.
In the non-isolated case the N\'{e}methi--Szil\'{a}rd book provides a general algorithm to determine the Milnor fibre boundary, however for concrete examples it is preferable to find a more direct procedure.

In Chapter 5 the Milnor fibre boundary is constructed as a surgery of the $ 3$-sphere along the double point curves of the immersion associated with the holomorphic germ. The gluing of the pieces along their torus boundaries is described by the newly defined vertical indices and homological invariants. For a certain class of germs the sum of these invariants is expressed with the number of Whitney umbrellas. In practice our algorithm provides a plumbing graph of the Milnor fibre boundary by modifying a good embedded resolution graph of the double point locus of the finitely determined germ. The examples we discuss are the simple germs from Mond's list and a corank--$2$ germ from an article of Marar and Nu\~{n}o-Ballesteros.

}

\newpage
\thispagestyle{empty}
\mbox{}

%\chapter*{}
\newpage
\phantomsection
\thispagestyle{plain}
\addcontentsline{toc}{chapter}{\"{O}sszefoglal\'{a}s (Summary in Hungarian)}

{\bf {\Large \"{O}sszefoglal\'{a}s (Summary in Hungarian)}}
\smallskip

{\setstretch{1.0}
Az \'{e}rtekez\'{e}sben olyan, $ \C^2 $-b\H{o}l $ \C^3$-be k\'{e}pez\H{o} holomorf
lek\'{e}pez\'{e}scs\'{i}r\'{a}kkal foglalkozunk, melyek csak az orig\'{o}ban
szingul\'{a}risak, illetve ezeknek egy speci\'{a}lis
oszt\'{a}ly\'{a}val, a v\'{e}gesen
determin\'{a}lt cs\'{i}r\'{a}kkal. Az eredm\'{e}nyeket k\'{e}t,
N\'{e}methi Andr\'{a}ssal
k\"{o}z\"{o}s cikkben publik\'{a}ltuk, ezeket a 3. \'{e}s 5.
fejezet tartalmazza.

Az 1. fejezetben cs\'{i}r\'{a}k stabilit\'{a}s\'{a}t \'{e}s v\'{e}ges
determin\'{a}lts\'{a}g\'{a}t t\'{a}rgyaljuk Mather \'{e}s Gaffney
eredm\'{e}nyei, Wall, Mond, Marar \'{e}s Nu\~{n}o-Ballesteros cikkei
alapj\'{a}n.
Ismertetj\"{u}k Mond analitikus invari\'{a}nsait, a
stabiliz\'{a}l\'{a}s sor\'{a}n megjelen\H{o} Whitney-eserny\H{o}k ill.
h\'{a}romszoros
pontok sz\'{a}m\'{a}t.

A 2. fejezetben bevezetj\"{u}k a g\"{o}mb\"{o}k immerzi\'{o}it
regul\'{a}ris homot\'{o}pia erej\'{e}ig
teljesen oszt\'{a}lyoz\'{o} Smale-invari\'{a}nst. Kimondjuk a $3$
dimenzi\'{o}s g\"{o}mb $5$ dimenzi\'{o}s t\'{e}rbe men\H{o}
immerzi\'{o}ir\'{o}l sz\'{o}l\'{o} Hughes--Melvin \'{e}s
Ekholm--Sz\H{u}cs formul\'{a}kat, melyek szingul\'{a}ris Seifert-fel\"{u}let
seg\'{i}ts\'{e}g\'{e}vel fejezik ki a Smale-invari\'{a}nst. Ismertetj\"{u}k a
stabil immerzi\'{o}k Ekholm-f\'{e}le invari\'{a}nsait. F\H{o}
forr\'{a}saink Hughes, Takase, Ekholm \'{e}s Sz\H{u}cs cikkei.

A 3. fejezetben holomorf cs\'{i}r\'{a}hoz tartoz\'{o} immerzi\'{o}
Smale-invari\'{a}ns\'{a}t azonos\'{i}tjuk a cs\'{i}ra egyik analitikus invari\'{a}ns\'{a}val, a  Whitney-eserny\H{o}k sz\'{a}m\'{a}val,
v\'{a}laszolva
ezzel Mumford 1961-es k\'{e}rd\'{e}s\'{e}re.
Mond list\'{a}j\'{a}r\'{o}l sz\'{a}rmaz\'{o} cs\'{i}r\'{a}khoz
tartoz\'{o} immerzi\'{o}kkal reprezent\'{a}ljuk a $3$ dimenzi\'{o}s
g\"{o}mbb\H{o}l $5$ dimenzi\'{o}s g\"{o}mbbe
k\'{e}pez\H{o} immerzi\'{o}k minden regul\'{a}ris
homot\'{o}piaoszt\'{a}ly\'{a}t, v\'{a}laszolva Smale
k\'{e}rd\'{e}s\'{e}re ebben a dimenzi\'{o}ban. 
%Kisz\'{a}moljuk az
% $A-D-E$
%h\'{a}nyadosszingularit\'{a}sok fed\H{o}cs\'{i}r\'{a}ihoz tartoz\'{o}
%immerzi\'{o}k
%Smale-invari\'{a}ns\'{a}t. 
%\'{u}j bizony\'{i}t\'{a}st adunk arra, hogy
%Hughes--Melvin-f\'{e}le
%nemstandard be\'{a}gyaz\'{a}sok nem \'{a}llnak el\H{o} holomorf
%cs\'{i}r\'{a}hoz tartoz\'{o}
%immerzi\'{o}k\'{e}nt. 
\'{U}j bizony\'{i}t\'{a}st adunk Mond
Whitney-eserny\H{o}k sz\'{a}m\'{a}r\'{o}l
sz\'{o}l\'{o} t\'{e}tel\'{e}re $1$--korang\'{u} cs\'{i}r\'{a}k
eset\'{e}n az \'{a}ltalunk bevezetett komplex Smale-invari\'{a}ns seg\'{i}ts\'{e}g\'{e}vel. Azonos\'{i}tjuk a Hughes--Melvin \'{e}s az
Ekholm--Sz\H{u}cs formul\'{a}k el\H{o}jel\'{e}t.
% konkr\'{e}t p\'{e}ld\'{a}kon kereszt\"{u}l. 
V\'{e}gesen
determin\'{a}lt cs\'{i}r\'{a}khoz tartoz\'{o}
stabil immerzi\'{o}kra az Ekholm-invari\'{a}nst kifejezz\"{u}k a
Mond-f\'{e}le analitikus invari\'{a}nsokkal. Megmutatjuk, hogy az
eml\'{i}tett analitikus invari\'{a}nsok $
\mathcal{C}^{\infty} $ invari\'{a}nsok is, \'{e}s az Ekholm-invari\'{a}nst
kifejez\H{o} kombin\'{a}ci\'{o}juk a v\'{e}gesen
determin\'{a}lt cs\'{i}r\'{a}k
tolopol\'{o}gikus invari\'{a}nsa.

V\'{e}gesen determin\'{a}lt cs\'{i}r\'{a}k k\'{e}pei speci\'{a}lis
nem-izol\'{a}lt
fel\"{u}letszingularit\'{a}sokat hat\'{a}roznak meg $ \C^3 $-ban. Az 5.
fejezetben
ezek
Milnor-fibrum\'{a}nak a perem\'{e}t
hat\'{a}rozzuk meg. Ehhez a 4. fejezetben
\"{o}sszefoglaljuk az izol\'{a}lt fel\"{u}letszingularit\'{a}sok
Milnor-fibrum\'{a}val \'{e}s rezol\'{u}ci\'{o}j\'{a}val
kapcsolatos eredm\'{e}nyeket a de Jong--Pfister k\"{o}nyv, Milnor
\'{e}s Looijenga cikkei \'{e}s
N\'{e}methi \"{o}sszefoglal\'{o} cikkei alapj\'{a}n. Izol\'{a}lt
fel\"{u}letszingularit\'{a}s
Milnor-fibrum\'{a}nak a pereme diffeomorf a szingularit\'{a}s
linkj\'{e}vel, mely
egyben a szingularit\'{a}s rezol\'{u}ci\'{o}j\'{a}nak a pereme is.
Nem-izol\'{a}lt esetben a
N\'{e}methi--Szil\'{a}rd k\"{o}nyvben tal\'{a}lhat\'{o} algoritmus a
Milnor-fibrum perem\'{e}nek meghat\'{a}roz\'{a}s\'{a}ra, konkr\'{e}t
csal\'{a}dokra azonban \'{e}rdemes k\"{o}zvetlenebb
elj\'{a}r\'{a}sokat keresni.

Az 5. fejezetben a $3$-dimenzi\'{o}s g\"{o}mbb\H{o}l kiindulva, a
holomorf cs\'{i}r\'{a}hoz tartoz\'{o} immerzi\'{o}
kett\H{o}spont-g\"{o}rb\'{e}i ment\'{e}n v\'{e}gzett m\H{u}t\'{e}ttel
szerkesztj\"{u}k meg a
Milnor-fibrum
pe\-re\-m\'{e}t. A t\'{o}ruszok menti
ragaszt\'{a}sokat az \'{a}ltalunk bevezetett vertik\'{a}lis
%vari\'{a}ci\'{o}s sz\'{a}mokkal 
indexekkel \'{e}s homol\'{o}gikus
invari\'{a}nsokkal \'{i}rjuk le,
melyek \"{o}sszeg\'{e}t kifejezz\"{u}nk a
Whitney-eserny\H{o}k sz\'{a}m\'{a}val a cs\'{i}r\'{a}k egy
oszt\'{a}ly\'{a}ra. Technikailag a v\'{e}gesen determin\'{a}lt
cs\'{i}ra kett\H{o}spont-g\"{o}rb\'{e}j\'{e}nek
be\'{a}gyazott rezol\'{u}ci\'{o}s gr\'{a}fj\'{a}t m\'{o}dos\'{i}tva
adjuk meg a cs\'{i}ra \'{a}ltal
meghat\'{a}rozott nem-izol\'{a}lt fel\"{u}letszingularit\'{a}s
Milnor-fibrum\'{a}nak
perem\'{e}t plumbing gr\'{a}f form\'{a}j\'{a}ban. A p\'{e}ld\'{a}k
Mond list\'{a}j\'{a}r\'{o}l \'{e}s egy Marar--Nu\~{n}o-Ballesteros
cikkb\H{o}l sz\'{a}rmaznak.

}

\end{document}